\documentclass[11pt]{article}
\usepackage[utf8]{inputenc} 
\usepackage{amstext, amsmath,latexsym,amsbsy,amssymb, amsfonts, amsthm}
\usepackage{hyperref}
\usepackage{amsfonts}
\usepackage{mathtools}
\usepackage{graphicx}
\usepackage{epsf}
\usepackage{graphics}
\usepackage{psfrag}
\usepackage{times}
\usepackage{epsfig}
\usepackage{subfig}
\usepackage[titletoc]{appendix}
\usepackage{multirow}
\usepackage{hhline}
\usepackage{tikz}
\usetikzlibrary{matrix}
\usepackage{caption}
\usepackage{titlesec}

\DeclarePairedDelimiter\floor{\lfloor}{\rfloor}

\long\def\comment#1{}
\renewcommand\vec[1]{\ensuremath\boldsymbol{#1}}

\newcommand{\Ocal}{\ensuremath{\mathcal{O}}}
\newcommand{\Ecal}{\ensuremath{\mathcal{E}}}
\newcommand{\Gcal}{\ensuremath{\mathcal{G}}}

\newcommand{\Scal}{\ensuremath{\mathcal{S}}}
\newcommand{\Pcal}{\ensuremath{\mathcal{P}}}
\newcommand{\myeta}{\ensuremath{\boldsymbol{\eta}}}
\newcommand{\rVMS}{\ensuremath{\operatorname{\rho}}}
\newcommand{\rgeneric}{\ensuremath{\operatorname{R}}}

\newcommand{\setgen}{\ensuremath{\operatorname{\Xi_{1}}}}

\newcommand{\extra}{\ensuremath{\bar{l}}}
\newcommand{\nonsingset}{\ensuremath{\mathcal{M}}}
\newcommand{\singset}{\ensuremath{\mathcal{L}}}
\newcommand{\singmat}{\ensuremath{\mathcal{M}}}

\newcommand{\lev}{\ensuremath{\ell}}

\makeatletter
\newcommand\mathcircled[1]{%
  \mathpalette\@mathcircled{#1}%
}
\newcommand\@mathcircled[2]{%
  \tikz[baseline=(math.base)] \node[draw,circle,inner sep=1pt] (math) {$\m@th#1#2$};%
}
\makeatother

\comment{
\setlength{\topmargin}{0 in}
\setlength{\textwidth}{5.5 in}
\setlength{\textheight}{8 in}
\setlength{\oddsidemargin}{0.5 in}
}

\theoremstyle{plain}

\newtheorem{theorem}{Theorem}

\numberwithin{theorem}{section}

\newtheorem{proposition}{Proposition}

\numberwithin{proposition}{section}

\newtheorem{lemma}{Lemma}

\numberwithin{lemma}{section}

\newtheorem{definition}{Definition}

\numberwithin{definition}{section}

\numberwithin{condition}{section}

\numberwithin{problem}{section}

\numberwithin{corollary}{section}

\numberwithin{assumption}{section}

\newtheorem{example}{Example}

\numberwithin{example}{section}

\numberwithin{conjecture}{section}

\theoremstyle{definition}

\numberwithin{remark}{section}

\renewenvironment{abstract}
 {\small
  \begin{center}
  \bfseries \abstractname\vspace{-.5em}\vspace{0pt}
  \end{center}
  \list{}{%
    \setlength{\leftmargin}{15mm}
    \setlength{\rightmargin}{\leftmargin}%
  }%
  \item\relax}
 {\endlist}

\begin{document}
\small \normalsize
\begin{center}

\textbf{\Large Singularity structures and impacts on parameter \\
\vspace{3mm}
estimation in finite mixtures of distributions}

\vspace{3 mm}

\large{Nhat Ho and XuanLong Nguyen}

\vspace{3 mm}

\large{Department of Statistics \\
University of Michigan}

\end{center}

\begin{abstract}
Singularities of a statistical model are the elements of the model's parameter space 
which make the corresponding Fisher information matrix degenerate.
These are the points for which estimation techniques
such as the maximum likelihood estimator and standard Bayesian procedures
do not admit the root-$n$ parametric rate of convergence. 
We propose a general framework for the identification of singularity structures 
of the parameter space of finite mixtures, and study the impacts of the 
singularity structures on minimax lower bounds and rates of convergence for the 
maximum likelihood estimator over a compact parameter space. 
Our study makes explicit the deep links between model singularities, parameter
estimation convergence rates and minimax lower bounds, and the algebraic geometry 
of the parameter space for mixtures of continuous distributions. 
The theory is applied to establish concrete convergence rates of parameter 
estimation for finite mixture of skew-normal distributions. This
rich and increasingly popular mixture model is shown to exhibit a 
remarkably complex range of asymptotic behaviors which have not
been hitherto reported in the literature.
\footnote{This research is supported in part by grants
NSF CAREER DMS-1351362 and NSF CNS-1409303 to XN.
The authors would like to acknowledge Michael I. Jordan, Antonio Lijoi, Judith Rousseau, Ya'acov Ritov, Martin J. Wainwright, Larry Wasserman, and Bin Yu for valuable 
discussions related to this work.


}
\end{abstract}

\section{Introduction}
In the standard asymptotic theory of parametric estimation,
a customary regularity assumption is the non-singularity of 
the Fisher information matrix defined by the statistical model 
(see, for example,~\cite{Lehmann-Casella} (pg. 124); or~\cite{vanderVaart-98}, Sec. 5.5). 
This condition leads to the cherished root-$n$ consistency, and
in many cases the asymptotic normality of parameter estimates. 
When the non-singularity condition 
 to hold, that is, when the true 
parameters represent a singular point in the statistical model,
very little is known about the asymptotic behavior of their estimates.

The singularity situation might have been brushed aside as 
idiosyncratic by some parametric statistical modelers in the past. As
complex and high-dimensional models are increasingly embraced by statisticians and
practitioners alike, singularities are no longer a rarity --- they start to
take a highly visible place in modern statistics and data science. For example, the many
zeros present in a high-dimensional linear regression problem represent a type of 
singularities of the underlying model, points corresponding to rank-deficient
Fisher information matrices~\cite{Hastie-Tibshirani-Wainwright}. 
In another example, the zero skewness in the family of skewed distributions 
represents a singular point~\cite{Chiogna-2005}.
In both examples, singularity points are quite easy to spot out --- it is
the impacts of their presence on improved parameter estimation procedures and the asymptotic 
properties such procedures entail that are nontrivial matters occupying the best efforts of 
many researchers in the past decade. The textbooks by~\cite{Buehlmann-vandeGeer,Hastie-Tibshirani-Wainwright}, for instance, address such
issues for high-dimensional regression problems, while the recent papers by~\cite{Ley-2010,Hallin-2012,Hallin-2014} investigate statistical inference in the skewed families
for distribution. 
By contrast, with finite mixture models --- a popular and rich class of modeling tools 
for density estimation and heterogeneity inference~\cite{Mclachlan-1988, Lindsay-1995, Mengersen-2005, Holzmann_2006} and a subject of this
paper, the singularity phenomenon is not quite well understood, to the best our knowledge,  
except for specific instances.

One of the simplest instances is the singularity of Fisher information matrix in an (over-fitted) 
finite mixture that includes a homogeneous distribution, a setting studied by~\cite{Kiefer-82}.~\cite{Chesher-Lee-86} analyzed a test of heterogeneity based on finite mixtures, addressing the challenge arising from
the aforementioned singularity. Recent works on the related topic 
include~\cite{Castelle-1997, Castelle-1999, Hanfeng_Chen-2001, Chen-2003, Drton_AOS_2009,Gassiat_2014, Shimotsu-2014}. Moreover,
model selection under singular models is given a book-length treatment in a seminal contribution by Watanabe~\cite{Watanabe-book}. Building upon this framework, Drton and Plummer~\cite{Drton_JRSSB_2017} studied finite mixture based model selection under singularity. Focusing on the estimation of \emph{parameters} of interest,
\cite{Rotnitzky-2000} investigated likelihood-based estimation methods in a somewhat 
general parametric modeling framework, subject to the constraint that the Fisher
information matrix is one rank deficient. For overfitted finite mixtures,~\cite{Chen1992} showed that under a condition of strong identifiability, 
there are estimators which achieve the generic convergence rate $n^{-1/4}$ 
for parameter estimation. 
Recent works also established generic behaviors of estimation under somewhat 
broader settings of over-fitted finite mixture models with both maximum likelihood 
estimation and Bayesian estimation~\cite{Rousseau-Mengersen-11,Nguyen-13,Ho-Nguyen-EJS-16}. 
Under sufficiently strong identifiability conditions for kernel densities, a sharp
local minimax lower bound of parameter estimation in over-fitted finite mixture models 
was recently obtained~\cite{Kahn-2015}.

The family of mixture models is far too rich to submit an uniform kind of 
behavior of parameter estimation, due to a weak identifiability phenomenon induced by
underlying singularities that are much more pervasive than previously thought.
In fact, it was shown recently that even classical models such as the location-scale
Gaussian mixtures, and the shape-rate Gamma mixtures, do not admit such a generic rate of 
convergence for an estimation method such as MLE or Bayesian estimation with a 
non-informative prior~\cite{Ho-Nguyen-Ann-16}. 
For instance, singularities arise in the finite mixtures of Gamma distributions, 
even when the number of mixing components is known --- this phenomenon results in an extremely 
slow convergence behavior for the model parameters lying in the vicinity of singular points,
even though such parameters are (perfectly) 
identifiable. Finite mixtures of Gaussian distributions,
though identifiable, exhibit both minimax lower bounds and maximum likelihood 
estimation rates that are directly linked to the solvability of a system of 
real polynomial equations, rates which deteriorate quickly with the increasing 
number of extra mixing components. The results obtained for such specific instances 
contain considerable insights about parameter estimation in finite
mixture models, but they only touch upon the surface of a general and complex phenomenon.
Indeed, as we shall see in this paper there is a rich spectrum
of asymptotic behavior in which regular (non-singular) mixtures, strongly identifiable mixtures, and 
weakly identifiable mixture models 
such as the ones studied by aforementioned works
occupy but a small spot.

\subsection{Main results}
The goal of this paper is to present a general and theoretical framework for analyzing parameter estimation
behavior in finite mixture models. We address directly the situations where the non-singularity 
condition of the Fisher information matrix may not hold.  Our approach is to take on a 
systematic investigation of the singularity structure of a compact and 
multidimensional parameter space of mixture models, to identify such singularities and then study the impacts of 
their presence on parameter estimation. There is a 
remarkable heterogeneity of the mixture model parameter space that we can shed some light on:
it will be shown that different parts of the parameter space may admit different convergence
rates, by several standard estimation methods. Parameters of different types
may possess different estimation rates, e.g., location vs scale of the same mixture component. 
Even parameters of the same type may carry distinct rates
of estimation, such as
shape parameters associated with different mixture components.

To obtain such a fine-grained picture of the parameter space,
several fundamental concepts will be introduced. In particular, the natural-valued
\emph{singularity level} will be useful in describing the convergence behavior of the (discrete) 
mixing measure that arises in the mixture model. 
Specifically, a mixture density of the form
$p_G(x) = \int f(x|\eta) \textrm{d}G(\eta)$, where $f$ denotes a kernel density,
corresponds to mixing measure $G$ on a suitable parameter space. If $G= \sum_{i=1}^{k} p_i\delta_{\eta_i}$, then it is often denoted that $p_G(x) = \sum_{i=1}^{k}p_if(x|\eta_i)$.
The singularity level for a mixing measure $G$
describes in a precise manner the variation of the mixture likelihood $p_G(x)$ with respect to
changes in mixing measure $G$. Now, Fisher information singularities 
simply correspond to points in the parameter space which identify a mixing measure 
whose singularity level is non-zero. 
Within the set of Fisher information singularities the parameter space can be partitioned into 
disjoint subsets determined by different singularity levels. 

Given an i.i.d.
$n$-sample from a (true) mixture density $p_{G_0}$, where $G_0$ admits a singularity
level $r$. This will imply, under some mild conditions on $f$, that a
standard estimation method such as maximum likelihood estimation and Bayesian estimation
with a non-informative prior
carries the rate of convergence $n^{-1/2(r+1)}$, which is also a minimax lower bound
(up to a logarithmic factor). Here, the convergence rate is
expressed in terms of a suitable Wasserstein metric on the space of mixing measures. 
Thus, singularity level 0 results in root-$n$ convergence rate for mixing measure estimation. 
Fisher singularity corresponds to singularity level 1 or greater than 1, resulting
in convergence rates $n^{-1/4}, n^{-1/6}, n^{-1/8}$ or so on.

Convergence in Wasserstein metric on mixing measures is easily translated into
convergence of the supporting atoms~\cite{Nguyen-13}. But each atom of
the mixing measure may be composed of different types (e.g., location, scale, shape).
To anticipate the heterogeneity of parameters of different types, we introduce
vector-valued \emph{singularity index}, which extends the notion of natural-valued singularity level described
earlier. The singularity index describes the variation of the mixture likelihood with 
respect to changes of individual parameter of each type. 
A singularity index $\kappa$ corresponds to singularity level $(\|\kappa\|_\infty-1)$ for
the mixing measure, but it tells us much more: the
convergence rate for estimating the $j$-th component of the atoms $\eta$ via MLE
or Bayesian method will be $n^{-1/2\kappa_{j}}$. One can in fact go further to capture
``complete heterogeneity'': via \emph{singularity matrix}, it can be shown that
each parameter may allow a possibly different convergence rate depending on the parameter's
values.
The complete picture of the distribution of singularity structure, however, can be extremely 
complex to derive. Remarkably, there are examples of finite mixtures for which the 
compact parameter space can be partitioned into disjoint subsets whose 
singularity level or elements of singularity index and singularity matrix range from 0 to 1 to 2,\ldots, up to infinity.
As a result, if we were to vary the true parameter values,
we would encounter a phenomenon akin to that of ``phase transition'' on 
the statistical efficiency of parameter estimation occuring 
within the same model class.

\subsection{Techniques}
A major component of our general framework is a procedure for identifying
subsets of points having the same singularity structure, via common singularity level and so on. 
It will be shown that these points 
are in fact a subset of a real affine variety. A real affine variety is a set of solutions 
to a system of real polynomial equations. The polynomial equations can be derived explicitly by the 
kernel density functions that define a given mixture distribution. The study of the solutions
of polynomial equations is a central subject of algebraic geometry~\cite{Stumfel-2002,Cox-etal}.
The connections between specific statistical models and algebraic geometry have received
a steady stream of contributions in the last two decades, including the analysis of
latent class analysis models~\cite{Geiger_2001}, factor analysis~\cite{Drton_2007}, algebraic statistical models~\cite{Drton_2007b}, discrete Markov random fields~\cite{Drton-etal-09}, finite mixtures of categorical data~\cite{Allman-2009}, Gaussian graphical models~\cite{Uhler_AOS_2012}, and the EM algorithms~\cite{Sturmfels_AOS_2015}. As mentioned earlier, singular mixture models were studied in the work of~\cite{Watanabe-book,Drton_JRSSB_2017}, who focused primarily on aspects of density estimation and model selection (e.g., estimation of the number of parameters), not the estimation of parameters per se.
By focusing on the statistical efficiency of parameter estimation
for finite mixtures of continuous distributions, we have found that the link to algebraic 
geometry is distilled from a new source of 
algebraic structure, in addition to the presence of mixing measures: 
it is traced to the partial differential equations satisfied by 
the mixture model's kernel density function. For Gaussian mixtures, it is the relation 
captured by Eq.~\eqref{key-normal} for the Gaussian kernel. 
The partial differential equations can be nonlinear, with coefficients given by rational
functions defined in terms of model parameters. It is this relation that is primarily
responsible for the complexity of the singularity structure.
A quintessential example of such a relation is given by 
Eq.~\eqref{eqn:overfittedskewnormaldistributionzero} for the skew-normal kernel densities. 

Starting from the aforementioned partial differential structure, we seek to represent likelihood
function $p_G(x)$ in terms of linearly independent functions, which lead to a representation we call \emph{minimal forms}.
These forms provide the basis for
studying the behavior of the likelihood as $G$ varies in a suitable neighborhood of mixing measures.  
It turns out that, as we move through increasingly sophisticated concepts of singularity structure (e.g., level, index), there are correspondingly structured notions of transportation
distance for the space of the underlying mixing measures. These distances, which generalize the Wasserstein metric
and behave asymptotically
as \emph{semi-polynomials} of the parameter perturbations, prove to be
the right object for linking up the information culled from a data sample (via its likelihood function)
and the algebraic structure of the parameter space of inferential interest.

Although our method for the analysis of singularity structure and the 
asymptotic theory for parameter estimation can be used to re-derive old and refine existing
results such as those of~\cite{Chen1992, Ho-Nguyen-Ann-16}, a substantial outcome is to
establish fresh new results on mixture models for which no asymptotic theory have hitherto been
achieved. This leads us to a story of finite mixtures of skew-normal distributions.
The skew-normal distribution was originally proposed in \cite{Azzalini-1986, Azzalini-1996, Azzalini-1999}.
The skew-normal generalizes normal (Gaussian) distribution, which is 
enhanced by the capability of handling asymmetric (skewed) data distributions. Due to 
its more realistic modeling capability for multi-modality 
and asymmetric components, skew-normal mixtures are increasingly adopted in recent years 
for model based inference of heterogeneity by many researchers~\cite{Tsung-2007,Genton-2008,Genton-2009,
Lin-2009,Sylvia-2009, Ghosal-13,Lee-13,Prates-2013, Canale-2015, Zeller-2015}. Due to its 
usefulness, a thorough understanding of the asymptotic behavior of parameter estimation 
for skew-normal mixtures is also of interest in its own right.

\comment{The singularity structure of the skew-normal mixtures is perhaps one of the more complex
among the parametric mixture models that we have typically encountered in the literature.
By comparison, strongly identifiable models admit the singularity level 0 (and singularity indices
of all ones) for all parameter values residing in a compact space, which results in
the $n^{-1/2}$ convergence rates of model parameters. 
Most mixture models whose kernel density function has only one type of parameter,
such as location mixtures or scale mixtures,
are in this category. Location-scale Gaussian mixtures are a step up in the complexity, in that 
all their parameter values carry the same singularity structure, which depends only on 
the number of extra mixing components. Yet this is not the picture of skew-normal mixtures, which
exhibits the kind of complete heterogeneity described earlier.
We will be able to identify subsets with singularity level (also, index/matrix) that vary all the way up to infinity. 
Even in the setting of mixtures with known number of mixing components, the singularity structure is remarkably complex. Thus, the results for skew-normal mixtures present an useful illustration for the full
power of the general theory for finite mixtures of continuous distributions.}

The source of complexity of skew-normal mixtures is the structure of the skew-normal 
kernel density.
The evidence for the latter was already made clear by 
\cite{Chiogna-2005, Ley-2010, Hallin-2012, Hallin-2014},
whose works provided a thorough picture of the singularities for the class of skew-normal 
densities, and their impacts on the non-standard rates of convergence of MLE. 
Not only can we recover the results of \cite{Hallin-2012,Hallin-2014} in terms of rates
of convergence, because they correspond to a trivial 
``mixture'' that has exactly one skew-normal component, an entirely new set of results are established for 
mixtures of two or more components. It is in this setting that new types of singularities arise 
out of the interactions between distinct skew-normal components. These interactions define 
the subset of singular points of a given structure that can be characterized by a system of real
 polynomial equations. This algebraic geometric characteration allows us to establish either the 
precise singularity structure or an upper bound for the mixture model's entire parameter space.
Due to space constraint, we shall describe only a small set of results pertaining the skew-normal
mixtures in this manuscript.

\subsection{Implications} Mixture models are one of the most popular and versatile tools in
modern statistics and data science. Despite numerous efforts, our understanding of
the parameter estimation behavior in mixture models remains woefully incomplete from both theoretical and computational standpoints.
As noted by in a recent textbook, "mixture models are riddle with difficulties such
as non-identifiability"~\cite{DasGupta-08}. Perhaps, as we shall show, the complexity lies not in
non-identifiability per se, but the varying levels of identifiability and the roles they play, which are captured precisely by the concepts of singularity level and index introduced
in this paper. We note that our theory of singularity structures also
carries important consequences on the computational complexity of parameter estimation procedures,
including both optimization and sampling based methods. Indeed,
the inhomogeneous nature of the singularity structures reveals a complex picture of
the likelihood function: regions in parameter space that carry low singularity levels/indices
may observe a relatively high curvature of the likelihood surface, while high singularity levels imply
a ``flatter'' likelihood surface along a certain subspace of the parameters.
Such a subspace 
is manifested by our construction of sequences of mixing measures that 
attest to the condition of singularities (e.g., $r$-singularity or $\kappa$-singularity) in general. Reposing upon this insight,~\cite{Raaz_Ho_Koulik_2018, Raaz_Ho_Koulik_2018_second, Wenlong_Ho_2019_sampling} recently established the slow convergence rates of expectation-maximization (EM) updates for approximating the MLE or developed polynomial mixing time MCMC algorithm for approximating the power posterior distribution for several specific settings of mixture models. It is of interest to further exploit the 
explicit knowledge of singularity structures obtained for a given mixture 
model class, so as to improve upon the 
computational efficiency of the optimization and sampling procedures
that operate on the model's parameter space.

The concepts of singularity level and index in the paper can be seen as complementary to the theory of singular models in general and applications to mixture models in particular~\cite{Watanabe-book, Aoyagi_2010, Drton_JRSSB_2017}. Indeed, the theory of~\cite{Watanabe-book} aims for a systemmatic approximation of functionals of a parametric density of interest, including Bayes generalization error or Kullback-Leibler distance. Although the rate of estimation for a parametric density function generally remains root-$n$ under Hellinger distance or a related functional, accounting for singularities yields a more accurate approximation for the marginal likelihood, which results in improved information criterion for model choice~\cite{Drton_JRSSB_2017}. On the other hand, we study how the model's likelihood and associated distance functionals vary with respect to model parameters. This study requires an elaborate excursion into the singularity structures of the parameter space, because convergence behavior of individual parameters is considerably more complex than that of a density function. Our singularity concepts provide an efficient way to characterize such structures, which directly yield non-regular parameter estimation rates in mixture models.

The plan for the remainder of our paper is as follows. 
Section \ref{Section:background} lays out the notations and relevant
concepts such as parameter spaces and the underlying geometries. 
Section \ref{Section:general_procedure_singularity} presents the general framework of 
analysis of singularity structure and the impact on convergence rates of parameter 
estimation for singular points of a given singularity structure. 
Section \ref{Section:overfitskew} illustrates the theory 
on the finite mixture of skew-normal distributions, by giving concrete minimax bounds and MLE 
convergence rates for this class of models for the first time. 
We conclude with a discussion in Section \ref{Section:discussion}. Additional concepts and results
and full proofs are given in the Appendices.

\comment{Here are the following updates (in red color) of the current draft:
\begin{itemize}
\item All of the minor issues are fixed (according to your comments).
\item All of the missing conditions in Theorem \ref{theorem:singularity_connection_liminf} and Theorem \ref{theorem:convergence_and_minimax} are filled.
\item Section \ref{Section:general_bound_omixtures_generic} and Section \ref{Section:nonlinear_system_study} are greatly revised.
\item Proof of Lemma \ref{lemma:reduced_linearly_independent} is more transparent.
\item The important missing proofs (greatly modified according to the presentation of this paper) of the results in the paper are put in the Appendix. Here are the detail
\begin{itemize}
\item Detail proof of Theorem \ref{theorem:conformant_setting} is included as it gives the direction how to deal with all the possible cases of $G_{0} \in \mathcal{S}_{1}$. This proof helps the readers to extend this argument to other full cases of $G_{0} \in \mathcal{S}_{2}, \mathcal{S}_{31}$, etc.
\item Detail proof of Theorem \ref{theorem:conformant_symmetry_setting} is included to illustrate the argument after \eqref{eqn:system_limits_symmetry_emixtures}.
\item Detail proofs of all propositions/ lemmas are included.
\end{itemize}
\end{itemize}}

\paragraph{Notation} We utilize several familiar notions of distance for mixture densities, with
respect to Lebesgue measure $\mu$. They are total variation
distance 
$V(p_{G},p_{G_{0}})={\displaystyle \dfrac{1}{2}\int {|p_{G}(x)-p_{G_{0}}(x)|}d\mu(x)}$ and Hellinger distance
$h^{2}(p_{G},p_{G_{0}})=\dfrac{1}{2} {\displaystyle \int {\left(\sqrt{p_{G}(x)}-\sqrt{p_{G_{0}}(x)}\right)^{2}}d\mu(x)}$. 

Additionally, for any $\kappa_{1}, \kappa_{2} \in \mathbb{R}^{d}$, we denote $\kappa_{1} \preceq \kappa_{2}$ iff all the components of $\kappa_{1}$ are less than or equal to the corresponding components of $\kappa_{2}$. Furthermore, $\kappa_{1} \prec \kappa_{2}$ iff $\kappa_{1} \preceq \kappa_{2}$ and $\kappa_{1} \neq \kappa_{2}$. Additionally, the expression "$\gtrsim$" will be used to denote the inequality up 
to a constant multiple where the value of the constant is fixed within our setting. We write 
$a_{n} \asymp b_{n}$ if both $a_{n} \gtrsim b_{n}$ and $a_{n} \lesssim b_{n}$ hold. Finally, for any $x \in \mathbb{R}$, $\floor{x}$ denotes the greatest integer that is less than or equal to $x$.

\section{Preliminaries} \label{Section:background}
\newcommand{\veceta}{\ensuremath{\vec{\eta}}}
\newcommand{\vecp}{\ensuremath{\vec{p}}}
\newcommand{\real}{\ensuremath{\mathbb{R}}}
\newcommand{\veca}{\ensuremath{\vec{a}}}
\newcommand{\vecb}{\ensuremath{\vec{b}}}

A finite mixture of continuous distributions admits density of the form
$p_G(x) = \int f(x|\eta) \textrm{d} G(\eta)$ with respect to Lebesgue measure
on an Euclidean space for $x$, where $f(x|\eta)$ denotes a probability density kernel, 
$\eta$ is a multi-dimensional parameter taking values in a subset of an Euclidean space 
$\Theta$, $G$ denotes a discrete \emph{mixing distribution} on $\Theta$. 
The number of support points of $G$ represents the number of mixing 
components in mixture model. Suppose that $G = \sum_{i=1}^{k}p_i \delta_{\eta_i}$, then 
$p_G(x) = \sum_{i=1}^{k}p_i f(x|\eta_i)$. 

\subsection{Parameter spaces and geometries}
There are different kinds of parameter space and geometries that they carry
which are relevant to the analysis of mixture models. We proceed to describe them in the following.

\paragraph{Natural parameter space}
The customarily defined parameter space of the $k$-mixture of distributions is 
that of the mixing component parameters $\eta_i$, and mixing probabilities $p_i$.
Throughout this paper, it is assumed that $\eta_i \in \Theta$, which is 
a compact subset of $\real^d$ for some $d\geq 1$, for $i=1,\ldots,k$.
The mixing probability vector $\vec{p}=(p_1,\ldots,p_k) \in \Delta^{k-1}$, the 
$(k-1)$-probability simplex. For the remainder of the paper, we also use $\Omega$ to denote the compact 
subset of the Euclidean space for parameters $(\vecp,\veceta)$.

\begin{example}
\textup{
The skew-normal density kernel on the real line has three parameters
$\eta = (\theta,\sigma,m) \in \real \times \real_+ \times \real$, namely,
the location, scale and skewness (shape) parameters. It is given by,
for $x\in \real$, 
\[f(x|\theta,\sigma,m) :=
\dfrac{2}{\sigma}f\left(\dfrac{x-\theta}{\sigma}\right) 
\Phi(m (x-\theta)/\sigma),\] 
where $f(x)$ is the standard normal density and 
${\displaystyle \Phi(x)=\int {f(t)1(t\leq x)}
\; \textrm{d}t}$. This generalizes the Gaussian density kernel, which corresponds to 
fixing $m=0$. 
The parameter space for the $k$-mixture of skew-normals is therefore a 
subset of an Euclidean space for the mixing probabilities $p_i$ and 
mixing component parameters $\eta_i = (\theta_i,v_i=\sigma_i^2,m_i) \in \real^3$. 
For each $i=1,\ldots, k$, $\theta_i, \sigma_i, m_{i}$ are restricted to reside
in compact subsets $\Theta_1 \subset \real, \Theta_2 \subset \real_{+},
\Theta_3 \subset \real$ respectively, i.e., 
$\Theta=\Theta_{1} \times \Theta_{2} \times \Theta_{3}$.
}
\end{example}

\paragraph{Semialgebraic sets} The singularity structure of the parameter space
submits to a different geometry. It will be described in terms of 
the zero sets (sets of solutions) of systems of real polynomial equations. 
The zero set of a system of real polynomial equations is called a (real)
affine variety \cite{Cox-etal}.
In fact, the sets which describe the singularity structure of finite mixtures 
are not affine varieties per se. We will see that they are the intersection 
between real affine varieties -- the real-valued solutions of a finite collection
of equations of the form $P(\vec{p},\vec{\eta}) = 0$, and the set of parameter values
satisfying $Q(\vec{p},\vec{\eta}) > 0$, for some real polynomials $P$ and $Q$. 
The intersection of these sets is also referred to as semialgebraic sets.

\begin{example}
\textup{Continuing on the example of skew-normal mixtures, 
we will see that first two types of singularities that arise from the mixture
of skew-normals are solutions of the following polynomial equations
\begin{itemize}
\item [(i)] Type A: $P_1(\veceta) = \prod_{j=1}^{k} m_j$.
\item [(ii)] Type B: $P_2(\veceta) = \prod_{1\leq i\neq j \leq k} 
\biggr \{ (\theta_i-\theta_j)^2 + \biggr [\sigma_i^2(1+m_j^2) - \sigma_j^2(1+m_i^2)
\biggr ]^2 \biggr \}$.
\end{itemize}
These are just two among many more polynomials and types of singularities
that we will be able to enumerate in the sequel. 
We quickly note that Type A refers to the one (and only) kind
of singularity intrinsic to the skew-normal kernel: $P_1 = 0$ if either one of 
the $m_j = 0$ --- one of the skew-normal components is actually normal (symmetric).
This type of singularity has received in-depth treatments by a number of 
authors~\cite{Chiogna-2005,Ley-2010,Hallin-2012,Hallin-2014}. 
One the other hand, Type B is intrinsic to a mixture model,
as it describes the relation of parameters of distinct  mixing components $i$ and $j$. 
}
\end{example}

\paragraph{Space of mixing measures and transportation distance}

As described in the Introduction, a goal of this work is to turn the knowledge 
about the nature of singularities of parameter space $\Omega$ into that of statistical
efficiency of parameter estimation procedures. 
For this purpose, the convergence of parameters in a mixture model is most naturally analyzed
in terms of the convergence in the space of mixing measures endowed by
transportation distance (Wasserstein distance) metrics~\cite{Nguyen-13}. 
This is because the role played by parameters $\vec{p},\veceta$ in the mixture model 
is via mixing measure $G$. It is mixing measure $G$ that determines the mixture
density $p_G$ according to which the data are drawn from.
Since the map $(\vec{p},\veceta) \mapsto G(\vec{p},\veceta) = G = \sum p_i\delta_{\veceta_i}$ 
is many-to-one, we shall treat 
a pair of parameter vectors $(\vec{p}, \vec{\eta}) = (p_1,\ldots,p_k; 
\eta_1,\ldots, \eta_k)$ and $(\vec{p'},\veceta') = (p'_1,\ldots, p'_{k'};
\eta'_1,\ldots, \eta'_{k'})$
to be equivalent if the corresponding mixing measures are equal,
$G(\vecp,\veceta) = G(\vecp',\veceta')$. For ease of notation we often omit the arguments
when the context is clear for $G=G(\vecp,\veceta)$ and
$G'=G(\vecp',\veceta')$.

For $r \geq 1$, the Wasserstein distance of order $r$ between $G$ 
and $G'$ takes the form (cf. \cite{Villani-03}),
\begin{eqnarray}
W_{r}(G,G') = \biggr (\inf \sum_{i,j} q_{ij} 
\|\eta_i - \eta'_j\|_{r}^{r} \biggr )^{1/r}, \nonumber
\end{eqnarray} 
where $\|\cdot\|_{r}$ is the $\ell_r$ norm endowed by the natural parameter
space,
the infimum is taken over all couplings $\vec{q}$ between $\vec{p}$ and $\vec{p}'$, 
i.e., $\vec{q}=(q_{ij})_{ij}\in [0,1]^{k \times k'}$ such that 
$\mathop {\sum }\limits_{i=1}^{k'}{q_{ij}}=p_{j}$ and 
$\mathop {\sum }\limits_{j=1}^{k}{q_{ij}}=p'_{i}$ for any $i=1,\ldots, k$ and 
$j=1,\ldots, k'$. 
For the specific example of skew-normal mixtures,
if $\eta = (\theta,v,m)$ and $\eta' = (\theta',v',m')$, then
$\|\eta - \eta'\|_{r}^{r}:= 
|\theta -\theta'|^r+|v - v'|^r+|m-m'|^r$.

Suppose that a sequence of probability measures $G_n = \sum_{i} p_i^{n}\delta_{\veceta_i^n}$ 
tending to $G_0$ under $W_r$ metric at a vanishing rate $\omega_n = o(1)$. 
If all $G_n$ have the same number of atoms $k_n = k_0$ as that of $G_0$, 
then the set of atoms of $G_n$ converge to the $k_0$ atoms of $G_0$, 
up to a permutation of the atoms, at the same rate $\omega_n$ under $\|\cdot\|_{r}$. 
If $G_n$ have the varying $k_n \in [k_0, k]$ number of atoms, where $k$ is a fixed 
upper bound, then a subsequence of $G_n$ can be constructed so that each atom of 
$G_0$ is a limit point of a certain subset of atoms of $G_n$ — 
the convergence to each such limit also happens at rate $\omega_n$. 
Some atoms of $G_n$ may have limit points that are not among $G_0$’s atoms 
— the total mass associated with those “redundant” atoms of $G_n$ must vanish at 
the generally faster rate $\omega_n^r$, since $r\geq 1$.

\subsection{Inhomogeneity and generalized transportation distance}
Although the standard Wasserstein metrics $W_r$ proved to be a convenient tool for the analysis of the space of mixing measures~\cite{Nguyen-13,Ho-Nguyen-EJS-16,Ho-Nguyen-Ann-16},
they are inadequate in describing the inhomogenenous behavior of individual parameters present in the model.
As we have seen in the previous paragraph, the convergence rate of mixing measures $G_n$
under Wasserstein metric $W_r$ induces the same rate of convergence for the atoms of $G_n$,
denoted by $\eta$. 
By way of example, suppose that $\eta$ is in fact made up of three parameters $\eta = (\theta, v, m)$, as
illustrated in the case of skew-normal mixtures,
this implies that same upper bound $\omega_n$ holds for all individual components $\theta$, $v$ and $m$
of the model parameter. 
Thus this fails to demonstrate the situations in which different parameter components may in fact
exhibit distinct convergence behavior. 
For instance, we will see that in normal and skew-normal mixtures, scale parameters may converge more
slowly than location parameters. 

To derive inhomogeneous 
convergence rates of different model parameters, we introduce a general version of 
the optimal transport distance, which can be formulated as follows.

\begin{definition} \label{def:generalized_Wasserstein}
For any $\kappa=(\kappa_{1},\ldots,\kappa_{d}) \in \mathbb{N}^{d}$, let
\begin{eqnarray}
d_{\kappa}(\theta_{1},\theta_{2}) : = \biggr(\sum \limits_{i=1}^{d}{|\theta_{1}^{(i)}-\theta_{2}^{(i)}|^{\kappa_{i}}}\biggr)^{1/\| \kappa \|_{\infty}} \nonumber
\end{eqnarray}
where $\theta_{i}=(\theta_{i}^{(1)},\ldots,\theta_{i}^{(d)}) \in \mathbb{R}^{d}$, as $i=1,2$, and
 $\|\kappa\|_{\infty}=\max \limits_{1 \leq i \leq d}{\left\{\kappa_{i}\right\}}$. 
The generalized transportation distance with respect to $\kappa$ is given by
\begin{eqnarray}
\widetilde{W}_{\kappa}(G,G') := \biggr (\inf \sum_{i,j} q_{ij} 
d_{\kappa}^{\| \kappa \|_{\infty}}(\eta_i ,\eta'_j) \biggr )^{1/\| \kappa \|_{\infty}} \nonumber
\end{eqnarray}
where the infimum is taken over all couplings $\vec{q}$ between $\vec{p}$ and $\vec{p'}$. 
\end{definition}

For instance, in skew-normal mixtures, if $\eta=(\theta,v,m)$, $\eta'=(\theta',v',m')$, and $\kappa=(2,1,1)$, then $d_{\kappa}(\eta,\eta')= (|\theta-\theta'|^{2}+|
v-v'|+|m-m'|)^{1/2}$. For general $\kappa$, $d_{\kappa}$ is a semi-metric --- it satisfies ``weak'' 
triangle inequality, i.e., it only satisfies the triangle inequality up to some positive constant less than 1, except when all $\kappa_{i}$ are identical. This implies that $\widetilde{W}_{\kappa}(G,G_{0})$ is a semi-metric 
that only satisfies weak triangle inequality (a proof of this fact is given by Lemma \ref{lemma:Wasserstein_semimetric} in Appendix F). An easy relation between the generalized 
transportation distance and the standard Wasserstein distance can be obtained.

\begin{lemma}\label{lemma:generalized_standard_Wasserstein}
For any $\kappa \in \mathbb{N}^{d}$ such that $\|\kappa\|_{\infty}=r \geq 1$, we have
\begin{eqnarray}
\widetilde{W}_{\kappa}(G,G') \gtrsim W_{r}(G,G'), \nonumber
\end{eqnarray}
where equality holds when $\kappa_{i}=r$ for all $1 \leq i \leq d$. 
\end{lemma}

The role of $\kappa$ in the definition of generalized transportation distance is to
capture the inhomogeneous behaviors of different parameters. In 
particular, assume that a sequence of probability measures $G_n \in \Ocal_{k}$ 
tending to $G_0$ under $\widetilde{W}_\kappa$ metric at a rate $\omega_n = o(1)$ 
for some $\kappa \in \mathbb{N}^{d}$. If all $G_n$ have the same number of atoms $k_n = k_0$ as that of $G_0$, 
then the set of atoms of $G_n$ converges to the $k_0$ atoms of $G_0$, 
up to a permutation of the atoms, at the same rate $\omega_n$ under $d_{\kappa}$ 
metric. Therefore, the $i$-th component of each atom of $G_{n}$ will converge to the $i
$-th component of corresponding atom of $G_{0}$ at rate $(\omega_{n})^{\|\kappa\|
_{\infty}/\kappa_{i}}$ for any $1 \leq i \leq d$. Similar implication also holds for the case 
$G_{n}$ having its number of components varying from $k_{0}$ to $k$. 

\paragraph{Complete inhomogeneity of parameter space} 
As we have discussed above, different types of parameter (location, scale, skewness) may exhibit 
inhomogeneous convergence behavior. It is a tribute to the complexity of finite mixture models that
it is not uncommon that different parameters of the same type within the same model 
may also exhibit such inhomogeneity as well. To properly handle such situations, one needs a notion of ``blocked 
generalized transportation distance''. Due to space constraint, the detailed formulation of this notion as well 
as convergence rate analysis of parameter estimation under this semi-metric are deferred to 
Appendix C. Our analysis with singularity structure of finite mixture models in the main 
text will focus on using Wasserstein distance or generalized transportation distance.

\subsection{Estimation settings: e- and o-mixtures}

The impact of singularities on parameter estimation behavior is dependent on
whether the mixture model is fitted with a known
number of mixing components, or if only an upper bound on the 
number of mixing components is known. The former model fitting
setting is called ``e-mixtures'' for short, while the latter ``o-mixtures''
(``e'' for exact-fitted and ``o'' for over-fitted).

Specifically, given an i.i.d. $n$-sample $X_{1},X_{2},\ldots,X_{n}$ 
according to the mixture density $p_{G_{0}}(x)=\int f(x|\eta) G_{0}(\textrm{d}\eta)$,
where $G_{0}=G(\vecp^0,\veceta^0) = \sum_{i=1}^{k_{0}} p_{i}^{0}\delta_{\eta_i^0}$ 
is unknown mixing measure with exactly $k_{0}$ distinct support points. 
We are interested in fitting a mixture of $k$ mixing components, where $k \geq k_0$, 
using the $n$-sample $X_1,\ldots, X_n$. 
In the e-mixture setting, 
$k=k_0$ is known, so an estimate $G_n$ (such as the maximum likelihood estimate) 
is selected from ambient space $\mathcal{E}_{k_{0}}$, the set of probability measures
$G=G(\vecp,\veceta)$ with exactly $k_{0}$ support points, where
$(\vecp,\veceta) \in \Omega$.  In the o-mixture setting,
$G_n$ is selected from ambient space $\Ocal_{k}$, the set of probability
measures $G=G(\vecp,\veceta)$ with at most $k$ support points,
where $(\vecp,\veceta) \in \Omega$. 

Throughout this paper, several conditions on the kernel density $f(x|\eta)$ are assumed
to hold. Firstly, the collection of kernel densities $f$ as $\eta$ varies is
linearly independent. It follows that the mixture model is
identifiable, i.e., $p_G(x) = p_{G_0}(x)$ for almost all $x$ entails $G=G_0$. 
Secondly, we say $f(x|\eta)$ satisfies a uniform Lipschitz
condition of order $r$, for some $r\geq 1$, if $f$ as a function of $\eta$
is differentiable up to order $r$, and that the partial derivatives with
respect to $\eta$, namely $\partial^{|\kappa|} f/\partial \eta^{\kappa}$,
for any $\kappa = (\kappa_1,\ldots,\kappa_d) \in \mathbb{N}^d$ such that 
$|\kappa|:= \kappa_1+\ldots+\kappa_d = r$ satisfy the following:
for any $\gamma \in \mathbb{R}^d$,
\[\sum_{|\kappa| = r} \biggr | \bigg (\frac{\partial^{|\kappa|} f}{\partial \eta^\kappa} 
(x|\eta_1) -
 \frac{\partial^{|\kappa|} f}{\partial \eta^\kappa}(x|\eta_2) \biggr ) \biggr | \gamma^\kappa
\leq C \|\eta_1-\eta_2\|_r^\delta \|\gamma\|_r^r\]
for some positive constants $\delta$ and $C$ independent of $x$ and $\eta_1,\eta_2 \in \Theta$.
It is simple to verify that most kernel densities used in mixture modeling,
including the skew-normal kernel, satisfy the uniform Lipschitz property over
compact domain $\Theta$, for any finite $r\geq 1$.

\section{Singularity structure in finite mixture models}
\label{Section:general_procedure_singularity}
\subsection{Beyond Fisher information} 
\label{subsection:fisher}
Given a mixture model
\[\biggr \{p_G(x) \biggr | G=G(\vecp,\veceta) = \sum_{i=1}^{k}p_i \delta_{\eta_i},
(\vecp,\veceta) \in \Omega \biggr \}\]
from some given finite $k$ and $f$ a given kernel density (e.g., skew-normal),
let $l_G$ denote the score vector ---
column vector made of the partial derivatives of the log-likelihood 
function $\log p_G(x)$ with respect to each of the model parameters
represented by $(\vecp,\veceta)$. The Fisher information
matrix is then given by $I(G) = \mathbb{E} (l_G l_G^\top)$, where the expectation is taken
with respect to $p_G$. We say that the parameter vector 
$(\vec{p},\veceta)$ (and the corresponding mixing measure $G$) is
a singular point in the parameter space (resp., ambient space of mixing measures),
if $I(G)$ is degenerate. Otherwise, $(\vecp,\veceta)$ (resp., $G$) is a non-singular point. 

According to the standard asymptotic theory, if
the true mixing measure $G_0$ is non-singular, \emph{and}
the number of mixing components $k_0=k$ (that is, we are in the e-mixture
setting), then basic estimators such as the MLE or Bayesian 
estimator yield the optimal root-$n$ rate of convergence. 
Although the standard theory remains silent when $I(G_0)$ is degenerate,
it is clear that the root-$n$ rate may no longer hold.
Moreover, there may be a richer range of behaviors for parameter 
estimation, requiring us to look into the deep structure of the
zeros of $I(G_0)$. This will be our story for both settings of e-mixtures
and o-mixtures. In fact, the (determinant of the) Fisher information matrix $I(G_0)$ is no longer 
sufficient in assessing parameter estimation behaviors.

\begin{example}
\label{example-skew-1}
\textup{
To illustrate in the context of skew-normal mixtures,
where parameter $\eta = (\theta, v, m)$,
observe that the mixture density 
structure allows the following characterization: $I(G)$ is
degenerate if and only if the collection of partial derivatives
\[\biggr \{\frac{\partial p_G(x)}{\partial p_j}, \frac{\partial p_G(x)}{\partial \eta_j}
\biggr \} := 
\biggr \{\frac{\partial p_{G}(x)}{\partial p_j}, 
\frac{\partial p_{G}(x)}{\partial \theta_j}, 
\frac{\partial p_{G}(x)}{\partial v_j}, 
\frac{\partial p_{G}(x)}{\partial m_j} \biggr | j=1,\ldots,k \biggr \}\]
as functions of $x$ are not linearly independent.
This is equivalent to, for some coefficients $(\alpha_{ij})$,
$i=1,\ldots,4$ and $j=1,\ldots,k$, not all of which are zeros, there holds
\begin{eqnarray}
\mathop {\sum }\limits_{j=1}^{k}{\alpha_{1j}f(x|\eta_{j})+
\alpha_{2j}\dfrac{\partial{f}}{\partial{\theta}}(x|\eta_{j})+\alpha_{3j}
\dfrac{\partial{f}}{\partial{v}}(x|\eta_{j}) +\alpha_{4j}
\dfrac{\partial{f}}{\partial{m}}(x|\eta_{j})}  =  0, \label{eqn:firstorderskew}
\end{eqnarray}
for almost all $x \in \mathbb{R}$. Lemma \ref{proposition-notskewnormal} later shows that the 
(Fisher information matrix's) singular points are the zeros of some polynomial equations. 
}
\end{example}

We have seen that for the e-mixtures $G$ is non-singular if
the collection of density kernel functions $f(x|\eta)$ and their first partial 
derivatives with respect to each model parameter are linearly independent.
This condition is also known as the first-order identifiability. 
For o-mixtures, the relevant notion is the second-order identifiability 
\cite{Chen1992,Nguyen-13,Ho-Nguyen-EJS-16}: the condition that
the collection of kernel density functions $f(x|\eta)$, their first and second
partial derivatives, are linearly independent. 
This condition fails to hold for skew-normal kernel densities, 
whose first and second partial derivatives are linked by the following 
nonlinear partial differential equations:
\begin{eqnarray}
\begin{cases}
\dfrac{\partial^{2}{f}}{\partial{\theta}^{2}}-2\dfrac{\partial{f}}{\partial{v}}+\dfrac{m^{3}+m}{v}\dfrac{\partial{f}}{\partial{m}}=0. \\
2m \dfrac{\partial{f}}{\partial{m}}+(m^{2}+1)\dfrac{\partial^{2}{f}}{\partial{m^{2}}}+2vm\dfrac{\partial^{2}{f}}{\partial{v}\partial{m}}=0.
\end{cases}
\label{eqn:overfittedskewnormaldistributionzero} 
\end{eqnarray}
The proof of these identities 
can be found in Lemma \ref{lemma:skewnormaldistribution} in Appendix F. 
Note that if $m=0$, the skew-normal kernel becomes normal kernel, which
admits a (simpler) linear relationship: 
\begin{equation}
\label{key-normal}
\frac{\partial^2 f}{\partial \theta^2} = 2\frac{\partial f}{\partial v}.
\end{equation}
This relation, also noted previously by \cite{Chen-2003, Shimotsu-2014}, 
plays a fundamental role in the analysis of finite mixtures of location-scale
normal distributions~\cite{Ho-Nguyen-Ann-16}. 
Unlike Gaussian kernel's, the nonlinear relations expressed by above PDEs for
skew-normal density kernel underscore the exceptional complexity of general mixture models --- 
the inhomogeneity of the parameter space. 
Analyzing this requires the development of a general theory that we now embark on.


\subsection{Likelihood in Wasserstein neighborhood}
\label{Section:likelihood_Wasserstein_neighborhood}
Instead of dwelling on the Fisher information matrix, we employ a direct approach
by studying the behavior of the likelihood function $p_G(x)$ as $G$ varies in a
Wasserstein neighborhood of a mixing measure $G_0 = \sum_{i=1}^{k_0} p_i^0 \delta_{\eta_i^0}$.

Fix $r\geq 1$, and consider an arbitrary sequence of $G_n \in \Ocal_{k}$, such that
$W_r(G_n,G_0) \rightarrow 0$. Let $k_n \leq k$ be the number of distinct
support points of $G_n$. 
There exists a subsequence of $G_n$ for which $k_n$ is constant in $n$ and each supporting atom 
$\eta_i^0$ as $i \in \left\{1,\ldots,k_{0} \right\}$ of $G_{0}$ is the limit point of exactly $s_{i}$ 
atoms of $G_n$. Additionally, there may be also a subset of atoms of 
$G_{n}$ whose limits are not among the atoms of $G_{0}$. Without loss of generality, 
we assume that there are $\extra \geq 0$ such limit points. By relabelling its atoms,
we can express $G_n$ as
\begin{eqnarray}
\label{eqn:representation_overfit}
G_n= \sum_{i=1}^{k_{0}+\extra} \sum_{j=1}^{s_{i}} p_{ij}^n \delta_{\eta_{ij}^{n}}, 
\end{eqnarray}
where $\eta_{ij}^{n} \to \eta_{i}^{0}$ for all $i=1,\ldots, k_{0}+\extra$, $j= 1,\ldots, s_{i}$. 
Additionally, $\sum \limits_{i=1}^{k_{0}+\extra}{s_{i}} = k_n$. 
Thus, from here on we replace the sequence of $G_n$ by this subsequence. 
To simplify the notation, $n$ will be dropped from the superscript when the context is clear.
In addition, we use the notation
$\Delta \eta_{ij} := \eta_{ij}-\eta_i^0$ for $i=1,\ldots, k_{0}+\extra, j=1,\ldots, s_{i}$.
Also, $p_{i\cdot} :=\sum \limits_{j=1}^{s_{i}}{p_{ij}}$,
and $\Delta p_{i\cdot} := p_{i\cdot}-p_{i}^{0}$, for $i=1,\ldots, k_0+\extra$ where $p_{i}^{0}=0$ as $k_0+1 \leq i \leq k_0+\extra$. 
For the setting of e-mixtures, the sequence of elements $G_n$ is restricted
to $\Ecal_{k_0} \subset \Ocal_k$, so $k_n = k_0$ for all $n$. It follows that
$s_{i}^{n}=1$ for all $i=1,\ldots,k_0$ and $\extra=0$,
 \comment{For o-mixtures, to simplify the presentation, we have omitted
the cases where some $G_n$ may have atoms that do not converge to the atoms of $G_0$
}
so the notation is simplified further: let $\Delta \eta_{i} :=  
\Delta \eta_{i1} = \eta_{i}-\eta_{i}^0, 
\Delta p_{i} :=\Delta p_{i\cdot} = p_i-p_i^0$ for all $i=1,\ldots,k_{0}$.

The following lemma relates Wasserstein metric to a semipolynomial
of degree $r$ (a semipolynomial is a polynomial of a collection of variables and/or
 the absolute value of some of the variables). 
\begin{lemma} \label{lemma:bound_overfit_Wasserstein}
Fix $r\geq 1$. For any element $G$ represented by 
Eq.~\eqref{eqn:representation_overfit}, define 
\[D_r(G,G_0) := 
\sum_{i=1}^{k_{0}+\extra} \sum_{j=1}^{s_{i}} p_{ij}
\|\Delta \eta_{ij}\|_r^r + \sum_{i=1}^{k_0+\extra} |\Delta p_{i\cdot}| .\]
Then $W_r^r(G,G_0) \asymp D_r(G,G_0)$, as $W_r(G,G_0) \downarrow 0$.
\end{lemma}
To investigate the behavior of likelihood function $p_{G}(x)$ as $G$ tends to
$G_0$ in Wasserstein distance $W_r$,  by representation \eqref{eqn:representation_overfit},
\begin{eqnarray}
p_G(x)-p_{G_0}(x) = \sum_{i=1}^{k_0+\extra} \sum_{j=1}^{s_i} p_{ij}(f(x|\eta_{ij}) - f(x|\eta_{i}^{0}))
+ \sum_{i=1}^{k_0+\extra} \Delta p_{i\cdot} f(x|\eta_{i}^{0}).
\end{eqnarray}
By Taylor expansion up to order $r$,
\begin{eqnarray}
p_{G}(x)-p_{G_{0}}(x) = \sum \limits_{i=1}^{k_{0}+\extra}{\sum \limits_{j=1}^{s_{i}}{p_{ij} 
\sum \limits_{|\kappa|=1}^{r}
\frac{(\Delta \eta_{ij})^{\kappa}}{\kappa!}
\frac{\partial^{|\kappa|}f}{\partial \eta^{\kappa}}
(x|\eta_i^0)}}
+ \sum_{i=1}^{k_0+\extra} \Delta p_{i\cdot} f(x|\eta_{i}^{0}) + R_{r}(x), 
\label{eqn:generalTaylorexpansion_overfit}
\end{eqnarray}
where $R_{r}(x)$ is the Taylor remainder. Moreover, it can be verified that
$\sup_{x}|R_{r}(x)/W_{r}^{r}(G,G_{0})| \to 0$ since $f$ is uniform Lipschitz up to order $r$.
\comment{
We have used the notation that for
$\kappa = (\kappa_1,\kappa_2,\kappa_3) \in \mathbb{N}_+^3$, 
$|\kappa| = \kappa_1+\kappa_2+\kappa_3$,
$(\Delta \eta)^{\kappa} = 
(\Delta \theta_{ij})^{\kappa_{1}}(\Delta v_{ij})^{\kappa_{2}}(\Delta m_{ij})^{\kappa_{3}}$,
$\kappa! = \kappa_{1}!\kappa_{2}!\kappa_{3}!$, and
$\frac{\partial^{|\kappa|}f}{\partial \eta^{\kappa}} = 
\dfrac{\partial^{|\kappa|}{f}}{\partial{\theta^{\kappa_{1}}}\partial{v^{\kappa_{2}}}
\partial{m^{\kappa_{3}}}}$.}
We arrive at the following formulae, which measures the speed of change of the likelihood
function as $G$ varies in the Wasserstein neighborhood of $G_0$:
\begin{equation}
\label{Eqn-ratio}
\frac{p_{G}(x)-p_{G_{0}}(x)}{W_r^r(G,G_0)}
=
\sum_{|\kappa|=1}^{r}
\sum_{i=1}^{k_{0}+\extra}
\sum \limits_{j=1}^{s_{i}} 
\biggr (\frac{p_{ij}(\Delta \eta_{ij})^{\kappa}/\kappa!}
{W_r^r(G,G_{0})}\biggr)
\frac{\partial^{|\kappa|}f}{\partial \eta^{\kappa}}
(x|\eta_i^0)
+ \sum_{i=1}^{k_0+\extra} \frac{\Delta p_{i\cdot}}{W_r^r(G,G_{0})}
f(x|\eta_{i}^{0}) + o(1).
\end{equation}
The right hand side of Eq.~\eqref{Eqn-ratio} is a linear combination of 
the partial derivatives of $f$ evaluated at $G_0$. It is crucial to note,
by Lemma~\ref{lemma:bound_overfit_Wasserstein}, each coefficient
of this linear representation is asymptotically equivalent to the ratio
of two semipolynomials.

Equation~\eqref{Eqn-ratio} highlights the distinct roles
of model parameters and the kernel density function in the formation of 
a mixture model's likelihood. The former appears
only in the coefficients, while the latter provides the partial derivatives
which appear as if they are basis functions for the linear combination.
We wrote ``as if'', because the partial derivatives of kernel $f$
may not be linearly independent functions
(recall the examples in Section \ref{subsection:fisher}). 
When a partial derivative of $f$ can be represented as a linear combination 
of other partial derivatives, it can be eliminated from the expression 
in the right hand side. This reduction process may be repeatedly applied 
until all partial derivatives that remain are linearly independent functions.
This motivates the following concept:

\begin{definition} \label{definition:rminimal}
The following representation is called $r$-\emph{minimal} form of
the mixture likelihood for a sequence of mixing measures $G$ tending to $G_0$
in $W_r$ metric:
\begin{eqnarray}
\frac{p_{G}(x)-p_{G_{0}}(x)}{W_{r}^{r}(G,G_{0})} = 
\sum_{l=1}^{T_r}
\biggr (\frac{\xi_{l}^{(r)}(G)}{W_r^r(G,G_{0})} \biggr ) H_{l}^{(r)}(x) + o(1), 
\label{eqn:generallinearindependencerepresentation}
\end{eqnarray}
which holds for almost all $x$, with the index $l$ ranging from 1 to a finite $T_r$,
if 
\begin{itemize}
\item [(1)] $H_{l}^{(r)}(x)$ for all $l$ are linearly independent functions of $x$, and
\item [(2)] coefficients $\xi_{l}^{(r)}(G)$ are polynomials of
the components of $\Delta \eta_{ij}$, and $\Delta p_{i\cdot}, p_{ij}$.
\end{itemize}
\end{definition}
It is sufficient, but not necessary, to select functions $H_l^{(r)}$ from
the collection of partial derivatives $\partial^{|\kappa|}f/\partial \eta^\kappa$
evaluated at particular atoms $\eta_i^0$ of $G_0$, where $|\kappa| \leq r$,
by adopting the elimination technique. The precise formulation of $\xi_{l}^{(r)}(G)$ and 
$H_{l}^{(r)}(x)$ will be determined explicitly by the specific $G_{0}$. 
The $r$-minimal form for each $G_0$ is not unique, but 
they play a fundamental role in the notion of singularity level of a mixing measure
relative to an ambient space that we now define.
\begin{definition}
\label{def-rsingular}
Fix $r\geq 1$ and let $\Gcal$ be a class of discrete probability measures which has 
a bounded number of support points in $\Theta$. 
We say $G_0$ is $r$-singular relative to $\Gcal$, if $G_0$ admits a $r$-minimal form 
given by Eq. \eqref{eqn:generallinearindependencerepresentation}, according to
which there exists a sequence of $G \in \Gcal$ tending to $G_0$ under $W_r$
such that
\[\xi_l^{(r)}(G)/W_r^r(G,G_0) \rightarrow 0 \;\textrm{for all}\; l=1,\ldots,T_r.\]
\end{definition}
We now verify that the $r$-singularity notion is well-defined, in that it does not
depend on any specific choice of the $r$-minimal form. This invariant property
is confirmed by part (a) of the following lemma. Part (b) establishes a crucial
monotonic property.
\begin{lemma} \label{lemma:minimal_form}
\label{invariance}
(a) (Invariance) The existence of the sequence of $G$ in the statement of Definition~\ref{def-rsingular}
holds for all $r$-minimal forms once it holds for at least one $r$-minimal form.

(b) (Monotonicity) If $G_0$ is $r$-singular for some $r>1$, then $G_0$ is $(r-1)$-singular.
\end{lemma}
The monotonicity of $r$-singularity naturally leads to the notion of
singularity level of a mixing measure $G_0$ (and the corresponding parameters) 
relative to an ambient space $\Gcal$.

\begin{definition} \label{definition:singularity_level_emixture} 
The singularity level of $G_0$ relative to a given class $\Gcal$,
denoted by $\lev(G_0|\Gcal)$, is
\begin{itemize}
\item [] 0, if $G_0$ is not $r$-singular for any $r\geq 1$; 
\item [] $\infty$, if $G_0$ is $r$-singular for all $r\geq 1$; 
\item [] otherwise, the largest natural number $r\geq 1$ for which $G_0$ is $r$-singular.
\end{itemize}
\end{definition}
The role of the ambient space $\Gcal$ is critical in determining
the singularity level of $G_0\in \Gcal$. 
Clearly, by definition if $G_0 \in \Gcal \subset \Gcal'$, 
$r$-singularity relative to $\Gcal$ entails $r$-singularity relative to $\Gcal'$. This
means $\lev(G_0|\Gcal) \leq \lev(G_0|\Gcal')$. Let us look at the following examples.
\begin{itemize}
\item Take $\Gcal = \Ecal_{k_0}$, i.e.,
the setting of e-mixtures.  If the collection of 
$\{\partial^\kappa f/\partial \eta^\kappa(x|\eta_j) | j=1,\ldots, k_0; |\kappa| \leq 1 \}$ 
evaluated at $G_0$ is linearly independent, then $G_0$ is not 1-singular relative to $\Ecal_{k_0}$. 
It follows that $\lev(G_0|\Gcal) = 0$.
\item Take $\Gcal= \Ocal_{k}$ for any $k> k_0$, i.e., the setting
of o-mixtures.  It is not difficult to check that $G_0$ is always 1-singular relative to 
$\Ocal_{k}$ for any $k>k_0$. Thus, $\lev(G_0|\Gcal) \geq 1$.
Now, if the collection of $\{\partial^\kappa f/\partial \eta^\kappa(x|\eta_j)
| j=1,\ldots, k_0; |\kappa| \leq 2 \}$ evaluated at $G_0$ is linearly independent,
then it can be observed that $G_0$ is not 2-singular relative to $\Ocal_{k}$.
Thus, $\lev(G_0|\Gcal) = 1$. 
\end{itemize}
The conditions described in the two examples above are in fact referred to as strong identifiability 
conditions studied by~\cite{Chen1992,Nguyen-13,Ho-Nguyen-EJS-16}.
The notion of singularity level generalizes such identifiability conditions, by
allowing us to consider situations where such conditions fail to hold. 
This is the situation where $\lev(G_0|\Gcal) = 2, 3,\ldots,\infty$.
The significance of this concept can be appreciated by the following theorem.

\begin{theorem} \label{theorem:singularity_connection_liminf} 
Let $\Gcal$ be a class of probability measures on $\Theta$ that have a bounded number of
support points, and fix $G_0 \in \Gcal$.
Suppose that $\lev(G_0|\Gcal) = r$, for some $0 \leq r \leq \infty$. 
\begin{itemize}
\item [(i)] If $r<\infty$, then $\inf \limits_{G \in \mathcal{G}} \dfrac{\|p_{G}-p_{G_0}\|_\infty}{W_{s}^{s}
(G,G_0)} > 0$ for any $s\geq r+1$.
\item [(ii)] If $r<\infty$, then $\inf \limits_{G \in \mathcal{G}} \dfrac{V(p_{G},p_{G_0})}{W_{s}^{s}
(G,G_0)} > 0$ for any $s\geq r+1$.
\comment{\item [(i)] If $r<\infty$, then $\liminf \limits_{G \in \mathcal{G}: W_s(G,G_0)\rightarrow 0} \dfrac{\|p_{G}-p_{G_0}\|_\infty}{W_{s}^{s}
(G,G_0)} > 0$ for any $s\geq r+1$.
\item [(ii)] If $r<\infty$, then $\liminf \limits_{G \in \mathcal{G}: W_s(G,G_0)\rightarrow 0} \dfrac{V(p_{G},p_{G_0})}{W_{s}^{s}
(G,G_0)} > 0$ for any $s\geq r+1$.}
\end{itemize}
\end{theorem}
\label{theorem:tight}

The following theorem establishes that the bounds obtained above are tight under some condition.
\begin{theorem}
Consider the same setting as that of Theorem~\ref{theorem:singularity_connection_liminf}.
\begin{itemize}
\item [(i)] If $1 \leq r<\infty$ and in addition,
\begin{itemize}
\item [(a)] $f$ is $(r+1)$-order differentiable with respect to $\eta$ and for 
some constant $c_{0}>0$,
\begin{eqnarray}
\label{cond-integrated}
{\displaystyle \mathop {\sup }\limits_{\|\eta -\eta'\| \leq c_0}
{\int \limits_{x \in \mathcal{X}}{\left(\dfrac{\partial^{s}{f}}
{\partial{\eta^{\alpha}}}(x|\eta)\right)^{2}/f(x|\eta')}dx}} <\infty 
\end{eqnarray}
for any index $\alpha$ such that $|\alpha|=s$ and $s \in \left\{1,\ldots,r\right\}$.
\item [(b)] There is a sequence of $G \in \mathcal{G}$ tending
to $G_0$ in Wasserstein metric $W_r$ and the coefficients of the $r$-minimal
form $\xi_{l}^{(r)}(G)$ satisfy $\xi_{l}^{(r)}(G)/W_{1}^{s}(G,G_{0}) \rightarrow 0$ for all real number $s \in [1,r+1)$ and $l=1,\ldots, T_{r}$. Additionally, all the masses $p_{ij}$ in the representation \eqref{eqn:representation_overfit} of $G$ are bounded away from 0. 
\end{itemize}
Then, for any real number $s \in [1,r+1)$, 
\[\liminf \limits_{G \in \mathcal{G}: W_1(G,G_0)\rightarrow 0} \dfrac{h(p_{G},p_{G_0})}{W_{1}^{s}
(G,G_0)} = \liminf \limits_{G \in \mathcal{G}: W_s(G,G_0)\rightarrow 0} \dfrac{V(p_{G},p_{G_0})}{W_{s}^{s}
(G,G_0)} = 0.\]
\item[(ii)] If $\lev(G_0|\Gcal) = \infty$ and the conditions (a), (b) in part (ii) hold for any $l \in 
\mathbb{N}$ (here, the parameter $r$ in these conditions is replaced by $l$), then the 
conclusion of part (i) holds for any $s \geq 1$.
\end{itemize}
\end{theorem}

We make a few remarks.
\begin{itemize}
\item Part (i) and part (ii) of Theorem~\ref{theorem:singularity_connection_liminf} 
show how the singularity level of $G_0$ relative to ambient space
$\Gcal$ may be used to translate the convergence of mixture densities (under the sup-norm and the
total variation distance) into the convergence of mixing measures under a suitable Wasserstein metric $W_s$.
Part (i) of Theorem~\ref{theorem:tight} shows a sufficient condition under which the power $r+1$ in the bounds 
from Theorem~\ref{theorem:singularity_connection_liminf}  cannot be improved.
\item In part (i) of Theorem~\ref{theorem:tight} the condition regarding the integrand of the partial derivative of 
$f$ (cf. Eq. \eqref{cond-integrated} can be easily checked to be satisfied by many kernels, such as Gaussian kernel, Gamma kernel, 
Student t's kernel, and skew-normal kernel.
\item Condition (b) regarding the sequence of 
$G$ appears somewhat opaque in general, but it will be illustrated
in specific examples for skew-normal mixtures in the sequel. It is sufficient,
but not necessary, for verifying the $r$-singularity of $G_0$ to 
construct the sequence of $G$ so that 
$\xi_{l}^{(r)}(G)=0$ for all $1 \leq l \leq T_{r}$, provided such a sequence
exists. This requires finding an 
appropriate parameterization of a sequence of $G$ tending toward $G_0$ that
satisfy a number of polynomial equations defined in terms of the parameter
perturbations.
\end{itemize}

We are ready to state the impact of the singularity level of mixing measure $G_0$
relative to an ambient space $\Gcal$ on the rate of convergence for an estimate
of $G_0$, where $\Gcal = \Ecal_{k_0}$ in e-mixtures, and $\Gcal = \Ocal_{k}$ in
o-mixtures. Let $\Gcal$ be structured into a sieve of subsets defined by 
the maximum singularity level relative to $\Gcal$.

\begin{eqnarray*}
\Gcal = \bigcup_{r=1}^{\infty} \Gcal_r, \;\textrm{where}\;\;\;
\Gcal_{r} & := & \biggr \{G \in \Gcal \biggr | \lev(G|\Gcal) \leq r \biggr \}, \; r = 0,1,\ldots,\infty.\\
\end{eqnarray*}
\comment{
To obtain minimax lower bound, we assume that for a given $r\geq 1$, all elements $G_0$
$\mathcal{G}_{r}' \neq \emptyset$, the following two conditions hold
\begin{itemize}
\item $f$ is $(r+1)$-order differentiable with respect to $\eta$ and for some sufficiently small $c_{0}>0$,
\begin{eqnarray}
{\displaystyle \mathop {\sup }\limits_{\|\eta -\eta'\| \leq c_0}
{\int \limits_{x \in \mathcal{X}}{\left(\dfrac{\partial^{r+1}{f}}{\partial{\eta^{\alpha}}}(x|\eta)\right)^{2}/f(x|\eta^{'})}dx}} <\infty \nonumber
\end{eqnarray}
for any $|\alpha|=r+1$.
\item We can find some sequence $G \in \bigcup \limits_{i=0}^{r}{\mathcal{G}_{i}'}$ such that $W_{r}(G,G_{0}) \to 0$ and $\xi_{l}^{(r)}(G)=0$ for any $1 \leq l \leq L_{r}$.
\end{itemize}
}

The first part of the following theorem gives a minimax lower bound for the estimation of 
the mixing measure $G_0$, given that the singularity level of $G_0$ is known up to a 
constant $r \geq 1$. The second part gives a quick result on the convergence rate
of a point estimate such as the MLE. 
\comment{For concreteness, we state the results for skew-normal
mixtures, but this kind of result holds for most kernel densities
currently used in the literature (e.g., location-scale Gaussian, Student's t, etc).}

\begin{theorem} \label{proposition:convergence_and_minimax} \comment{Let $f$
be the skew-normal kernel density. $p_G$ denotes skew-normal mixture
density whose parameters lie in compact set $(\vecp,\veceta) \in \Omega$.}
\begin{itemize}
\item [(a)] Fix $r\geq 1$. Assume that for any $G_{0} \in \Gcal_{r}$, the conclusion of part (i) of Theorem \ref{theorem:tight} holds for $\Gcal_{r}$
(i.e., $\Gcal$ is replaced by $\Gcal_{r}$ in that theorem). 
Then, for any real number $s \in [1, r+1)$ there holds
\begin{eqnarray}
\mathop {\inf }\limits_{\widehat{G}_{n} \in \Gcal_r}
\sup_{G_0 \in \Gcal_r} E_{p_{G_{0}}} W_{s} (\widehat{G}_{n},G_{0})
\gtrsim n^{-1/2s}. \nonumber
\end{eqnarray}
Here, the infimum is taken over all sequences of estimates $\widehat{G}_{n} \in \Gcal_{r}$ 
and $E_{p_{G_{0}}}$ denotes the expectation taken with respect to product measure with mixture density $p_{G_{0}}^{n}$.

\item [(b)] Let $G_0 \in \Gcal_{r}$ for some fixed $r\geq 1$. Let $\widehat{G}_n 
\in \Gcal_{r}$ be a point estimate for $G_0$, which is obtained from an $n$-sample of i.i.d. observations
drawn from $p_{G_0}$. As long as $h(p_{\widehat{G}_n},p_{G_0}) = O_P(n^{-1/2})$, we have
\begin{eqnarray}
 W_{r+1}(\widehat{G}_n, G_0) = O_P(n^{-1/2(r+1)}). \nonumber
\end{eqnarray}
\end{itemize}
\end{theorem}
\begin{proof}
Part (a) of this theorem is a consequence of the conclusion of 
Theorem \ref{theorem:tight}, part (iii).
The proof of this fact is rather standard, and similar to that of 
Theorem 1.1. of~\cite{Ho-Nguyen-Ann-16}
and is omitted. Part (b) follows immediately from 
part (ii) of Theorem \ref{theorem:singularity_connection_liminf},
as we have $h(p_{\widehat{G}_{n}},p_{G_{0}}) \geq 
V(p_{\widehat{G}_{n}},p_{G_{0}}) \gtrsim W_{r+1}^{r+1}(\widehat{G}_{n},G_{0})$. 
\end{proof}

We conclude this section with some comments. It is well-known that many density
estimation methods, such as MLE and Bayesian estimation applied to a compact
parameter space for parametric mixture models,
guarantee a root-$n$ rate (up to a logarithmic term) of convergence under 
Hellinger distance metric on the density functions (cf. 
\cite{Vandegeer-2000,Ghosal-2001,DasGupta-08}). It follows that, as far as we are concerned,
the remaining work in establishing the convergence behavior of mixing measure estimation (as opposed to density estimation) lies in the calculation of the singularity levels, i.e., 
the identification of sets $\Gcal_r$. 
For skew-normal mixtures, such calculations will be carried out in Section \ref{Section:overfitskew}
(with further details given in Appendices E and F). For the settings of $G_{0}$ where we are able to 
obtain the exact singularity levels, we can also construct the sequence of $G$ required by
part (i) of Theorem \ref{theorem:tight}. 
Whenever the exact singularity level is obtained, we automatically obtain
a (local) minimax lower bound and a matching upper bound
for MLE convergence rate under a Wasserstein distance metric, thanks to the above
theorem.
In some cases, however, the singularity level of $G_0$ may be not determined exactly, 
but only an upper bound given. In such cases, only an upper bound to the 
convergence rate of the MLE can be obtained, while minimax lower bounds may be unknown.

\comment{As a result, there is relatively little left to prove as far as the present theorem 
is concerned. We will only give a basic sketch of the remaining steps, 
while omitting the details. For part (a), one needs to verify that both the 
skew-normal kernel density $f$ and any $G_0 \in \Gcal_r$ satisfy the 
assumption of part (iii) of Theorem ~\ref{theorem:singularity_connection_liminf}. 
Then, the minimax lower bound is obtained in exactly the same way as the proof of part (a) of 
Theorem 1.1. of \cite{Ho-Nguyen-Ann-16}. For part (b), we can first show that
\[h(p_{\widehat{G}_n},p_{G_0}) = O_P((\log n/n)^{1/2}).\]
The proof of this is similar to part (b) of Theorem 1.1. of \cite{Ho-Nguyen-Ann-16} for 
location-scale Gaussian mixtures. Now, invoking part (ii) of Theorem 
\ref{theorem:singularity_connection_liminf}, we
immediately obtain that $W_r(\widehat{G}_n,G_0) = O_P((\log n/n)^{1/2(r+1)})$.}


\subsection{Inhomogeneity of parameter space}
\label{Section:inhomogeneity_singularity}

\comment{
Thus far, our previous studies of singularity level demonstrate the best possible 
convergence rate $n^{-1/2(r+1)}$ of any point estimate $\widehat{G}_{n}$ of 
true mixing measure $G_{0}$ under $W_{r}$ distance as long as $G_{0}$ is $r$-singular relative to $\Gcal$ 
where $\Gcal$ is a class of probability measures on $\Theta$ that have a bounded 
number of support points. At a deeper level, it implies that components of each atom 
of $\widehat{G}_{n}$ will converge to components of corresponding atom 
(possibly outside) of $G_{0}$ at the same rate $n^{-1/2(r+1)}$. However, it may 
happen that different components of each atom may not share similar rates of 
convergence. To be able to tackle this issue, we utilize a general version of Wasserstein 
distance, which is the generalized transportation distance in Definition 
\ref{def:generalized_Wasserstein}.}

Singularity level is an useful concept for deriving rates of convergence for 
the mixing measure under a suitable Wasserstein metric, which in turn entails
upper bound on the rate of convergence for individual model parameters (which
make up the atoms of the mixing measure). However, different parameters may
actually admit different convergence behaviors. To study this phenomenon of inhomogeneity, 
a more elaborate notion is required to describe the singularity structure of the parameter
space.  In this section we shall introduce such a notion, which we call \emph{singularity index},
along with generalized transportation distance for the mixing measures
already given by Def.~\eqref{def:generalized_Wasserstein}.

As in the previous section, we adopt the strategy 
of investigating the behavior of the likelihood function $p_G(x)$ as $G$ varies in a
generalized transportation neighborhood of $G_0$. In particular, for any 
fixed $\kappa \in \mathbb{N}^d$, consider a sequence of $G_{n} \in \Ocal_{k}$ such that $
\widetilde{W}_{\kappa}(G_{n},G_{0}) \to 0$.  Let $G_{n}$ be 
represented as in \eqref{eqn:representation_overfit}. To avoid notational cluttering, we 
again drop $n$ from the superscript when the context is clear. The following lemma
relates 
generalized transportation distance to a semipolynomial of order $\|\kappa\|_{\infty}$.

\begin{lemma} \label{lemma:bound_overfit_Wasserstein_first}
Fix $\kappa \in \mathbb{N}^{d}$. For any element $G$ represented by 
Eq.~\eqref{eqn:representation_overfit}, define 
\[D_\kappa(G,G_0) := 
\sum_{i=1}^{k_{0}+\extra} \sum_{j=1}^{s_{i}} p_{ij}
d_{\kappa}^{\|\kappa\|_{\infty}}(\eta_{ij},\eta_{i}^{0}) + \sum_{i=1}^{k_0+\extra} |\Delta p_{i\cdot}| .\]
Then $\widetilde{W}_\kappa^{\|\kappa\|_{\infty}}(G,G_0) \asymp D_\kappa(G,G_0)$, as $\widetilde{W}_\kappa(G,G_0) \downarrow 0$.
\end{lemma}

Denote $r:=\|\kappa\|_{\infty}$ (this notational choice is deliberate, as we will see shortly). 
By virtue of the argument from Section 
\ref{Section:likelihood_Wasserstein_neighborhood}, using Taylor expansion up to the $r$-order we have
\begin{eqnarray}
\frac{p_{G}(x)-p_{G_{0}}(x)}{\widetilde{W}_{\kappa}^{\|\kappa\|_{\infty}}(G,G_{0})}
=
\sum_{|\alpha|=1}^{r}
\sum_{i=1}^{k_{0}+\extra}
\sum \limits_{j=1}^{s_{i}} 
\biggr (\frac{p_{ij}(\Delta \eta_{ij})^{\alpha}/\alpha!}
{\widetilde{W}_{\kappa}^{\|\kappa\|_{\infty}}(G,G_{0})}\biggr)
\frac{\partial^{|\alpha|}f}{\partial \eta^{\alpha}}
(x|\eta_i^0) + \nonumber \\
\sum_{i=1}^{k_0+\extra} \frac{\Delta p_{i\cdot}}{\widetilde{W}_{\kappa}^{\|\kappa\|_{\infty}}(G,G_{0})}
f(x|\eta_{i}^{0})
+ \dfrac{R_{r}(x)}{\widetilde{W}_{\kappa}^{\|\kappa\|_{\infty}}(G,G_{0})}, \nonumber
\end{eqnarray}
where $R_{r}(x)$ is the Taylor remainder. By Lemma \ref{lemma:generalized_standard_Wasserstein} 
we readily verify that
\begin{eqnarray}
\sup \limits_{x} \left|R_{r}(x)/\widetilde{W}_{\kappa}^{\|\kappa\|_{\infty}}(G,G_{0})\right| \lesssim \sup \limits_{x} \left|R_{r}(x)/W_{r}^{r}(G,G_{0})\right| \to 0 \nonumber
\end{eqnarray}
as long as $f$ is uniform Lipschitz up to order $r$. Thus,
\begin{eqnarray}
\frac{p_{G}(x)-p_{G_{0}}(x)}{\widetilde{W}_{\kappa}^{\|\kappa\|_{\infty}}(G,G_{0})}
=
\sum_{|\alpha|=1}^{r}
\sum_{i=1}^{k_{0}+\extra}
\sum \limits_{j=1}^{s_{i}} 
\biggr (\frac{p_{ij}(\Delta \eta_{ij})^{\alpha}/\alpha!}
{\widetilde{W}_{\kappa}^{\|\kappa\|_{\infty}}(G,G_{0})}\biggr)
\frac{\partial^{|\alpha|}f}{\partial \eta^{\alpha}}
(x|\eta_i^0) \nonumber \\
+ \sum_{i=1}^{k_0+\extra} \frac{\Delta p_{i\cdot}}{\widetilde{W}_{\kappa}^{\|\kappa\|_{\infty}}(G,G_{0})}
f(x|\eta_{i}^{0}) + o(1). \nonumber
\end{eqnarray}
We are ready to define the following concept.

\begin{definition} \label{definition:rminimal_new}
For any $\kappa \in \mathbb{N}^{d}$, the following representation is called $\kappa$-\emph{minimal} form of
the mixture likelihood for a sequence of mixing measures $G$ tending to $G_0$
in $\widetilde{W}_{\kappa}$ distance:
\begin{eqnarray}
\frac{p_{G}(x)-p_{G_{0}}(x)}{\widetilde{W}_\kappa^{\|\kappa\|_{\infty}}(G,G_0)} = 
\sum_{l=1}^{T_\kappa}
\biggr (\frac{\xi_{l}^{(\kappa)}(G)}{\widetilde{W}_\kappa^{\|\kappa\|_{\infty}}(G,G_0)} \biggr ) H_{l}^{(\kappa)}(x) + o(1), 
\label{eqn:generallinearindependencerepresentation_general}
\end{eqnarray}
which holds for almost all $x$, with the index $l$ ranging from 1 to a finite $T_\kappa$,
if 
\begin{itemize}
\item [(1)] $H_{l}^{(\kappa)}(x)$ for all $l$ are linearly independent functions of $x$, and
\item [(2)] coefficients $\xi_{l}^{(\kappa)}(G)$ are polynomials of
the components of $\Delta \eta_{ij}$, and $\Delta p_{i\cdot}, p_{ij}$.
\end{itemize}
\end{definition}

It is clear that $\kappa$-minimal form is a general version of $r$-minimal form when $
\kappa=(r,\ldots,r)$. The procedure for constructing $\kappa$-minimal forms is similar
that of $r$-minimal forms, and will be given in Section \ref{sec:r-form},
where we specifically search for a subset of linearly independent partial derivatives up to the order $\|\kappa\|
_{\infty}$. 
The multi-index $\kappa$-form provides the basis for the notion of multi-index singularity 
that we now define.

\begin{definition}\label{def-rsingular_set}
Let $\Gcal$ be a class of discrete probability measures which has 
a bounded number of support points in $\Theta$. For any $\kappa \in \mathbb{N}^{d}$, we say that $G_{0}$ is $\kappa$-singular relative to $\Gcal$, if $G_0$ admits a $\kappa$-minimal form 
given by Eq. \eqref{eqn:generallinearindependencerepresentation_general}, according to
which there exists a sequence of $G \in \Gcal$ tending to $G_0$ under $\widetilde{W}_\kappa$ distance
such that
\[\xi_l^{(\kappa)}(G)/\widetilde{W}_\kappa^{\|\kappa\|_{\infty}}(G,G_0) \rightarrow 0 \;\textrm{for all}\; l=1,\ldots,T_\kappa.\]
\end{definition}
\comment{\begin{definition}\label{def-rsingular_set}
Fix $r \geq 1$ and let $\Gcal$ be a class of discrete probability measures which has 
a bounded number of support points in $\Theta$. For any $\kappa \in \mathbb{N}^{d}
$, we say that $\kappa \in \singset_{r}(G_{0}|\Gcal)$ if $\|\kappa\|_{\infty}=r$ and 
$G_0$ admits a $\kappa$-minimal form 
given by Eq. \eqref{eqn:generallinearindependencerepresentation_general}, according to
which for any sequence of $G \in \Gcal$ tending to $G_0$ under generalized $\widetilde{W}_\kappa$ metric, not all
\[\xi_l^{(\kappa)}(G)/\widetilde{W}_\kappa^{\|\kappa\|_{\infty}}(G,G_0) \rightarrow 0 \;\textrm{as}\; l=1,\ldots,T_\kappa.\]
\end{definition}}

Like $r$-singularity, the notion of $\kappa$-singularity possesses a crucial
monotonic property in terms of partial order with vector:

\begin{lemma} \label{lemma:minimal_form_general}
\label{invariance}
(a) (Invariance) The existence of the sequence of $G$ in the statement of Definition~\ref{def-rsingular_set}
holds for all $\kappa$-minimal forms once it holds for at least one $\kappa$-minimal form.

(b) (Monotonicity) If $G_0$ is $\kappa$-singular for some $\kappa \in \mathbb{N}^{d}$, then $G_0$ is $\kappa'$-singular for any $\kappa' \preceq \kappa$.
\end{lemma}
Let $\overline{\mathbb{N}} := \mathbb{N} \cup \left\{\infty\right\}$. The monotonicity 
of $\kappa$-singularity naturally leads to the following notion of
singularity index of a mixing measure $G_0$ (and the corresponding parameters) 
relative to an ambient space $\Gcal$:

\begin{definition} \label{definition:singularity_index} 
For any $\kappa \in \overline{\mathbb{N}}^{d}$, we say $\kappa$ is a singularity index of $G_{0}$ relative to a given class $\Gcal$ if and only if $G_{0}$ is $\kappa'$-singular relative to $\Gcal$ for any $\kappa' \prec \kappa$, and there is no $\kappa' \succeq \kappa$ such that $G_{0}$ remains $\kappa'$-singular relative to $\Gcal$.
Define the \emph{singularity index set} 
\[\singset(G_0|\Gcal) := \left\{\kappa \in \overline{\mathbb{N}}^{d}: \kappa \ \text{is a singularity index of} \ G_{0} \ \text{relative to} \ \Gcal \right\}.\] 
\end{definition}

This definition suggests that the singularity index set may not be always a singleton in general.
The following proposition clarifies the relation between singularity level $\lev(G_{0}|\Gcal)$ and 
singularity index set $\singset(G_{0}|\Gcal)$.

\begin{proposition} \label{proposition:singularity_set_level}
Assume that $\lev(G_{0}|\Gcal)=r$ for some $r \geq 0$. Then,
\begin{itemize}
\item[(i)] If $r=0$, then $\singset(G_{0}|\Gcal)=\left\{(1,\ldots,1)\right\}$.
\item[(ii)] If $r=\infty$, then $\singset(G_{0}|\Gcal)=\left\{(\infty,\ldots,\infty)\right\}$.
\item[(iii)] If $1 \leq r <\infty$, then there exists $\kappa \in \singset(G_{0}|\Gcal)$ such that $\kappa \preceq (r+1,\ldots,r+1)$ and at least one component of $\kappa$ is $r+1$. 
\item[(iv)] If $r \geq 1$, and $G_{0}$ is not $\overline{\kappa}$-singular relative to $\Gcal$
for some $\overline{\kappa} \in \mathbb{N}^d$, then there exists
$\kappa \in \singset(G_0|\Gcal)$ such that $\kappa \preceq \overline{\kappa}$. 
\item[(v)] If some finite $\kappa \in \singset(G_{0}|\Gcal)$, then $\lev(G_{0}|\Gcal)\leq \|\kappa\|_{\infty}-1$.
Moreover, if $\kappa$ is unique, then $\lev(G_{0}|\Gcal)= \|\kappa\|_{\infty}-1$.
\end{itemize}
\end{proposition} 
This proposition establishes that when the singularity index set of a mixing measure $G_0$ relative
to $\Gcal$ is a singleton,
one can determine the corresponding singularity level of $G_0$ immediately. 
We will give several examples of ambient space $\Gcal$ 
and kernel $f$ under which this situation holds (cf. the examples in Section \ref{Section:singularity_index_classical_mixtures}). 
The role of the singularity index is in determining minimax lower bound and convergence rate of
estimation for
individual parameters that make up the mixing measure $G_0$. Briefly speaking,
provided that $\kappa$ is a singularity index of $G_{0}$, 
then any density estimation method, such as 
MLE or Bayesian estimation, that guarantees, say, a root-$n$ rate of convergence toward density 
$p_{G_{0}}$  under Hellinger metric will lead to the convergence rate $n^{-1/2\|\kappa\|_{\infty}}$ of 
estimating $G_{0}$ under generalized transportation metric $\widetilde{W}_{\kappa}$, 
which is also minimax under additional conditions on kernel density $f$. The implication 
of such results is that the $i$-th component of each atom of $G_{0}$ can be estimated 
with rate $n^{-1/2\kappa_{i}}$ where $\kappa_{i}$ is the i-th component 
of $\kappa$. This theory will be presented in Appendix B.

\paragraph{Complete inhomogeneity}
Although the singularity index captures the inhomogeneous convergence behavior of different components of
an atom of $G_0$, i.e., parameters of different types such as location, scale and skewness,
it is possible that every parameter in a mixture model admits a different convergence rate, including those of the same type but associating with different 
mixture components.
We call this phenomenon "complete inhomogeneity". To characterize this, we shall introduce
 blocked generalized transportation distance by replacing uniform semi-metric 
$d_{\kappa}$ in the formulation of generalized transportation distance as possibly 
different semi-metrics $d_{K_{i}}$ with respect to the $i$-th atom $\eta_{i}^{0}$ of 
$G_{0}$ where $K_{i} \in \mathbb{N}^{d}$ for all $1 \leq i \leq k_{0}$. The best 
possible convergence rates of $\eta_{i}^{0}$, therefore, will be determined by the 
optimal choices of $K_{i}$ for any $i$. To quantify these optimal choices, we define a 
new notion of \emph{singularity matrix} in terms of a matrix $K$ which treat all $K_{i}$ as its rows. 
With this concept in place, we can establish rates of convergence 
for estimating $G_{0}$, its atoms, as well as components of these atoms based 
on specific values of the singularity matrix. Due to space constraint, the detailed formulation 
and discussion of singularity matrix are deferred to Appendix C.




\subsection{Revisiting known results on finite mixtures} 
\label{Section:singularity_index_classical_mixtures}

In this section, singularity structures of parameter space will be examined to shed some light on
previously known or recent results on the parameter estimation behavior of
several classes of finite mixtures.  

\paragraph{O-mixtures with second-order identifiable kernels} As being studied by 
\cite{Chen1992, Nguyen-13,Ho-Nguyen-EJS-16,Rousseau-Mengersen-11}, the second order identifiability condition of 
kernel density $f$ simply means that the collection of $\{\partial^\kappa f/\partial \eta^\kappa(x|\eta_j)
| j=1,\ldots, k_0; |\kappa| \leq 2 \}$ evaluated at $G_{0}$ is linearly independent. 
 
\begin{proposition} \label{proposition:strong_identifiability_singularity_index}
Assume that $f$ is second-order identifiable and admits uniform Lipschitz condition up 
to the second order. Then, $\lev(G_{0}|\Ocal_{k})=1$ and $\singset(G_{0}|\Ocal_{k})=
\left\{(2,\ldots,2)\right\}$.
\end{proposition}

By Proposition \ref{proposition:strong_identifiability_singularity_index} and Theorem~\ref{proposition:convergence_and_minimax}  
the convergence rate of estimating $G_{0}$ under o-mixtures of second order identifiable kernel $f$ is 
$n^{-1/4}$. Moreover, by Theorem~\ref{proposition:convergence_and_minimax_general} in Appendix B, the components of 
each atom of $G_{0}$ also admit uniform convergence rate $n^{-1/4}$. 

\paragraph{Univariate Gaussian o-mixtures} Location-scale Gaussian mixtures are among the most popular mixture models
in statistics. For simplicity, consider univariate Gaussian o-mixtures and let $G_0 \in \Ecal_{k_0}
\subset \Ocal_{k,c_{0}}$ for some $k>k_0$ and small constant $c_{0}>0$. That is, the ambient space
$O_{k,c_{0}} \subset \Ocal_k$ contains only (discrete) probability measures
whose point masses are bounded from below by $c_{0}$. Let $\left\{f(x|\theta, v=\sigma^2)
\right\}$ be the family of univariate location-scale Gaussian distributions. Recall the partial differential equation, Eq.~\eqref{key-normal},
satisfied by Gaussian kernels \cite{Chen-2003,Shimotsu-2014,Ho-Nguyen-Ann-16}. Following \cite{Ho-Nguyen-Ann-16},
denote by $\overline{r}(k-k_{0})$ the minimal value of $r>0$ such that the following system of 
polynomial equations 
\begin{eqnarray}
\mathop {\sum }\limits_{j=1}^{k-k_{0}+1}{\mathop {\sum }\limits_{\substack{n_{1}+2n_{2} = 
\alpha \\ n_{1},n_{2} \geq 0}}{\dfrac{c_{j}^{2}a_{j}^{n_{1}}b_{j}^{n_{2}}}{
n_{1}!n_{2}!}}}=0 \ \ \text{for each} \;
\alpha = 1,\ldots, s
\label{eqn:generalovefittedGaussianzero_Gaussian_mulindex}
\end{eqnarray}
does not have any solution for the unknowns $(a_{j},b_{j},c_{j})_{j=1}^{k-k_{0}+1}$
such that all of $c_{j}$s are non-zeros, and at least one of 
the $a_{j}$s is non-zero. By means of the argument from Proposition 2.2 in \cite{Ho-Nguyen-Ann-16}, we can quickly verify that the singularity level of $G_{0}$ is 
$\lev(G_{0}|\Ocal_{k,c_{0}})=r(k-k_{0})-1$. It leads to the 
convergence rate $n^{-1/2r(k-k_{0})}$ of estimating mixing measure $G_{0}$ when we 
overfit Gaussian mixture models by $k$ components, as established by~\cite{Ho-Nguyen-Ann-16}. 
However, we can say more: it turns out that  
the location parameters and the scale parameters in the Gaussian mixtures admit different rates of convergence. 
This is due to examining the singularity index of $G_{0}$.
\begin{proposition} \label{proposition:Gaussian_mulindex}
 For any $G_{0} \in \Ecal_{k_{0}} \cap \Ocal_{k_{0},c_{0}}$, we obtain 
\begin{eqnarray}
\singset(G_{0}|\Ocal_{k,c_{0}})=\begin{cases} \biggr\{\biggr(\overline{r}(k-k_{0}),\dfrac{\overline{r}(k-k_{0})}{2}\biggr)\biggr\}, & \mbox{if} \ \overline{r}(k-k_{0}) \ \mbox{is an even number} \\ \left\{\biggr(\overline{r}(k-k_{0}),\dfrac{\overline{r}(k-k_{0})+1}{2}\biggr)\right\}, & \mbox{if} \ \overline{r}(k-k_{0}) \ \mbox{is an odd number}. \end{cases} \nonumber
\end{eqnarray}
\end{proposition}
The result of Proposition \ref{proposition:Gaussian_mulindex} indicates that under 
Gaussian o-mixtures the best possible convergence rate of location parameter is 
$n^{-1/2\overline{r}(k-k_{0})}$ while that of scale parameter is $n^{-1/\overline{r}(k-
k_{0})}$ when $\overline{r}(k-k_{0})$ is even or is $n^{-1/(\overline{r}(k-k_{0})+1)}$ 
when $\overline{r}(k-k_{0})$ is even. These convergence rates are sharp, 
thanks to part (a) of Theorem 
\ref{proposition:convergence_and_minimax_general} in Appendix B. 
Thus, in an overfitted Gaussian mixture, the more overfitted the model is the slower the estimation rate. Moreover, 
the scale parameter is generally more efficient to estimate than the location parameter.

\paragraph{Gamma mixtures} The Gamma family of densities takes the form
$f(x|a,b) := \dfrac{b^{a}}{\Gamma(a)}x^{a-1}\exp(-bx)$ for $x > 0$, and $0$ otherwise, 
where $a,b$ are positive shape and rate parameters, respectively. 
The Gamma kernel admits the following partial differential equation:
\begin{eqnarray}
\frac{\partial f}{\partial b}(x|a,b) = \frac{a}{b}f(x|a,b) - \frac{a}{b}f(x|a+1,b). \nonumber
\end{eqnarray}
As demonstrated by \cite{Ho-Nguyen-Ann-16}, this identity leads to two disjoint categories of the parameter 
values of $G_{0}$, which are respectively called ``generic cases'' and ``pathological cases''. In particular, denote $G_{0}=\mathop {\sum }\limits_{i=1}^{k_{0}}{p_{i}^{0}\delta_{(a_{i}^{0},b_{i}^{0})}}$ where $k_{0} \geq 2$. Assume that $a_{i}^{0} \geq 1$ for all $1 \leq i \leq k_{0}$. Now, we define
\begin{itemize}
\item[(A1)] Generic cases: $\left\{|a_{i}^{0}-a_{j}^{0}|,|b_{i}^{0}-b_{j}^{0}|\right\} \neq \left\{1,0\right\}$
 for all $1 \leq i,j \leq k_{0}$.
\item[(A2)] Pathological cases: $\left\{|a_{i}^{0}-a_{j}^{0}|,|b_{i}^{0}-b_{j}^{0}|\right\} = \left\{1,0\right\}$
 for some $1 \leq i,j \leq k_{0}$.
\end{itemize}
For the Gamma o-mixtures setting, following \cite{Ho-Nguyen-Ann-16} we also define the constrained set of $\Ocal_{k}$
\begin{eqnarray}
\overline{\mathcal{O}}_{k,c_{0}}=
\biggr \{G=\mathop {\sum }\limits_{i=1}^{k^{'}}{p_{i}\delta_{(a_{i},b_{i})}}
\biggr |  k^{'} \leq k \ \text{and } \ |a_{i}-a_{j}^{0}| \not \in [1-c_{0},1+c_{0}] \nonumber \\
\cup [2-c_{0},2+c_{0}] \forall \  (i,j) 
\biggr \} \nonumber
\end{eqnarray} 
where $c_{0}>0$. The singularity structure of $G_{0}$ is clarified by the following result.

\begin{proposition} \label{proposition:singularity_index_level_Gamma} Fix any $G_{0} \in \Ecal_{k_{0}}$.
\begin{itemize}
\item[(a)] For generic cases specified by (A1), there holds 
\begin{itemize}
\item[(a1)] For e-mixtures: $\lev(G_{0}|\Ecal_{k_{0}})=0$, $\singset(G_{0}|\Ecal_{k_{0}})=\left\{(1,1)\right\}$.
\item[(a2)] For o-mixtures: $\lev(G_{0}|\overline{\Ocal}_{k,c_{0}})=1$, $\singset(G_{0}|\overline{\Ocal}_{k,c_{0}})=\left\{(2,2)\right\}$.
\end{itemize}
\item[(b)] For pathological cases specified by (A2), there holds
\begin{itemize}
\item[(b1)] For e-mixtures: $\lev(G_{0}|\Ecal_{k_{0}})=\infty$, $\singset(G_{0}|\Ecal_{k_{0}})=\left\{(\infty,\infty)\right\}$.
\item[(b2)] For o-mixtures: $\lev(G_{0}|\overline{\Ocal}_{k,c_{0}})=\infty$, $\singset(G_{0}|\overline{\Ocal}_{k,c_{0}})=\left\{(\infty,\infty)\right\}$.
\end{itemize}
\end{itemize}
\end{proposition}

Part (a) of Proposition \ref{proposition:singularity_index_level_Gamma} entails the generic $n^{-1/2}$ and $n^{-1/4}$ rates of parameter estimation 
in the e-mixture and o-mixture settings, respectively. If true parameter values belong to pathological cases, however, part (b) of the proposition implies
that polynomial rates of parameter estimation are not possible, due to infinite singularity level.

Although non-trivial convergence behavior obtained in previous work can also recovered
via our notion of singularity levels and indices, with the exception of location-scale Gaussian mixtures, 
none of these examples exhibit the complex inhomogeneity of the parameter space. 
To demonstrate the full spectrum of complexity of finite mixture models, we will apply the theory on 
finite mixtures of skew-normal distributions starting in Section~\ref{Section:overfitskew}.

\subsection{Construction of $r$-minimal and $\kappa$-minimal forms}
\label{sec:r-form}
We have seen that minimax rates of parameter estimation for finite mixtures
can be read off from the singularity structures, via notions of singularity levels and indices, of the parameter space.
For the remainder of Section \ref{Section:general_procedure_singularity} we shall present a general procedure or calculating singularity level/index for a given
mixing measure. 

To this end, one needs to first construct $r$-minimal forms and $\kappa$-minimal forms.
Since the latter can be constructed in much the same manner as the former, we will focus our presentation
on $r$-minimal forms.
A simple way of constructing an $r$-minimal form is to
select a subset of partial derivatives of $f$ taken up to order $r$
such that all these functions are linearly independent. 
A simple procedure is to start from the smallest order $r=1$ and
then move up to $r=2,3,\ldots$ and so on. For each $r$, assume that
we have obtained a linearly independent subset of partial derivatives
up to order $r-1$. Now, going over the ordered list of 
$r$-th partial derivatives:
$\{\partial^{|\kappa|}f/\partial \eta^\kappa | \kappa \in \mathbb{N}^d,
|\kappa | = r\}$. For each $\kappa$ such that $|\kappa|=r$, if
the partial derivative of $f$ of order $\kappa$ can be expressed
as a linear combination of other partial derivatives among those already selected,
then this one is eliminated. The process goes on until we exhaust the
list of the partial derivatives.

\begin{example} 
\label{ex-reduce-1}
\textup{
Continuing from Example~\ref{example-skew-1}, suppose
that $G_0$ satisfies Eq.~\eqref{eqn:firstorderskew}. 
According to the proof of Lemma ~\ref{proposition-notskewnormal}, we
can choose $\alpha_{4k}\neq 0$, so 
the partial derivative may be eliminated via the reduction:
\[\dfrac{\partial{f(x|\eta_k^0)}}{\partial{m}}
= -\sum_{j=1}^{k}
\frac{\alpha_{1j}}{\alpha_{4k}}f(x|\eta_{j}^0)+
\frac{\alpha_{2j}}{\alpha_{4k}}
\dfrac{\partial{f(x|\eta_j^0)}}{\partial{\theta}}+
\frac{\alpha_{3j}}{\alpha_{4k}}
\dfrac{\partial{f(x|\eta_j^0)}}{\partial{v}}
-\sum_{j=1}^{k-1}\frac{\alpha_{4j}}{\alpha_{4k}}
\dfrac{\partial{f(x|\eta_j^0)}}{\partial{m}} 
\]
Note that this elimination step is valid only for a subset
of $G_0 = G(\vecp^0,\veceta^0)$ for which Eq.~\eqref{eqn:firstorderskew} holds.
That is, only if $P_1(\veceta^0) = 0$ or $P_2(\veceta^0) = 0$.
}
\end{example}

\begin{example} 
\label{ex-reduce-2}
\textup{
If $f(x|\eta) = f(x|\theta,v,m)$ where $m=0$,
the skew-normal kernel becomes the Gaussian kernel. Due to~\eqref{key-normal}, 
all partial derivatives with respect to both $\theta$ and $v$ can be eliminated
via the following reduction:
for any $\kappa_1,\kappa_2 \in \mathbb{N}$,
for any $j=1,\ldots,k_0$,
\[\frac{\partial^{\kappa_1+\kappa_2}f(x|\eta_j^0)}
{\partial \theta^{\kappa_1}v^{\kappa_2}} = \frac{1}{2^{\kappa_2}}
\frac{\partial^{\kappa_1+2\kappa_2}f(x|\eta_j^0)}
{\partial \theta^{\kappa_1+2\kappa_2}}.\]
This elimination is valid for all parameter values $(\vecp^0,\veceta^0)$,
and $r$-minimal forms for all orders.
}
\end{example}

\begin{example} 
\label{ex-reduce-3}
\textup{
For the skew-normal kernel density $f(x|\eta)=f(x|\theta,v,m)$, 
Eq.~\eqref{eqn:overfittedskewnormaldistributionzero} yields the following 
reductions: for any $j=1,\ldots, k_0$, any $\eta=(\theta,v,m)=\eta_j^0=
(\theta_j^0,v_j^0,m_j^0)$ such that $m \neq 0$
\begin{eqnarray}
\label{Eqn-reduction-1}
\frac{\partial^2 f}{\partial \theta^2}
& = & 2\frac{\partial f}{\partial v} - \frac{m^3+m}{v}
\frac{\partial f}{\partial m}, \\
\label{Eqn-reduction-2}
\dfrac{\partial^{2}{f}}{\partial{v}\partial{m}}
& = & -\frac{1}{v}\frac{\partial f}{\partial m}
- \frac{m^2+1}{2vm}\frac{\partial^2 f}{\partial m^2}.
\end{eqnarray}
Differentiating results in a ripple effect on subsequent eliminations at higher orders.
For examples, partial derivatives up to the third order of 
$f$ evaluated at $\eta=\eta^0_j=(\theta_j^0,v_j^0,m_j^0)$ for
any $j=1,\ldots,k_0$ where $m_{j}^{0} \neq 0$ can be expressed as follows:
\begin{eqnarray}
\dfrac{\partial^{3}{f}}{\partial{\theta^{3}}} & = & 2\dfrac{\partial^{2}{f}}{\partial{\theta}
\partial{v}}-\dfrac{m^{3}+m}{v}\dfrac{\partial^{2}{f}}{\partial{\theta}\partial{m}}, 
\nonumber \\
\dfrac{\partial^{3}{f}}{\partial{\theta^{2}}\partial{v}} & = & 2\dfrac{\partial^{2}{f}}
{\partial{v^{2}}}+\dfrac{m^{3}+m}{v^{2}}\dfrac{\partial{f}}{\partial{m}}-\dfrac{m^{3}+m}
{v}\dfrac{\partial^{2}{f}}{\partial{v}\partial{m}}, \nonumber \\
\dfrac{\partial^{3}{f}}{\partial{\theta^{2}}\partial{m}} & = & 2\dfrac{\partial^{2}{f}}
{\partial{v}\partial{m}}-\dfrac{3m^{2}+1}{v}\dfrac{\partial{f}}{\partial{m}}-\dfrac{m^{3}
+m}{v}\dfrac{\partial^{2}{f}}{\partial{m^{2}}}, \nonumber \\
\dfrac{\partial^{3}{f}}{\partial{v}\partial{m^{2}}} & = & -\dfrac{m^{2}+1}{2vm}\dfrac{\partial^{3}{f}}{\partial{m^{3}}}-\dfrac{3m^{2}-1}{2vm^{2}}\dfrac{\partial^{2}{f}}{\partial{m^{2}}}, \nonumber \\
\dfrac{\partial^{3}{f}}{\partial{v^{2}}\partial{m}} & = & -\dfrac{2}{v}\dfrac{\partial^{2}{f}}{\partial{v}\partial{m}}-\dfrac{m^{2}+1}{2vm}\dfrac{\partial^{3}{f}}{\partial{v}\partial{m^{2}}} \nonumber \\
& = & \dfrac{(m^{2}+1)^{2}}{4v^{2}m^{2}}\dfrac{\partial^{3}{f}}{\partial{m^{3}}}+\dfrac{(m^{2}+1)(7m^{2}-1)}{4m^{3}v^{2}}\dfrac{\partial^{2}{f}}{\partial{m^{2}}}+\dfrac{2}{v^{2}}\dfrac{\partial{f}}{\partial{m}}, \nonumber \\
\dfrac{\partial^{3}{f}}{\partial{\theta}\partial{v}\partial{m}} & = & -\dfrac{m^{2}+1}{2vm}\dfrac{\partial^{3}{f}}{\partial{\theta}\partial{m^{2}}}-\dfrac{1}{v}\dfrac{\partial^{2}{f}}{\partial{\theta}\partial{m}}. 
\label{eqn:nonlinearequations_third}
\end{eqnarray}
}
\end{example}
All three examples above demonstrate how the dependence among partial derivatives of 
kernel density $f$, among different orders $\kappa$, 
and among those evaluated at different component $i$,
has a deep impact on the representation of $r$-minimal forms. 

In general, the $r$-minimal form \eqref{eqn:generallinearindependencerepresentation}
may be expressed somewhat more explicitly as follows
\[\frac{p_G(x)-p_{G_0}(x)}{W_r^r(G,G_0)}
= \sum_{(i,\kappa) \in \mathcal{I,K}} 
\frac{\xi_{i,\kappa}^{(r)}(G)}{W_r^r(G_0,G)} H_{i,\kappa}^{(r)}(x|G_0) + 
\sum_{i=1}^{k_0} \frac{\zeta_i^{(r)}(G)}{W_r^r(G_0,G)}
f(x|\eta_{i}^{0}) + o(1).\]
where $\mathcal{I} \subset \{1,\ldots,k_0\}$ and
$\mathcal{K} \subset \mathbb{N}^d$ of elements $\kappa$ such that 
$|\kappa|\leq r$. It is emphasized that the sets $\mathcal{I}$ and
$\mathcal{K}$ are specific to a particular $r$-minimal form under
consideration. $H_{i,\kappa}^{(r)}$ are a collection of
linearly independent partial derivatives of $f$ that are also
independent of all functions $f(x|\eta_i^0)$. $H_{i,\kappa}^{(r)}$
are taken from the
collection of partial derivatives with order at most $r$. 
Moreover,
$\xi_{i,\kappa}^{(r)}$ and $\zeta_i^{(r)}$ take the following 
polynomial forms:
\begin{eqnarray}
\label{coeff-1}
\xi_{i,\kappa}^{(r)}(G) & = & \sum_{j=1}^{s_i}
\frac{p_{ij}(\Delta\eta_{ij})^{\kappa}}{\kappa!}
+ \sum_{i',\kappa'} \beta_{i,\kappa,i',\kappa'}(G_0) 
\sum_{j=1}^{s_{i'}} \frac{p_{ij}(\Delta\eta_{ij})^{\kappa'}}{\kappa'!},\\
\label{coeff-2}
\zeta_{i}^{(r)}(G) & = & 
\Delta p_{i\cdot} 
+ \sum_{i',\kappa'} \gamma_{i,i',\kappa'}(G_0) 
\sum_{j=1}^{s_{i'}} \frac{p_{ij}(\Delta\eta_{ij})^{\kappa'}}{\kappa'!}.
\end{eqnarray}
In the right hand side of each of the last two equations, $i'$ is taken from
a subset of $\{1,\ldots,k_0\}$ and $\kappa'$ is from a 
subset of $\mathbb{N}^d$ such that $|\kappa|\leq |\kappa'|\leq r$.
The actual detail of these subsets depend on the specific elimination scheme
leading to the $r$-minimal form. Likewise,
the non-zero coefficients $\beta_{i,\kappa,i',\kappa'}(G_{0})$ and 
$\gamma_{i,\kappa,i',\kappa'}(G_{0})$ arise from the
specific elimination scheme. We include
argument $G_0$ in these coefficients to highlight the fact that
they may be dependent on the atoms of $G_0$ (cf. Example~\ref{ex-reduce-1}
and \ref{ex-reduce-3}).

By the definition of $r$-singularity for any $r\geq 1$, $G_0$ is $r$-singular relative to $\Gcal$
if there exists a sequence of $G$ tending to $G_0$ in the ambient space $\Gcal$
such that the sequences of semipolynomial fractions,
namely, $\xi_{i,\kappa}^{(r)}(G)/W_r^r(G,G_0)$ and
$\zeta_{i}^{(r)}(G)/W_r^r(G,G_0)$ (whose numerators are given by Eq.~\eqref{coeff-1} and
Eq.~\eqref{coeff-2}), must vanish. As a consequence, the question of $r$-singularity 
for a given element $G_0$ is determined by 
the limiting behavior of a finite collection of infinite sequences of 
semipolynomial fractions.

\subsection{Polynomial limits of $r$-minimal and $\kappa$-minimal form coefficients} \label{Section:polynomial_limits}

The limiting behavior of semipolynomial fractions described
above is independent of a particular choice of the $r$-minimal form, in a sense that we now explain.
In part (a) of Lemma \ref{invariance}, we established an invariance property of the $r$-singularity, 
which does not depend on a specific form of the $r$-minimal form. Let us 
restrict the basis functions to be members of the collection of all partial 
derivatives of $f$ up to order $r$.  In the proof of part (b) of
Lemma ~\ref{invariance} it was shown that the coefficients $\xi_l^{(r)}(G)$ have
to be elements of a set of polynomials of $\Delta \eta_{ij}$,
$\Delta p_{i\cdot}$, and $p_{ij}$, which are closed under linear combinations
of its elements. Let us denote this set by $\Pcal(G,G_0)$, which is
invariant with respect to any specific choice of the basis functions (from the 
collection of partial derivatives) for the $r$-minimal form. 
Moreover, $G_0$ is $r$-singular if and only if a sequence of $G$ tending to $G_0$ in $W_r$ 
can be constructed such that for any element
$\xi_{l}^{(r)}(G) \in \Pcal(G,G_0)$, we have $\xi_{l}^{(r)}(G)/W_r^r(G,G_0) \rightarrow 0$. 
Equivalently, 
\begin{equation}
\label{eqn-polynomial-limits}
\xi_l^{(r)}(G)/D_r(G,G_0) \rightarrow 0 \;\; \textrm{for all}\; \xi_l^r(G) \in \Pcal(G,G_0).
\end{equation}

Extracting the limits of a single multivariate semipolynomial fraction 
(a.k.a. rational semipolynomial functions)
is quite challenging in general, due to the interaction among multiple variables involved \cite{Xiao-14}.
Analyzing the limits of not one but a collection of multivariate rational semipolynomials
is considerably more difficult. To obtain meaningful and
concrete results, we need to consider specific systems of multivariate rational semipolynomials
that arise from the $r$-minimal form.

In the remainder of this paper we will proceed to do just that. By working with
specific choices of kernel density $f$, it will be shown that
under the compactness of the parameter spaces, one can extract
a subset of limits from the system of rational semipolynomials 
$\xi_l^{(r)}(G)/D_r(G,G_0)$. These limits are expressed as a system of polynomials
admitting non-trivial solutions.
For a given $r\geq 1$, if the extracted system of polynomial limits does not contain 
admissible solutions, then it means that there does not exist any sequence of mixing measures 
$G$ for which a valid $r$-minimal form can be constructed, because~\eqref{eqn-polynomial-limits}
is not fullfilled. This would entail the upper bound
$\lev(G_0|\Gcal) < r$. On the other hand, if the extracted system of
polynomial limits does contain at least one admissible solution, this is a hint that
the $r$-singularity level of $G_0$ relative to the ambient space $G\cal$ \emph{might} hold. 
Whether this is actually the case or not
requires an explicit construction of a sequence of $G \in \Gcal$ 
(often building upon the admissible solutions of the polynomial limits) and then 
the verification that condition~\eqref{eqn-polynomial-limits} indeed holds.
For the verification purpose, it is sufficient (and simpler) to 
work with a specific choice of $r$-minimal form, as Definition ~\ref{def-rsingular} allows.

Due to the asymptotic equivalence between generalized transportation distance and a 
semipolynomial in Lemma \ref{lemma:bound_overfit_Wasserstein_first}, the studies of 
$\kappa$-singularity also reduce to an investigation of limiting behaviors of 
semipolynomial fractions as those in the case of $r$-singularity described earlier. 
However, as  $D_\kappa(G,G_0)$ is an inhomogeneous 
semipolynomial, the limiting behaviors of ratios $\xi_l^{(\kappa)}(G)/\widetilde{W}_
\kappa^{\|\kappa\|_{\infty}}(G,G_0)$ is generally more challenging to investigate than 
those of ratios from $r$-minimal form in Definition \ref{def-rsingular}. This will be seen via examples in the sequel.

The foregoing description, along with the presentation in the previous subsection
on the construction of $r$-minimal and $\kappa$-minimal forms, 
provides the outline of a general procedure which
links the determination of the singularity structure of parameter space to the solvability of 
a system of polynomial limits. This procedure will be illustrated carefully in Section \ref{Section:overfitskew} 
for the remarkable world of mixtures of skew-normal distributions.


\section{O-mixtures of skew-normal distributions}
\label{Section:overfitskew}

In this section, we study parameter estimation behavior for skew-normal mixtures.
Our motivation is two-fold: First, as discussed in the Introduction, skew-normal mixture models
are widely embraced in applications despite little or no known theoretical results, so understanding their 
theoretical properties is of interest in their own right. Second, skew-normal mixtures appear to be an ideal 
illustration for the general theory and tools developed in the previous section, which 
helps to shed some light on
the remarkably complex structure of a mixture distribution's parameter space. 

On the flip side, our presentation will be necessarily (and unfortunately)
quite technical. To alleviate the technicality, 
we focus the presentation in this section on singularity structures
in the overfitted setting subject to certain restrictions on probability mass and other parameters.\footnote{Specifically, consider 
$G_0 \in \Ecal_{k_0}$ relative to ambient space $\Ocal_{k,c_{0}}$ for some $k>k_0$ and small constant $c_{0}>0$ where $O_{k,c_{0}} \subset \Ocal_k$ contains 
only (discrete) probability measures whose point masses are bounded from below by $c_{0}$.
Moreover, we will analyze the singularity structure of $G_0 \in \Scal_0$,
a subset to be defined shortly by Eq. \eqref{eqn:generic_components}.}
This case is quite interesting because it illustrates the full power
of the general method of analysis that was described in 
Section \ref{Section:general_procedure_singularity} in a concrete fashion, yielding a
general result while revealing sufficiently complex structures.

For readers interested in the finer details of skew-normal mixtures, in Appendix E we summarize the singularity 
structure of $G_0$ relative to 
the ambient space $\Ecal_{k_0}$ (that is, e-mixture setting), for which
a more complete picture of the singularity structure is achieved, with full details laid out in Appendix F.
For o-mixtures, further results can be found in a technical report \cite{Ho-Nguyen-skewnormalTR}. 
Such elaborate picture might be appreciated by practitioners and experts of the skew-normals, but 
they can be safely skipped by the rest of the audience.

\begin{lemma}\label{proposition-notskewnormal}
For skew-normal density kernel $f(x|\veceta)$, the collection of 
$\left\{\partial^\kappa f/\partial \eta^\kappa(x|\eta_j)
| j=1,\ldots, k_0; \right. \\ \left. 0\leq |\kappa| \leq 1 \right\}$
is not linearly independent if and only if $\veceta = (\eta_1,\ldots,\eta_k)$ 
are the zeros of either polynomial $P_1$ or $P_2$, which are defined as follows:
\begin{itemize}
\item [] Type A: $P_{1}(\myeta) =  \prod \limits_{j=1}^{k_{0}}{m_{j}}$.
\item [] Type B: $P_{2} (\myeta) = \prod \limits_{1 \leq i \neq j \leq k_{0}}
\biggr 
\{ (\theta_{i}-\theta_{j})^{2} + \biggr [ \sigma_i^2(1+m_j^2) - \sigma_j^2(1+m_i^2) 
\biggr ]^2 \biggr \}$.
\end{itemize}
\end{lemma}
This lemma leads us to consider
\begin{eqnarray}
\Scal_0 = \biggr \{G = G(\vecp, \veceta) \biggr | (\vecp, \veceta) \in \Omega,
P_1(\veceta) \neq 0, P_2(\veceta) \neq 0 \biggr \}. \label{eqn:generic_components}
\end{eqnarray}

\comment{
The partitioning of $\Ecal_{k_0}$ is as follows
\[\Ecal_{k_0} = \cup_{i=0}^{3}\Scal'_i,\]
where $\Scal'_0 = \Scal_0$, and $\Scal'_1 = \Scal_1$. In addition,
\begin{itemize}
\item $\Scal'_{2} = \left\{G \in \mathcal{E}_{k_{0}} \ | \ P_{1}(\myeta) = 0  \right\}$.
\item $\Scal'_{3} = \left\{G \in \mathcal{E}_{k_{0}} \ | \ P_{1}(\myeta) \neq 0 \ , \ 
P_{2}(\myeta) = 0 \ \text{and} \ G \ \text{is nonconformant} \right\}$.
\end{itemize}
As we can see, $\Scal_{2} \subset \Scal'_{2}$ as we allow the nonconformant 
homologous sets to appear in $\Scal_{2}'$. On the other hand, $\Scal_{3}' \subset 
\Scal_{3}$ as we do not allow Gaussian components to appear in the elements of $
\Scal_{3}'$. As in the case with $\Scal_3$, the singularity structure of $\Scal'_3$
is very complex, and will be delineated based on singularity types
C(1) and C(2) described in Section \ref{Section:nonconformant_setting}. In fact,
\begin{itemize}
\item $\Scal'_{31}=\left\{G \in \Scal'_{3} \ | \ P_{3}(\myeta) \neq 0 \right\}$
\item $\Scal'_{32} =\left\{G \in \Scal'_{3} \ | \ P_{3}(\myeta) = 0 \ , P_{4} (\myeta) \neq 0 \right\}$
\item $\Scal'_{33}=\left\{G \in \Scal'_{3} \ | \ P_{3}(\myeta) = 0 \ , 
P_{4} (\myeta) = 0 \right\}$.
\end{itemize} 
}


In o-mixtures, we will see that $\lev(G_0|\Ocal_{k,c_0})$ and $\singset(G_{0}|\Ocal_{k,c_{0}})$ may grow with
$k-k_0$, the number of extra mixing components. The main excercise is to
arrive at suitable $r$-minimal and $\kappa$-minimal forms, for which the behavior of
its coefficients can be analyzed. Section~\ref{sec:r-form} describes a general
strategy for the construction of $r$-minimal form (or equivalently $(r,r,r)$-minimal form) based on the partial derivatives
of the density kernel $f$ with respect to the parameters $\eta=(\theta, v, m)$ up
to order $r$. This is also a strategy that we would like to utilize for $\kappa$-minimal forms for any $\kappa \in \mathbb{N}^{3}$.

For skew-normal kernel density $f$, the following lemma provides an explicit form 
for reducing a partial derivative of $f$ to other partial derivatives of lower orders. 
\begin{lemma}\label{lemma:recursiveequation}
For any $r \geq 1$, denote
\begin{eqnarray}
A_{1}^{r} & = & \left\{(\alpha_{1},\alpha_{2},\alpha_{3}): \ \alpha_{1} \leq 1, \ \alpha_{3}=0, \ \text{and} \ |\alpha| \leq r \right\}. \nonumber \\
A_{2}^{r} & = & \left\{(\alpha_{1},\alpha_{2},\alpha_{3}): \ \alpha_{1} \leq 1, \alpha_{2}=0, \alpha_{3} \geq 1, \ \text{and} \ |\alpha| \leq r \right\}. \nonumber \\
\mathcal{F}_{r} & = & A_{1}^{r} \cup A_{2}^{r}.\nonumber
\end{eqnarray}
Let $f(x|\eta) = f(x|\theta, v, m)$ denote the skew-normal kernel. Then, for any $\alpha=(\alpha_{1},\alpha_{2},\alpha_{3}) \in \mathbb{N}^3$ and $m \neq 0$, there holds
\begin{eqnarray}
\dfrac{\partial^{|\alpha|}{f}}{\partial{\theta^{\alpha_{1}}\partial{v^{\alpha_{2}}}
\partial{m^{\alpha_{3}}}}} &=& \sum 
\dfrac{P_{\alpha_{1},\alpha_{2},\alpha_{3}}
^{\kappa_{1},\kappa_{2},\kappa_{3}}(m)}{H_{\alpha_{1},\alpha_{2},\alpha_{3}}
^{\kappa_{1},\kappa_{2},\kappa_{3}}(m)Q_{\alpha_{1},\alpha_{2},\alpha_{3}}
^{\kappa_{1},\kappa_{2},\kappa_{3}}(v)}\dfrac{\partial^{|\kappa|}{f}}
{\partial{\theta^{\kappa_{1}}\partial{v}^{\kappa_{2}}\partial{m}^{\kappa_{3}}}}, 
\nonumber
\end{eqnarray}
where $\kappa$ in the above sum satisfies $\kappa \in \mathcal{F}_{|\alpha|}$ and $\kappa_{1}+2\kappa_{2}+2\kappa_{3} \leq \alpha_{1}+2\alpha_{2}+2\alpha_{3}$. Additionally, $P_{\alpha_{1},\alpha_{2},\alpha_{3}}
^{\kappa_{1},\kappa_{2},\kappa_{3}}(m)$, $H_{\alpha_{1},\alpha_{2},\alpha_{3}}
^{\kappa_{1},\kappa_{2},\kappa_{3}}(m)$, and $Q_{\alpha_{1},\alpha_{2},\alpha_{3}}
^{\kappa_{1},\kappa_{2},\kappa_{3}}(v)$ are polynomials in terms of $m,m$ and $v$, 
respectively.
\end{lemma}

Next, we show that the partial derivatives on the RHS of the above identity are in fact linearly independent,
under additional assumptions on $G_0$.

\begin{lemma}\label{lemma:reduced_linearly_independent} 
Recall the notation from Lemma \ref{lemma:recursiveequation}.
If $G_0 \in \mathcal{S}_0$, then for any $r \geq 1$,
the collection of partial derivatives of the skew-normal density kernel $f(x|\eta)$, namely
\[\left\{\dfrac{\partial^{|\kappa|}{f}(x|\eta)}
{\partial \theta^{\kappa_{1}}\partial v^{\kappa_{2}}\partial m^{\kappa_{3}}} \biggr |
\ \kappa=(\kappa_{1},\kappa_{2},\kappa_{3}) \in \mathcal{F}_{r}, \eta=\eta_1^{0},\ldots,\eta_{k_0}^{0}\right\}\] 
are linearly independent.
\end{lemma}

Figure \ref{figure:reduced_derivatives} gives an illustration of Lemma \ref{lemma:reduced_linearly_independent} when $r=3$. Armed with the foregoing lemmas we can easily obtain a suitable minimal form for the
mixture densities of skew-normals.

\subsection{Illustrations via special cases} \label{Section:illustration_omixture_byone}

To illustrate our techniques and results, consider a special case 
in which $G_0$ has exactly one atom, and $k=k_0+1=2$. 
The general result is presented in Section \ref{Section:general_bound_omixtures_generic}.
As a warm-up exercise, we verify that

\paragraph{Claim: $G_0$ is 1-singular and (1,1,1)-singular.} Indeed, $G_0\in \Scal_0$ implies that all
first order derivatives of $f$ are linearly independent. 
Hence, from Eq.~\eqref{eqn:generallinearindependencerepresentation} we obtain
(1,1,1)-minimal form:
\begin{eqnarray}
\dfrac{p_{G}(x)-p_{G_{0}}(x)}{W_{1}(G,G_{0})} 
\asymp \dfrac{1}{W_{1}(G,G_{0})}\biggr(\Delta p_{1.} 
f(x|\eta_1^0) + 
\sum \limits_{i=1}^{2}{p_{1i} \Delta \theta_{1i}}\dfrac{\partial{f}}
{\partial{\theta}}(x|\eta_1^0) \nonumber \\  
+ \sum \limits_{i=1}^{2}{p_{1i} \Delta v_{1i}}\dfrac{\partial{f}}{\partial{v}}
(x|\eta_1^0) + \sum \limits_{i=1}^{2}{p_{1i} \Delta m_{1i}}\dfrac{\partial{f}}{\partial{m}}(x|\eta_1^0)\biggr)+o(1). \label{eqn:taylorexpansionfirstorder_overfit}
\end{eqnarray}
Since $k=2$ and $k_{0}=1$, $\Delta p_{1.}=0$. A sequence of $G$ can be easily chosen 
so that $\sum \limits_{i=1}^{2}{p_{1i} 
\Delta \theta_{1i}}=0$, $\sum \limits_{i=1}^{2}{p_{1i} \Delta v_{1i}}=0$, $\sum 
\limits_{i=1}^{2}{p_{1i} \Delta m_{1i}}=0$, so that all of the coefficients in 
\eqref{eqn:taylorexpansionfirstorder_overfit} are $0$. 
Hence, $G_0$ is 1-singular and $(1,1,1)$-singular relative to $\mathcal{O}_{2,c_{0}}$.
In light of Prop.~\ref{proposition:strong_identifiability_singularity_index}, it
is non-trivial to show that

\paragraph{Claim: $G_0$ is 2-singular and (2,2,2)-singular.} 
Indeed, using the method of elimination described in Example \ref{ex-reduce-3}
we obtain the following 2-minimal and $(2,2,2)$-minimal form:
\begin{eqnarray}
\dfrac{1}{W_{2}^{2}(G,G_{0})}\biggr(\sum \limits_{\kappa \in \mathcal{F}_2}{\xi_{\kappa_{1},\kappa_{2},\kappa_{3}}^{(2)}\dfrac{\partial^{|\kappa|}{f}}{\partial{\theta^{\kappa_{1}}}\partial{v^{\kappa_{2}}}\partial{m^{\kappa_{3}}}}(x|\eta_1^0)}\biggr)+o(1), \label{eqn:taylorexpansion_s0_second}
\end{eqnarray}
where $\xi_\kappa^{(2)} \equiv \xi_{\kappa_{1},\kappa_{2},\kappa_{3}}^{(2)}$ are given by
\begin{eqnarray}
\xi_{1,0,0}^{(2)}=\sum \limits_{i=1}^{2} {p_{1i}\Delta \theta_{1i}}, \ \xi_{0,1,0}^{(2)}=\sum \limits_{i=1}^{2} p_{1i}\Delta v_{1i}+\sum 
\limits_{i=1}^{2}{p_{1i}(\Delta \theta_{1i})^{2}}, \nonumber \\
\xi_{0,0,1}^{(2)}=-\dfrac{(m_{1}^{0})^{3}+m_{1}^{0}}{2v_{1}^{0}}
\sum \limits_{i=1}^{2}{p_{1i}(\Delta \theta_{1i})^{2}}-\dfrac{1}{v_{1}^{0}}\sum \limits_{i=1}^{2}{p_{1i}\Delta v_{1i}\Delta m_{1i}}+ \sum \limits_{i=1}^{2}{p_{1i}
\Delta m_{1i}}, \nonumber \\
\xi_{0,2,0}^{(2)}=\sum \limits_{i=1}^{2} {p_{1i}(\Delta v_{1i})^{2}}, \;
\xi_{0,0,2}^{(2)}=-\dfrac{(m_{1}^{0})^{2}+1}{2v_{1}^{0}m_{1}^{0}}
\sum \limits_{i=1}^{2}{p_{1i}\Delta v_{1i}\Delta m_{1i}} +
\sum \limits_{i=1}^{2} {p_{1i}\Delta (m_{1i})^{2}}, \nonumber \\
\xi_{1,1,0}^{(2)}=\sum \limits_{i=1}^{2} {p_{1i}\Delta \theta_{1i}\Delta v_{1i}}, \;
\xi_{1,0,1}^{(2)}=\sum \limits_{i=1}^{2} {p_{1i}\Delta \theta_{1i}\Delta m_{1i}}, \nonumber
\end{eqnarray} 
\comment{
\begin{eqnarray}
\xi_{1,0,0}^{(2)}=\sum \limits_{i=1}^{2} {p_{1i}\Delta \theta_{1i}}, \ \xi_{0,1,0}^{(2)}=\sum \limits_{i=1}^{2} p_{1i}\Delta v_{1i}+\sum 
\limits_{i=1}^{2}{p_{1i}(\Delta \theta_{1i})^{2}}, \nonumber \\
\xi_{0,0,1}^{(2)}=-\dfrac{(m_{1}^{0})^{3}+m_{1}^{0}}{2v_{1}^{0}}
\sum \limits_{i=1}^{2}{p_{1i}(\Delta \theta_{1i})^{2}}+ \sum \limits_{i=1}^{2}{p_{1i}
\Delta m_{1i}}, \ \xi_{0,2,0}^{(2)}=\sum \limits_{i=1}^{2} {p_{1i}(\Delta v_{1i})^{2}}, \nonumber \\
\xi_{0,0,2}^{(2)}=\sum \limits_{i=1}^{2} {p_{1i}(\Delta m_{1i})^{2}}, \ \xi_{1,1,0}^{(2)}=\sum \limits_{i=1}^{2} {p_{1i}\Delta \theta_{1i}\Delta v_{1i}}, \nonumber \\
\xi_{1,0,1}^{(2)}=\sum \limits_{i=1}^{2} {p_{1i}\Delta \theta_{1i}\Delta m_{1i}}, \ \xi_{0,1,1}^{(2)}=\sum \limits_{i=1}^{2} {p_{1i}\Delta v_{1i}\Delta m_{1i}}. \nonumber
\end{eqnarray} 
}
where we note in particular that the formulas for
$\xi_{0,1,0}^{(2)}$, $\xi_{0,0,1}^{(2)}$ and $\xi_{0,0,2}^{(2)}$
are the results of the elimination step via Eqs.
~\eqref{Eqn-reduction-1} and \eqref{Eqn-reduction-2}. 

It remains to construct a sequence of $G$ tending to $G_0$
so that $\xi_{\kappa}^{(2)}/W_2^2(G,G_0)$ vanish for all 
$\kappa = (\kappa_1,\kappa_2,\kappa_3) \in \mathcal{F}_2$. Define
\begin{eqnarray}
\overline{M}=\mathop {\max }{\left\{|\Delta \theta_{11}|, |\Delta \theta_{12}|, |\Delta v_{11}|^{1/2}, |\Delta v_{12}|^{1/2}, |\Delta m_{11}|^{1/2}, |\Delta m_{12}|^{1/2}\right\}}. \nonumber
\end{eqnarray} 
Then, $\xi_{\kappa_{1},\kappa_{2},\kappa_{3}}^{(2)} = O(\overline{M}^{\kappa_{1}+2\kappa_{2}+2\kappa_{3}})$. Moreover,
it follows from Lemma ~\ref{lemma:bound_overfit_Wasserstein} that $W_{2}^{2}(G,G_{0}) \lesssim \overline{M}^{2}$. 
The sequence of $G$ and the corresponding $\Delta \theta_{1i}, \Delta v_{1i}$, and $\Delta m_{1i}$ will be chosen so that 
$W_{2}^{2}(G,G_{0}) \asymp \overline{M}^{2}$, and for all $\kappa \in \mathcal{F}_2$, 
$\xi_{\kappa_{1},\kappa_{2},\kappa_{3}}^{(2)}/W_{2}^{2}(G,G_{0}) \to 0$. 
For any $\kappa \in \mathcal{F}_2$ such that 
$\kappa_{1}+2\kappa_{2}+2\kappa_{3} \geq 3$, we have
$\xi_{\kappa_{1},\kappa_{2},\kappa_{3}}^{(2)}=O(\overline{M}^{s})$ where $s \geq 3$, so
$\xi_{\kappa_{1},\kappa_{2},\kappa_{3}}^{(2)}/W_{2}^{2}(G,G_{0}) \to 0$. Thus,
we only need to consider the coefficients for which $\kappa_{1}+2\kappa_{2}+2\kappa_{3} \leq 2$ and $\kappa_{1} \leq 1$. 
They include $\xi_{1,0,0}^{(2)}/W_2^2(G,G_0)$, $\xi_{0,1,0}^{(2)}/W_2^2(G,G_0)$, 
and $\xi_{0,0,1}^{(2)}/W_2^2(G,G_0)$. 
We will see shortly that if these terms vanish, we must be able to extract a solvable system of polynomial equations. 

Let $\Delta \theta_{1i}/\overline{M} \to a_{i}$, $\Delta v_{1i}/\overline{M}^{2} \to 
b_{i}$, $\Delta m_{1i}/\overline{M}^{2} \to c_{i}$, $p_{1i} \to d_{i}^{2}$ for 
all $1 \leq i \leq 2$ (such limits always exist for some subsequence of $G$, in which case we use to replace $G$). 
For $\xi_{1,0,0}^{(2)}/W_2^2(G,G_0) \to 0$ to hold, by dividing its numerator and denominator by $\overline{M}$, it is necessary that
\begin{align}
d_{1}^{2}a_{1}+d_{2}^{2}a_{2}=0. \label{eqn:taylorexpansion_s0_second_first_term}
\end{align}
For $\xi_{0,1,0}^{(2)}/W_2^2(G,G_0) \to 0$, by dividing its numerator and denominator by $\overline{M}^{2}$, the following holds
\begin{align}
d_{1}^{2}a_{1}^{2}+d_{2}^{2}a_{2}^{2}+d_{1}^{2}b_{1}+d_{2}^{2}b_{2}=0. \label{eqn:taylorexpansion_s0_second_second_term}
\end{align} 
Finally, for $\xi_{0,0,1}^{(2)}/W_2^2(G,G_0) \to 0$, by dividing its numerator and denominator by $\overline{M}^{2}$
and noting that
$\left(\sum_{i=1}^{2}{p_{1i}\Delta v_{1i}\Delta m_{1i}}\right)/W_{2}^{2}(G,G_{0}) = O(\overline{M}^{2}) \to 0$,
we obtain
\begin{align}
-\dfrac{(m_{1}^{0})^{3}+m_{1}^{0}}{2v_{1}^{0}}(d_{1}^{2}a_{1}^{2}+d_{2}^{2}a_{2}^{2})+d_{1}^{2}c_{1}+d_{2}^{2}c_{2}=0. \label{eqn:taylorexpansion_s0_second_third_term}
\end{align}
Thus, we have obtained a system of polynomial equations 
\eqref{eqn:taylorexpansion_s0_second_first_term}, \eqref{eqn:taylorexpansion_s0_second_second_term}, and \eqref{eqn:taylorexpansion_s0_second_third_term}.
\comment{
Combining the equations from \eqref{eqn:taylorexpansion_s0_second_first_term}, \eqref{eqn:taylorexpansion_s0_second_second_term}, and \eqref{eqn:taylorexpansion_s0_second_third_term}, as long as $\xi_{\kappa_{1},\kappa_{2},\kappa_{3}}^{(2)}/W_{2}^{2}(G,G_{0}) \to 0$ for $\kappa_{1}+2\kappa_{2}+2\kappa_{3} \leq 2$ and $\kappa_{1} \leq 1$,
we achieve the following system of polynomial equations
\begin{eqnarray}
d_{1}^{2}a_{1}+d_{2}^{2}a_{2}=0, \nonumber \\
d_{1}^{2}a_{1}^{2}+d_{2}^{2}a_{2}^{2}+d_{1}^{2}b_{1}+d_{2}^{2}b_{2}=0, \nonumber \\
-\dfrac{(m_{1}^{0})^{3}+m_{1}^{0}}{2v_{1}^{0}}(d_{1}^{2}a_{1}^{2}+d_{2}^{2}a_{2}^{2})+d_{1}^{2}c_{1}+d_{2}^{2}c_{2}=0. \label{eqn:systemnonlinear_second}
\end{eqnarray} 
}
Since $p_{11}+p_{12}=p_{1}^{0}=1$ and $p_{1i} \geq 
c_{0}$ for all $1 \leq i \leq 2$, we have $d_{1}^{2}+d_{2}^{2}=1$ and $d_{i}^{2}$ are bounded away from 0
for $1 \leq i \leq 2$. Additionally, at least one among $a_{1},a_{2},b_{1},b_{2},c_{1},c_{2}$ is non-zero. 

\comment{Note that the constraint $d_{1}^{2}
+d_{2}^{2}=1$ can be efficiently removed by reparametrizing $d_{i}=(d_{i}^{'})^{2}/
(d_{1}^{'})^{2}+(d_{2}^{'})^{2})$ for all $1 \leq i \leq 2$ where $d_{1}^{'},d_{2}^{'} \neq 0$. 
Therefore, herethereupon, under this section, when we talk about system of polynomial equations like 
\eqref{eqn:systemnonlinear_second}, we implicitly remove the constraint on the summation 
of $d_{1}^{2},d_{2}^{2}$. }

One non-trivial solution to the above system of polynomial 
equations is  $d_{1}=d_{2}$, $a_{1}=-a_{2}$, $b_{1}=b_{2}=-a_{1}^{2}$, $c_{1}
=c_{2}=[(m_{1}^{0})^{3}+m_{1}^{0}]a_{1}^2/(2v_{1}^{0})$. Given this solution, we can now select a sequence of $G$ 
by letting $p_{11}=p_{12} = 1/2$,
$\Delta \theta_{11} = -\Delta \theta_{12}$, 
$\Delta v_{11} = \Delta v_{12} = -(\Delta \theta_{11})^2$, and
$\Delta m_{11} = \Delta m_{12} =  
(\Delta \theta_{11})^2[(m_{1}^{0})^{3}+m_{1}^{0}]/(2v_{1}^{0})$. It is simple to verify that
$W_{2}^{2}(G,G_{0}) \asymp \overline{M}^{2}$ and $\xi_{\kappa_{1},\kappa_{2},\kappa_{3}}^{(2)}/W_{2}^{2}(G,G_{0}) \to 0$ for $\kappa_{1}+2\kappa_{2}+2\kappa_{3} \leq 2$ and $\kappa_{1} \leq 1$, i.e., 
all coefficients of the 2-minimal form vanish to 0. Hence, $G_0$ is 2-singular and (2,2,2)-singular
relative to $\Ocal_{2,c_{0}}$.
   
\begin{figure}
\begin{tikzpicture}
  \matrix (m) [matrix of math nodes,row sep=2.0 em,column sep= 0 em,minimum width=0.05em]
{ \dfrac{\partial{f}}{\partial{\theta}}& \dfrac{\partial{f}}{\partial{v}}& \dfrac{\partial{f}}{\partial{m}}& & & & & & & \\
\mathcircled{\dfrac{\partial^{2}{f}}{\partial{\theta^{2}}}}& \dfrac{\partial^{2}{f}}{\partial{v^{2}}}& \dfrac{\partial^{2}{f}}{\theta{m^{2}}} & & \dfrac{\partial^{2}{f}}{\partial{\theta}\partial{v}}& \dfrac{\partial^{2}{f}}{\partial{\theta}\partial{m}}&  \mathcircled{\dfrac{\partial^{2}{f}}{\partial{v}\theta{m}}}& & & &  \\
\mathcircled{\dfrac{\partial^{3}{f}}{\partial{\theta^{3}}}}&\dfrac{\partial^{3}{f}}{\partial{v^{3}}}& \dfrac{\partial^{3}{f}}{\theta{m^{3}}}& \mathcircled{\dfrac{\partial^{3}{f}}{\partial{\theta^{2}}\partial{v}}} & \mathcircled{\dfrac{\partial^{3}{f}}{\partial{\theta^{2}}\partial{m}}} &\mathcircled{\dfrac{\partial^{3}{f}}{\partial{v^{2}}\partial{m}}} &\dfrac{\partial^{3}{f}}{\partial{\theta}\partial{v^{2}}} & \dfrac{\partial^{3}{f}}{\partial{\theta}\partial{m^{2}}} & \mathcircled{\dfrac{\partial^{3}{f}}{\partial{v}\partial{m^{2}}}} &\mathcircled{\dfrac{\partial^{3}{f}}{\partial{\theta}\partial{v}\partial{m}}}  \\
};

\path[-stealth]
(m-2-1) edge (m-1-2) edge (m-1-3)
(m-3-1) edge (m-2-5) edge (m-2-6)
(m-3-4) edge (m-2-2) edge (m-2-7) edge (m-1-3)
(m-3-5) edge (m-2-3) edge (m-2-7) edge (m-1-3)
(m-2-7) edge [bend right] (m-2-3)
(m-2-7) edge (m-1-3)
(m-3-6) edge [bend left=35] (m-3-3)
(m-3-6) edge (m-2-3)
(m-3-6) edge (m-1-3)
(m-3-9) edge [bend left=35] (m-3-3)
(m-3-9) edge (m-2-3)
(m-3-10) edge [bend left=40] (m-3-8)
(m-3-10) edge (m-2-6);
\draw [->] (m-3-5) -- (m-2-3);

\end{tikzpicture}

\caption{Illustration of the elimination steps from a complete collection of 
derivatives of $f$ up to the order 3 to a reduced system of linearly independent
partial derivatives, cf. Lemma \ref{lemma:reduced_linearly_independent}. The circled 
partial derivatives are eliminated from the partial derivatives present in the 3-minimal form.
$A \to B$ means that 
$B$ is included in the representation of the minimal form when $A$ is eliminated.}
\label{figure:reduced_derivatives}
\end{figure}
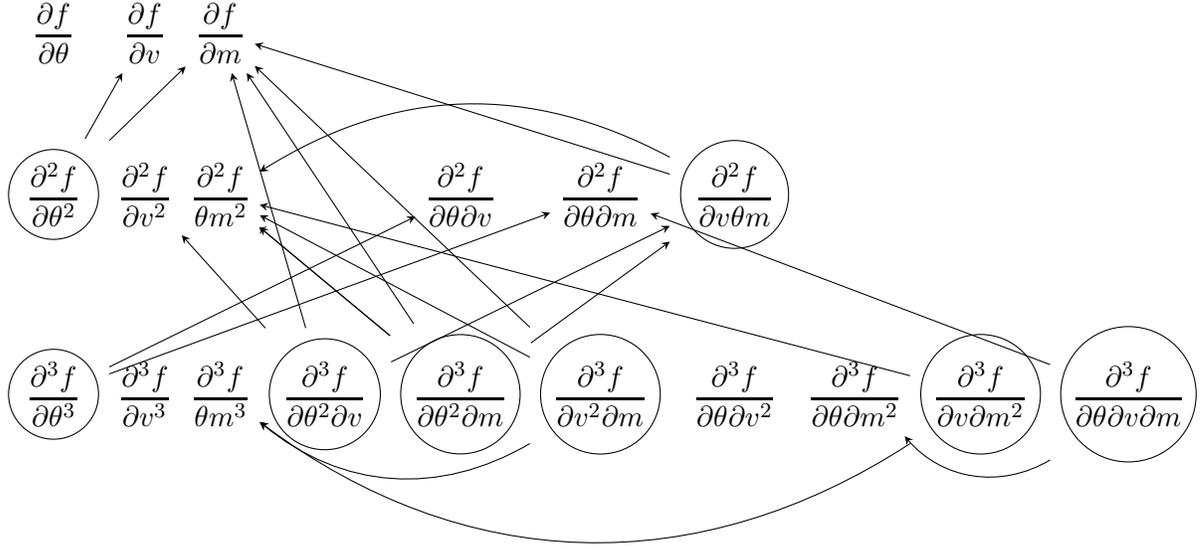

\paragraph{Claim: $G_0$ is 3-singular and (3,3,3)-singular.} The proof for this is similar to the argument that $G_{0}$ is 2-singular and $(2,2,2)$-singular. 
In particular, a 3-minimal and (3,3,3)-minimal
form can be obtained by applying the reductions \eqref{eqn:nonlinearequations_third},
which eliminate all third order partial derivatives in terms of lower order ones
that are in fact linearly independent by the condition that $G_0 \in \Scal_0$.
As in the foregoing paragraphs, as long as $\Delta \theta_{1i}, \Delta v_{1i}$, and $\Delta m_{1i}$ are chosen such that $W_{3}^{3}(G,G_{0}) \asymp \overline{M}^{3}$, we can obtain a system of polynomials
that turn out to share the same solution as the one described. 
This leads to the same choice of sequence for $G$ according to which all 
coefficients of the 3-minimal form vanish to 0. Thus, $G_0$ is 3-singular and (3,3,3)-singular relatively
to $\Ocal_{2,c_0}$. In order to establish the singularity level and singularity index of $G_0$, we will show that

\paragraph{Claim: $G_0$ is not (4,2,2)-singular.} 
This claim immediately entails, thanks to Lemma \ref{lemma:minimal_form} and Lemma \ref{lemma:minimal_form_general}, 
that $G_{0}$ is not 4-singular and (4,4,4)-singular relative to $\Ocal_{k,c_{0}}$, and so by definition
$\lev(G_{0}|\Ocal_{2,c_{0}})=3$.  Indeed, following the same approach as above, we obtain a $(4,2,2)$-minimal form and 
their rational semipolynomial coefficients, from which we extract the following system of real polynomial limits:
\begin{eqnarray}
d_{1}^{2}a_{1}+d_{2}^{2}a_{2}=0, \nonumber \\
d_{1}^{2}a_{1}^{2}+d_{2}^{2}a_{2}^{2}+d_{1}^{2}b_{1}+d_{2}^{2}b_{2}=0, \nonumber \\
-\dfrac{(m_{1}^{0})^{3}+m_{1}^{0}}{2v_{1}^{0}}(d_{1}^{2}a_{1}^{2}+d_{2}^{2}a_{2}
^{2})+d_{1}^{2}c_{1}+d_{2}^{2}c_{2}=0, \nonumber \\
\dfrac{1}{3}(d_{1}^{2}a_{1}^{3}+d_{2}^{2}a_{2}^{3})+d_{1}^{2}a_{1}b_{1}+d_{2}^{2}
a_{2}b_{2}=0, \nonumber \\
-\dfrac{(m_{1}^{0})^{3}+m_{1}^{0}}{6v_{1}^{0}}(d_{1}^{2}a_{1}^{3}+d_{2}^{2}a_{2}
^{3})+d_{1}^{2}a_{1}c_{1}+d_{2}^{2}a_{2}c_{2}=0, \nonumber \\
\dfrac{1}{6}(d_{1}^{2}a_{1}^{4}+d_{2}^{2}a_{2}^{4})+d_{1}^{2}a_{1}^{2}b_{1}+d_{2}
^{2}a_{2}^{2}b_{2}+\dfrac{1}{2}(d_{1}^{2}b_{1}^{2}+d_{2}^{2}b_{2}^{2})=0, \nonumber \\
\dfrac{((m_{1}^{0})^{3}+m_{1}^{0})^{2}}{12(v_{1}^{0})^{2}}(d_{1}^{2}a_{1}^{4}
+d_{2}^{2}a_{2}^{4})-\dfrac{(m_{1}^{0})^{3}+m_{1}^{0}}{v_{1}^{0}}(d_{1}^{2}a_{1}
^{2}c_{1}+d_{2}^{2}a_{2}^{2}c_{2}) - \nonumber \\
\dfrac{(m_{1}^{0})^{2}+1}{v_{1}^{0}m_{1}^{0}}(d_{1}^{2}b_{1}c_{1}+d_{2}^{2}b_{2}c_{2})+d_{1}^{2}c_{1}^{2}+d_{2}^{2}c_{2}^{2}=0, \label{eqn:systemnonlinear_fourth}
\end{eqnarray}
such that at least one among 
$a_{1},a_{2},b_{1},b_{2},c_{1},c_{2}$ is non-zero and $d_{1},d_{2} \neq 0$. 

At the first glance, the behavior of this system appears dependent on the specific
value of $v_1^0,m_1^0$. However, if we remove the third, fifth and eighth equations,
we obtain a system of real polynomials that does not depend on the specific value of $G_0$. In fact, it
can be verified that this system does \emph{not} admit any non-trivial real solution, using
a standard tool (Groebner bases method) from computational algebra \cite{Cox-etal}.
Thus, there does not exist any sequence of $G \in \Ocal_{2,c_0}$
according to which all coefficients of the 4-minimal form vanish. 
This implies that $G_0$ is \emph{not} (4,2,2)-singular relative to $\Ocal_{2,c_{0}}$. 
We proceed to show 
\paragraph{Claim: $\singset(G_{0}|\Ocal_{2,c_{0}})=\left\{(4,2,2)\right\}$.}
It suffices to verify that $G_{0}$ is $(3,r,r)$-singular, $(r,1,r)$-singular, and $(r,r,1)
$-singular for any $r \geq 1$. As $G_{0}$ is $(3,3,3)$-singular relative to $
\Ocal_{2,c_{0}}$, it is sufficient to validate the previous claims when $r \geq 4$. 
Select the same sequence of $G$ as in the foregoing argument, i.e., $p_{11}=p_{12} = 1/2$,
$\Delta \theta_{11} = -\Delta \theta_{12}$, 
$\Delta v_{11} = \Delta v_{12} = -(\Delta \theta_{11})^2$, and
$\Delta m_{11} = \Delta m_{12} =  
(\Delta \theta_{11})^2[(m_{1}^{0})^{3}+m_{1}^{0}]/(2v_{1}^{0})$. If $\kappa = (3,r,r)$, then $\|\kappa\|_\infty = r \geq 4$. It is simple to verify that
\begin{align}
\widetilde{W}_{\kappa}^{\|\kappa\|_{\infty}}(G,G_{0}) = \widetilde{W}_{\kappa}^{r}(G,G_{0}) \asymp |\Delta \theta_{11}|^{3}. \nonumber
\end{align}
 Similarly, if $\kappa \in \left\{(r,1,r), (r,r,1)\right\}$ and $r \geq 4$, we have
\begin{align}
\widetilde{W}_{\kappa}^{\|\kappa\|_{\infty}}(G,G_{0}) = \widetilde{W}_{\kappa}^{r}(G,G_{0}) \asymp |\Delta \theta_{11}|^2 \gg |\Delta \theta_{11}|^{3}. \nonumber
\end{align}
Hence, as long as $\kappa \in \mathcal{U}_{r} := \left\{(3,r,r),(r,1,r),(r,r,1)\right\}$ and $r \geq 4$, we have that $\widetilde{W}
_{\kappa}^{\|\kappa\|_{\infty}}(G,G_{0}) \gtrsim |\Delta \theta_{11}|
^{3}$. Due to the choices of $\Delta \theta_{1i}$, $\Delta v_{1i}$, 
and $\Delta m_{1i}$, it is easy to check that
\begin{align}
|\sum \limits_{i=1}^{2} p_{1i}(\Delta \theta_{1i})^{\alpha_{1}}(\Delta v_{1i})^{\alpha_{2}}(\Delta m_{1i})^{\alpha_{3}}| = O(|\Delta \theta_{11}|^{\alpha_{1}+2\alpha_{2}+2\alpha_{3}}) \nonumber
\end{align}
for any $\alpha = (\alpha_{1},\alpha_{2},\alpha_{3})$. Invoking the previous bounds, the following limit holds 
\begin{eqnarray}
\dfrac{|\sum \limits_{i=1}^{2} p_{1i}(\Delta \theta_{1i})^{\alpha_{1}}(\Delta v_{1i})^{\alpha_{2}}(\Delta m_{1i})^{\alpha_{3}}|}{\widetilde{W}_{\kappa}^{\|\kappa\|_{\infty}}(G,G_{0})} \lesssim \dfrac{|\Delta \theta_{11}|^{\alpha_{1}+2\alpha_{2}+2\alpha_{3}}}{|\Delta \theta_{11}|^{3}} \to 0 \label{eqn:fundamental_limits}
\end{eqnarray} 
for any $\alpha_{1}+2\alpha_{2}+2\alpha_{3} \geq 4$ and $\kappa \in \mathcal{U}_{r}$ where $r \geq 4$.

Now, for any $\kappa \in \mathcal{U}_{r}$ and $r \geq 4$, we obtain the following $\kappa$-minimal form:
\begin{eqnarray}
\dfrac{1}{\widetilde{W}_{\kappa}^{\|\kappa\|_{\infty}}(G,G_{0})}\biggr(\sum \limits_{\tau \in \mathcal{F}_r}{\xi_{\tau_{1},\tau_{2},\tau_{3}}^{(r)}\dfrac{\partial^{|\tau|}{f}}{\partial{\theta^{\tau_{1}}}\partial{v^{\tau_{2}}}\partial{m^{\tau_{3}}}}(x|\eta_1^0)}\biggr)+o(1). \nonumber
\end{eqnarray}

Similar to the formulations of $\xi_{\tau_{1},\tau_{2},
\tau_{3}}^{(2)}$  in ~\eqref{eqn:taylorexpansion_s0_second}, $
\xi_{\tau_{1},\tau_{2},\tau_{3}}^{(r)}$ will only consist of 
monomials of the following form $\sum \limits_{i=1}^{2} p_{1i}(\Delta 
\theta_{1i})^{\alpha_{1}}(\Delta v_{1i})^{\alpha_{2}}(\Delta 
m_{1i})^{\alpha_{3}}$ where $\alpha_{1}+2\alpha_{2}+2\alpha_{3} \geq 
\tau_{1}+2\tau_{2}+2\tau_{3}$. Note that, the constraint with $\alpha 
= (\alpha_{1},\alpha_{2},\alpha_{2})$ stems from the representation 
in Lemma~\ref{lemma:recursiveequation}. 
By Eq. \eqref{eqn:fundamental_limits} and the structure of $\xi_{\tau_{1},\tau_{2},\tau_{3}}^{(r)}$, for any $\tau 
\in \mathcal{F}_{r}$ such that $\tau_{1}+2\tau_{2}+2\tau_{3} \geq 4$, we have $
\xi_{\tau_{1},\tau_{2},\tau_{3}}^{(r)}/\widetilde{W}_{\kappa}^{\|\kappa\|
_{\infty}}(G,G_{0}) \to 0$. 
On the other hand, for any $\tau 
\in \mathcal{F}_{r}$ such that $\tau_{1}+2\tau_{2}+2\tau_{3} \leq 3$, 
denote $\overline{\xi}_{\tau_{1},\tau_{2},\tau_{3}}^{(r)}$ a 
coefficient by leaving out all the monomials $\sum \limits_{i=1}^{2} 
p_{1i}(\Delta \theta_{1i})^{\alpha_{1}}(\Delta v_{1i})^{\alpha_{2}}
(\Delta m_{1i})^{\alpha_{3}}$ with $
\alpha_{1}+2\alpha_{2}+2\alpha_{3} \geq 4$ from the original 
coefficient $\xi_{\tau_{1},\tau_{2},\tau_{3}}^{(r)}$. Invoking again
~\eqref{eqn:fundamental_limits}, the following holds
\begin{align}
\xi_{\tau_{1},\tau_{2},\tau_{3}}^{(r)}/\widetilde{W}_{\kappa}^{\|\kappa\|_{\infty}}(G,G_{0}) - \overline{\xi}_{\tau_{1},\tau_{2},\tau_{3}}^{(r)}/\widetilde{W}_{\kappa}^{\|\kappa\|_{\infty}}(G,G_{0}) \rightarrow 0, \nonumber
\end{align}
where, by direct computations, the formulations of $\overline{\xi}
_{\tau_{1},\tau_{2},\tau_{3}}^{(r)}$ when $\tau 
\in \mathcal{F}_{r}$ and $\tau_{1}+2\tau_{2}+2\tau_{3} \leq 3$ are given as follows
\begin{align}
\overline{\xi}_{1,0,0}^{(r)} = \sum \limits_{i=1}^{2} {p_{1i}\Delta \theta_{1i}}, \ \overline{\xi}_{0,1,0}^{(r)}=\sum \limits_{i=1}^{2} p_{1i}\Delta v_{1i}+\sum 
\limits_{i=1}^{2}{p_{1i}(\Delta \theta_{1i})^{2}}, \nonumber \\
\overline{\xi}_{0,0,1}^{(r)}=-\dfrac{(m_{1}^{0})^{3}+m_{1}^{0}}{2v_{1}^{0}}
\sum \limits_{i=1}^{2}{p_{1i}(\Delta \theta_{1i})^{2}} + \sum \limits_{i=1}^{2}{p_{1i}
\Delta m_{1i}}, \nonumber \\
\overline{\xi}_{1,1,0}^{(r)}=\sum \limits_{i=1}^{2} {p_{1i}\Delta \theta_{1i}\Delta v_{1i}}, \;
\overline{\xi}_{1,0,1}^{(r)}=\sum \limits_{i=1}^{2} {p_{1i}\Delta \theta_{1i}\Delta m_{1i}}. \nonumber
\end{align}
According to the choices of $p_{1i}, \Delta 
\theta_{1i}$, $\Delta v_{1i}$, and $\Delta m_{1i}$, we have $
\overline{\xi}_{\tau_{1},\tau_{2},\tau_{3}}^{(r)} = 0$. Therefore, we obtain that $\xi_{\tau_{1},\tau_{2},
\tau_{3}}^{(r)}/\widetilde{W}_{\kappa}^{\|\kappa\|_{\infty}}(G,G_{0}) \to 0$ for any $\tau 
\in \mathcal{F}_{r}$ such that $\tau_{1}+2\tau_{2}+2\tau_{3} \leq 3$. As a 
consequence, $G_{0}$ is $\kappa$-singular for any $\kappa \in 
\mathcal{U}_{r}$ as $r \geq 1$, which leads to the conclusion that $\singset(G_{0}|\Ocal_{2,c_{0}})=\left\{(4,2,2)\right\}$.

We end the foregoing laborious exercise with a few comments.
The fact that there exists a subset of the limiting polynomials of the coefficients
of $r$-minimal forms that do not depend on specific value of $G_0$ is very useful, because it
results in a non-trivial upper bound on the singularity level the holds \emph{uniformly} for all 
$G_0 \in \Scal_0$. It is interesting to note that this subset of polynomials also arises from 
the same analysis applied to the Gaussian kernels studied by~\cite{Ho-Nguyen-Ann-16}.
This observation can be partially explained by the fact that Gaussian kernels are 
a special case of skew-normal kernels with zero skewness. A striking
consequence from this observation is that the singularity level in a skew-normal
mixture is always bounded from above by the singularity level in a Gaussian
mixture. Thanks to Theorem \ref{proposition:convergence_and_minimax} 
we arrive at a remarkable conclusion that the MLE and minimax bounds for estimating mixing measure
in skew-normal o-mixtures are generally \emph{faster} than that of Gaussian o-mixtures. 
In terms of individual parameter estimation, we also arrive at another interesting 
phenemenon as the convergence rates of location and scale parameters in skew-normal 
o-mixtures are also faster than those in Gaussian o-mixtures.
At this point we are ready for a general result 
which quantifies our remarks precisely.

\subsection{A general theorem for skew-normal o-mixtures} \label{Section:general_bound_omixtures_generic} In this section we shall
present results on singularity structures of $G_{0}$ for the general case $k>k_{0}$. 
To do so, we define the system of the limiting polynomials
that characterizes both the singularity level and singularity index of $G_0$.  
Recall the notation introduced by the statement of Lemma \ref{lemma:recursiveequation}, where
$P_{\alpha_{1},\alpha_{2},\alpha_{3}}
^{\kappa_{1},\kappa_{2},\kappa_{3}}(m)$, $H_{\alpha_{1},\alpha_{2},\alpha_{3}}
^{\kappa_{1},\kappa_{2},\kappa_{3}}(m)$, and $Q_{\alpha_{1},\alpha_{2},\alpha_{3}}
^{\kappa_{1},\kappa_{2},\kappa_{3}}(v)$ are polynomials in terms of $m,m$ and $v$, respectively,
that arise in the decomposition of partial derivatives of the skew-normal kernel function.

For given $r \geq 1$, for each
$i=1,\ldots, k_0$, the system of limiting polynomial is given by the following equations of real unknowns
$(a_j,b_j,c_j,d_j)_{j=1}^{k-k_0+1}$:
\begin{eqnarray}
\biggr \{ \sum \limits_{j=1}^{k-k_0+1}{\sum \limits_{\alpha}{\dfrac{P^{\beta_{1},\beta_{2},\beta_{3}}_{\alpha_{1},\alpha_{2},\alpha_{3}}(m_{i}^{0})}{H^{\beta_{1},\beta_{2},\beta_{3}}_{\alpha_{1},\alpha_{2},\alpha_{3}}(m_{i}^{0})Q^{\beta_{1},\beta_{2},\beta_{3}}_{\alpha_{1},\alpha_{2},\alpha_{3}}(v_{i}^{0})}\dfrac{d_{j}^{2}a_{j}^{\alpha_{1}}b_{j}^{\alpha_{2}}c_{j}^{\alpha_{3}}}{\alpha_{1}!\alpha_{2}!\alpha_{3}!}}} = 0 
\biggr |
\; \beta \in \mathcal{F}_r \cap \{\beta_1 + 2\beta_2 + 2\beta_3 \leq r \}
\biggr \} \label{eqn:systemnonlinear}
\end{eqnarray}
where the range of $\alpha=(\alpha_{1},\alpha_{2},\alpha_{3}) \in \mathbb{N}^3$ in the above sum satisfies $\alpha_{1}+2\alpha_{2}+2\alpha_{3}=\beta_{1}+2\beta_{2}+2\beta_{3}$. 

Note that the above system of polynomial equations is the general version of the systems of polynomial equations 
described in the examples of Section \ref{Section:illustration_omixture_byone}. 
There are $2r-1$ equations in the above system of $4(k-k_0+1)$
unknowns.
A solution of \eqref{eqn:systemnonlinear} is considered \textit{non-trivial} 
if all of $d_{j}$ are non-zeros while at 
least one among $a_{1},\ldots,a_{i},b_{1},\ldots,b_{i},c_{1},
\ldots,c_{i}$ is non-zero. We say that system \eqref{eqn:systemnonlinear} is unsolvable if it does not have 
any non-trivial (or admissible) solution. Note that increasing $r$ makes the system more constrained.
The main result of this section is the following.

\begin{theorem}
\label{theorem:generic_setting_omixtures} 
For each $i=1,\ldots,k_0$, let $\rVMS(v_{i}^{0},m_{i}^{0},k-k_0)$ be the minimum $r$
for which system of polynomial equations \eqref{eqn:systemnonlinear} does not admit non-trivial solutions.
Let $G_0\in  \Scal_0$ and define
\begin{eqnarray}
\rgeneric(G_{0},k) 
& = &  \mathop {\max }\limits_{ 1 \leq i \leq k_{0}}{\rVMS(v_{i}^{0},m_{i}^{0},k-k_0)}. 
\label{eqn:general_rate_generic_setting}
\end{eqnarray}
\begin{itemize}
\item [(i)]
Then, $\lev(G_0| \Ocal_{k,c_0}) \leq \rgeneric(G_{0},k)-1$. 
\item [(ii)] Moreover, there exists $\kappa \in \singset(G_{0}|\Ocal_{k,c_{0}})$ such that 
\begin{eqnarray}
\kappa \preceq \begin{cases} \biggr(\rgeneric(G_{0},k),\dfrac{\rgeneric(G_{0},k)}{2},\dfrac{\rgeneric(G_{0},k)}{2}\biggr), & \mbox{if} \ \rgeneric(G_{0},k) \ \mbox{is an even number} \\ \biggr(\rgeneric(G_{0},k),\dfrac{\rgeneric(G_{0},k)+1}{2},\dfrac{\rgeneric(G_{0},k)+1}{2}\biggr), & \mbox{if} \ \rgeneric(G_{0},k) \ \mbox{is an odd number}. \end{cases} \nonumber
\end{eqnarray} 
\end{itemize}
\end{theorem} 

\paragraph{Remark} We make the following comments regarding the results of Theorem \ref{theorem:generic_setting_omixtures}.
\begin{itemize}
\item[(i)] If $k-k_{0}=1$, we can obtain $\rgeneric(G_{0},k)=4$ from the examples given in Section 
\ref{Section:illustration_omixture_byone} (although in the examples we only worked out the case that $k_0=1$, 
for general $k_0 \geq 1$ the techniques are the same).  Since $(4,2,2)$ is the unique singularity index of $G_{0}$, the bounds with singularity level and singularity index are tight.
\item[(ii)] Since the index component for the location parameter dominates that of shape and scale pararameters $(4>2)$,
estimating shape-scale parameters may be more efficient than estimating location parameter in skew-normal o-mixtures. 
\item[(iii)] In order to determine $\rgeneric(G_{0},k)$, 
we need to find the value of $\rVMS(v_{i}^{0},m_{i}^{0},k-k_{0})$ for all 
$1 \leq i \leq k_{0}$. One may ask whether the value of $\rVMS(v_{i}^{0},m_{i}^{0},k-k_{0})$ 
depends on the specific values of $v_{i}^{0},m_{i}^{0}$. The 
structure of
$\rVMS(v_{i}^{0},m_{i}^{0},k-k_{0})$ will be looked at in more detail in the next subsection.
\end{itemize}

\subsection{Properties of the system of limiting polynomial equations} 
\label{Section:nonlinear_system_study} 
The goal of this subsection is the present additional results on the structure
of function $\rVMS(v,m,k-k_0)$, which is a fundamental quantity in Theorem
\ref{theorem:generic_setting_omixtures} (Here, $v_{i}^{0},m_{i}^{0}$ are replaced by $v,m$).
It is difficult to obtain explicit values for $\rVMS(v,m,k-k_0)$ in general.
Nonetheless, we can obtain a nontrivial upper bound for $\rVMS$. Now, let
$\Xi_{1} :=\left\{(v,m) \in \Theta_{2} \times \Theta_{3}: \ m 
\neq 0 \right\}$.
Recall that $\rVMS(v,m,l)$, where $l=k-k_0 \geq 1$, is the minimum value according to which 
system \eqref{eqn:systemnonlinear} does not admit non-trivial real-solution.

\begin{proposition}\label{lemma:upperbound_system_nonlinear}
Let $\overline{r}(l)$ be defined as in \eqref{eqn:generalovefittedGaussianzero_Gaussian_mulindex}. 
For all $l=1,2,\ldots$, there holds
\[\sup_{(v,m) \in \setgen} \rVMS(v,m,l) \leq \overline{r}(l).\] 
\end{proposition}

\paragraph{Remarks} (i) The proof of this proposition is given in Appendix F, 
which proceeds by verifying that system \eqref{eqn:generalovefittedGaussianzero_Gaussian_mulindex} forms a 
subset of equations that defines system \eqref{eqn:systemnonlinear}. Combining
with the statement of Theorem \ref{theorem:generic_setting_omixtures}, we immediately obtain, 
provided that $G_0 \in \Scal_0$,
\[\lev(G_0|\Ocal_{k,c_0}) \leq \overline{r}(l)-1.\]
Moreover, there exists $\kappa \in \singset(G_{0}|\Ocal_{k,c_{0}})$ such that
\begin{eqnarray}
\kappa \preceq \begin{cases} \biggr(\overline{r}(l),\dfrac{\overline{r}(l)}{2},\dfrac{\overline{r}(l)}{2}\biggr), & \mbox{if} \ \overline{r}(l) \ \mbox{is an even number} \\ \biggr(\overline{r}(l),\dfrac{\overline{r}(l)+1}{2},\dfrac{\overline{r}(l)+1}{2}\biggr), & \mbox{if} \ \overline{r}(l) \ \mbox{is an odd number}. \end{cases} \nonumber
\end{eqnarray} 
Thus we are able to relate the singularity structure of a mixing measure in a location-scale Gaussian mixture model
to that of the same mixing measure in a skew-normal mixture model.

(ii) Combining the above remark with the results established by Theorem 
\ref{proposition:convergence_and_minimax}, Proposition \ref{proposition:Gaussian_mulindex}, and Theorem \ref{proposition:convergence_and_minimax_general}
leads us to conclude the following interesting results: in terms of mixing measure, it is statistically more efficient to estimate mixing measure of 
skew-normal o-mixtures than to estimate mixing measure of
Gaussian o-mixtures that carry the same number of extra mixing components.
Regarding individual parameters, estimating location and scale parameters of 
skew-normal o-mixtures is more efficient than estimating location and scale parameters of
Gaussian o-mixtures.

\paragraph{Dependence of $\rho$ on $(v,m)$}
To understand the role of parameter value $(v,m)$ on singularity levels and indices, we shall
construct a partition of the parameter space for $(v,m)$ based on the value of function $\rho$. 
For each $l,r \geq 1$, define an ``inverse'' function
\begin{eqnarray}
\rho^{-1}_{l}(r)=\left \{(v,m) \in \setgen: \rho(v,m,l) = r \right\}. \nonumber
\end{eqnarray}
Additionally, take
\begin{eqnarray}
\underbar{\rVMS}(l) =\mathop {\min } {\left\{r : \rho^{-1}_{l}(r) \neq \emptyset \right\}}, \ 
\overline{\rVMS}(l) = \mathop {\max} {\left\{r : \rho^{-1}_{l}(r) \neq \emptyset \right\}}. \nonumber
\end{eqnarray}
It follows from Proposition\ref{lemma:upperbound_system_nonlinear} that 
$\overline{\rVMS}(l) \leq \overline{r}(l)$. In addition, 
$\rho_l^{-1}(r)$ are mutually disjoint for different values of $r$. 
So, for each fixed amount of overfitting $l \geq 1$, the parameter space may be partioned by
\begin{eqnarray}
\setgen = \mathop {\bigcup } \limits_{r=\underbar{\rVMS}(l)}^{\overline{\rVMS}(l)} \rho^{-1}_{l}(r). \nonumber
\end{eqnarray} 

\begin{proposition} For each $l\geq 1, r\geq 1$, $\rho^{-1}_{l}(r)$ is a semialgebraic set.
\end{proposition}
\begin{proof} For each $r \geq 1$, let $\mathbb{A}_{r}$ be the collection of all $(v,m) 
\in \Xi_{1}$ such that the system of polynomial equations \eqref{eqn:systemnonlinear} 
contains admissible solutions.
Furthermore, $\mathbb{B}_{r}$ denotes the collection of all solutions $(v,m,
\left\{a_{i}\right\}_{i=1}^{l},\left\{b_{i}\right\}_{i=1}^{l},\left\{c_{i}\right\}_{i=1}^{l},
\left\{d_{i}\right\}_{i=1}^{l})$ of the system of polynomial equations 
\eqref{eqn:systemnonlinear}, i.e., we treat $v, m$ as two additional unknowns of the system. 
Since $P^{\beta_{1},\beta_{2},\beta_{3}}_{\alpha_{1},
\alpha_{2},\alpha_{3}}(m)$, $H^{\beta_{1},\beta_{2},\beta_{3}}_{\alpha_{1},
\alpha_{2},\alpha_{3}}(m)$, and $Q^{\beta_{1},\beta_{2},\beta_{3}}_{\alpha_{1},
\alpha_{2},\alpha_{3}}(v)$ are polynomial functions of $m,m$ and $v$, resp., for all $\alpha, \beta$, by
definition $\mathbb{B}_{r}$ is a semialgebraic set for all $r \geq 1$. By Tarski-Seidenberg 
theorem \cite{Basu-2006}, since $\mathbb{A}_{r}$ is the projection of $\mathbb{B}_{r}$ 
from dimension $(4l+2)$ to dimension 2, $\mathbb{A}_{r}$ is a semialgebraic set for all $r 
\geq 1$. It follows that $\mathbb{A}_{r}^{c}$ is semialgebraic for all $r \geq 
1$. Since $\rho^{-1}_{l}
(r)=\mathbb{A}_{r}^{c} \cap \mathbb{A}_{r-1}$ for all $r \geq 1$, the conclusion of the proposition
follows.
\end{proof}
The following result gives us some exact values of $\underbar{\rVMS}(l)$ and $\overline{\rVMS}(l)$ in
specific cases.
\begin{proposition} \label{proposition:overfittedproposition} 
\begin{itemize}
\item[(a)] If $l=k-k_0=1$, then $\underbar{\rVMS}(l)=\overline{\rVMS}(l)=4$. 
\item[(b)] If $l=k-k_0=2$, then $\underbar{\rVMS}(l)=5$ and $\overline{\rVMS}(l)=6$. 
Thus, $\setgen$ is partitioned into two subsets, both of which are non-empty because
$\left\{(1,-2),(1,2)\right\} \subset \rho_{l}^{-1}(5)$, and 
$(1,\dfrac{1}{10}) \in \rho_{l}^{-1}(6)$.
\end{itemize}
\end{proposition}
From the definition of $\rgeneric(G_0,k)$, we can write
\[\rgeneric(G_0,k) = \max \biggr \{r  \biggr | \mbox{there is}\; 
i=1,\ldots,k_0 \; \mbox{such that} (v_i^0,m_i^0) \in \rho_{k-k_0}^{-1}(r) \biggr \}.\]
According to the Proposition \ref{proposition:overfittedproposition}, 
if $k-k_0=1$, we have $\rgeneric(G_{0},k)=4$ (see also our earlier remark).
If $k-k_0=2$, we may have either $\rgeneric(G_0,k) = 5$ or $6$, depending
on the value of parameters $(v,m)$ that provide the support for $G_0$.

We end this section by noting that we have just provided specific examples
in which $\rgeneric(G_0,k)-1$ may vary with the actual parameter values that define $G_0$.  
Although this provides upper bounds of the singularity level and singularity index, we
have \emph{not} actually proved that the singularity level and singularity index of $G_0$ may generally 
vary with its parameter values. We will be able to do so when we work with the e-mixture setting. 
The analysis of singularity structure of parameter space under that setting is laid out in Appendices E and F
for the interested.

\section{Discussion and concluding remarks} \label{Section:discussion}

Understanding the behavior of parameter estimates of mixture models is useful because
the mixing parameters represent explicitly the heterogeneity of the underlying data 
population that mixture models are most suitable for. 
In this paper, a general theory for the identification of singularity structure
arising from finite mixture models is proposed. It is shown that the singularity
structures of the model's parameter space directly determine minimax lower
bounds and maximum likelihood estimation convergence rates, under conditions on
the compactness of the parameter space.

The systematic identification of singularity structures and the implications on parameter 
estimation is a crucial step toward the development of more
efficient model-based inference procedures. 
It is our view that such procedures must account for the presence of singular 
points residing in the parameter space of the model. 
As a matter of fact, there are quite a few examples of such efforts applied to 
specific statistical models, even if the picture of the singularity structures associating
with those models might not have been discussed explicitly.  
This raises a question of whether or not it is possible to extend
and generalize such techniques in order to address the presence of singularities in a
direct fashion. We give several examples:

\begin{itemize}
\item [(1)] For overfitted mixture models, methods based on likelihood-based 
penalization techniques were shown to be quite effective (e.g., \cite{Gassiat-2009, Chen-2016}).
Our work shows that parameter values residing in the vicinity of regions of high singularity
levels should be hard to estimate efficiently. Can a
penalization technique be generalized to regularize the estimates
toward subsets containing singularity points of lower levels?

\item [(2)] Suitable choices of Bayesian prior have been proposed to induce favorable posterior
contraction behavior for overfitted finite mixtures~\cite{Rousseau-Mengersen-11}.
Can we develop an appropriate prior for the mixture model parameters, given our
knowledge of singular points residing in the parameter space?

\item [(3)] Reparametrization is an effective technique that can be employed to combat 
singularities present in the class of skewed distributions \cite{Hallin-2014}. 
It would be interesting to study if such reparameterization technique
can be systematically developed for mixture models as well.
\end{itemize}

Finally, we also expect that the theory of singularity structures 
carries important consequences on the computational complexity of parameter estimation procedures,
including both optimization and sampling based methods.
The inhomogeneous nature of the singularity structures reveals a complex picture of
the likelihood function: regions in parameter space that carry low singularity levels/indices
may observe a relatively high curvature of the likelihood surface, while high singularity levels imply
a ``flatter'' likelihood surface along a certain subspace of the parameters.
Such a subspace 
is manifested by our construction of sequences of mixing measures that 
attest to the condition of $r$-singularity or $\kappa$-singularity in general. It is of interest to exploit the 
explicit knowledge of singularity structures obtained for a given mixture 
model class, so as to improve upon the 
computational efficiency of the optimization and sampling procedures
that operate on the model's parameter space.


\bibliography{Nhat,NPB,Nguyen}

\newpage

\section{Appendix A: Proofs of key results}
This Appendix contains the proofs of key results in the paper.

\subsection{Proof of statements in Section \ref{Section:general_procedure_singularity}}
\paragraph{PROOF OF LEMMA ~\ref{lemma:minimal_form}}
(a) The existence of the sequence of $G$ described in the definition of a $r$-minimal form
implies for that sequence, $(p_G(x)-p_{G_0}(x))/W_r^r(G,G_0) \rightarrow 0$ holds
for almost all $x$. Now take any $r$-minimal form \eqref{eqn:generallinearindependencerepresentation}
given by the same sequence of $G$. Let $C(G) = \max_{l=1}^{T_r}
\frac{\xi_{l}^{(r)}(G)}{W_r^r(G_0,G)}$. We will show that $\liminf C(G) = 0$, which concludes
the proof. Suppose that this is not the case, so we have $\liminf C(G) > 0$. It follows that
\[\sum_{l=1}^{T_r}
\biggr (\frac{\xi_{l}^{(r)}(G)}{C(G) W_r^r(G,G_0)} \biggr ) H_{l}^{(r)}(x)
\rightarrow 0.\]
Moreover, all the coefficients in the above display are bounded from above by 1,
one of which is in fact 1. There exists a subsequence of $G$ by which
these coefficients have limits, one of which is 1. This is a contradiction
due to the linear independence of functions $H_l^{(r)}(x)$. 

(b) Let $G$ be an element in the sequence that admits a $r$-minimal form
such that $\xi_{l}^{(r)}(G)/W_r^r(G_0,G)$\\ $\rightarrow 0$
for all $l=1,\ldots,T_r$. It suffices to assume that the basis functions $H_{l}^{(r)}$
are selected from the collection of partial derivatives of $f$. 
We will show that the same sequence of $G$ and the elimination procedure for 
the $r$-minimal form can be used to construct a $r-1$-minimal form by which
\[\xi_{l}^{(r-1)}(G)/W_{r-1}^{r-1}(G_0,G) \rightarrow 0\]
for all $l=1,\ldots,T_{r-1}$. There are two possibilities to consider.

First, suppose that each of the $r$-th partial derivatives of density kernel $f$
(i.e., $\partial^\kappa f/\partial \eta^\kappa$, where $|\kappa|= r$)
is not in the linear span of the collection of
partial derivatives of $f$ at order $r-1$ or less. Then, for each $l=1,\ldots, T_{r-1}$,
$\xi_{l}^{(r-1)}(G) = \xi_{l'}^{(r)}(G)$ for some $l'\in [1,T_{r}]$.
Since $W_{r-1}^{r-1}(G,G_0)\gtrsim W_r^r(G,G_0)$, due to the fact 
that the support points of $G$ and $G_0$ are in a bounded set, 
we have that 
\[\xi_{l}^{(r-1)}(G)/W_{r-1}^{r-1}(G_0,G) 
\lesssim \xi_{l'}^{(r)}(G)/W_r^r(G_0,G)\]
which vanishes by the hypothesis.

Second, suppose that some of the $r$-th partial derivatives, say,
$\partial^{|\kappa|}f/\partial \eta^\kappa$ where $|\kappa|=r$,
can be eliminated because they can be represented by a linear combination of a subset of 
other partial derivatives $H_l^{(r-1)}$ (in addition to possibly
a subset of other partial derivatives $H_l^{(r)}$) with corresponding
finite coefficients $\alpha_{\kappa,i,l}$. 
It follows that for each $l=1,\ldots, T_{r-1}$, the coefficient $\xi_l^{(r-1)}(G)$
that defines the $r-1$-minimal form is transformed into
a coefficient in the $r$-minimal form by

\begin{equation}
\label{coeffs-r}
\xi_{l'}^{(r)}(G) := \xi_{l}^{(r-1)}(G) + 
\sum_{\kappa; |\kappa|=r} 
\sum_{i=1}^{k_0} \alpha_{\kappa,i,l}
\sum_{j=1}^{s_i}
{p_{ij}(\Delta \eta_{ij})^{\kappa}/\kappa!}.
\end{equation}
Since $\xi_{l'}^{(r)}(G)/W_r^r(G,G_0)$ tends to 0,
so does $\xi_{l'}^{(r)}(G)/W_{r-1}^{r-1}(G,G_0)$.
By Lemma~\ref{lemma:bound_overfit_Wasserstein} 
for each $\kappa$ such that $|\kappa|=r$,
$\sum_{i=1}^{k_0}\sum_{j=1}^{s_i}
{p_{ij}(\Delta \eta_{ij})^{\kappa}/\kappa!} 
= o(D_{r-1}(G_0,G)) = o(W_{r-1}^{r-1}(G,G_0))$.
Combining with Eq.~\eqref{coeffs-r} it follows
that $\xi_l^{(r-1)}(G)/W_{r-1}^{r-1}(G,G_0)$ tends to 0,
for each $l = 1,\ldots, T_{r-1}$. This completes the proof.

\comment{
It remains to show that a subsequence of this tends to zero. By the definition of
$r$-singularity, for any $x$:
\[\frac{p_{G}(x)-p_{G_{0}}(x)}{W_{r}^{r}(G,G_{0})} = 
\sum_{l=1}^{T_r}
\biggr (\frac{\xi_{l}^{(r)}(G)}{W_r^r(G_0,G)} \biggr ) H_{l}^{(r)}(x) + o(1)
\rightarrow 0.\]
Since $W_r^r(G,G_0)\lesssim W_{r-1}^{r-1}(G,G_0)$, 
we have $(p_{G}(x)-p_{G_{0}}(x))/W_{r-1}^{r-1}(G,G_{0}) \rightarrow 0$.
Under the $(r-1)$-minimal form
\[\frac{p_{G}(x)-p_{G_{0}}(x)}{W_{r-1}^{r-1}(G,G_{0})} = 
\sum_{l=1}^{L_{r-1}}
\biggr (\frac{\xi_{l}^{(r-1)}(G)}{W_{r-1}^{r-1}(G_0,G)} \biggr ) H_{l}^{(r-1)}(x) + o(1)
\rightarrow 0.\]
Since $\frac{\xi_{l}^{(r-1)}(G)}{W_{r-1}^{r-1}(G_0,G)}$ are bounded sequences,
there exists a converging subsequence for each $l$ and hence all $l$. 
These converging subsequences must all converge to 0, because $H_l^{(r-1)}$ 
are linearly independent functions. This concludes the proof.}
\paragraph{PROOF OF THEOREM ~\ref{theorem:singularity_connection_liminf}}

(i) It suffices to prove the first inequality for $s=r+1$. Firstly, we will demonstrate that
\begin{eqnarray}
\liminf \limits_{G \in \mathcal{G}: W_s(G,G_0)\rightarrow 0} \|p_{G}-p_{G_0}\|_\infty/ W_{s}^{s}
(G,G_0) > 0. \nonumber
\end{eqnarray}
If this is not true, then there exists
a sequence of $G$ such that $W_s(G,G_0)\rightarrow 0$, and
for almost all $x$, $(p_G(x)-p_{G_0}(x))/W_s^s(G,G_0) \rightarrow 0$.
Take any $s$-minimal form for this ratio, we have
\[\frac{p_{G}(x)-p_{G_0}(x)}{W_s^s(G,G_0)} = \sum_{l=1}^{T_s}
\biggr (\frac{\xi_{l}^{(s)}(G)}{W_s^s(G,G_0)} \biggr ) H_{l}^{(s)}(x) + o(1)
\rightarrow 0.\]
For each $G$ in the sequence, let $C(G) = \max_{l}
\dfrac{\xi_{l}^{(s)}(G)}{W_s^s(G_0,G)}$. If $\liminf C(G) = 0$, then
this means $G_0$ is $s$-singular, so $\lev(G_0|\Gcal) \geq s$.
This violates the given assumption.
So we have $\liminf C(G) > 0$. It follows that
\[\sum_{l=1}^{T_s}
\biggr (\frac{\xi_{l}^{(s)}(G)}{C(G) W_s^s(G,G_0)} \biggr ) H_{l}^{(s)}(x)
\rightarrow 0.\]
Moreover, all coefficients in the above display are bounded from above by 1,
one of which is in fact 1. There exists a subsequence of $G$ by which
these coefficients have a limit, one of which is 1. This is also a contradiction
due to the linear independence of functions $H_l^{(s)}$. 

Therefore, we can find a positive number $\epsilon_{0}$ such that $\|p_{G}-p_{G_{0}}\|
_{\infty} \gtrsim W_{s}^{s}(G,G_{0})$ as soon as $W_{s}(G,G_{0}) \leq \epsilon_{0}$. Now, 
to obtain the conclusion of part (i), it suffices to demonstrate that
\begin{eqnarray}
\inf \limits_{G \in \mathcal{G}: W_{s}(G,G_{0})>\epsilon_{0}} \| p_{G}-p_{G_0}\|_{\infty}/ W_{s}^{s}
(G,G_0) > 0. \nonumber
\end{eqnarray}
If this is not the case, there is a sequence $G'$ such that $W_{s}(G',G_{0})>\epsilon_{0}$ 
and $\|p_{G'}-p_{G_0}\|_{\infty}/ W_{s}^{s}
(G',G_0)$\\$ \to 0$. Since $\Theta$ is compact and $\Gcal$ contains only probability measures 
with bounded number of support points in $\Theta$, we can find $G^* \in \mathcal{G}$ such that 
$W_{s}(G',G^*) \to 0$ and $W_{s}(G^*,G_{0}) \geq \epsilon_{0}$. As $W_{s}
(G',G_{0}) \to W_{s}(G^*,G_{0})>0$, we have $\|p_{G'}-p_{G_0}\|_{\infty} \to 0$. Now, 
due to the first order uniform Lipschitz condition of $f$, we obtain $p_{G'}(x) \to p_{G^*}
(x)$ for all $x \in \mathcal{X}$. Thus, $p_{G^*}(x)=p_{G_{0}}(x)$ for almost all $x \in \mathcal{X}
$, which entails that $G^* = G_{0}$, a contradiction. This completes
the proof.

(ii) Turning to the second inequality, we also firstly demonstrate that
\begin{eqnarray}
\liminf \limits_{G \in \mathcal{G}: W_s(G,G_0)\rightarrow 0} V(p_{G},p_{G_0})/ W_{s}^{s}
(G,G_0) > 0. \nonumber
\end{eqnarray}
If it is not true, then we have a sequence of $G$
such that $W_s(G,G_0)\rightarrow 0$ and $V(p_G,p_{G_0})/W_s^s(G,G_0) \rightarrow 0$. By Fatou's lemma
\[0 = \liminf \frac{V(p_G,p_{G_0})}{C(G) W_s^s(G,G_0)}
\geq \int \liminf_{G} \biggr |\frac{\xi_{l}^{(s)}(G)}{C(G) W_s^s(G,G_0)}  H_{l}^{(s)}(x) \biggr | 
\textrm{d}x.\]
The integrand must be zero for almost all $x$, leading to a contradiction as before. Hence, to obtain the conclusion of part (ii), we only need to show that
\begin{eqnarray}
\inf \limits_{G \in \mathcal{G}: W_{s}(G,G_{0})>\epsilon_{0}} V(p_{G},p_{G_0})/ W_{s}^{s}
(G,G_0) > 0. \nonumber
\end{eqnarray}
where $\epsilon_{0}>0$ such that $V(p_{G},p_{G_0}) \gtrsim W_{s}^{s}(G,G_{0})$ for any $W_{s}(G,G_{0}) \leq \epsilon_{0}$. If it is not true, then using the same argument as that of part (i), there is a sequence of
$G'$ such that $W_s(G',G^*) \rightarrow 0$, $V(p_{G'},p_{G_{0}}) \to 0$,
while $W_{s}(G^*,G_{0}) \geq \epsilon_{0}$ and $p_{G'}(x) \to p_{G^*}(x)$ for all $x \in \mathcal{X}$. By Fatou's lemma,
\begin{eqnarray}
0 = \liminf V(p_{G'},p_{G_{0}}) \geq \int {\liminf |p_{G'}(x)-p_{G_{0}}(x)|}dx = V(p_{G^*},p_{G_{0}}), \nonumber
\end{eqnarray}
which leads to $G^* =  G_{0}$, a contradiction. The proof is concluded.

\subsection{Proofs for section~\ref{Section:overfitskew}}

\paragraph{PROOF OF THEOREM ~\ref{theorem:generic_setting_omixtures}}
\comment{
We will provide for the case $\rgeneric(G_{0},k)$ is an even number as 
the argument for the case of odd $\rgeneric(G_{0},k)$ is similar. To obtain 
the conclusion of the theorem under that setting of $\rgeneric(G_{0},k)$, it is sufficient 
to demonstrate that $G_{0}$ is not $\biggr(\rgeneric(G_{0},k),
\dfrac{\rgeneric(G_{0},k)}{2},\dfrac{\rgeneric(G_{0},k)}{2}\biggr)$-singular relative to 
$\Ocal_{k,c_{0}}$. 
}

The reader is recommended to go over the special cases given earlier in 
Section~\ref{Section:illustration_omixture_byone} before embarking on this proof.
Our strategy is clear: First, 
we obtain a valid $\kappa$-minimal form for $G_0$, cf. Eq.
\eqref{eqn:generallinearindependencerepresentation_general}. This requires
a method for obtaining linearly independent basis functions $H_l(x)$
out of the partial derivatives of kernel density $f$. Second, we obtain
the polynomial limits of collection of coefficients of the $\kappa$-minimal
form. Third, we obtain bounds on $r$ according to which this
system of limiting polynomials does not admit non-trivial real solutions.
This yields upper bounds on the singularity level and singularity index of $G_0$.

\paragraph{Step 1: Construction of $\kappa$-minimal form}
Let $\kappa=(\kappa_1,\kappa_2,\kappa_3) \in \mathbb{N}^{3}$ be an index vector, and
put $r = \|\kappa\|_\infty$.
By Lemma \ref{lemma:recursiveequation} and Lemma \ref{lemma:reduced_linearly_independent}
a $\kappa$-th minimal form for $G_0$ can be obtained as
\begin{eqnarray}
\dfrac{p_{G}(x)-p_{G_{0}}(x)}{\widetilde{W}_{\kappa}^{r}(G,G_{0})} \asymp \dfrac{A_{1}(x)+B_{1}(x)}{\widetilde{W}_{\kappa}^{r}(G,G_{0})}, \nonumber
\end{eqnarray}
where $A_{1}(x)$ and $B_{1}(x)$ are given as follows
\begin{eqnarray}
A_{1}(x) & = & \sum \limits_{i=1}^{k_{0}}{\sum \limits_{\beta \in \mathcal{F}_{r}}{\biggr(\sum \limits_{j=1}^{s_{i}} \sum_{\alpha} {\dfrac{P^{\beta_{1},\beta_{2},\beta_{3}}_{\alpha_{1},\alpha_{2},\alpha_{3}}(m_{i}^{0})}{H^{\beta_{1},\beta_{2},\beta_{3}}_{\alpha_{1},\alpha_{2},\alpha_{3}}(m_{i}^{0})Q^{\beta_{1},\beta_{2},\beta_{3}}_{\alpha_{1},\alpha_{2},\alpha_{3}}(v_{i}^{0})}\dfrac{p_{ij}(\Delta \theta_{ij})^{\alpha_{1}}(\Delta v_{ij})^{\alpha_{2}}(\Delta m_{ij})^{\alpha_{3}}}{\alpha_{1}!\alpha_{2}!\alpha_{3}!}}}\biggr)} \times \nonumber \\
& & \dfrac{\partial^{|\beta|}{f}}{\partial{\theta^{\beta_{1}}}\partial{v^{\beta_{2}}}
\partial{m^{\beta_{3}}}}(x|\theta_{i}^{0},\sigma_{i}^{0},m_{i}^{0}), \nonumber \\
B_{1}(x) & = & \mathop {\sum }\limits_{i=1}^{k_{0}}{\Delta p_{i\cdot}f(x|\theta_{i}^{0},\sigma_{i}^{0},m_{i}^{0})}. \nonumber
\end{eqnarray} 
In the above expression, due to condition $G_0 \in \Ocal_{k,c_0}$, the number of redundant limit points in the $\kappa$-minimal form is $\bar{l} = 0$, so $i$
 in the above sum runs from 1 to $k_0$. It is also important to note that $\alpha$ in the above sum satisfies $|\alpha| \leq r$ and 
 $\alpha_{1}+2\alpha_{2}+2\alpha_{3} \geq \beta_{1}+2\beta_{2}+2\beta_{3}$.

Suppose that there exists a sequence of $G$ tending to $G_0$ under $\widetilde{W}_{\kappa}$ such that all the coefficients of  $A_{1}(x)/\widetilde{W}_{\kappa}^{r}(G,G_{0})$ and 
$B_{1}(x)/\widetilde{W}_{\kappa}^{r}(G,G_{0})$ vanish, so that $G_0$ is $\kappa$-singular relative to $\Ocal_{k,c_0}$. 
Then for all $1 \leq i \leq k_{0}$, we obtain that $\Delta p_{i\cdot}/\widetilde{W}_{\kappa}^{r}(G,G_{0}) \to 0$ and 
\begin{eqnarray}
E_{\beta_{1},\beta_{2},\beta_{3}}(\theta_{i}^{0},v_{i}^{0},m_{i}^{0}) := \dfrac{\sum \limits_{j=1}^{s_{i}}{\sum \limits_{\alpha}{\dfrac{P^{\beta_{1},\beta_{2},\beta_{3}}_{\alpha_{1},\alpha_{2},\alpha_{3}}(m_{i}^{0})}{H^{\beta_{1},\beta_{2},\beta_{3}}_{\alpha_{1},\alpha_{2},\alpha_{3}}(m_{i}^{0})Q^{\beta_{1},\beta_{2},\beta_{3}}_{\alpha_{1},\alpha_{2},\alpha_{3}}(v_{i}^{0})}\dfrac{p_{ij}(\Delta \theta_{ij})^{\alpha_{1}}(\Delta v_{ij})^{\alpha_{2}}(\Delta m_{ij})^{\alpha_{3}}}{\alpha_{1}!\alpha_{2}!\alpha_{3}!}}}}{\widetilde{W}_{\kappa}^{r}(G,G_{0})} \to 0, \nonumber 
\end{eqnarray}
for all $\beta \in \mathcal{F}_{r}$. 

By Lemma \ref{lemma:bound_overfit_Wasserstein_first}, $\widetilde{W}_{\kappa}^{r}(G,G_{0}) \asymp D_{\kappa}(G_{0},G)$. 
So, $\sum \limits_{i=1}^{k_{0}}{|\Delta p_{i\cdot}|}/D_{\kappa}(G_{0},G) \to 0$. It follows that
\begin{eqnarray}
\biggr\{\mathop {\sum }\limits_{i=1}^{k_{0}}{\mathop {\sum }\limits_{j=1}^{s_{i}}{p_{ij}(|\Delta \theta_{ij}|^{\kappa_1}+|\Delta v_{ij}|^{\kappa_2}+|\Delta m_{ij}|^{\kappa_3})}}\biggr\}/D_{r}(G_{0},G) \to 1. \nonumber
\end{eqnarray}
This means there exists some index $i^{*} \in \left\{1,\ldots,k_{0}\right\}$ such that 
\begin{eqnarray}
\mathop {\sum }\limits_{j=1}^{s_{i^{*}}}{p_{i^{*}j}(|\Delta \theta_{i^{*}j}|^{\kappa_1}+|\Delta v_{i^{*}j}|^{\kappa_2}+|\Delta m_{i^{*}j}|^{\kappa_3})}/D_{\kappa}(G_{0},G) \not \to 0. \nonumber
\end{eqnarray} 
By multiplying the inverse of 
the above term with $E_{\beta_{1},\beta_{2},\beta_{3}}(\theta_{i^{*}}^{0},v_{i^{*}}
^{0},m_{i^{*}}^{0})$ as $\beta \in \mathcal{F}_{r}$ and using the fact that $\widetilde{W}_{\kappa}^{r}(G,G_{0}) \asymp D_{\kappa}(G_{0},G)$, we obtain
\begin{multline}
 F_{\beta_{1},\beta_{2},\beta_{3}}(\theta_{i^{*}}^{0},v_{i^{*}}^{0},m_{i^{*}}^{0}) := \nonumber \\
 \dfrac{\sum \limits_{j=1}^{s_{i^*}}{\sum_{\alpha} {\dfrac{P^{\beta_{1},\beta_{2},\beta_{3}}_{\alpha_{1},\alpha_{2},\alpha_{3}}(m_{i^{*}}^{0})}{H^{\beta_{1},\beta_{2},\beta_{3}}_{\alpha_{1},\alpha_{2},\alpha_{3}}(m_{i^{*}}^{0})Q^{\beta_{1},\beta_{2},\beta_{3}}_{\alpha_{1},\alpha_{2},\alpha_{3}}(v_{i^{*}}^{0})}\dfrac{p_{i^{*}j}(\Delta \theta_{i^{*}j})^{\alpha_{1}}(\Delta v_{i^{*}j})^{\alpha_{2}}(\Delta m_{i^{*}j})^{\alpha_{3}}}{\alpha_{1}!\alpha_{2}!\alpha_{3}!}}}}{
\mathop {\sum }\limits_{j=1}^{s_{i^*}}{p_{i^{*}j}(|\Delta \theta_{i^{*}j}|^{\kappa_1}+|\Delta v_{i^{*}j}|^{\kappa_2}+|\Delta m_{i^{*}j}|^{\kappa_3})}} \to 0 \nonumber
\end{multline}
where $\alpha$ in the above sum satisfies $|\alpha| \leq r$ and $\alpha_{1}+2\alpha_{2}+2\alpha_{3} \geq \beta_{1}+2\beta_{2}+2\beta_{3}$.

\paragraph{Step 2: Greedy extraction of polynomial limits} 
We proceed to extract polynomial limit of coefficient $F_{\beta_{1},
\beta_{2},\beta_{3}}(\theta_{i^{*}}^{0},v_{i^{*}}^{0},m_{i^{*}}^{0})$
for each admissible index vector $\beta$.
Let
\begin{multline}
\overline{M}_{g}=\\
\mathop {\max }\biggr\{|\Delta \theta_{i^{*}1}|^{\kappa_1/r},\ldots,|\Delta \theta_{i^{*}s_{i^{*}}}|^{\kappa_1/r},
|\Delta v_{i^{*}1}|^{\kappa_2/r},\ldots,|\Delta v_{i^{*}s_{i^{*}}}|^{\kappa_2/r}, |\Delta m_{i^{*}1}|^{\kappa_3/r},\ldots,|\Delta m_{i^{*}s_{i^{*}}}|^{\kappa_3/r}\biggr\}. \nonumber
\end{multline} 
Then the denominator of coefficient $F_{\beta_{1},
\beta_{2},\beta_{3}}$ satisfies
 $\mathop {\sum }\limits_{j=1}^{s_{i^*}}{p_{i^{*}j}(|\Delta \theta_{i^{*}j}|^{\kappa_1}+|\Delta 
v_{i^{*}j}|^{\kappa_2}+|\Delta m_{i^{*}j}|^{\kappa_3})} \asymp \overline{M}_{g}^{r}$,
while the numerator consists of a finite number of monomials of power indices $\alpha = (\alpha_1,\alpha_2,\alpha_3)$,
each of which is asymptotically bounded from above by $M_g^{(\alpha_1/\kappa_1+\alpha_2/\kappa_2+\alpha_3/\kappa_3)r}$.
Accordingly, the only meaningful monomials in the numerator of $F_{\beta_{1},\beta_{2},\beta_{3}}$ are those
associated with $\alpha$ such that 
\begin{equation}
\label{kappa-constraint}
\alpha_1/\kappa_1+\alpha_2/\kappa_2+\alpha_3/\kappa_3 \leq 1.
\end{equation}
Coupling this with the constraints that 
 $|\alpha| \leq r$ and $\alpha_{1}+2\alpha_{2}+2\alpha_{3} \geq \beta_{1}+2\beta_{2}+2\beta_{3}$,
 we only need to consider asymptotically dominating monomials, i.e., those with minimal index vector $\alpha$.
 They are the indices that satisfy $\alpha_1 + 2\alpha_2 + 2\alpha_3 = \beta_1 + 2\beta_2 + 2\beta_3$
 (which would also entail that $|\alpha| \leq \beta_1 + 2\beta_2 + 2\beta_3 \leq r$, due to the definition of 
 index set $\mathcal{F}_r$). 
 In general, for a given index vector $\beta$, $\alpha = \beta$ is clearly one such minimal index, even though there may be others.
 Now for a fixed $r\geq 1$, we wish to select a minimal index $\kappa$, subject to $\|\kappa\|_\infty = r$, for which 
 the ``greedy'' choice $\alpha = \beta$ always satisfies constraint ~\eqref{kappa-constraint}. 
 The clear answer is to select $\kappa = (r,r/2,r/2)$ (if $r$ is not even, 
 $\kappa = (r, (r+1)/2, (r+1)/2)$ is selected instead).
 
\comment{
{\color{red} Note that, the respective choice of powers $1,1/2, 1/2$ for $|\Delta \theta_{i^{*}j}|$,$|\Delta v_{i^{*}j}|$, $|\Delta m_{i^{*}j}|$ in the above definition of $\overline{M}_{g}$ is motivated from the condition $\alpha_{1}+2\alpha_{2}+2\alpha_{3} \geq \beta_{1}+2\beta_{2}+2\beta_{3}$ in the formulation of $F_{\beta_{1},\beta_{2},\beta_{3}}(\theta_{i^{*}}^{0},v_{i^{*}}^{0},m_{i^{*}}^{0})$}. Governed by such choices of powers, for any index vector $\beta=(\beta_{1},\beta_{2},\beta_{3})$ such that $\beta \in \mathcal{F}_{r}$, the lowest order 
of $\overline{M}_{g}$ in the numerator of $F_{\beta_{1},\beta_{2},\beta_{3}}(\theta_{i^{*}}
^{0},v_{i^{*}}^{0},m_{i^{*}}^{0})$ is $\overline{M}_{g}^{\beta_{1}+2\beta_{2}+2\beta_{3}}$ . 
Since $
\mathop {\sum }\limits_{j=1}^{s_{i^*}}{p_{i^{*}j}(|\Delta \theta_{i^{*}j}|^{r}+|\Delta 
v_{i^{*}j}|^{r/2}+|\Delta m_{i^{*}j}|^{r/2})} \asymp \overline{M}_{g}^{r}$, 
it is clear that
$F_{\beta_{1},\beta_{2},\beta_{3}}(\theta_{i^{*}}^{0},v_{i^{*}}^{0},m_{i^{*}}^{0})$ vanishes as long as
$\beta_{1}+2\beta_{2}+2\beta_{3} \geq r+1$.
Thus, we only need to concern with $F_{\beta_{1},\beta_{2},\beta_{3}}(\theta_{i^{*}}^{0},v_{i^{*}}^{0},m_{i^{*}}^{0})$ 
when $\beta \in \mathcal{F}_{r}$ and $\beta_{1}+2\beta_{2}+2\beta_{3} \leq r$.
}

From this point on, let $\kappa = (r,r/2,r/2)$. 
Denote the limits for the relevant subsequences, which exist due to the boundedness: 
$\Delta \theta_{i^{*}j}/\overline{M}_{g} \to a_{j}$, 
$\Delta v_{i^{*}j}/
\overline{M}_{g}^{2} \to b_{j}$, and $\Delta m_{i^{*}j}/\overline{M}_{g}^2 \to c_{j}$, 
and $p_{i^{*}j} \to d_{j}^{2}$ for each $j=1,\ldots, s_{i^{*}}$. Here, at least one element 
of $\left(a_{j},b_{j},c_{j}\right)_{j=1}^{s_{i^{*}}}$ equals to -1 or 
1. For any $\beta=(\beta_{1},\beta_{2},\beta_{3})$ such that $\beta \in \mathcal{F}_{r}$ and $
\beta_{1}+2\beta_{2}+2\beta_{3} \leq r$, by dividing the numerator and denominator of 
$F_{\beta_{1},\beta_{2},\beta_{3}}(\theta_{i^{*}}^{0},v_{i^{*}}^{0},m_{i^{*}}^{0})$ by 
$\overline{M}_{g}^{\beta_{1}+2\beta_{2}+2\beta_{3}}$ (i.e., the lowest order of $
\overline{M}_{g}$ in the numerator of $F_{\beta_{1},\beta_{2},\beta_{3}}(\theta_{i^{*}}
^{0},v_{i^{*}}^{0},m_{i^{*}}^{0})$), we obtain the following system of equations
\begin{eqnarray}
\sum \limits_{j=1}^{s_{i^{*}}}{\sum \limits_{\alpha}{\dfrac{P^{\beta_{1},\beta_{2},\beta_{3}}_{\alpha_{1},\alpha_{2},\alpha_{3}}(m_{i^{*}}^{0})}{H^{\beta_{1},\beta_{2},\beta_{3}}_{\alpha_{1},\alpha_{2},\alpha_{3}}(m_{i^{*}}^{0})Q^{\beta_{1},\beta_{2},\beta_{3}}_{\alpha_{1},\alpha_{2},\alpha_{3}}(v_{i^{*}}^{0})}\dfrac{d_{j}^{2}a_{j}^{\alpha_{1}}b_{j}^{\alpha_{2}}c_{j}^{\alpha_{3}}}{\alpha_{1}!\alpha_{2}!\alpha_{3}!}}} = 0, \label{eqn:systemnonlinear_proof}
\end{eqnarray}
where the range of $\alpha=(\alpha_{1},\alpha_{2},\alpha_{3})$ in the above sum satisfies 
$\alpha_{1}+2\alpha_{2}+2\alpha_{3}=\beta_{1}+2\beta_{2}+2\beta_{3}$. The above 
system of polynomial equations is the general version of system of polynomial equations 
\eqref{eqn:systemnonlinear_fourth} that we 
considered in Section \ref{Section:illustration_omixture_byone}. Now, one of the elements 
of $a_{j}$s, $b_{j}$s, $c_{j}$s is non-zero. Since $G \in \mathcal{O}_{k,c_{0}}$ and $
\sum \limits_{j=1}^{s_{i^{*}}}{p_{i^{*}j}} \to p_{i^{*}}^{0}$, we have the constraints 
$d_{j}^{2}>0$ and $\sum \limits_{j=1}^{s_{i^{*}}}{d_{j}^{2}}=p_{i^{*}}^{0}$. However, 
we can remove the constraint on the summation of $d_{j}^{2}$ by putting $d_{j}^{2}
=p_{i^{*}}^{0}(d_{j}^{'})^{2})/\sum \limits_{j=1}^{s_{i}}{(d_{j}^{'})^{2})}$ where the
only constraint on $d_{j}^{'}$s is $d_{j}^{'} \neq 0$ for all $1 \leq j \leq s_{i^{*}}$. As a 
consequence,  when we talk about system of polynomial equations 
\eqref{eqn:systemnonlinear_proof}, we may consider only the constraint $d_{j}^{2} \neq 0$ 
for any $1 \leq j \leq s_{i^{*}}$. 

Recall Definition \ref{def-rsingular_set}, we have established that if $G_0$ is $\kappa$-singular relative to $\Ocal_{k,c_0}$, where $\kappa = (r,r/2,r/2)$,
then the system \eqref{eqn:systemnonlinear_proof} must admit a non-trivial solution for unknowns
$(a_j,b_j,c_j,d_j)_{j=1}^{s_{i^*}}$.  In other words, $G_0$ is \emph{not} $\kappa$-singular 
relative to $\Ocal_{k,c_0}$ as long as the system \eqref{eqn:systemnonlinear_proof} 
does not admit any non-trivial solution.

\paragraph{Step 3: Upper bound for singularity level and indices} 
There are two distinct features of system of polynomial equations \eqref{eqn:systemnonlinear_proof}. First,
$i^{*}$ varies in $\left\{1,2,\ldots,k_{0}\right\}$ as $G \in \mathcal{O}_{k,c_{0}}$ tends to $G_{0}$. Second,
the value of $s_{i^{*}}$ of the subsequence of $G$ is subject to the constraint that $s_{i^{*}} \leq k-k_{0}+1$. 
(This constraint arises due to number of distinct atoms of $G$,  $\sum \limits_{j=1}^{k_{0}}{s_{j}} \leq 
k^{'} \leq k$ and all $s_{j} \geq 1$ for all $1 \leq j \leq k_{0}$). It follows from these two observations that
the system \eqref{eqn:systemnonlinear_proof} admits a non-trivial solution \emph{only if} the system
\eqref{eqn:systemnonlinear} also admits a non-trivial solution, because the latter has a larger number of unknowns (i.e., thus less constrained). 
This cannot be the case if $r \geq \rgeneric(G_0,k)$,
by the definition given in Eq. \eqref{eqn:general_rate_generic_setting}. 
Thus, $G_0$ is \emph{not} $(\rgeneric(G_0,k), \rgeneric(G_0,k)/2, \rgeneric(G_0,k)/2)$-singular relative to $\Ocal_{k,c_0}$.
The conclusion of the theorem is immediate by invoking Proposition\ref{proposition:singularity_set_level}
part (iv) and (v).

\subsection{Auxiliary lemmas}

\begin{lemma} \label{lemma:Wasserstein_semimetric}
For any $\kappa=(\kappa_{1},\ldots,\kappa_{d}) \in \mathbb{N}^{d}$, if all the components of $\kappa$ are not identical then $\widetilde{W}_{\kappa}(G,G_{0})$ is a semi-metric satisfying weak triangle inequality.
\end{lemma}
\begin{proof} The proof of this lemma follows the same argument from Chapter 6 in 
\cite{Villani-03}. In particular, it is clear that $\widetilde{W}_{\kappa}(G,G_{0})=
\widetilde{W}_{\kappa}(G_{0},G)$. Additionally, we can verify that $d_{\kappa}$ 
satisfies weak triangle inequality. In particular, we can demonstrate the following simple 
bound
\begin{eqnarray}
d_{\kappa}(\theta_{1},\theta_{2})+d_{\kappa}(\theta_{2},\theta_{3}) \geq C d_{\kappa}(\theta_{1},\theta_{3}) \nonumber
\end{eqnarray}
where $C = \min \limits_{1 \leq i \leq d}{\dfrac{1}{2^{\kappa_{i}-1}}}$. The proof for 
this inequality is straightforward from the application of Cauchy-Schwarz's inequality. 
However, $d_{\kappa}$ does not satisfy the standard triangle inequality as $\kappa_{i}
$ are not all identical.

Now, from the definition of $\widetilde{W}_{\kappa}$ we have the following equivalent representation
\begin{eqnarray}
\widetilde{W}_{\kappa}(G,G_{0})= \inf \left\{\biggr[E d_{\kappa}^{\|\kappa\|_{\infty}}(X,Y)\biggr]^{1/\|\kappa\|_{\infty}}, \ \text{law}(X)=G, \ \text{law}(Y)=G_{0} \right\}. \nonumber
\end{eqnarray}  
For any finitely discrete probability measures $G_{1},G_{2},G_{3}$, let $(X_{1},X_{2})$ 
represent the optimal coupling for $\widetilde{W}_{\kappa}(G_{1},G_{2})$ and $
(Z_{2},Z_{3})$ be optimal coupling for $\widetilde{W}_{\kappa}(G_{2},G_{3})
$. According to Gluing Lemma (cf. page 11 in \cite{Villani-03}), there exists random 
variables $(X_{1}',X_{2}',X_{3}')$ such that $\text{law}(X_{1}',X_{2}')=\text{law}
(X_{1},X_{2})$ and $\text{law}(X_{2}',X_{3}')=\text{law}(Z_{2},Z_{3})$. Therefore, 
we obtain
\begin{eqnarray}
\widetilde{W}_{\kappa}(G_{1},G_{3}) = \biggr[E d_{\kappa}^{\|\kappa\|_{\infty}}(X_{1}',X_{3}')\biggr]^{1/\|\kappa\|_{\infty}} & \leq & \dfrac{1}{C}\biggr[E \biggr(d_{\kappa}(X_{1}',X_{2}')+d_{\kappa}(X_{2}',X_{3}')\biggr)^{\|\kappa\|_{\infty}}\biggr]^{1/\|\kappa\|_{\infty}} \nonumber \\
& \leq & \dfrac{1}{C}\biggr\{\biggr[E d_{\kappa}^{\|\kappa\|_{\infty}}(X_{1}',X_{2}')\biggr]^{1/\|\kappa\|_{\infty}}+ \nonumber \\
& + &  \biggr[E d_{\kappa}^{\|\kappa\|_{\infty}}(X_{2}',X_{3}')\biggr]^{1/\|\kappa\|_{\infty}}\biggr\} \nonumber \\
& = & \dfrac{1}{C} \biggr[\widetilde{W}_{\kappa}(G_{1},G_{2}) + \widetilde{W}_{\kappa}(G_{2},G_{3})\biggr]. \nonumber
\end{eqnarray} 
As a consequence, generalized transportation distance $\widetilde{W}$ satisfies 'weak' 
triangle inequality. To demonstrate that $\widetilde{W}$ does not satisfy the standard 
triangle inequality, we choose $G_{1}=\delta_{\theta_{1}}$, $G_{2}=
\delta_{\theta_{2}}$, and $G_{3}=\delta_{\theta_{3}}$. As there exist $\theta_{1}, 
\theta_{2}$, and $\theta_{3}$ such that $d_{\kappa}$ does not satisfy the standard 
triangle inequality, it directly implies that these $\widetilde{W}$ does not satisfy 
standard triangle inequality with these choices of $G_{1},G_{2}$, and $G_{3}$. We 
achieve the conclusion of the lemma. 
\end{proof}

\begin{lemma} \label{lemma:skewnormaldistribution} Let $\left\{f(x|\theta,\sigma,m),
\theta \in \Theta_{1}, \sigma \in \Theta_{2}, m \in \Theta_{3} \right\}$ be a class of skew 
normal distribution. Denote $v:=\sigma^{2}$, then 
\begin{eqnarray}
\begin{cases}
\dfrac{\partial^{2}{f}}{\partial{\theta}^{2}}(x|\theta,\sigma,m)-2\dfrac{\partial{f}}{\partial{v}}(x|\theta,\sigma,m)+\dfrac{m^{3}+m}{v}\dfrac{\partial{f}}{\partial{m}}(x|\theta,\sigma,m)=0. \\
2m \dfrac{\partial{f}}{\partial{m}}(x|\theta,\sigma,m)+(m^{2}+1)\dfrac{\partial^{2}{f}}{\partial{m^{2}}}(x|\theta,\sigma,m)+2vm\dfrac{\partial^{2}{f}}{\partial{v}\partial{m}}(x|\theta,\sigma,m)=0.
\end{cases}
\nonumber 
\end{eqnarray}
\end{lemma}
\begin{proof}
Direct calculation yields
\begin{eqnarray}
\dfrac{\partial^{2}{f}}{\partial{\theta^{2}}}(x|\theta,\sigma,m) = \biggr \{\left(-\dfrac{2}{\sqrt{2\pi}\sigma^{3}}+\dfrac{2(x-\theta)^{2}}{\sqrt{2\pi}\sigma^{5}}\right)\Phi\left(\dfrac{m(x-\theta)}{\sigma}\right)- \nonumber \\
\dfrac{2m(m^{2}+2)(x-\theta)}{\sqrt{2\pi}\sigma^{4}}f\left(\dfrac{m(x-\theta)}{\sigma}\right)\biggr \}\exp\left(-\dfrac{(x-\theta)^{2}}{2\sigma^{2}}\right), \nonumber \\
\dfrac{\partial{f}}{\partial{v}}(x|\theta,\sigma,m)  =  \biggr \{\left(-\dfrac{1}{\sqrt{2\pi}\sigma^{3}}+\dfrac{(x-\theta)^{2}}{\sqrt{2\pi}\sigma^{5}}\right)\Phi\left(\dfrac{m(x-\theta)}{\sigma}\right)-\nonumber \\
\dfrac{m(x-\theta)}{\sqrt{2\pi}\sigma^{4}}f\left(\dfrac{m(x-\theta)}{\sigma}\right)\biggr \} \exp\left(-\dfrac{(x-\theta)^{2}}{2\sigma^{2}}\right), \nonumber \\
\dfrac{\partial{f}}{\partial{m}}(x|\theta,\sigma,m)  =  \dfrac{2(x-\theta)}{\sqrt{2\pi}\sigma^{2}}f\left(\dfrac{m(x-\theta)}{\sigma}\right)\exp\left(-\dfrac{(x-\theta)^{2}}{2\sigma^{2}}\right), \nonumber \\
\dfrac{\partial^{2}{f}}{\partial{m^{2}}}(x|\theta,\sigma,m)=\dfrac{-2m(x-\theta)^{3}}{\sqrt{2\pi}\sigma^{4}}f\left(\dfrac{m(x-\theta)}{\sigma}\right)\exp\left(-\dfrac{(x-\theta)^{2}}{2\sigma^{2}}\right), \nonumber \\
\dfrac{\partial^{2}{f}}{\partial{v}\partial{m}}(x|\theta,\sigma,m)=\biggr(-\dfrac{2(x-\theta)}{\sqrt{2\pi}\sigma^{4}}+\dfrac{(m^{2}+1)(x-\theta)^{3}}{\sqrt{2\pi}\sigma^{6}}\biggr)\exp\left(-\dfrac{(x-\theta)^{2}}{2\sigma^{2}}\right). \nonumber
\end{eqnarray}
From these equations, we can easily verify the conclusion of our lemma.
\end{proof}



\section{Appendix B: Impacts of singularity index on parameter estimation}

In this appendix we present minimax lower bounds and convergence rates for parameter estimation that arise from mixing measure's singularity index. 
These results underscore the inhomogeneity of the parameter space that make up the mixing measure's atoms, while generalizing the 
theorems established in Section~\ref{Section:likelihood_Wasserstein_neighborhood}.
We start with the following result that lowers bound distances between mixing densities in terms of generalized transportation metric of corresponding mixing measures based on given singularity index.

\begin{theorem} \label{theorem:singularity_connection_liminf_general} 
Let $\Gcal$ be a class of probability measures on $\Theta$ that have a bounded number of
support points, and fix $G_0 \in \Gcal$. Suppose that $\kappa \in \singset(G_{0}|\Gcal)$, let $r=\|\kappa\|_{\infty}$. There hold:
\begin{itemize}
\item [(i)] If $r<\infty$, then $\inf \limits_{G \in \mathcal{G}} \dfrac{\|p_{G}-p_{G_0}\|_\infty}{\widetilde{W}_{\kappa'}^{\|\kappa'\|_{\infty}}
(G,G_0)} > 0$ for any $\kappa' \succeq \kappa$.
\item [(ii)] If $r<\infty$, then $\inf \limits_{G \in \mathcal{G}} \dfrac{V(p_{G},p_{G_0})}{\widetilde{W}_{\kappa'}^{\|\kappa'\|_{\infty}}
(G,G_0)} > 0$ for any $\kappa' \succeq \kappa$.
\end{itemize}
\begin{itemize}
\item [(iii)] If $1 \leq r<\infty$ and in addition,
\begin{itemize}
\item [(a)] $f$ is $(r+1)$-order differentiable with respect to $\eta$ and for 
some constant $c_{0}>0$,
\begin{eqnarray}
\label{cond-integrated}
{\displaystyle \mathop {\sup }\limits_{\|\eta -\eta'\| \leq c_0}
{\int \limits_{x \in \mathcal{X}}{\left(\dfrac{\partial^{r+1}{f}}
{\partial{\eta^{\alpha}}}(x|\eta)\right)^{2}/f(x|\eta')}dx}} <\infty 
\end{eqnarray}
for any $|\alpha|=r+1$. 
\item [(b)] For any $\kappa' \in \mathbb{R}^d$ such that $(1,\ldots,1) \preceq \kappa' \prec \kappa$, there is a sequence of $G \in \mathcal{G}$ tending
to $G_0$ in generalized transportation distance $\widetilde{W}_{\kappa''}$ where $\kappa''_{i}=\floor{\kappa'_{i}}$ for all $1 \leq i \leq d$ and the coefficients of the $\kappa''$-minimal
form $\xi_{l}^{(\kappa'')}(G)$ satisfy $\xi_{l}^{(\kappa'')}(G)/\widetilde{W}_{\kappa'}^{\|\kappa'\|_{\infty}}(G,G_{0}) \to 0$ for all $l=1,\ldots, T_{\kappa''}$. Additionally, all the masses $p_{ij}$ in the representation \eqref{eqn:representation_overfit} of $G$ are bounded away from 0.
\end{itemize}
Then, for any $\kappa' \in \mathbb{R}^d$ such that $(1,\ldots,1) \preceq \kappa' \prec \kappa$, 
\[\liminf \limits_{G \in \mathcal{G}: \widetilde{W}_{\kappa'}(G,G_0)\rightarrow 0} \dfrac{h(p_{G},p_{G_0})}{\widetilde{W}_{\kappa'}^{\|\kappa'\|_{\infty}}
(G,G_0)} = 0.\]
\item[(iv)] If $r=\infty$ and the condition (a) in part (iii) holds for any $l \in 
\mathbb{N}$ (here, the parameter $r$ in these conditions is replaced by $l$) while the condition (b) in part (iii) holds, then the 
conclusion of part (iii) holds for any $\kappa' \in \mathbb{R}^d$ such that $(1,\ldots,1) \preceq \kappa'  \prec \kappa$.
\end{itemize}
\end{theorem}

The proof of Theorem \ref{theorem:singularity_connection_liminf_general} is similar to that of Theorem \ref{theorem:singularity_connection_liminf} and Theorem \ref{theorem:tight}; it is 
omitted for the brevity of the paper. Now, denote $\nonsingset(\Gcal)=\left\{\kappa \in \mathbb{N}^{d}: \ \exists \ G \in \Gcal \ \text{such that} \ \kappa \in \singset(G|\Gcal) \right\}$. 
Instead of partitioning via singularity levels, we shall use finer structures via singularity indices. Let $\Gcal$ be structured into a sieve of subsets defined by 
the elements of $\nonsingset(\Gcal)$ as follows

\begin{eqnarray*}
\Gcal = \bigcup_{\kappa \in \nonsingset(\Gcal)} \Gcal_{\kappa}, \;\textrm{where}\;\;\;
\Gcal_{\kappa} & := & \biggr \{G \in \Gcal \biggr | \ \exists \ \kappa' \in \singset(G|\Gcal) \ \text{such that} \ \kappa' \preceq \kappa \biggr \}, \textrm{as} \ \kappa \in \nonsingset(\Gcal).\\
\end{eqnarray*}

The following theorem is the counterpart of Theorem~\ref{proposition:convergence_and_minimax}.
\begin{theorem} \label{proposition:convergence_and_minimax_general} 
\begin{itemize}
\item [(a)] Fix $\kappa \in \nonsingset(\Gcal)$. Assume that for any $G_{0} \in 
\Gcal_{\kappa}$, the conditions of part (iii) of Theorem 
\ref{theorem:singularity_connection_liminf_general} hold for $\Gcal_{\kappa}$
(i.e., $\Gcal$ is replaced by $\Gcal_{\kappa}$ in that theorem) and $\kappa'$ satisfies condition (b) of part (iii) of that theorem as long as $\kappa' \in \lev(G_{0}|\Gcal)$. 
Then, for any $\kappa' \in \mathbb{R}^d$ such that $(1,\ldots,1) \preceq \kappa'  \prec \kappa$ there holds
\begin{eqnarray}
\mathop {\inf }\limits_{\widehat{G}_{n} \in \Gcal_\kappa}
\sup_{G_0 \in \Gcal_\kappa} E_{p_{G_{0}}} \widetilde{W}_{\kappa'} (\widehat{G}_{n},G_{0})
\gtrsim n^{-1/2\|\kappa'\|_{\infty}}. \nonumber
\end{eqnarray}
Here, the infimum is taken over all sequences of estimates $\widehat{G}_{n} \in \Gcal_{\kappa}$ 
and $E_{p_{G_{0}}}$ denotes the expectation taken with respect to product measure with mixture density $p_{G_{0}}^{n}$.

\item [(b)] Let $G_0 \in \Gcal_{\kappa}$ for some fixed $\kappa \in \nonsingset(\Gcal)$. Let $\widehat{G}_n 
\in \Gcal_{\kappa}$ be a point estimate for $G_0$, which is obtained from an $n$-sample of i.i.d. observations
drawn from $p_{G_0}$. As long as $h(p_{\widehat{G}_n},p_{G_0}) = O_P(n^{-1/2})$, we have
\begin{eqnarray}
 \widetilde{W}_{\kappa}(\widehat{G}_{n},G_{0}) = O_P(n^{-1/2\|\kappa\|_{\infty}}). \nonumber
\end{eqnarray}
\end{itemize}
\end{theorem}
We have some remarks regarding Theorem 
\ref{proposition:convergence_and_minimax_general}: (i) If we define $\kappa_{i}
^{(\min)}=\min \limits_{\kappa \in \singset(G_{0}|\Gcal)}{\kappa_{i}}$ for any $1 \leq i 
\leq d$, then the result of part (b) 
implies that the 
best possible convergence rate of estimating i-th component of atoms of $G_{0}$, in a local minimax sense, is 
$n^{-1/2\kappa_{i}^{(\min)}}$ as $1 \leq i \leq d$; (ii) The result of part (b) implies that 
if $G_{0}$ is not $\kappa'$-singular relative to $\Gcal$ for some $\kappa' \in \mathbb{N}^{d}$, then we obtain 
\begin{eqnarray}
\widetilde{W}_{\kappa'}(\widehat{G}_{n},G_{0})=O_{P}(n^{-1/2\|\kappa'\|_{\infty}}) \nonumber
\end{eqnarray}
where $\widehat{G}_{n}$ is a sequence of estimates specified in part (b) of Theorem \ref{proposition:convergence_and_minimax_general}.

\newpage

%

\section{Appendix C: Complete inhomogeneity via singularity matrices} \label{Section:singularity_matrix}

 In this appendix, we provide a new notion, namely, singularity matrix, which characterizes the inhomogeneity of convergence behavior of different atoms of a fixed mixing measure under a finite mixture model. 
To simplify the presentation, we will only consider the e-mixtures setting, i.e.,
the number of mixture components is known. Let $G_{0}=G_{0}(\vec{p}^{0},\vec{\eta}^{0})$ 
be a true mixing measure with weights $\vec{p}^{0}$ and atoms $\vec{\eta}^{0}$ 
such that it has $k_{0}$ components. To capture distinct 
convergence rates of atoms of $G_{0}$, we introduce a blocked version of transportation distance
with respect to $G_{0}$ as follows.
Recall that for any $\kappa=(\kappa_{1},\ldots,\kappa_{d}) \in \mathbb{N}^{d}$, we defined
\begin{eqnarray}
d_{\kappa}(\theta_{1},\theta_{2}) : = \biggr(\sum \limits_{i=1}^{d}{|\theta_{1}^{(i)}-\theta_{2}^{(i)}|^{\kappa_{i}}}\biggr)^{1/\| \kappa \|_{\infty}} \nonumber
\end{eqnarray}
for any $\theta_{i}=(\theta_{i}^{(1)},\ldots,\theta_{i}^{(d)}) \in \mathbb{R}^{d}$ as 
$1 \leq i \leq 2$. 
\begin{definition} \label{def:blocked_generalized_Wasserstein}
Given any matrix $K \in \mathbb{N}^{k_{0} \times d}$ and mixing 
measure $G_{0}(\vec{p}^{0},\vec{\eta}^{0})$ with $k_{0}$ components, the blocked 
transportation distance with respect to matrix $K$ and mixing measure $G_{0}
$ is given by
\begin{eqnarray}
\widehat{W}_{K}(G(\vec{p},\vec{\eta}),G_{0}(\vec{p}^{0},\vec{\eta}^{0})) := \biggr (\inf \sum_{i,j} q_{ij} 
d_{K_{j}}^{\| K_{j} \|_{\infty}}(\eta_i ,\eta_j^{0}) \biggr )^{1/\| K \|_{\infty}} \nonumber
\end{eqnarray}
for any probability measure $G=G(\vec{p},\vec{\eta})$ with weights $\vec{p}$ and 
atoms $\vec{\eta}$ such that it has exactly $k_{0}$ components. The infimum in the 
above formulation is taken over all couplings $\vec{q}$ between $\vec{p}$ and $
\vec{p}^0$, while $K_{j}$ is the $j$-th row of matrix $K$ for any $1 \leq j \leq k_{0}$ 
and $\|K\|_{\infty} := \max \limits_{1 \leq i \leq k_{0}}{\left\{\|K_{i}\|_{\infty}\right\}}$. 
\end{definition}
Henceforth in this appendix, for any matrix $K$ we denote $K_{i}$ as its $i$-th row and 
$K_{ij}$ as its element in the $i$-th row and $j$-th column. Additionally, for any two 
matrices $K, K' \in \mathbb{N}^{k_{0} \times d}$, we denote $K \preceq K'$ if $K_{i} 
\preceq K_{i}'$ for all $1 \leq i \leq k_{0}$. In general, the blocked transportation
distance is not a metric as it is not symmetric except when all the columns of $K$ are 
identical. Furthermore, it also does not satisfy the standard triangle inequality except when 
all the elements of $K$ are identical. We have the following result relating blocked 
transportation distance to generalized transportation distance and standard 
Wasserstein distance.

\begin{proposition} \label{proposition:blocked_generalized_Wasserstein}
Let $K \in \mathbb{N}^{k_{0} \times d}$ be any matrix. Then, we have
\begin{eqnarray}
\widehat{W}_{K}(G,G_{0}) \gtrsim \widetilde{W}_{\kappa}(G,G_{0}) \gtrsim W_{\|\kappa\|_{\infty}}(G,G_{0}) \nonumber
\end{eqnarray}
for any $\kappa$ such that $K_{i} \preceq \kappa$ for all $1 \leq i \leq d$ and $\|K\|_{\infty}=\|\kappa\|_{\infty}$. The first inequality holds when $K_{i}$ are all equal to $\kappa$.
The second inequality holds when all the elements of $K$ are equal to $\|\kappa\|_{\infty}$.
\end{proposition}
The matrix $K$ utilized in the definition of blocked transportation distance helps to 
capture fine-grained differences in convergence behavior toward individual components
of each distinct atom of the true mixing measure $G_{0}$. 
In particular, assume that a sequence of probability measures $G_n \in \Ecal_{k_{0}}$ 
tending to $G_0$ under $\widehat{W}_K$ distance at a rate $\omega_n = o(1)$ 
for some matrix $K \in \mathbb{N}^{k_{0} \times d}$. Similar to the interpretation with generalized transportation 
metric, the $i$-th atom of $G_0$ receives a converging sequence from an atom of $G_n$, 
to be also labeled $i$, at the rate $(\omega_n)^{\|K\|_{\infty}/\|K_{i}\|_{\infty}}$ under $d_{K_{i}}$ 
semi-metric. Hence, the $j$-th component of $i$-th atom of $G_{n}$ converges to the $j
$-th component of corresponding atom of $G_{0}$ at rate $(\omega_{n})^{\|K\|
_{\infty}/K_{ij}}$ for any $1 \leq i \leq k_{0}$ and $1 \leq j \leq d$.

Similar to approaches to the singularity level and singularity index, we are also interested in
analyzing the behavior of likelihood function $p_{G}$ as $G \in \Ecal_{k_{0}}$ 
varies in a neighborhood of $G_{0}$ according to the blocked transportation distance.
To do so we also develop notions corresponding to our choice of blocked transportation distance. 
To avoid the ambiguity in our argument, we will 
repeat the key steps from establishing singularity level and singularity index in
developing the key notions. In particular, for 
any fixed matrix $K \in \mathbb{N}^{k_{0} \times d}$ we consider a sequence $G_{n} 
\in \Ecal_{k_{0}}$ such that $\widehat{W}(G_{n},G_{0}) \to 0$. We can argue that up 
to a permutation of atoms' labels, $G_{n}$ can be represented as 
\eqref{eqn:representation_overfit} where $\overline{l}=0$ and $s_{i}=1$ for all $1 \leq 
i \leq k_{0}$, i.e., $G_{n} = \sum \limits_{i=1}^{k_{0}}{p_{i}^{n}\delta_{\eta_{i}
^{n}}}$.  To avoid notational cluttering, we 
also drop $n$ from the superscript when the context is clear. Now, denote $\|K\|_{\infty}=r$. By carrying Taylor expansion up to $r$-th order, we achieve that
\begin{eqnarray}
\frac{p_{G}(x)-p_{G_{0}}(x)}{\widehat{W}_{K}^{r}(G,G_{0})}
=
\sum_{|\alpha|=1}^{r}
\sum_{i=1}^{k_{0}}
\biggr (\frac{p_{i}(\Delta \eta_{i})^{\alpha}/\alpha!}
{\widehat{W}_{K}^{r}(G,G_{0})}\biggr)
\frac{\partial^{|\alpha|}f}{\partial \eta^{\alpha}}
(x|\eta_i^0)
+ \sum_{i=1}^{k_0} \frac{\Delta p_{i}}{\widehat{W}_{K}^{r}(G,G_{0})}
f(x|\eta_{i}^{0}) +o(1) \nonumber
\end{eqnarray}
as long as $f$ is uniform Lipschitz up to order $r$. The above representation motivates the following definition of matrix minimal form
\begin{definition} \label{definition:minimal_matrix}
For any $K \in \mathbb{N}^{k_{0} \times d}$, the following representation is called $K$-\emph{minimal} form of
the mixture likelihood for a sequence of mixing measures $G$ tending to $G_0$
in $\widehat{W}_{K}$ distance:
\begin{eqnarray}
\frac{p_{G}(x)-p_{G_{0}}(x)}{\widehat{W}_K^{\|K\|_{\infty}}(G,G_{0})} = 
\sum_{l=1}^{T_K}
\biggr (\frac{\xi_{l}^{(K)}(G)}{\widehat{W}_K^{\|K\|_{\infty}}(G,G_{0})} \biggr ) H_{l}^{(K)}(x) + o(1), 
\label{eqn:generallinearindependencerepresentation_matrix}
\end{eqnarray}
which holds for all $x$, with the index $l$ ranging from 1 to a finite $T_K$,
if 
\begin{itemize}
\item [(1)] $H_{l}^{(K)}(x)$ for all $l$ are linearly independent functions of $x$, and
\item [(2)] coefficients $\xi_{l}^{(K)}(G)$ are polynomials of
the components of $\Delta \eta_{ij}$, and $\Delta p_{i\cdot}, p_{ij}$.
\end{itemize}
\end{definition}
It is clear that $K$-minimal form is a general version of $\kappa$-minimal form for any $\kappa \in \mathbb{N}^{d}$ when $K_{i}=\kappa$ for all $1 \leq i \leq k_{0}$. Similar to $\kappa$-minimal form in Section \ref{Section:inhomogeneity_singularity}, matrix $K$-minimal form leads to our notion of matrix singularity, which we
now define.
\begin{definition}\label{def-singular_matrix}
For any $K \in \mathbb{N}^{k_{0} \times d}$, we say that $G_{0}$ is $K$-singular relative to $\Ecal_{k_{0}}$, if $G_0$ admits a $K$-minimal form 
given by Eq. \eqref{eqn:generallinearindependencerepresentation_matrix}, according to
which there exists a sequence of $G \in \Ecal_{k_{0}}$ tending to $G_0$ under $\widehat{W}_{K}$ distance
such that
\[\xi_l^{(\kappa)}(G)/\widehat{W}_{K}^{\|K\|_{\infty}}(G,G_0) \rightarrow 0 \;\textrm{for all}\; l=1,\ldots,T_K.\]
\end{definition}
We note that the limiting behavior of ratios (coefficients) $\xi_l^{(K)}(G)/\widehat{W}_{K}^{\|K\|
_{\infty}}(G,G_0)$ is generally more challenging to investigate than those of 
coefficients of $\kappa$-minimal form in Definition \ref{def-rsingular_set} and $r$-minimal 
form in Definition \ref{def-rsingular}, due to the complex nature of blocked transportation distance
 $\widehat{W}_{K}^{\|K\|_{\infty}}(G,G_0)$. In particular, we can 
demonstrate that as $\widehat{W}_K^{\|K\|_{\infty}}(G,G_0)  \to 0$, $
\widehat{W}_K^{\|K\|_{\infty}}(G,G_0) \asymp D_K(G,G_0)$ where
\begin{eqnarray}
D_K(G,G_0) := 
\sum_{i=1}^{k_{0}} p_{i}
d_{K_{i}}^{\|K_{i}\|_{\infty}}(\eta_{i},\eta_{i}^{0}) + \sum_{i=1}^{k_0} |\Delta p_{i}| . \nonumber
\end{eqnarray}
As $\widehat{W}_{K}^{\|K\|_{\infty}}(G,G_0)$ is asymptotically equivalent to a rather
complicated inhomogeneous semipolynomial form, the vanishing of ratios $\xi_l^{(K)}(G)/
\widehat{W}_{K}^{\|K\|_{\infty}}(G,G_0)$ will be difficult to fathom if the values of 
vector $K_{i}$ are very different.
Similar to singularity levels and singularity indices, the matrix 
singularity notion also possesses a crucial monotonic property in terms of a partial order of matrices:
\begin{lemma} \label{lemma:minimal_form_matrix}
(a) (Invariance) The existence of the sequence of $G$ in the statement of Definition~\ref{def-singular_matrix}
holds for all $K$-minimal forms once it holds for at least one $K$-minimal form.

(b) (Monotonicity) If $G_0$ is $K$-singular to $\Ecal_{k_{0}}$ for some $K \in \mathbb{N}^{k_{0} \times d}$, then $G_0$ is $K'$-singular for any $K' \preceq K$.
\end{lemma} 
The monotonicity of $K$-singularity leads to the following notion of
singularity matrix of a mixing measure $G_0$ relative to an ambient space $\Ecal_{k_{0}}$.

\begin{definition} \label{definition:singularity_matrix} 
For any $K \in \overline{\mathbb{N}}^{k_{0} \times d}$, we say $K$ is a singularity 
matrix of $G_{0}$ relative to a given class $\Ecal_{k_{0}}$ if and only if $G_{0}$ is 
$K'$-singular relative to $\Ecal_{k_{0}}$ for any $K' \prec K$, and there is no $K' 
\succeq K$ such that $G_{0}$ remains $K'$-singular relative to $\Ecal_{k_{0}}$.
\end{definition}
Denote $\singmat(G_{0}|\Ecal_{k_{0}})= \left\{K \in \mathbb{N}^{k_{0} \times d}: \ 
K \ \text{is singularity matrix of} \ G_{0} \right\}$, i.e., the set of all singularity 
matrices of $G_{0}$ relative to $\Ecal_{k_{0}}$. The significance of singularity matrix 
notion can be summarized by the following results
\begin{theorem} \label{theorem:singularity_connection_liminf_convergence_rate_matrix} 
Fix $G_0 \in \Ecal_{k_{0}}$. Take a $K \in \singmat(G_{0}|\Ecal_{k_{0}})$, let $r=\|K\|_{\infty}$. 
\begin{itemize}
\item [(i)] If $r<\infty$, then $\inf \limits_{G \in \mathcal{G}} \dfrac{\|p_{G}-p_{G_0}\|_\infty}{\widehat{W}_{K'}^{\|K'\|_{\infty}}
(G,G_0)} > 0$ for any $K' \succeq K$.
\item [(ii)] If $r<\infty$, then $\inf \limits_{G \in \mathcal{G}} \dfrac{V(p_{G},p_{G_0})}{\widehat{W}_{K'}^{\|K'\|_{\infty}}
(G,G_0)} > 0$ for any $K' \succeq K$.
\item [(iii)] Let $\widehat{G}_n 
\in \Ecal_{k_{0}}$ be a point estimate for $G_0$, which is obtained from an $n$-sample of i.i.d. observations
drawn from $p_{G_0}$. As long as $h(p_{\widehat{G}_n},p_{G_0}) = O_P(n^{-1/2})$ and $r<\infty$, we obtain
\begin{eqnarray}
 \widehat{W}_{K'}(\widehat{G}_{n},G_{0}) = O_P(n^{-1/2\|K'\|_{\infty}}). \nonumber
\end{eqnarray}
for any $K' \succeq K$.
\end{itemize}
\end{theorem}
Let $K_{ij}^{(\min)}=\min \limits_{K \in \singmat(G_{0}|\Ecal_{k_{0}})}
{K_{ij}}$ for any $1 \leq i \leq k_{0}$ and $1 \leq j \leq d$. Part (iii) provides a
guarantee for the estimation rate of the $j$-th component of $i$-th atom of $G_{0}$ 
to be at most $n^{-1/2K_{ij}^{(\min)}}$.

\subsection{Examples of singularity matrices with e-mixtures}
\label{Section:example_singularity_matrix}
In this section, we provide several examples of singularity matrices of e-mixtures models 
that we have studied thus far in the paper, including e-mixtures of first order identifiable 
kernels, Gamma e-mixtures, and skew-normal e-mixtures. 
The proofs of these results are quite similar to those for singularity levels and indices,
and will be omitted for the brevity of the paper.

\paragraph{E-mixtures of first order identifiable kernels} As studied by 
\cite{Ho-Nguyen-EJS-16}, the first order identifiability of kernel density $f$ means that the 
collection of $\left\{\partial^\kappa f/\partial \eta^\kappa(x|\eta_j)
| j=1,\ldots, k_0; |\kappa| \right. \\ \left. \leq 1 \right\}$ evaluated at $G_{0}$ are
linearly independent. 
It is not difficult to establish the following result regarding singularity matrices of 
$G_{0}$ under first order identifiability condition of $f$.
\begin{proposition} \label{proposition:strong_identifiability_singularity_matrix}
Assume that $f$ is first order identifiable and admits uniform Lipschitz condition up to 
the first order. Then, $K=\vec{1}_{k_{0} \times d}$ is the unique element of $
\singmat(G_{0}|\Ecal_{k_{0}})$ where $\vec{1}_{k_{0} \times d}$ is the matrix with 
all elements to be 1.
\end{proposition}

\paragraph{Gamma e-mixtures} 
The generic and pathological cases of $G_{0}$ for Gamma mixtures have been
discussed in detail in Section \ref{Section:singularity_index_classical_mixtures}. 
We have the following result for Gamma mixtures:

\begin{proposition} \label{proposition:singularity_matrix_Gamma_emix}
For any $G_{0} \in \Ecal_{k_{0}}$, we obtain
\begin{itemize}
\item[(a)] Generic cases: $K=\vec{1}_{k_0 \times 2}$ is the unique singularity matrix 
of $\singmat(G_{0}|\Ecal_{k_{0}})$.
\item[(b)] Pathological cases: let $A=\left\{i: \exists j \ \text{such that} \left\{|a_{i}
^{0}-a_{j}^{0}|,|b_{i}^{0}-b_{j}^{0}|\right\} = \left\{1,0\right\} \right\}$. Let $K \in 
\overline{\mathbb{N}}^{k_{0} \times 2}$ such that $K_{i}=(\infty,\infty)$ when $i \in A$ 
and $K_{i}=(\infty,\infty)$ when $i \in A^{c}$. Then, matrix $K$ is the unique element of $
\singmat(G_{0}|\Ecal_{k_{0}})$.
\end{itemize}
\end{proposition}
We wish to emphasize that the non-polynomial convergence rates of 
atoms in $A$ under pathological cases are just the upper bounds. It is possible 
that the actual convergence rates of these atoms in $A$ may be 
better. We leave this question for future exploration.  
\paragraph{Skew-normal e-mixtures} Again, our result is restricted to e-mixtures,
under the setting of $G_{0} \in \Ecal_{k_{0}}$ under skew 
normal e-mixtures. We first start with the following result when $G_{0} \in \Scal_{0}$, 
i.e., $P_{1}(\vec{\eta}^{0})P_{2}(\vec{\eta}^{0}) \neq 0$.

\begin{proposition} \label{proposition:singularity_matrix_S0}
If $G_{0} \in \Scal_{0}$, then $K=\vec{1}_{k_{0} \times 3}$ is the unique 
element of $\singmat(G_{0}|\Ecal_{k_{0}})$.
\end{proposition}

When $G_{0} \not \in \Scal_{0}$, we further consider subsets of the complement of $\Scal_0$.
In particular, we give results for two subsets of the complement: $G_{0} 
\in \Scal_{1}$ or $G_{0} \in \Scal_{2}$, which is defined in 
Section \ref{Section:singularity_level_index_skewemix}.  
Briefly speaking, in that definition, for each index $i=1,\ldots, k_0$,
$I_{i}$ collects all the atoms homologous to the $i$-th atom of $G_{0}$.
The following result establishes singularity matrix of $G_{0} \in \Scal_{1}$.

\begin{proposition} \label{proposition:singularity_matrix_S1}
Given $G_{0} \in \Scal_{1}$. Denote $A=\left\{i: I_{i} \ \text{has more than one 
elements}\right\}$. Let $K \in \overline{\mathbb{N}}^{k_{0} \times 3}$ be such that 
$K_{i}=(1,1,2)$ for all $i \in A$ and $K_{i}=(1,1,1)$ for all $i \in A^{c}$. Then, matrix
$K$ is the unique element of $\singmat(G_{0}|\Ecal_{k_{0}})$.
\end{proposition}

As a consequence of Proposition \ref{proposition:singularity_matrix_S1}, the atoms 
of $G_{0} \in \Scal_{1}$ can be divided into two blocks according to their convergence 
rates. For those without any homologous structure, their convergence rates are $n^{-1/
2}$ (up to a log factor); however, for those with homologous structure with conformant property, their 
convergence rates of location, scale parameters admit $n^{-1/2}$ convergence rate
while shape parameters admit $n^{-1/4}$ rate. Under the setting of $G_{0} \in \Scal_{2}$, 
the singularity matrix becomes somewhat more complicated
as being demonstrated by the following result.

\begin{proposition} \label{proposition:singularity_matrix_S2}
Given $G_{0} \in \Scal_{2}$. Denote $A=\left\{i: I_{i} \ \text{has more than one 
elements}\right\}$ and $B=\left\{j: m_{j}^{0} \right. \\ \left. =0 \right\}$. Let $K \in 
\overline{\mathbb{N}}^{k_{0} \times 3}$ be such that $K_{i}=(1,1,2)$ for all $i \in A
$, $K_{i}=(3,2,3)$ for all $i \in B$, and $K_{i}=(1,1,1)$ for all $i \in (A \cup B)^{c}$. 
Then, matrix $K$ is the unique element of $\singmat(G_{0}|\Ecal_{k_{0}})$.
\end{proposition}
The above result indicates that the atoms of $G_{0} \in \Scal_{3}$ can be divided into 
three blocks according to their convergence rates. For the ``Gaussian atoms'' in $G_{0}$, 
i.e., those indexed by set $B$,
the convergence rates of location and shape parameters of these atoms are $n^{-1/6}$ 
while those of scale parameters of these atoms are $n^{-1/4}$. For other atoms of 
$G_{0}$, we arrive at the same regime of convergence as those in Proposition 
\ref{proposition:singularity_matrix_S1}.

\newpage


\section{Appendix D: Proofs for Sections \ref{Section:general_procedure_singularity} and \ref{Section:overfitskew}}
For completeness we collect the remaining proofs of the statements described in the main text.

\subsection{Proofs for Section \ref{Section:general_procedure_singularity}}
\paragraph{PROOF OF THEOREM~\ref{theorem:tight}}
Since the proofs for part (i) and (ii) are similar, we only provide the proof for part (i). 
The proof of this part is the generalization of that of part (c) in 
Theorem 3.2 in \cite{Ho-Nguyen-EJS-16}. 
By means of Taylor expansion up to $r$-th order, we have
\begin{eqnarray}
h^{2}(p_{G},p_{G_{0}})< \int \limits_{\mathcal{X}}{\dfrac{(p_{G}(x)-p_{G_{0}}(x))^{2}}{p_{G_{0}}(x)}}\textrm{d}x & = & \int \limits_{\mathcal{X}}{\dfrac{\biggr(\sum \limits_{l=1}^{T_{r}}{\xi_{l}^{(r)}(G)H_{l}^{(r)}}(x)+R_{r}(x)\biggr)^{2}}{p_{G_{0}}(x)}}\textrm{d}x \nonumber \\
& \lesssim & \int \limits_{ \mathcal{X}}{\dfrac{\sum \limits_{l=1}^{T_{r}}{\biggr(\xi_{l}^{(r)}(G)H_{l}^{(r)}(x)\biggr)^{2}}+R_{r}^{2}(x)}{p_{G_{0}}(x)}}\textrm{d}x  \nonumber
\end{eqnarray}
where the last inequality is due to Cauchy-Schwarz's inequality. Here, $R_{r}(x)$ has the following form
\begin{eqnarray}
R_{r}(x)=\sum \limits_{i=1}^{k_{0}}{\sum \limits_{j=1}^{s_{i}}{\sum \limits_{|\alpha|=r+1}{\dfrac{r+1}{\alpha !}{(\Delta \eta_{ij})^{\alpha}\int \limits_{0}^{1}{(1-t)^{r}\dfrac{\partial^{r+1}{f}}{\partial{\eta^{\alpha}}}(x|\eta_{i}^{0}+t\Delta \eta_{ij})}\textrm{d}t}}}}. \nonumber
\end{eqnarray}
Due to condition (a), the following result holds
\begin{align}
\int \limits_{\mathcal{X}} \biggr(H_{l}^{(r)}(x)\biggr)^2/p_{G_{0}}(x)dx < \infty \nonumber
\end{align}
for any $1 \leq l \leq T_{r}$. Combining the above result with the assumption that $\xi_{l}^{(r)}/W_{1}^{s}(G,G_{0}) \to 0$ in (b) for all $s \in [1,r+1)$ and $l = 1,\ldots,T_{r}$, we achieve that 
\begin{eqnarray}
\int \limits_{ \mathcal{X}}\dfrac{\biggr(\xi_{l}^{(r)}(G)H_{l}^{(r)}(x)\biggr)^{2}}{W_{1}^{2s}(G,G_{0})p_{G_{0}}(x)}\textrm{d}x \to 0 \label{eqn:singularity_connection_liminf_first}
\end{eqnarray}
for all $1 \leq l \leq T_{r}$. Additionally, as $p_{G_{0}}(x)>p_{i}^{0}f(x|\eta_{i}^{0})$ for all $1 \leq i \leq k_{0}$, for any $s<r+1$, we have
\begin{eqnarray}
\dfrac{h^{2}(p_{G},p_{G_{0}})}{W_{1}^{2s}(G,G_{0})} & \lesssim & \int \limits_{ \mathcal{X}}{\dfrac{R_{r}^{2}(x)}{W_{1}^{2s}(G,G_{0})p_{G_{0}}(x)}}\textrm{d}x \nonumber \\
& \lesssim & \sum \limits_{i=1}^{k_{0}}{\int \limits_{ \mathcal{X}}{\dfrac{\biggr(\sum \limits_{j=1}^{s_{i}}{\sum \limits_{|\alpha|=r+1}{\dfrac{r+1}{\alpha !}{(\Delta \eta_{ij})^{\alpha}\int \limits_{0}^{1}{(1-t)^{r}\dfrac{\partial^{r+1}{f}}{\partial{\eta^{\alpha}}}(x|\eta_{i}^{0}+t\Delta \eta_{ij})} \textrm{d}t\biggr)^{2}}}}}{W_{1}^{2s}(G,G_{0})p_{i}^{0}f(x|\eta_{i}^{0})}}\textrm{d}x} \nonumber \\
& \lesssim & \sum \limits_{i=1}^{k_{0}}{\int \limits_{ \mathcal{X}}{\dfrac{\sum \limits_{j=1}^{s_{i}}{\sum \limits_{|\alpha|=r+1}{\biggr(\dfrac{r+1}{\alpha !}{(\Delta \eta_{ij})^{\alpha}\int \limits_{0}^{1}{(1-t)^{r}\dfrac{\partial^{r+1}{f}}{\partial{\eta^{\alpha}}}(x|\eta_{i}^{0}+t\Delta \eta_{ij})} \textrm{d} t\biggr)^{2}}}}}{W_{1}^{2s}(G,G_{0})p_{i}^{0}f(x|\eta_{i}^{0})}}\textrm{d}x}, \nonumber
\end{eqnarray}
where the last inequality is due to Cauchy-Schwarz's inequality. Now, for any $s<r+1$, by utilizing Lemma \ref{lemma:bound_overfit_Wasserstein}  and the assumption that $p_{ij}$ are bounded away from 0 in condition (b), 
we obtain that
\begin{eqnarray}
\dfrac{|(\Delta \eta_{ij})^{\alpha}|}{W_{1}^{s}(G,G_{0})} \asymp \dfrac{|(\Delta \eta_{ij})^{\alpha}|}{D_{1}^{s}(G_{0},G)} \lesssim \dfrac{|(\Delta \eta_{ij})^{\alpha}|}{\|\Delta \eta_{ij}\|^{s}} \to 0, \label{eqn:singularity_connection_liminf_second}
\end{eqnarray} 
for any $|\alpha|=r+1$.
According to the hypothesis, as $\Delta \eta_{ij}<c_{0}$, we have
\begin{eqnarray}
\int \limits_{ \mathcal{X}}{\dfrac{\biggr(\int \limits_{0}^{1}{(1-t)^{r}\dfrac{\partial^{r+1}{f}}{\partial{\eta^{\alpha}}}(x|\eta_{i}^{0}+t\Delta \eta_{ij})} \textrm{d}t\biggr)^{2}}{p_{i}^{0}f(x|\eta_{i}^{0})}}\textrm{d}x < \int \limits_{x \in \mathcal{X}}{\dfrac{\biggr(\dfrac{\partial^{r+1}{f}}{\partial{\eta^{\alpha}}}(x|\eta_{i}^{0}+t\Delta \eta_{ij})\biggr)^{2}}{p_{i}^{0}f(x|\eta_{i}^{0})}}\textrm{d}x < \infty. \label{eqn:singularity_connection_liminf_third}
\end{eqnarray}
By combining \eqref{eqn:singularity_connection_liminf_first}, \eqref{eqn:singularity_connection_liminf_second}, and \eqref{eqn:singularity_connection_liminf_third}, we achieve $h(p_{G},p_{G_{0}})/W_{1}^{s}(G,G_{0}) \to 0$, which yields the conclusion of this part.

\paragraph{PROOF OF LEMMA~\ref{lemma:minimal_form_general}}
The proof idea of Lemma \ref{lemma:minimal_form_general} is a generalization of that of Lemma \ref{lemma:minimal_form}. 

(a) The existence of the sequence of $G$ described in the definition of a $\kappa$-minimal form
implies for that sequence, $(p_G(x)-p_{G_0}(x))/\widetilde{W}_{\kappa}^{\|\kappa\|_{\infty}}(G,G_0) 
\rightarrow 0$ holds for almost all $x$. 
Now take any $\kappa$-minimal form \eqref{eqn:generallinearindependencerepresentation_general}
given by the same sequence. Let $C(G) = \max_{l=1}^{T_\kappa}
\frac{\xi_{l}^{(\kappa)}(G)}{\widetilde{W}_{\kappa}^{\|\kappa\|_{\infty}}(G_0,G)}$. 
We will show that $\liminf C(G) = 0$, which concludes the proof.
Suppose this is not the case, so we have $\liminf C(G) > 0$. It follows that
\[\sum_{l=1}^{T_r}
\biggr (\frac{\xi_{l}^{(\kappa)}(G)}{C(G) \widetilde{W}_{\kappa}^{\|\kappa\|_{\infty}}(G_{0},G)} \biggr ) H_{l}^{(\kappa)}(x)
\rightarrow 0.\]
Moreover, all the coefficients in the above display are bounded from above by 1,
one of which is in fact 1. There exists a subsequence of $G$ by which
these coefficients have limits, one of which is 1. This is a contradiction
due to the linear independence of functions $H_l^{(\kappa)}(\cdot)$. 

(b) It suffices to establish the conclusion of this part when $\|\kappa\|_{\infty} -1 \leq \|\kappa'\|_{\infty}
\leq \|\kappa\|_\infty$. Let $G$ be an element in the sequence that admits a $\kappa$-minimal form
such that $\xi_{l}^{(\kappa)}(G)/\widetilde{W}_{\kappa}^{\|\kappa\|_{\infty}}(G_0,G) \rightarrow 0$
for all $l=1,\ldots,T_\kappa$. It suffices to assume that the basis functions $H_{l}^{(\kappa)}$
are selected from the collection of partial derivatives of $f$. 
We will show that the same sequence of $G$ and the elimination procedure for 
the $\kappa$-minimal form can be used to construct a $\kappa'$-minimal form by which
\[\xi_{l}^{(\kappa')}(G)/\widetilde{W}_{\kappa'}^{\|\kappa'\|_{\infty}}(G_0,G) \rightarrow 0\]
for all $l=1,\ldots,T_{\kappa'}$. When $\|\kappa'\|_{\infty}=\|\kappa\|_{\infty}$, as $\kappa' \preceq \kappa$ and the support points of $G$ and $G_0$ are in a bounded set, it is straightforward that $\widetilde{W}_{\kappa'}^{\|\kappa'\|_{\infty}}(G,G_{0}) \gtrsim \widetilde{W}_{\kappa}^{\|\kappa\|_{\infty}}(G,G_{0})$.  Therefore, by choosing $T_{\kappa}=T_{\kappa'}$ and $\xi_{l}^{(\kappa')}(G)=\xi_{l}^{(\kappa)}(G)$ for any $1 \leq l \leq T_{\kappa'}$, we obtain
\begin{eqnarray}
\xi_{l}^{(\kappa')}(G)/\widetilde{W}_{\kappa'}^{\|\kappa'\|_{\infty}}(G_0,G) \lesssim \xi_{l}^{(\kappa)}(G)/\widetilde{W}_{\kappa}^{\|\kappa\|_{\infty}}(G_0,G) \to 0 \nonumber
\end{eqnarray}
for any $1 \leq l \leq T_{\kappa'}$. This results in a valid $\kappa'$-form.
It remains to consider the case $\|\kappa'\|_{\infty}=\|\kappa\|_{\infty}-1$. There are two possibilities.

First, suppose that each of the $\|\kappa\|_{\infty}$-th partial derivatives of density kernel $f$
(i.e., $\partial^\alpha f/\partial \eta^\alpha$, where $|\alpha|= \|\kappa\|_{\infty}$)
is not in the linear span of the collection of
partial derivatives of $f$ at order $\|\kappa\|_{\infty}-1$ or less. Then, for each $l=1,\ldots, T_{\kappa'}$,
$\xi_{l}^{(\kappa')}(G) = \xi_{l'}^{(\kappa)}(G)$ for some $l'\in [1,T_{\kappa}]$.
Since $\widetilde{W}_{\kappa'}^{\|\kappa'\|_{\infty}}(G,G_{0}) \gtrsim \widetilde{W}_{\kappa}^{\|\kappa\|_{\infty}}(G,G_{0})$,  
we have that 
\[\xi_{l}^{(\kappa')}(G)/\widetilde{W}_{\kappa'}^{\|\kappa'\|_{\infty}}(G,G_{0}) 
\lesssim \xi_{l'}^{(\kappa)}(G)/\widetilde{W}_{\kappa}^{\|\kappa\|_{\infty}}(G,G_{0})\]
which vanishes by the hypothesis.

Second, suppose that some of the $\|\kappa\|_{\infty}$-th partial derivatives, say,
$\partial^{|\beta|}f/\partial \eta^\beta$ where $|\beta|=\|\kappa\|_{\infty}$,
can be eliminated because they can be represented by a linear combination of a subset of 
other partial derivatives $H_l^{(\kappa')}$ (in addition to possibly
a subset of other partial derivatives $H_l^{(\kappa)}$) with corresponding
finite coefficients $\alpha_{\beta,i,l}$. 
It follows that for each $l=1,\ldots, T_{\kappa'}$, the coefficient $\xi_l^{(\kappa')}(G)$
that defines the $\kappa'$-minimal form is transformed into
a coefficient in the $\kappa$-minimal form by

\[\xi_{l'}^{(\kappa)}(G) := \xi_{l}^{(\kappa')}(G) + 
\sum_{\beta; |\beta|=\|\kappa\|_{\infty}} 
\sum_{i=1}^{k_0+\extra} \alpha_{\beta,i,l}
\sum_{j=1}^{s_i}
{p_{ij}(\Delta \eta_{ij})^{\beta}/\beta!}.\]
Since $\xi_{l'}^{(\kappa)}(G)/\widetilde{W}_{\kappa}^{\|\kappa\|_{\infty}}(G,G_{0})$ tends to 0,
so does $\xi_{l'}^{(\kappa)}(G)/\widetilde{W}_{\kappa'}^{\|\kappa'\|_{\infty}}(G,G_{0})$.
By Lemma~\ref{lemma:bound_overfit_Wasserstein_first} 
for each $\beta$ such that $|\beta|=\|\kappa\|_{\infty} = r$,
$\sum_{i=1}^{k_0+\extra}\sum_{j=1}^{s_i}
{p_{ij}(\Delta \eta_{ij})^{\beta}/\beta!} = o(W_{r-1}^{r-1}(G_{0},G))
= o(D_{\kappa'}(G_0,G)) = o(\widetilde{W}_{\kappa'}^{\|\kappa'\|_{\infty}}(G_{0},G))$.
It follows
that $\xi_l^{(\kappa')}(G)/\widetilde{W}_{\kappa'}^{\|\kappa'\|_{\infty}}(G,G_{0})$ tends to 0,
for each $l = 1,\ldots, T_{\kappa'}$. This completes the proof.

\paragraph{PROOF OF PROPOSITION ~\ref{proposition:singularity_set_level}}
Part (i) and (ii) are 
immediate from the definition of $\lev(G_{0}|\Gcal)$ and $\singset(G_{0}|\Gcal)$.

To prove part (iii), note that there are two possibilities.
First, $G_{0}$ is $\kappa$-singular relative to $\Gcal$ for any $\kappa \prec (r+1,\ldots,r
+1)$. Thus, $(r+1,\ldots,r+1) \in \singset(G_{0}|\Gcal)$, we are done. Second, there exists $\kappa 
\prec (r+1,\ldots,r+1)$ such that $G_{0}$ is not $\kappa$-singular relative to $\Gcal$. 
Denote $\mathcal{A}=\left\{\kappa' \in \mathbb{N}^{d}: \kappa' \prec \kappa \ 
\text{and} \ G_{0} \ \text{is not} \ \kappa'\text{-singular} \right\}$. It is clear that $|
\mathcal{A}|$ is finite. Since $\lev(G_{0}|\Gcal)=r$, $G_{0}$ is $(r,\ldots,r)$-singular 
relative to $\Gcal$, which also implies that it is $\kappa'$-singular for any $\kappa' \preceq 
(r,\ldots,r)$ according to part (b) of Lemma \ref{lemma:minimal_form_general}. Therefore, 
at least one component of $\kappa$ is $r+1$. If $G_{0}$ is $\kappa'$-singular relative to $
\Gcal$ for all $\kappa' \prec \kappa$, then $\kappa \in \singset(G_{0}|\Gcal)$, which 
concludes the proof. If there exists $\kappa^{(1)} \prec \kappa$ such that $G_{0}$ is not $
\kappa^{(1)}$-singular relative to $\Gcal$, then at least one component of $\kappa^{(1)}$ 
is $r+1$ by the fact that $\lev(G_{0}|\Gcal)=r$. If $G_{0}$ is $\kappa'$-singular relative to 
$\Gcal$ for any $\kappa' \prec \kappa^{(1)}$, then $\kappa' \in \singset(G_{0}|\Gcal)$. If 
the previous assumption does not hold, then we also achieve $\kappa^{(2)}$ such that 
$G_{0}$ is not $\kappa^{(2)}$-singular relative to $\Gcal$. By repeating the same 
argument, we eventually will have an index $1 \leq s \leq |\mathcal{A}|$ such that $G_{0}$ is 
not $\kappa^{(s)}$-singular while it is $\kappa'$-singular for any $\kappa' \prec 
\kappa^{(s)}$, which implies that $\kappa^{(s)} \in \singset(G_{0}|\Gcal)$. As $
\kappa^{(s)} \prec \kappa \prec (r+1,\ldots,r+1)$, we achieve the conclusion of part (iii). 

For part (iv), $r\geq 1$ implies that $G_0$ is $r$-singular
relative to $\Gcal$. Moreover, $G_0$ is $\kappa$-singular relative to $\Gcal$ for
all $\kappa \preceq (r,\ldots,r)$. Since $G_0$ is not $\overline{\kappa}$-singular
by the hypothesis, we must have $(1,\ldots,1) \prec \overline{\kappa}$. Due to the boundedness
of $\overline{\kappa}$, the existence of non-empty $\mathcal{L}(G_0|\Gcal)$ then follows.

For part (v), since $\kappa \preceq \kappa':=(\|\kappa\|_\infty,\ldots,\|\kappa\|_\infty)$, by definition
$G_0$ is not $\kappa'$-singular relative to $\Gcal$. It follows that $\lev(G_0|\Gcal) < \|\kappa\|_\infty$.
If $\kappa$ is unique, then the conclusion is immediate from part (iii).

\paragraph{PROOF OF PROPOSITION ~\ref{proposition:strong_identifiability_singularity_index}}
The assumption of second order identifiability entails that $G_{0}$ is neither 2-singular  nor $(2,\ldots,2)$-singular relative to $\Ocal_{k}$
(cf. proof of Theorem 3.2 of \cite{Ho-Nguyen-Ann-16}). As a consequence of Proposition~\ref{proposition:singularity_set_level},
it suffices to demonstrate  that $G_{0}$ is $
\kappa$-singular relative to $\Ocal_{k}$ for any $\kappa$ such that all of its 
components are $r$ except for one component to be 1 as $r \geq 1$. Without loss of generality, let $\kappa=(1,r,\ldots,r)$.
For any $\eta \in \mathbb{R}^{d}$, let $
\eta^{(i)}$ denote the $i$-th component of $\eta$ for any $1 \leq i \leq d$. To simplify our proof argument, we firstly consider the
basic case of $r=2$. Now, construct
sequence of $G \rightarrow G_0$ such that $G$ always has $k_0+1$ support points. Specifically using the representation \eqref{eqn:representation_overfit}, 
$G=\sum \limits_{i=1}^{k_{0}}
{\sum \limits_{j=1}^{s_{i}}{p_{ij}\delta_{\eta_{ij}}}}$ where $s_{1}=2$ and $s_{i}
=1$ for all $2 \leq i \leq k_0$. Additionally, $\eta_{11}$ and $\eta_{12}$ are chosen 
such that $\Delta \eta_{11}= -\Delta \eta_{12}$ and $(\Delta \eta_{11}^{(i)})^{2}/
\Delta \eta_{11}^{(1)} \to 0$ for all $2 \leq i \leq d$. For other atoms of $G$, we 
choose $\eta_{ij}=\eta_{i}^{0}$ for all $2 \leq i \leq k_0$. As for the mass of $G$'s atoms, we 
choose $p_{11}=p_{12}=p_{1}^{0}/2$ and $p_{ij}=p_{i}^{0}$ for all $2 \leq i \leq k_0$. 
From this choice of $G$, we can verify that $\widetilde{W}_{\kappa}^{r}(G,G_{0}) 
\asymp p_1^0 (|\Delta \eta_{11}^{(1)}| + \sum_{i=2}^{d} (\Delta \eta_{11}^{(i)})^2)
\asymp |\Delta \eta_{11}^{(1)}|$. By carrying out Taylor expansion of the likelihood function up to the second order, 
we obtain a $\kappa$-minimal form for the sequence $G$, 
\begin{eqnarray}
\dfrac{p_{G}(x)-p_{G_{0}}(x)}{\widetilde{W}_{\kappa}^{r}(G,G_{0})} \asymp \dfrac{\sum \limits_{i=1}^{2} \sum \limits_{\alpha: |\alpha| \leq 2}(\Delta \eta_{1i})^{\alpha}\dfrac{\partial{f}^{|\alpha|}}{\partial{\eta^{\alpha}}}(x|\eta_{1}^{0})}{|\Delta \eta_{11}^{(1)}|}  \label{eqn:strong_identifiability_index_proof}
\end{eqnarray}
thanks to the second-order identifiability condition on $f$. Since the $\kappa$-minimal form coefficients all vanish due to our choice of 
$\eta_{11}$ and $\eta_{12}$,
we conclude that $G_{0}$ is $\kappa$-singular relative to $\Ocal_{k}$ when $\kappa=(1,r,\ldots,r)$ and $r=2$.

Now, for general value of $r \geq 3$ and $\kappa=(1,r,\ldots,r)$, with the choice of $
\eta_{11}$ and $\eta_{12}$ such that $\Delta \eta_{11}= -\Delta \eta_{12}$ and $
(\Delta \eta_{11}^{(i)})^{r}/\Delta \eta_{11}^{(1)} \to 0$ for any $2 \leq i \leq d$ we can verify that $\widetilde{W}_{\kappa}^{r}(G,G_{0}) 
\asymp |\Delta \eta_{11}^{(1)}|$ and 
\begin{eqnarray}
\biggr(p_{11}(\Delta \eta_{11})^{\alpha}+p_{12}(\Delta \eta_{12})^{\alpha}\biggr)/ \widetilde{W}_{\kappa}^{r}(G,G_{0}) \to 0 \label{eqn:strong_identifiability_index_proof_second}
\end{eqnarray}
for any $|\alpha| \geq 3$. By means of Taylor expansion up to the $r$ order, a $\kappa$-minimal form for the sequence $G$ is as follows
\begin{align}
\dfrac{p_{G}(x)-p_{G_{0}}(x)}{\widetilde{W}_{\kappa}^{r}(G,G_{0})} & \asymp \dfrac{\sum \limits_{i=1}^{2} \sum \limits_{\alpha: |\alpha| \leq r}(\Delta \eta_{1i})^{\alpha}\dfrac{\partial{f}^{|\alpha|}}{\partial{\eta^{\alpha}}}(x|\eta_{1}^{0})}{|\Delta \eta_{11}^{(1)}|} \nonumber \\
& \asymp \dfrac{\sum \limits_{i=1}^{2} \sum \limits_{\alpha: |\alpha| \leq 2}(\Delta \eta_{1i})^{\alpha}\dfrac{\partial{f}^{|\alpha|}}{\partial{\eta^{\alpha}}}(x|\eta_{1}^{0})}{|\Delta \eta_{11}^{(1)}|} \to 0. \nonumber
\end{align}
Here, the second asymptotic result is due to results in~
\eqref{eqn:strong_identifiability_index_proof_second} while the last 
limit is due to the choice of $G$.
Therefore, $G_{0}$ is $\kappa$-singular relative to $\Ocal_{k}$ when $\kappa=(1,r,
\ldots,r)$ for any $r \geq 1$. It follows that $(2,\ldots,2)$ is the unique singularity index 
of $G_{0}$ relative to $\Ocal_{k}$, i.e., $\singset(G_{0}|\Ocal_{k})=\left\{(2,\ldots,2)
\right\}$.

\paragraph{PROOF OF PROPOSITION ~\ref{proposition:Gaussian_mulindex}}
Here, we only provide the proof of this proposition for the case $\overline{r}(k-k_{0})$ is an even 
number as the argument for the case $\overline{r}(k-k_{0})$ is an odd number is similar. 
We follow the argument from the proof of Proposition 2.2 in \cite{Ho-Nguyen-Ann-16}. 
Denote $v=\sigma^{2}$ and $\overline{r}=\overline{r}(k-k_{0})$. For any $\|\kappa\|
_{\infty}=r$ and $r \geq 1$, let $G \in \Ocal_{k,c_{0}} \to G_{0}$ under $
\widetilde{W}_{\kappa}$ distance. According to Step 1 to Step 3 in the proof of 
Proposition 2.2 in \cite{Ho-Nguyen-Ann-16} we have the $\kappa$-minimal form for the sequence $G$ as
\begin{eqnarray}
\dfrac{p_{G}(x)-p_{G_{0}}(x)}{\widetilde{W}_{\kappa}^{r}(G,G_{0})} \asymp \dfrac{A_{1}(x)+B_{1}(x)}{\widetilde{W}_{\kappa}^{r}(G,G_{0})} \nonumber
\end{eqnarray}
where $A_{1}(x)=\sum \limits_{i=1}^{k_{0}}{(p_{i\cdot}-p_{i}^{0})f(x|\theta_{i}
^{0},v_{i}^{0})}$ and $B_{1}(x)=\mathop {\sum }\limits_{i=1}^{k_{0}}{\mathop 
{\sum }\limits_{j=1}^{s_{i}}{p_{ij}^{n}{\mathop {\sum }\limits_{\alpha \geq 1}
{\mathop {\sum }\limits_{n_{1}, n_{2}}{\dfrac{(\Delta \theta_{ij})^{n_{1}}(\Delta 
v_{ij})^{n_{2}}}{2^{n_{2}}n_{1}!n_{2}!}}\dfrac{\partial^{\alpha}{f}}{\partial{\theta}
^{\alpha}}(x|\theta_{i}^{0}}}}}$\\$,v_{i}^{0})$. The natural indices $n_{1}, n_{2}$ in the sum 
satisfy $n_{1}+2n_{2}=\alpha$ and $n_{1}+n_{2} \leq r$. To obtain the conclusion of 
the proposition, we divide the proof argument into the following key steps.

\paragraph{Step 1} We will show $G_{0}$ is not $(\overline{r},\overline{r}/2)$-singular relative to $
\Ocal_{k,c_{0}}$. In fact, by choosing $\kappa=(\overline{r},\overline{r}/2)$ and 
assume that all the coefficients of $A_{1}(x)/\widetilde{W}_{\kappa}^{r}(G,G_{0})$ 
and $B_{1}(x)/\widetilde{W}_{\kappa}^{r}(G,G_{0})$ go to 0, with the same argument 
as Step 8 of the proof of Proposition 2.2 in \cite{Ho-Nguyen-Ann-16} we eventually 
reach to the following system of limits
\begin{eqnarray}
E_{\alpha}=\dfrac{\mathop {\sum }\limits_{j=1}^{s_{1}}{p_{1j}^{n}\mathop {\sum }\limits_{\substack{n_{1}+2n_{2}=\alpha \\ n_{1}+n_{2} \leq \overline{r}}}{\dfrac{(\Delta \theta_{1j})^{n_{1}}(\Delta v_{1j})^{n_{2}}}{2^{n_{2}}n_{1}!n_{2}!}}}}{\mathop {\sum }\limits_{j=1}^{s_{1}}{p_{1j}^{n}(|\Delta \theta_{1j}|^{\overline{r}}+|\Delta v_{1j})|^{\overline{r}/2})}} \to 0 \nonumber
\end{eqnarray}
Denote $\overline{p}=\mathop {\max }\limits_{1 \leq j \leq s_{1}}{\left\{p_{1j}\right\}}$ and $\overline{M}=\mathop {\max }\biggr\{|\Delta \theta_{11}|,\ldots,|\Delta 
\theta_{1s_{1}}|,|\Delta v_{11}|^{1/2},\ldots,|\Delta v_{1s_{1}}|^{1/2}\biggr\}$. Let 
$\Delta \theta_{1j}/\overline{M} \to a_{j}$, $\Delta v_{1j}/\overline{M} \to 2b_{j}$, $
{p_{1j}/\overline{p}} \to c_{j}^{2}>0$ for all 
$j=1,\ldots,s_{1}$. By dividing both the numerator and denominator of $E_{\alpha}$ by 
$\overline{M}^{\alpha}$, we quickly achieve the system of polynomial equations 
\eqref{eqn:generalovefittedGaussianzero_Gaussian_mulindex}. From the definition of $
\overline{r}$, this system does not admit any non-trivial solution, which is a 
contradiction. As a consequence, $G_{0}$ is not $(\overline{r},\overline{r}/2)$-singular 
relative to $\Ocal_{k,c_{0}}$.
\paragraph{Step 2} We will show that $G_{0}$ is $(l,\overline{r}/2-1)$-singular and $(\overline{r}-1,l)$-singular relative to $\Ocal_{k,c_{0}}$ for any $l \geq \overline{r}$. Indeed, the sequence of $G$ provided in the proof of Proposition 2.2 in \cite{Ho-Nguyen-Ann-16} is sufficient to verify these results. In particular, let $G$ be constructed as
\[\theta_{1j}=\theta_{1}^{0}+\dfrac{a_{j}^{*}}{n},\;
v_{1j}=v_{1}^{0}+\dfrac{2b_{j}^{*}}{n^{2}},\; 
p_{1j}=\frac{p_{1}^{0}(c_{j}^{*})^{2}}{\mathop {\sum }\limits_{j=1}^{k-k_{0}+1}{(c_{j}^{*})^{2}}},\;
\text{for all}\; j=1,\ldots,k-k_{0}+1,\] 
and 
$\theta_{i1}=\theta_{i}^{0}, \; 
v_{i1}=v_{i}^{0}, \;
p_{i1}=p_{i}^{0}$
for all $i = 2,\ldots, k_{0}$ where $(c_{i}^{*},a_{i}^{*},b_{i}^{*})_{i=1}^{k-k_{0}+1}$ 
is a non-trivial solution of the system of equations \eqref{eqn:generalovefittedGaussianzero_Gaussian_mulindex} with $r=\overline{r}-1$. 
With these choices of $G$, we can easily verify that $\widetilde{W}_{\kappa}
^{\overline{r}}(G,G_{0}) \asymp (1/n)^{\overline{r}-2}$ when $\kappa=(l,\overline{r}/
2-1)$ or $\widetilde{W}_{\kappa}^{\overline{r}}(G,G_{0}) \asymp (1/
n)^{\overline{r}-1}$ when $\kappa=(\overline{r}-1,l)$ for any $l \geq \overline{r}$. As 
$B_{1}(x)=O(n^{-\overline{r}})$ (see Step 5 in the proof of Proposition 2.2 in 
\cite{Ho-Nguyen-Ann-16}) and $A_{1}(x)=0$, it implies that all the coefficients of 
$A_{1}(x)/\widetilde{W}_{\kappa}^{l}(G,G_{0})$ and $B_{1}(x)/\widetilde{W}
_{\kappa}^{l}(G,G_{0})$ go to 0 when $\kappa=(l,\overline{r}/2-1)$ or $
\kappa=(\overline{r}-1,l)$. Hence, $G_{0}$ is $(l,\overline{r}/2-1)$-singular 
and $(\overline{r}-1,l)$-singular relative to $\Ocal_{k,c_{0}}$ for any $l \geq \overline{r}$.

In summary, the results of Step 1 and Step 2 demonstrate that $(\overline{r},
\overline{r}/2)$ is the unique singularity index of $G_{0}$ relative to $\Ocal_{k,c_{0}}$, which concludes the proof.

\subsection{Proofs for Section \ref{Section:overfitskew}}
\paragraph{PROOF OF LEMMA~\ref{proposition-notskewnormal}}
For any $k_{0} \geq 1$ and $k_{0}$ different pairs $\eta_{1}=(\theta_{1},\sigma_{1},m_{1}),
\ldots,\eta_{k_{0}}=(\theta_{k_{0}},\sigma_{k_{0}},m_{k_{0}})$, 
let $\alpha_{ij} \in \mathbb{R}$ for $i=1,\ldots,4,\; j=1,\ldots, k_{0}$
such that for almost all $x \in \mathbb{R}$
\begin{eqnarray}
\mathop {\sum }\limits_{j=1}^{k_{0}}{\alpha_{1j}f(x|\eta_{j})+
\alpha_{2j}\dfrac{\partial{f}}{\partial{\theta}}(x|\eta_{j})+\alpha_{3j}
\dfrac{\partial{f}}{\partial{\sigma^{2}}}(x|\eta_{j})
\alpha_{4j}\dfrac{\partial{f}}{\partial{m}}(x|\eta_{j})}  =  0. \nonumber
\end{eqnarray}
We can rewrite the above equation as
\begin{eqnarray}
\mathop {\sum }\limits_{j=1}^{k_{0}} \biggr \{[\beta_{1j}+\beta_{2j}(x-\theta_{j})+
\beta_{3j}(x-\theta_{j})^{2}]\Phi\left(\dfrac{m_{j}(x-\theta_{j})}{\sigma_{j}}\right)\exp
\left(-\dfrac{(x-\theta_{j})^{2}}{2\sigma_{j}^{2}}\right)+ \nonumber \\
(\gamma_{1j}+\gamma_{2j}(x-\theta_{j})) f\left(\dfrac{m_{j}(x-\theta_{j})}{\sigma_{j}}\right)\exp\left(-\dfrac{(x-\theta_{j})^{2}}{2\sigma_{j}^{2}}\right) \biggr \}=0, \label{eqn:notidentifiableskewnormalone}
\end{eqnarray}
where $\beta_{1j}=\dfrac{2\alpha_{1j}}{\sqrt{2\pi}\sigma_{j}}-\dfrac{\alpha_{3j}}
{\sqrt{2\pi}\sigma_{j}^{3}}$,\;\; $\beta_{2j}=\dfrac{2\alpha_{2j}}{\sqrt{2\pi}\sigma_{j}
^{3}}$,\;\; 
$\beta_{3j}=\dfrac{\alpha_{3j}}{\sqrt{2\pi}\sigma_{j}^{5}}$,\;\; $\gamma_{1j}=-
\dfrac{2\alpha_{2j}m_{j}}{\sqrt{2\pi}\sigma_{j}^{2}}$, 
and $\gamma_{2j}=-\dfrac{\alpha_{3j}m_{j}}{\sqrt{2\pi}\sigma_{j}^{4}}+
\dfrac{2\alpha_{4j}}{\sqrt{2\pi}\sigma_{j}^{2}}$ 
for all $j=1,\ldots,k_{0}$. 

\paragraph{"Only if" direction:} 
Assume by contrary that the conclusion does not hold,
i.e., both type A and type B conditions do not hold. Denote $\sigma_{j+k_{0}}=\dfrac{\sigma_{j}
^{2}}{1+m_{j}^{2}}$ for all $1 \leq j \leq k_{0}$. For the simplicity of the argument, we 
assume that $\sigma_{i}$ are pairwise different and $\dfrac{\sigma_{i}^{2}}{1+m_{i}^{2}} 
\not \in \left\{\sigma_{j}^{2}: 1 \leq j \leq k_{0} \right\}$ for all $1 \leq i \leq k_{0}$. The 
argument for the other cases is similar. Now, $\sigma_{j}$ are pairwise different as $1 \leq j \leq 2k_{0}$. The equation 
\eqref{eqn:notidentifiableskewnormalone} can be rewritten as
\begin{eqnarray}
\mathop {\sum }\limits_{j=1}^{2k_{0}}\biggr \{[\beta_{1j}+\beta_{2j}(x-\theta_{j})+
\beta_{3j}(x-\theta_{j})^{2}]\Phi\left(\dfrac{m_{j}(x-\theta_{j})}{\sigma_{j}}\right) \exp\left(-\dfrac{(x-\theta_{j})^{2}}{2\sigma_{j}^{2}}\right)\biggr \} & = & 0, \label{eqn:exactfittedidentifiableskewnormalsecond}
\end{eqnarray}
where $m_{j}=0$, $\theta_{j+k_{0}}=\theta_{j}$, $\beta_{1(j+k_{0})}=\dfrac{2\gamma_{1j}}{\sqrt{2\pi}},
\beta_{2(j+k_{0})}=\dfrac{2\gamma_{2j}}{\sqrt{2\pi}}, \beta_{3j}=0$ as $k_{0}+1 \leq j \leq 2k_{0}$. 
Denote $\overline{i}=\mathop {\arg \max }\limits_{1 \leq i \leq 2k_{0}}{\left\{\sigma_{i}\right\}}$. 
Multiply both sides of \eqref{eqn:exactfittedidentifiableskewnormalsecond} with $\exp\left(\dfrac{(x-\theta_{\overline{i}})^{2}}{2\sigma_{\overline{i}}^{2}}\right)/
\Phi\left(\dfrac{m_{\overline{i}}(x-\theta_{\overline{i}})}{\sigma_{\overline{i}}}\right)$ 
and let $x \to +\infty$ if $m_{\overline{i}} \geq 0$ or let $x \to -\infty$ if $m_{\overline{i}}<0$ 
on both sides of the new equation, we obtain $\beta_{1\overline{i}}+\beta_{2\overline{i}}(x-\theta_{\overline{i}})+\beta_{2\overline{i}}(x-\theta_{\overline{i}})^{2} \to 0$. 
It implies that $\beta_{1\overline{i}}=\beta_{2\overline{i}}=\beta_{3\overline{i}}=0$. 
Repeatedly apply the same 
argument to the remaining $\sigma_{i}$ until we obtain 
$\beta_{1i}=\beta_{2i}=\beta_{3i}=0$ for all $1 \leq i \leq 2k_{0}$. 
It is equivalent to $\alpha_{1i}=\alpha_{2i}=\alpha_{3i}=\alpha_{4i}=0$ for all 
$1 \leq i \leq k_{0}$, which is a contradiction. 

\paragraph{"If" direction:} There are two possible scenarios.

\paragraph{Type A singularity} There exists some $m_{j}=0$ as $1 \leq j \leq k_{0}$. In this 
case, we assume that $m_{1}=0$. If we choose $\alpha_{1j}=\alpha_{2j}=\alpha_{3j}=
\alpha_{4j}=0$ for all $2 \leq j \leq k_{0}$, then equation 
\eqref{eqn:notidentifiableskewnormalone} can be rewritten as
\begin{eqnarray}
\dfrac{\beta_{11}}{2}+\dfrac{\gamma_{11}}{\sqrt{2\pi}}+\left(\dfrac{\beta_{21}}{2}+\dfrac{\gamma_{21}}{\sqrt{2\pi}}\right)(x-\theta_{1})+\dfrac{\beta_{31}}{2}(x-\theta_{1})^{2}=0. \nonumber
\end{eqnarray}
By choosing $\alpha_{31}=0$, $\alpha_{11}=\dfrac{\alpha_{21}m_{1}}{\sqrt{2\pi}
\sigma_{1}}$, $\alpha_{21}=-\dfrac{\alpha_{41}\sigma_{1}}{\sqrt{2\pi}}$, 
the above equation always equal to 0. Since $\alpha_{11},\alpha_{21},\alpha_{41}$ are not 
necessarily zero, the first-order identifiability (i.e., linear independence condition)
is violated.

\paragraph{Type B singularity} There exists indices $1 \leq i \neq j \leq k_{0}$ such 
that $\left(\dfrac{\sigma_{i}^{2}}{1+m_{i}^{2}},\theta_{i}\right)=\left(\dfrac{\sigma_{j}
^{2}}{1+m_{j}^{2}},\theta_{j} \right)$. Without loss of generality, we assume that $i=1, 
j=2$. If we choose $\alpha_{1j}=\alpha_{2j}=\alpha_{3j}=
\alpha_{4j}=0$ for all $3 \leq j \leq k_{0}$, then equation in 
\eqref{eqn:notidentifiableskewnormalone} can be rewritten as
\begin{eqnarray}
\mathop {\sum }\limits_{j=1}^{2} \biggr \{[\beta_{1j}+\beta_{2j}(x-\theta_{j})+\beta_{3j}(x-\theta_{j})^{2}]\Phi\left(\dfrac{m_{j}(x-\theta_{j})}{\sigma_{j}}\right)\exp\left(-\dfrac{(x-\theta_{j})^{2}}{2\sigma_{j}^{2}}\right) \biggr \}+ \nonumber \\
\dfrac{1}{\sqrt{2\pi}}\left(\mathop {\sum }\limits_{j=1}^{2}{\gamma_{1j}}+\mathop {\sum }\limits_{j=1}^{2}{\gamma_{2j}}(x-\theta_{1})\right)\exp\left(-\dfrac{(m_{1}^{2}+1)(x-\theta_{1})^{2}}{2\sigma_{1}^{2}}\right)=0. \nonumber
\end{eqnarray}
Now, we choose $\alpha_{1j}=\alpha_{2j}=\alpha_{3j}=0$ for all $1 \leq j \leq 2$, $
\dfrac{\alpha_{41}}{\sigma_{1}^{2}}+\dfrac{\alpha_{42}}{\sigma_{2}^{2}}=0$ then the 
above equation always hold. 
Since $\alpha_{41}$ and $\alpha_{42}$ need not be zero, the first-order identifiability
condition is violated. This concludes the proof.

\paragraph{PROOF OF LEMMA \ref{lemma:recursiveequation}}

The proof proceeds via induction on $\alpha_{1}+2\alpha_{2}+2\alpha_{3}$. To ease the presentation, we denote throughout this proof the following notation
\begin{align}
\mathcal{U}_{\alpha} = \left\{\kappa \in \mathcal{F}_{|\alpha|}: \ \kappa_{1}+2\kappa_{2}+2\kappa_{3} \leq \alpha_{1}+2\alpha_{2}+2\alpha_{3}\right\}. \nonumber
\end{align}
As $\alpha_{1}+2\alpha_{2}+2\alpha_{3} \leq 2$, we can easily check the conclusion of the lemma. Assume that the 
conclusion holds for any $\alpha_{1}+2\alpha_{2}+2\alpha_{3} \leq k-1$. We shall demonstrate that 
it also holds for $\alpha_{1}+2\alpha_{2}+2\alpha_{3} =k$. Indeed, there are two settings:
\paragraph{Case 1: $\alpha_{1}=k$} Under this setting, 
$\alpha_{2}=\alpha_{3}=0$. From the induction hypothesis,
\begin{eqnarray}
\dfrac{\partial^{|\alpha|}{f}}{\partial{\theta^{\alpha_{1}}\partial{v^{\alpha_{2}}}
\partial{m^{\alpha_{3}}}}} & = & \dfrac{\partial}{\partial{\theta}}\biggr(\dfrac{\partial^{|\alpha|-1}{f}}{\partial{\theta^{\alpha_{1}-1}}\partial{v^{\alpha_{2}}}\partial{m^{\alpha_{3}}}}\biggr) \nonumber \\
& = & \dfrac{\partial}{\partial{\theta}}\biggr(\sum \limits_{\kappa \in \mathcal{U}_{(\alpha_{1}-1,\alpha_{2},\alpha_{3})}}{\dfrac{P_{\alpha_{1}-1,\alpha_{2},\alpha_{3}}
^{\kappa_{1},\kappa_{2},\kappa_{3}}(m)}{H_{\alpha_{1}-1,\alpha_{2},\alpha_{3}}
^{\kappa_{1},\kappa_{2},\kappa_{3}}(m)Q_{\alpha_{1}-1,\alpha_{2},\alpha_{3}}
^{\kappa_{1},\kappa_{2},\kappa_{3}}(v)}\dfrac{\partial^{|\kappa|}{f}}
{\partial{\theta^{\kappa_{1}}\partial{v}^{\kappa_{2}}\partial{m}^{\kappa_{3}}}}}\biggr)  \nonumber \\
& = & \sum \limits_{\kappa \in \mathcal{U}_{(\alpha_{1}-1,\alpha_{2},\alpha_{3})}}{\dfrac{P_{\alpha_{1}-1,\alpha_{2},\alpha_{3}}
^{\kappa_{1},\kappa_{2},\kappa_{3}}(m)}{H_{\alpha_{1}-1,\alpha_{2},\alpha_{3}}
^{\kappa_{1},\kappa_{2},\kappa_{3}}(m)Q_{\alpha_{1}-1,\alpha_{2},\alpha_{3}}
^{\kappa_{1},\kappa_{2},\kappa_{3}}(v)}\dfrac{\partial^{|\kappa|+1}{f}}
{\partial{\theta^{\kappa_{1}+1}\partial{v}^{\kappa_{2}}\partial{m}^{\kappa_{3}}}}}, \nonumber \\
& = & \sum \limits_{\kappa \in \mathcal{U}_{(\alpha_{1}-1,\alpha_{2},\alpha_{3})}: \kappa_{1}=0}{\dfrac{P_{\alpha_{1}-1,\alpha_{2},\alpha_{3}}
^{\kappa_{1},\kappa_{2},\kappa_{3}}(m)}{H_{\alpha_{1}-1,\alpha_{2},\alpha_{3}}
^{\kappa_{1},\kappa_{2},\kappa_{3}}(m)Q_{\alpha_{1}-1,\alpha_{2},\alpha_{3}}
^{\kappa_{1},\kappa_{2},\kappa_{3}}(v)}\dfrac{\partial^{|\kappa|+1}{f}}
{\partial{\theta^{\kappa_{1}+1}\partial{v}^{\kappa_{2}}\partial{m}^{\kappa_{3}}}}} \nonumber \\
& + & \sum \limits_{\kappa \in \mathcal{U}_{(\alpha_{1}-1,\alpha_{2},\alpha_{3})}: \kappa_{1}=1}{\dfrac{P_{\alpha_{1}-1,\alpha_{2},\alpha_{3}}
^{\kappa_{1},\kappa_{2},\kappa_{3}}(m)}{H_{\alpha_{1}-1,\alpha_{2},\alpha_{3}}
^{\kappa_{1},\kappa_{2},\kappa_{3}}(m)Q_{\alpha_{1}-1,\alpha_{2},\alpha_{3}}
^{\kappa_{1},\kappa_{2},\kappa_{3}}(v)}\dfrac{\partial^{|\kappa|+1}{f}}
{\partial{\theta^{\kappa_{1}+1}\partial{v}^{\kappa_{2}}\partial{m}^{\kappa_{3}}}}} \quad \label{lemma:representation_first}
\end{eqnarray}
where the second equality is due to the application of the 
hypothesis for $\alpha_{1}-1+2\alpha_{2}+2\alpha_{3} = k-1$. 
For any $\kappa \in \mathcal{U}_{(\alpha_{1}-1,\alpha_{2},\alpha_{3})}$ such that $\kappa_{1}=1$, 
\begin{eqnarray}
\dfrac{\partial^{|\kappa|+1}{f}}
{\partial{\theta^{\kappa_{1}+1}\partial{v}^{\kappa_{2}}\partial{m}^{\kappa_{3}}}} & = & \dfrac{\partial^{|\kappa|-1}{f}}{\partial{v^{\kappa_{2}}}\partial{m^{\kappa_{3}}}}\biggr(2\dfrac{\partial{f}}{\partial{v}}-\dfrac{m^{3}+m}{v}\dfrac{\partial{f}}{\partial{m}}\biggr) \nonumber \\
& = & 2\dfrac{\partial^{|\kappa|}{f}}{\partial{v^{\kappa_{2}+1}}\partial{m^{\kappa_{3}}}}-\dfrac{\partial^{|\kappa|-1}{f}}{\partial{v^{\kappa_{2}}}\partial{m^{\kappa_{3}}}}\biggr(\dfrac{m^{3}+m}{v}\dfrac{\partial{f}}{\partial{m}}\biggr). \label{lemma:representation_second}
\end{eqnarray}
From the inductive hypothesis, since $\kappa_{1}+2\kappa_{2}+2\kappa_{3}=2\kappa_{2}+2\kappa_{3}+1 \leq k-1$, 
\begin{eqnarray}
\dfrac{\partial^{|\kappa|}{f}}{\partial{v^{\kappa_{2}+1}}\partial{m^{\kappa_{3}}}} = \sum \limits_{\kappa' \in \mathcal{U}_{(0,\kappa_{2}+1,\kappa_{3})}}{\dfrac{P_{0,\kappa_{2}+1,\kappa_{3}}
^{\kappa_{1}',\kappa_{2}',\kappa_{3}'}(m)}{H_{0,\kappa_{2}+1,\kappa_{3}}
^{\kappa_{1}',\kappa_{2}',\kappa_{3}'}(m)Q_{0,\kappa_{2}+1,\kappa_{3}}
^{\kappa_{1}',\kappa_{2}',\kappa_{3}'}(v)}\dfrac{\partial^{|\kappa'|}{f}}
{\partial{\theta^{\kappa_{1}'}\partial{v}^{\kappa_{2}'}\partial{m}^{\kappa_{3}'}}}}.\label{lemma:representation_third}
\end{eqnarray}
In addition,
\begin{eqnarray}
\dfrac{\partial^{|\kappa|-1}{f}}{\partial{v^{\kappa_{2}}}\partial{m^{\kappa_{3}}}}\biggr(\dfrac{m^{3}+m}{v}\dfrac{\partial{f}}{\partial{m}}\biggr) = \sum \limits_{\beta: |\beta| \leq |\kappa|,\beta_{1} \leq \kappa_{2}, \beta_{2} \leq \kappa_{3}+1}{\dfrac{A_{\beta_{1},\beta_{2}}(m)}{B_{\beta_{1},\beta_{2}}(v)}\dfrac{\partial^{|\beta|}{f}}{\partial{v^{\beta_{1}}}\partial{m^{\beta_{2}}}}}. \label{lemma:representation_fourth}
\end{eqnarray}
Since $2\beta_{1}+2\beta_{2} \leq 2\kappa_{2}+2\kappa_{3} \leq k-1$, from the hypothesis,
\begin{eqnarray}
\dfrac{\partial^{|\beta|}{f}}{\partial{v^{\beta_{1}}}\partial{m^{\beta_{2}}}}=\sum \limits_{\kappa'' \in \mathcal{U}_{(0,\beta_{1},\beta_{2})}}{\dfrac{P_{0,\beta_{1},\beta_{2}}
^{\kappa_{1}'',\kappa_{2}'',\kappa_{3}''}(m)}{H_{0,\beta_{1},\beta_{2}}
^{\kappa_{1}'',\kappa_{2}'',\kappa_{3}''}(m)Q_{0,\beta_{1},\beta_{2}}
^{\kappa_{1}'',\kappa_{2}'',\kappa_{3}''}(v)}\dfrac{\partial^{|\kappa''|}{f}}
{\partial{\theta^{\kappa_{1}''}\partial{v}^{\kappa_{2}''}\partial{m}^{\kappa_{3}''}}}}. \label{lemma:representation_fifth}
\end{eqnarray}
Combining equations \eqref{lemma:representation_first}, \eqref{lemma:representation_second}, \eqref{lemma:representation_third}, \eqref{lemma:representation_fourth}, and \eqref{lemma:representation_fifth}, we arrive at the conclusion of the lemma.
\paragraph{Case 2: $\alpha_{1} \leq k-1$} Under this setting, assume without loss of 
generality that $\alpha_{2} \geq 1$. 
\begin{eqnarray}
\dfrac{\partial^{|\alpha|}{f}}{\partial{\theta^{\alpha_{1}}\partial{v^{\alpha_{2}}}
\partial{m^{\alpha_{3}}}}} & = &  \dfrac{\partial}{\partial{v}}\biggr(\dfrac{\partial^{|\alpha|-1}{f}}{\partial{\theta^{\alpha_{1}}\partial{v^{\alpha_{2}-1}}
\partial{m^{\alpha_{3}}}}}\biggr) \nonumber \\
& = & \dfrac{\partial}{\partial{v}}\biggr( \sum \limits_{\kappa \in \mathcal{U}_{(\alpha_{1},\alpha_{2}-1,\alpha_{3})}}{\dfrac{P_{\alpha_{1},\alpha_{2}-1,\alpha_{3}}
^{\kappa_{1},\kappa_{2},\kappa_{3}}(m)}{H_{\alpha_{1},\alpha_{2}-1,\alpha_{3}}
^{\kappa_{1},\kappa_{2},\kappa_{3}}(m)Q_{\alpha_{1},\alpha_{2}-1,\alpha_{3}}
^{\kappa_{1},\kappa_{2},\kappa_{3}}(v)}\dfrac{\partial^{|\kappa|}{f}}
{\partial{\theta^{\kappa_{1}}\partial{v}^{\kappa_{2}}\partial{m}^{\kappa_{3}}}}}\biggr)\nonumber \\
& = & \sum \limits_{\kappa \in \mathcal{U}_{(\alpha_{1},\alpha_{2}-1,\alpha_{3})}}{\dfrac{\partial{}}{\partial{v}}\biggr(\dfrac{P_{\alpha_{1},\alpha_{2}-1,\alpha_{3}}
^{\kappa_{1},\kappa_{2},\kappa_{3}}(m)}{H_{\alpha_{1},\alpha_{2}-1,\alpha_{3}}
^{\kappa_{1},\kappa_{2},\kappa_{3}}(m)Q_{\alpha_{1},\alpha_{2}-1,\alpha_{3}}
^{\kappa_{1},\kappa_{2},\kappa_{3}}(v)}\biggr)\dfrac{\partial^{|\kappa|}{f}}
{\partial{\theta^{\kappa_{1}}\partial{v}^{\kappa_{2}}\partial{m}^{\kappa_{3}}}}} \nonumber \\
& + & \sum \limits_{\kappa \in \mathcal{U}_{(\alpha_{1},\alpha_{2}-1,\alpha_{3})}}{\dfrac{P_{\alpha_{1},\alpha_{2}-1,\alpha_{3}}
^{\kappa_{1},\kappa_{2},\kappa_{3}}(m)}{H_{\alpha_{1},\alpha_{2}-1,\alpha_{3}}
^{\kappa_{1},\kappa_{2},\kappa_{3}}(m)Q_{\alpha_{1},\alpha_{2}-1,\alpha_{3}}
^{\kappa_{1},\kappa_{2},\kappa_{3}}(v)}\dfrac{\partial^{|\kappa|+1}{f}}
{\partial{\theta^{\kappa_{1}}\partial{v}^{\kappa_{2}+1}\partial{m}^{\kappa_{3}}}}}. \label{lemma:representation_sixth}
\end{eqnarray}
Denote $A:= \sum \limits_{\kappa \in \mathcal{U}_{(\alpha_{1},\alpha_{2}-1,\alpha_{3})}}{\dfrac{P_{\alpha_{1},\alpha_{2}-1,\alpha_{3}}
^{\kappa_{1},\kappa_{2},\kappa_{3}}(m)}{H_{\alpha_{1},\alpha_{2}-1,\alpha_{3}}
^{\kappa_{1},\kappa_{2},\kappa_{3}}(m)Q_{\alpha_{1},\alpha_{2}-1,\alpha_{3}}
^{\kappa_{1},\kappa_{2},\kappa_{3}}(v)}\dfrac{\partial^{|\kappa|+1}{f}}
{\partial{\theta^{\kappa_{1}}\partial{v}^{\kappa_{2}+1}\partial{m}^{\kappa_{3}}}}}$, we further have that
\begin{eqnarray}
A & = & \sum \limits_{\kappa \in \mathcal{U}_{(\alpha_{1},\alpha_{2}-1,\alpha_{3})}: \kappa_{3}=0}{\dfrac{P_{\alpha_{1},\alpha_{2}-1,\alpha_{3}}
^{\kappa_{1},\kappa_{2},\kappa_{3}}(m)}{H_{\alpha_{1},\alpha_{2}-1,\alpha_{3}}
^{\kappa_{1},\kappa_{2},\kappa_{3}}(m)Q_{\alpha_{1},\alpha_{2}-1,\alpha_{3}}
^{\kappa_{1},\kappa_{2},\kappa_{3}}(v)}\dfrac{\partial^{|\kappa|+1}{f}}
{\partial{\theta^{\kappa_{1}}\partial{v}^{\kappa_{2}+1}\partial{m}^{\kappa_{3}}}}} \nonumber \\
& + & \sum \limits_{\kappa \in \mathcal{U}_{(\alpha_{1},\alpha_{2}-1,\alpha_{3})}: \kappa_{2}=0, \kappa_{3} \geq 1}{\dfrac{P_{\alpha_{1},\alpha_{2}-1,\alpha_{3}}
^{\kappa_{1},\kappa_{2},\kappa_{3}}(m)}{H_{\alpha_{1},\alpha_{2}-1,\alpha_{3}}
^{\kappa_{1},\kappa_{2},\kappa_{3}}(m)Q_{\alpha_{1},\alpha_{2}-1,\alpha_{3}}
^{\kappa_{1},\kappa_{2},\kappa_{3}}(v)}\dfrac{\partial^{|\kappa|+1}{f}}
{\partial{\theta^{\kappa_{1}}\partial{v}^{\kappa_{2}+1}\partial{m}^{\kappa_{3}}}}}. \label{lemma:representation_seventh}
\end{eqnarray}
Since $m \neq 0$, for any $\kappa \in \mathcal{U}_{(\alpha_{1},\alpha_{2}-1,\alpha_{3})}$ such that $\kappa_{2}=0$ and $\kappa_{3} \geq 1$, we have
\begin{eqnarray}
\dfrac{\partial^{|\kappa|+1}{f}}
{\partial{\theta^{\kappa_{1}}\partial{v}^{\kappa_{2}+1}\partial{m}^{\kappa_{3}}}} & = & \dfrac{\partial^{|\kappa|-1}{f}}{\partial{\theta^{\kappa_{1}}}\partial{m^{\kappa_{3}-1}}}\biggr(-\dfrac{1}{v}\dfrac{\partial{f}}{\partial{m}}-\dfrac{m^{2}+1}{2mv}\dfrac{\partial^{2}{f}}{\partial{m^{2}}}\biggr) \nonumber \\
& = & -\dfrac{1}{v}\dfrac{\partial^{|\kappa|}{f}}{\partial{\theta^{\kappa_{1}}}\partial{m^{\kappa_{3}}}}-\dfrac{\partial^{|\kappa|-1}{f}}{\partial{\theta^{\kappa_{1}}}\partial{m^{\kappa_{3}-1}}}\biggr(\dfrac{m^{2}+1}{2mv}\dfrac{\partial^{2}{f}}{\partial{m^{2}}}\biggr). \label{lemma:representation_eighth}
\end{eqnarray}
Since $|\kappa|=\kappa_{1}+\kappa_{3} \leq k-1$ and $\kappa_{1} \leq 1$, we have $(\kappa_{1},0,\kappa_{3}) \in \mathcal{F}_{k}$. Additionally, we can represent
\begin{eqnarray}
\dfrac{\partial^{|\kappa|-1}{f}}{\partial{\theta^{\kappa_{1}}}\partial{m^{\kappa_{3}-1}}}\biggr(\dfrac{m^{2}+1}{2mv}\dfrac{\partial^{2}{f}}{\partial{m^{2}}}\biggr) = \sum \limits_{1 \leq \tau \leq \kappa_{3}+1}{\dfrac{A'_{\tau}(m)}{B'_{\tau}(m)C'_{\tau}(v)}\dfrac{\partial^{\kappa_{1}+\tau}{f}}{\partial{\theta^{\kappa_{1}}\partial{m^{\tau}}}}}, \nonumber
\end{eqnarray}
where $A'_{\tau}(m)$, $B'_{\tau}(m)$, $C'_{\tau}(v)$ are some polynomials of $m$ and 
$v$. Since $\kappa_{1}+\tau \leq \kappa_{1}+\kappa_{3}+1 \leq k$ and $\kappa_{1} \leq 
1$, we have $(\kappa_{1},0,\tau) \in \mathcal{F}_{k}$. Combining these results with 
equations \eqref{lemma:representation_sixth}, \eqref{lemma:representation_seventh}, and 
\eqref{lemma:representation_eighth}, we achieve the conclusion of the lemma.

\paragraph{PROOF OF LEMMA \ref{lemma:reduced_linearly_independent}}
The proof of this lemma proceeds by induction on $r$. If $r=1$, 
\begin{eqnarray}
\left\{\dfrac{\partial^{|\alpha|}{f}}
{\theta^{\alpha_{1}}v^{\alpha_{2}}m^{\alpha_{3}}}: \ (\alpha_{1},\alpha_{2},\alpha_{3}) \in \mathcal{F}_{r}\right\}=\left\{\dfrac{\partial{f}}{\partial{\theta}},\dfrac{\partial{f}}{\partial{v}},\dfrac{\partial{f}}{\partial{m}}\right\}, \nonumber
\end{eqnarray} 
which are linearly independent with respect to $G_{0} \in \Scal_{0}$ due to the conclusion 
of Lemma \ref{proposition-notskewnormal}. 
Assume that the conclusion of the lemma holds up to $r$. We will 
demonstrate that it continues to hold for $r+1$. In fact, 
\begin{eqnarray}
\left\{\dfrac{\partial^{|\alpha|}{f}}
{\theta^{\alpha_{1}}v^{\alpha_{2}}m^{\alpha_{3}}}: \ (\alpha_{1},\alpha_{2},\alpha_{3}) \in \mathcal{F}_{r+1}\right\}=\left\{\dfrac{\partial^{|\alpha|}{f}}
{\theta^{\alpha_{1}}v^{\alpha_{2}}m^{\alpha_{3}}}: \ (\alpha_{1},\alpha_{2},\alpha_{3}) \in \mathcal{F}_{r}\right\} \cup \nonumber \\
\left\{\dfrac{\partial^{r+1}{f}}{\partial{\theta}\partial{v^{r}}},\dfrac{\partial^{r+1}{f}}{\partial{v^{r+1}}},\dfrac{\partial^{r+1}{f}}{\partial{\theta}\partial{m^{r}}},\dfrac{\partial^{r+1}{f}}{\partial{m^{r+1}}}\right\}. \label{eqn:lemma_reduced_first}
\end{eqnarray}
Assume that there are 
coefficients $\beta_{\alpha_{1},\alpha_{2},\alpha_{3}}^{(i)}$ where $1 \leq i \leq k_{0}$ and $(\alpha_{1},\alpha_{2},\alpha_{3}) \in \mathcal{F}_{r+1}$ such that for all $x$
\begin{eqnarray}
\sum \limits_{i=1}^{k_{0}}{\sum \limits_{(\alpha_{1},\alpha_{2},\alpha_{3}) \in \mathcal{F}_{r+1}}{\beta_{\alpha_{1},\alpha_{2},\alpha_{3}}^{(i)}\dfrac{\partial^{|\alpha|}{f}}
{\theta^{\alpha_{1}}v^{\alpha_{2}}m^{\alpha_{3}}}(x|\eta_{i}^{0})}}=0. \nonumber
\end{eqnarray}
Using the fact from \eqref{eqn:lemma_reduced_first}, we rewrite the above equation as
\begin{eqnarray}
\sum \limits_{i=1}^{k_{0}}{\sum \limits_{(\alpha_{1},\alpha_{2},\alpha_{3}) \in \mathcal{F}_{r}}{\beta_{\alpha_{1},\alpha_{2},\alpha_{3}}^{(i)}\dfrac{\partial^{|\alpha|}{f}}
{\theta^{\alpha_{1}}v^{\alpha_{2}}m^{\alpha_{3}}}(x|\eta_{i}^{0})}}+\beta_{1,r,0}^{(i)}\dfrac{\partial^{r+1}{f}}{\partial{\theta}\partial{v^{r}}}(x|\eta_{i}^{0})+ \nonumber \\
\beta_{0,r+1,0}^{(i)}\dfrac{\partial^{r+1}{f}}{\partial{v^{r+1}}}((x|\eta_{i}^{0})+\beta_{1,0,r}^{(i)}\dfrac{\partial^{r+1}{f}}{\partial{\theta}\partial{m^{r}}}(x|\eta_{i}^{0})+\beta_{0,0,r+1}^{(i)}\dfrac{\partial^{r+1}{f}}{\partial{m^{r+1}}}(x|\eta_{i}^{0})=0. \label{eqn:lemma_reduced_second}
\end{eqnarray}
Equation \eqref{eqn:lemma_reduced_second} can be rewritten as
\begin{eqnarray}
\sum \limits_{i=1}^{k_{0}}{{\biggr(\sum 
\limits_{j=1}^{2r+3}{\gamma_{j,i}^{(r+1)}(x-\theta_{i}^{0})^{j-1}}\biggr)f\left(\dfrac{x-\theta_{i}
^{0}}{\sigma_{i}^{0}}\right)\Phi\left(\dfrac{m_{i}^{0}(x-\theta_{i}^{0})}{\sigma_{i}
^{0}}\right)}} \nonumber \\
+\sum \limits_{i=1}^{k_{0}}{\biggr(\sum \limits_{j=1}^{2r+2}{\tau_{j,i}^{(r+1)}(x-\theta_{i}^{0})^{j-1}}\biggr)\exp
\left(-\dfrac{(m_{i}^{0})^{2}+1}{2v_{i}^{0}}(x-\theta_{i}^{0})^{2}\right)}=0, \nonumber
\end{eqnarray}
where $\gamma_{j,i}^{(r+1)}$ are a combination of $\beta_{\alpha_{1},\alpha_{2},
\alpha_{3}}^{(i)}$ when $(\alpha_{1},\alpha_{2},\alpha_{3}) \in \mathcal{F}_{r+1}$ and $
\alpha_{3}=0$. Additionally, $\tau_{j,i}^{(r+1)}$ are a combination of $\beta_{\alpha_{1},
\alpha_{2},\alpha_{3}}^{(i)}$ when $(\alpha_{1},\alpha_{2},\alpha_{3}) \in \mathcal{F}_{r+1}$. 
Due to the fact that there are no type A or type B singularities 
in $\biggr\{\eta_{1}^{0},\ldots,\eta_{k_{0}}^{0}\biggr\}$, by using the same argument as 
that of the proof of Lemma 
\ref{proposition-notskewnormal}, we obtain that $\gamma_{j,i}^{(r+1)}=0$ for all $1 \leq i \leq 
k_{0}$, $1 \leq j \leq 2r+3$ and $\tau_{j,i}^{(r+1)}=0$ for all $1 \leq i \leq k_{0}$, $1 
\leq j \leq 2r+2$. It can be checked that $\gamma_{2r+3,i}^{(r+1)}=0$ implies $\beta_{0,r
+1,0}^{(i)}=0$ while $\gamma_{2r+2,i}^{(r+1)}=0$ implies $\beta_{1,r,0}^{(i)}=0$ for all 
$1 \leq i \leq k_{0}$. Similarly, $\tau_{2r+2,i}^{(r+1)}=0$ implies $\beta_{0,0,r+1}^{(i)}
=0$ while $\tau_{2r+1,i}^{(r+1)}=0$ implies $\beta_{1,0,r}^{(i)}=0$ for all $1 \leq i \leq 
k_{0}$. As a consequence, Eq. \eqref{eqn:lemma_reduced_second} is reduced to
\begin{eqnarray}
\sum \limits_{i=1}^{k_{0}}{\sum \limits_{(\alpha_{1},\alpha_{2},\alpha_{3}) \in \mathcal{F}_{r}}{\beta_{\alpha_{1},\alpha_{2},\alpha_{3}}^{(i)}\dfrac{\partial^{|\alpha|}{f}}
{\theta^{\alpha_{1}}v^{\alpha_{2}}m^{\alpha_{3}}}(x|\eta_{i}^{0})}}=0.
\end{eqnarray} 
According to the hypothesis with $r$, we obtain 
that $\beta_{\alpha_{1},\alpha_{2},\alpha_{3}}^{(i)}=0$ for all $1 \leq i \leq k_{0}, (\alpha_{1},\alpha_{2},\alpha_{3}) \in \mathcal{F}_{r}$. 
This concludes our proof.
\paragraph{PROOF OF PROPOSITION \ref{lemma:upperbound_system_nonlinear}}
From the formation of system of 
polynomial equations \eqref{eqn:systemnonlinear}, if we choose $\beta_{3}=0$ (i.e., 
we only reduce to derivatives with respect to the location and scale parameter), then
we have $P_{\alpha_{1},\alpha_{2},\alpha_{3}}^{\beta_{1},\beta_{2},
\beta_{3}}(m)/H_{\alpha_{1},\alpha_{2},\alpha_{3}}^{\beta_{1},\beta_{2},
\beta_{3}}(m)Q_{\alpha_{1},\alpha_{2},\alpha_{3}}^{\beta_{1},\beta_{2},
\beta_{3}}(v) = 2^{\alpha_{2}}$ when $\alpha_{3}=0$ and $P_{\alpha_{1},\alpha_{2},
\alpha_{3}}^{\beta_{1},\beta_{2},
\beta_{3}}(m)/H_{\alpha_{1},\alpha_{2},\alpha_{3}}^{\beta_{1},\beta_{2},
\beta_{3}}(m)$\\$Q_{\alpha_{1},\alpha_{2},
\alpha_{3}}^{\beta_{1},\beta_{2},
\beta_{3}}(v) = 0$ as $\alpha_{3} \geq 1$ for any 
$v,m$ and $\alpha_{1}+2\alpha_{2}+2\alpha_{3}=\beta_{1}+2\beta_{2}+2\beta_{3}
$. This shows that the system of polynomial equations 
\eqref{eqn:systemnonlinear} contains the following system of equations
\begin{eqnarray}
\sum \limits_{j=1}^{l}{\sum \limits_{\alpha_{1}+2\alpha_{2}=\beta_{1}+2\beta_{2}}\dfrac{2^{\alpha_{2}}d_{j}^{2}a_{j}^{\alpha_{1}}b_{j}
^{\alpha_{2}}}{\alpha_{1}!\alpha_{2}!}} = 0, \label{eqn:system_lemma}
\end{eqnarray}
where $\beta_{1}+2\beta_{2} \leq r$ and $\beta_{1} \leq 1$. This is precisely the system of polynomial 
equations \eqref{eqn:generalovefittedGaussianzero_Gaussian_mulindex} if we replace $d_{j}$ by $x_{j}$, $a_{j}
$ by $y_{j}$, $2b_{j}$ by $z_{j}$, $\alpha_{1}, \alpha_{2}$ by $n_{1},n_{2}$. Now, if we 
choose $r>\overline{r}(l)$, the system of polynomial equations \eqref{eqn:system_lemma} 
has only trivial solution $a_{j}=b_{j}=0$ for all $1 \leq j \leq l$. Substitute these 
results back to system of polynomial equations \eqref{eqn:systemnonlinear}, we 
also obtain $c_{j}=0$ for all $1 \leq j \leq l$, which is a contradiction. 
This completes our proof.

\paragraph{PROOF OF PROPOSITION \ref{proposition:overfittedproposition}}
The proof of part (a) is straightforward from the discussion in Section \ref{Section:illustration_omixture_byone}. For the proof for part (b), we will present an explicit form for the system of polynomial equations 
to illustrate the variablity of $\underbar{\rVMS}(l)$ and $\overline{\rVMS}(l)$ based on the 
values of $(m,v)$.

(b) As $l=2$ and $r=6$, the system of polynomial equations \eqref{eqn:systemnonlinear} can be rewritten as
\begin{eqnarray}
\sum \limits_{i=1}^{3}{d_{i}^{2}a_{i}}=0, \ \sum \limits_{i=1}^{3}{d_{i}^{2}a_{i}^{2}+d_{i}^{2}b_{i}}=0, \ \sum \limits_{i=1}^{3}{-(m^{3}+m)d_{i}^{2}a_{i}^{2}+2v d_{i}^{2}c_{i}}=0, \nonumber \\
\sum \limits_{i=1}^{3}{\dfrac{1}{3} d_{i}^{2}a_{i}^{3}+d_{i}^{2}a_{i}b_{i}}=0, \ \sum \limits_{i=1}^{3}{-(m^{3}+m) d_{i}^{2}a_{i}^{3}+6v d_{i}^{2}a_{i}c_{i}}=0,  \nonumber \\
\sum \limits_{i=1}^{3}{\dfrac{(m^{3}+m)^{2}}{12v^{2}}d_{i}^{2}a_{i}^{4}
-\dfrac{m^{3}+m}{v}d_{i}^{2}a_{i}
^{2}c_{i}-\dfrac{m^{2}+1}{vm}d_{i}^{2}b_{i}c_{i}+d_{i}^{2}c_{i}^{2}}=0, \nonumber \\
\sum \limits_{i=1}^{3}{\dfrac{1}{6} d_{i}^{2}a_{i}^{4}+d_{i}^{2}a_{i}^{2}b_{i}+\dfrac{1}{2}d_{i}^{2}b_{i}^{2}}=0, \ \sum \limits_{i=1}^{3}{\dfrac{1}{30}d_{i}^{2}a_{i}^{5}+\dfrac{1}{3} d_{i}^{2}a_{i}^{3}b_{i}+\dfrac{1}{2}d_{i}^{2}a_{i}b_{i}^{2}}=0, \nonumber \\
\sum \limits_{i=1}^{3}{\dfrac{(m^{3}+m)^{2}}{120v^{2}}d_{i}^{2}a_{i}^{5}-\dfrac{(m^{3}+m)}{6v} d_{i}^{2}a_{i}^{3}c_{i}-\dfrac{m^{2}+1}{2vm}d_{i}^{2}a_{i}b_{i}c_{i}+\dfrac{1}{2}d_{i}^{2}a_{i}c_{i}^{2}}=0, \nonumber \\
\sum \limits_{i=1}^{3}{\dfrac{1}{90}d_{i}^{2}a_{i}^{6}+\dfrac{1}{12} d_{i}^{2}a_{i}^{4}b_{i}+\dfrac{1}{2}d_{i}^{2}a_{i}^{2}b_{i}^{2}+\dfrac{1}{6}d_{i}^{2}b_{i}^{3}}=0, \nonumber \\
\sum \limits_{i=1}^{3}{\dfrac{(m^{3}+m)^{3}}{720v^{3}}d_{i}^{2}a_{i}^{6}+\dfrac{(m^{3}+m)^{2}}{24v^{2}} d_{i}^{2}a_{i}^{4}c_{i}+\dfrac{m^{3}+m}{4v}d_{i}^{2}a_{i}^{2}c_{i}^{2}} + \nonumber \\
\dfrac{(m^{2}+1)^{2}}{8v^{2}m^{2}}d_{i}^{2}b_{i}^{2}c_{i}-\dfrac{m^{2}+1}{4mv}d_{i}^{2}b_{i}c_{i}^{2}+\dfrac{1}{6}d_{i}^{2}c_{i}^{3}=0.
\label{eqn:proposition_systemnonlinear_one}
\end{eqnarray}
When $r=4$, the system of polynomial equations \eqref{eqn:systemnonlinear} 
contains the first 7 equations in the system of polynomial equations 
\eqref{eqn:proposition_systemnonlinear_one}. Now, $m$ and $v$ are considered as two additional
variables in the above system of polynomial equations. Hence, there are 13 variables with 
only 7 equations. If we choose $d_{1}=d_{2}=d_{3}$ and take the lexicographical ordering
$a_{1} 
\succ a_{2} \succ a_{3} \succ b_{1} \succ b_{2} \succ b_{3} \succ c_{1} \succ c_{2} 
\succ c_{3} \succ m \succ v$, the Grobener bases (cf. \cite{Bruno-Thesis}) of the above 
system of polynomial 
equations will return a non-trivial solution (due to the complexity of the roots, we will not 
present them here). As a consequence, $\underbar{\rVMS}(l) \geq 5$ under the case $l=2$.

For $l=2$ and $r=5$, the system of polynomial equations 
\eqref{eqn:systemnonlinear} retains the first 9 equations in system 
\eqref{eqn:proposition_systemnonlinear_one}. It can be checked that if we 
choose $m= \pm 2, v=1$, then the system of polynomial equations when $r=5$ does not have any 
non-trivial solution (note that, we also use the same lexicographical order as that being used 
in the case $r=4$). So, $\underbar{\rVMS}(l) = 5$. However, we can check that the value of 
$m=\dfrac{1}{10}$ (close to 0 in general) and $v=1$ will lead the system of polynomial equations 
\eqref{eqn:proposition_systemnonlinear_one} to not having any non-trivial solution. 
Thus, $\overline{\rVMS}(l)=6$. 
This concludes the proof or part (b) of the proposition.

\newpage

\section{Appendix E: Theory for skew-normal e-mixtures -- a summary} 
\label{Section:singularity_level_index_skewemix}

E-mixtures are the setting in which the number of mixing components
is known $k=k_0$. In this appendix, we provide a summary of singularity level and singularity index of mixing measure 
$G_0$ relative to the ambient space $\Ecal_{k_0}$, where $k_0$ is the
number of supporting atoms for $G_0$.

Recall from the previous sections the definition of $\Scal_0$, 
the subset $\Scal_0 \subset \Ecal_{k_0}$ of measure $G_0=G_0(\vecp^0,\veceta^0)$
such that $(\vecp^0,\veceta^0)$ satisfy $P_1(\myeta^0)P_2(\myeta^0) \neq 0$.
$P_1$ and $P_2$ are polynomials given in the statement of Lemma
\ref{proposition-notskewnormal}. It is simple to verify that
for any $G_0 \in \Scal_0$, as a consequence of this lemma, the Fisher information 
matrix $I(G_0)$ is non-singular. It follows that

\begin{theorem} If $G_0 \in \Scal_0$, then $\lev(G_0|\Ecal_{k_0}) = 0$ and $\singset(G_{0}|\Ecal_{k_{0}})=\left\{(1,1,1)\right\}$.
\end{theorem}

We turn our attention to the singularity structure of set $\Ecal_{k_0} \setminus \Scal_0$. 
For any $G_0 \in \Ecal_{k_0}\setminus \Scal_0$, the parameters of $G_0$
satisfy $P_1(\veceta^0) P_2(\veceta^0) =0$. Accordingly, for each pair of $(i,j) = 1,\ldots, k_0$
the two components indexed by $i$ and $j$ are said to be \textbf{homologous} if 
\[(\theta_i^0-\theta_j^0)^2 + [v_i^0(1+(m_j^0)^2) - v_j^0(1+(m_i^0)^2)]^2 = 0.\]
Moreover, for each $1 \leq i \leq k_{0}$, let $I_{i}$ denote 
the set of all components homologous to (component) $i$. By definition,
it is clear that if $i$ and $j$ are homologous, $I_{i} \equiv I_{j}$. 
Therefore, these homologous sets form equivalence classes. From here on,
when we say a homologous set $I$, we implicitly mean that it is the representation 
of the equivalent classes. 

Now, the homologous set consists of the indices of skew-normal components that share the same
location and a rescaled version of the scale parameter. 
A non-empty homologous set $I$ is said to be
{\bf conformant} if for any $i \neq j \in I$, $m_{i}^0m_{j}^0>0$. A non-empty homologous set $I$ is said to be {\bf nonconformant} if we can find two indices $i,j \in I$ such that $m_{i}^0m_{j}^0<0$. Additionally, $G_{0}$ is said to be 
\textbf{conformant} if all the homologous sets are conformant or \textbf{nonconformant} 
(NC) if at least one homologous set is nonconformant.
Now, we define a partition of $\mathcal{E}_{k_{0}} \setminus \Scal_0$ 
as follows $\Ecal_{k_0} = \Scal_0 \cup \Scal_1 \cup \Scal_2 \cup \Scal_3$, where
\begin{equation*}
\begin{cases}
\Scal_{1} = \left\{G = G(\vecp,\veceta) \in \mathcal{E}_{k_{0}} 
| \ P_{1}(\myeta) \neq 0, P_{2}(\myeta) = 0, G \ \text{is conformant} \right\} \\
\Scal_{2} = \left\{G = G(\vecp,\veceta) \in \mathcal{E}_{k_{0}} \ | 
\ P_{1}(\myeta) = 0, \ \text{if} \ P_{2}(\myeta) = 0 \ \text{then} \ G \ \text{is conformant} \right\} \\
\Scal_{3} = \left\{G = G(\vecp,\veceta) \in \mathcal{E}_{k_{0}} \ | \ P_{2}(\myeta) = 0 \ \text{and} \ G \ \text{is nonconformant} \right\}.
\end{cases}
\end{equation*}
Figure \ref{figure:partitionGone} summarizes singularity levels of elements 
residing in $\Ecal_{k_0}$, except for $\Scal_3$.

\begin{figure}
\begin{tikzpicture}
  \matrix (m) [matrix of math nodes,row sep=1 em,column sep= 0 em,minimum width=-1em]
{ & & & &  & \Ecal_{k_{0}} & & & & \\
 & & P_{1} \neq 0 & & & & & P_{1} = 0 &  & \\
0 & P_{2} \neq 0 & & P_{2} = 0 &  & & & & P_{2} = 0 & \\
1 &  & \mathcal{C}  & &  & &  &  & &  \\
2 & & & & &  & P_{2} \neq 0 & \mathcal{C} & & \\
\textrm{level}\; \in[1,\infty] & & & & \mathcal{NC} & & & & & \mathcal{NC} \\
& \Scal_{0} & \Scal_{1} & & \Scal_{3} & & \Scal_{2} & \Scal_{2} & & \Scal_{3}  \\
};
\path[-stealth]
(m-1-6) edge (m-2-3) edge (m-2-8)
(m-2-3) edge (m-3-2) edge (m-3-4)
(m-2-8) edge (m-5-7) edge (m-3-9)
(m-3-4) edge (m-4-3) edge (m-6-5)
(m-3-9) edge (m-5-8) edge (m-6-10);
\draw[->] (m-3-2) edge (m-7-2);
\draw[->] (m-4-3) edge (m-7-3);
\draw[->] (m-6-5) edge (m-7-5);
\draw[->] (m-5-7) edge (m-7-7);
\draw[->] (m-5-8) edge (m-7-8);
\draw[->] (m-6-10) edge (m-7-10);
\draw[densely dotted] (m-3-2) -- (m-3-1);
\draw[densely dotted] (m-4-3) -- (m-4-1);
\draw[densely dotted] (m-5-7) -- (m-5-1);
\draw[densely dotted] (m-5-8) -- (m-5-7);
\draw[densely dotted] (m-6-5) -- (m-6-1);
\draw[densely dotted] (m-6-10) -- (m-6-5);
\draw[thick] (-6.5,-3.2) -- (8,-3.2);
\end{tikzpicture}

\caption{The singularity level of $G_{0}$ relative to $\Ecal_{k_0}$
is determined by partition based on zeros of polynomials $P_1,P_2$
into subsets $\Scal_0,\Scal_1,\Scal_2,\Scal_3$. Here, "$\mathcal{NC}$" stands for 
nonconformant.}
\label{figure:partitionGone}
\end{figure}
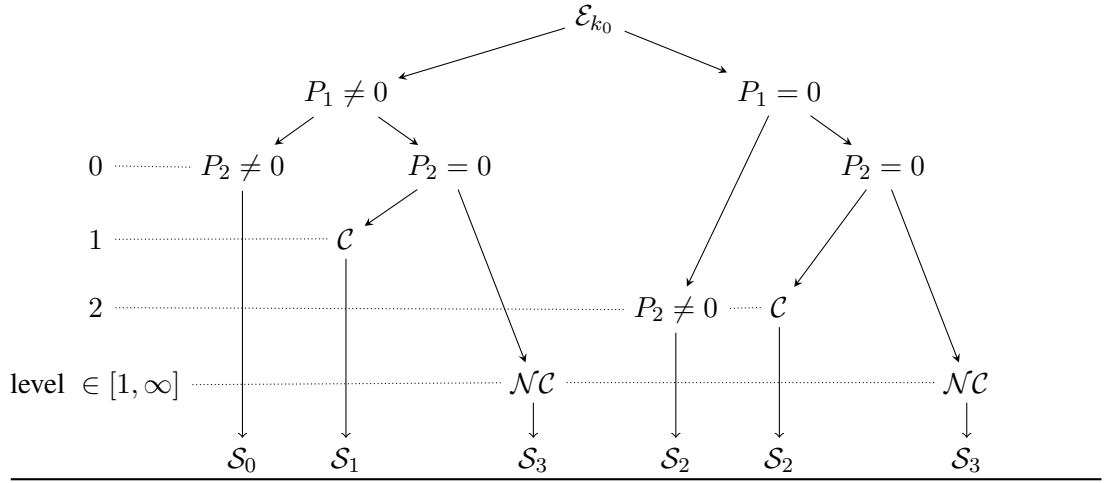

\subsection{Singularity structure of $G_0 \in \Scal_1 \cup \Scal_2$}
\label{Section:conformant_setting}
The main results of this subsection are the following two theorems.
\begin{theorem}
\label{theorem:conformant_setting}
If $G_0 \in \Scal_1$, then $\lev(G_0|\Ecal_{k_0}) = 1$ and $\singset(G_{0}|\Ecal_{k_{0}})=\left\{(1,1,2)\right\}$.
\end{theorem}
The results of Theorem \ref{theorem:conformant_setting} imply that the convergence 
rate of estimating mixing measure $G_{0}$ is $n^{-1/2}$ but individual parameters of $G_{0}$ 
admit different rates:  it is $n^{-1/2}$ for
location and scale parameters and $n^{-1/4}$ for skewness parameter.
Therefore, it is generally more efficient to estimate location and scale parameter than 
estimating skewness parameter under the setting $G_{0} \in \Scal_{1}$.
\begin{theorem}
\label{theorem:conformant_symmetry_setting}
If $G_0 \in \Scal_2$, 
then $\lev(G_0|\Ecal_{k_0}) = 2$ and $\singset(G_{0}|\Ecal_{k_{0}})=\left\{(3,2,3)\right\}$.
\end{theorem}
Unlike the results from Theorem \ref{theorem:conformant_setting}, under the setting 
$G_{0} \in \Scal_{2}$, the convergence rate of estimating mixing measure $G_{0}$ is 
$n^{-1/4}$. However, the convergence rate of location and skewness parameter 
is $n^{-1/6}$ while that of scale parameter is $n^{-1/4}$. 
The illustration of proof idea of the above theorems will be first given shortly, and while a complete proof is
presented in Appendix F.


\begin{figure}
\begin{tikzpicture}
  \matrix (m) [matrix of math nodes,row sep= 1 em,column sep= -0.2 em,minimum width=0.1em]
{& & P_{1} = 0 & & & \\  
& & \mathcal{NC} & & &  \\
\leq \mathop {\max }{\left\{\overline{s}(G_{0}),2\right\}} & P_{3} \neq 0 & & & 
P_{3} = 0 &  \\
\text{level} \geq 3 & & &  P_{4} \neq 0 & &  \\
& & & & & \\
\text{level} =\infty & & & & &  P_{4} = 0  \\
& & & & & \\
& \Scal_{31} & & \Scal_{32} & &  \Scal_{33} \\
};
\path[-stealth]
(m-1-3) edge (m-2-3)
(m-2-3) edge (m-3-2) edge (m-3-5)
(m-3-5) edge (m-4-4) edge (m-6-6);
\draw[->] (m-3-2) edge (m-8-2);
\draw[->] (m-4-4) edge (m-8-4);
\draw[->] (m-6-6) edge (m-8-6);
\draw[densely dotted] (m-3-2) -- (m-3-1);
\draw[densely dotted] (m-4-4) -- (m-4-1);
\draw[densely dotted] (m-6-6) -- (m-6-1);
\draw[thick] (-6.5,-3.3) -- (8,-3.3);
\end{tikzpicture}
\caption{The level of singularity structure of $G_{0} \in \Scal_{3}$ when $P_{1}(\veceta^0)=0$. Here, "$\mathcal{NC}$" stands for 
nonconformant. The term $\overline{s}(G_{0})$ is defined in \eqref{eqn:sufficient_order_nonconformant_emixture}.}
\label{figure:partitionnonconformant_symmetry}
\end{figure}

\begin{figure}
\begin{tikzpicture}
  \matrix (m) [matrix of math nodes,row sep=1 em,column sep= 0 em,minimum width=-1em]
{& & P_{1} \neq 0 & & & \\  
& & \mathcal{NC} & & &  \\
\leq \overline{s}(G_{0}) & P_{3} \neq 0 & & & P_{3} = 0 &  \\
\text{level} \geq 3 & & &  P_{4} \neq 0 & &  \\
& & & & & \\
\text{level} =\infty & & & & &  P_{4} = 0  \\
& & & & & \\
& \Scal_{31} & & \Scal_{32} & &  \Scal_{33} \\
};

\path[-stealth]
(m-1-3) edge (m-2-3)
(m-2-3) edge (m-3-2) edge (m-3-5)
(m-3-5) edge (m-4-4) edge (m-6-6);
\draw[->] (m-3-2) edge (m-8-2);
\draw[->] (m-4-4) edge (m-8-4);
\draw[->] (m-6-6) edge (m-8-6);
\draw[densely dotted] (m-3-2) -- (m-3-1);
\draw[densely dotted] (m-4-4) -- (m-4-1);
\draw[densely dotted] (m-6-6) -- (m-6-1);
\draw[thick] (-6.5,-3.3) -- (8,-3.3);
\end{tikzpicture}
\caption{The level of singularity structure of $G_{0} \in \mathcal{S}_{3}$ when $P_{1}(\veceta^0) \neq 0$. Here, "$\mathcal{NC}$" stands for 
nonconformant. The term $\overline{s}(G_{0})$ is defined in \eqref{eqn:sufficient_order_nonconformant_emixture}.}
\label{figure:partitionnonconformant_no_symmetry}
\end{figure}
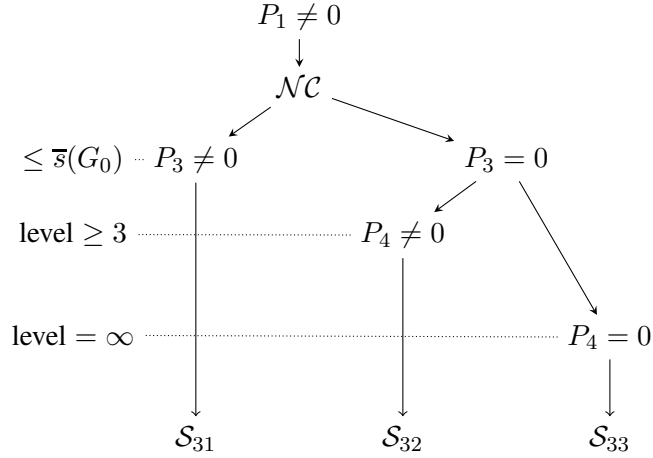


\subsection{Singularity structure of $G_{0} \in \Scal_{3}$: a summary} 
\label{Section:nonconformant_setting}

The singularity level and singularity index of $\Scal_3$ are much more complex than those of previous settings 
of $G_{0}$. $\Scal_3$ does not admit an uniform level of singularity structure for 
all its elements --- it needs to be partitioned into many subsets via 
intersections with additional semialgebraic sets of the parameter space. In addition,
we can establish the existence of
subsets that correspond to the infinite singularity level and singularity index. 
In most cases when the singularity level and singularity index are finite, we may be able to provide
some bounds rather than giving an exact value. As in o-mixtures setting in Section \ref{Section:general_bound_omixtures_generic}, the unifying theme
of such bounds is their connection to the solvability of
a system of real polynomial equations. 

If $G_0=G_0(\vecp^0,\veceta^0) \in \Scal_3$, then its corresponding parameters satisfy
$P_{2}(\myeta^0)=0$, i.e., there is at least one homologous set of $G_0$.
Moreover, at
least one such homologous set is nonconformant. 
For any $G_0 \in \Scal_{3}$, let $I_{1},\ldots,I_{t}$ be all nonconformant 
homologous sets of $G_0$. The singularity structures of $\Scal_3$ arise
from the zeros of the following polynomials:
\begin{itemize}
\item Type C(1): $P_{3}(\vecp^0, \veceta^0) := 
\prod \limits_{i=1}^{t}{\biggr(\prod \limits_{S \subseteq I_{i}, |S| \geq 2}{\biggr(\sum \limits_{j \in S}{p_{j}^0\prod \limits_{l \neq j}{m_{l}^0}}\biggr)}\biggr)}$.
\item Type C(2): $P_{4}(\vecp^0,\veceta^0) : = \prod \limits_{1 \leq i \neq j \leq k_{0}}
{\biggr [u_{ij}^2
+ (m_i^0\sigma_j^0+m_j^0\sigma_i^0)^{2}+
(p_i^0\sigma_j^0-p_j^0\sigma_i^0)^2\biggr ]}$,
where $u_{ij}^2 = (\theta_i^0-\theta_j^0)^2 + \left(v_i^0(1+(m_j^0)^2) - v_j^0(1+(m_i^0)^2)\right)^2$.
\end{itemize}
Type C singularities, including both C(1) and C(2), are distinguished from Type A and Type B singularities 
by the fact that the Type C polynomials are defined
by not only component parameters $\veceta^0$, but also mixing probability
parameters $\vecp^0$.
Note that C(1) singularity implies that there is some homologous set 
$I_{i}$ of $G_0$ such that
$\prod \limits_{S \subseteq I_{i}, |S| \geq 2}{\biggr(\sum \limits_{j \in S}{p_{j}^0\prod \limits_{l \neq j}{m_{l}^0}}\biggr)}=0$.
A homologous set of $G_0$ having the above property is said to contain 
type C(1) singularity locally. Similarly, type C(2) singularity implies that 
there is some pair $1 \leq i \neq j \leq k_{0}$ such that
$u_{ij}^{2}+ (m_i^0\sigma_j^0+m_j^0\sigma_i^0)^2 + (p_i^0\sigma_j^0-p_j^0\sigma_i^0)^2 = 0$.
A homologous set of $G_0$ having this pair is said to contain type C(2)
singularity locally. It can be easily checked that a homologous set containing 
type C(2) singularity must also contain type C(1) singularity,
since $P_4(\vecp^0, \veceta^0) = 0$ entails $P_3(\vecp^0,\veceta^0) = 0$. Now, we define the following partition of $\Scal_3$ according to the definition of type C(1) and C(2) singularity:
$\Scal_3 = \Scal_{31}\cup \Scal_{32} \cup \Scal_{33}$, where
\begin{equation*}
\begin{cases}
\Scal_{31}=\left\{G = G(\vecp,\veceta) \in \Scal_{3} \ | \ 
P_{3}(\vecp, \veceta) \neq 0  \right\} \\
\Scal_{32}=\left\{G = G(\vecp,\veceta)\in \Scal_{3} \ | \ P_{3}(\vecp, \veceta) = 0, 
P_{4}(\vecp, \veceta) \neq 0 \right\} \\
\Scal_{33}=\left\{G = G(\vecp,\veceta) \in \Scal_{3} \ | \ P_{3}(\vecp, \veceta) = 0, 
P_{4}(\vecp, \veceta) = 0 \right\}.
\end{cases}
\end{equation*}

Due to the highly technical nature of our analysis of the singularity structure of $\Scal_3$,
we defer the detailed analysis to Section~\ref{Section:singularity_level_S3_setting} in Appendix F. 
Here, we only provide a summary of such results. We can demonstrate that $(1,1,
\lev(G_{0}|\Ecal_{k_{0}})+1)$ is the unique singularity index of $G_{0} \in \Scal_{3}$ 
relative to $\Ecal_{k_{0}}$ when $P_1(\veceta^0) \neq 0$, i.e., there are no Gaussian 
components in $G_{0}$. Additionally, when $P_1(\veceta^0) = 0$, i.e., there are some 
Gaussian components in $G_{0}$, then we have $(3,2,\max\left\{2,\lev(G_{0}|
\Ecal_{k_{0}})\right\}+1)$ as the unique singularity index of $G_{0}$. Therefore, 
studying singularity index of $G_{0} \in \Scal_{3}$ is equivalent to studying singularity 
level of $G_{0} \in \Scal_{3}$. 

Figure
\ref{figure:partitionnonconformant_symmetry} and 
\ref{figure:partitionnonconformant_no_symmetry} provide illustrations of singularity levels of $G_{0} \in \Scal_{3}$.
Specifically, when
$G_0 \in \Scal_{31}$, it is shown that $\lev(G_0|\Ecal_{k_0}) \leq \max\{2,
\overline{s}(G_{0})\}$, where $\overline{s}(G_{0})$ is defined by a
system of polynomial
equations that we obtain via a method of greedy extraction of polynomial limits, see Section 
\ref{Section:nonconformant_no_typeC}.
In some specific cases, the precise
singularity level of $G_{0} \in \Scal_{31}$ will be given. 
If $G_0 \in \Scal_{32}$, we need a more sophisticated
method of extraction for polynomial limits; our technique is illustrated on
on a specific example of $G_{0}$ in Section \ref{Section:nonconformant_typeC}. Finally, if 
$G_0 \in \Scal_{33}$, it is shown that
$\lev(G_0|\Ecal_{k_0}) = \infty$ in Section \ref{Section:nonconformant_typeC_typeIV}.

%


\section{Appendix F: Theory of skew-normal e-mixtures -- with proofs}

This Appendix contains a self-contained and detailed treatment of singularity
structure of e-mixtures of skew-normal distributions for which a brief summary of the results 
was given earlier in Appendix E. This Appendix should be skipped at the first reading.

\subsection{Singularity structure of $\Scal_1$}
\label{Section:singularity_level_S1_setting}
In the following, we shall present the proof of Theorem \ref{theorem:conformant_setting} for a simple setting of $G_0 \in \Scal_1$,
which illustrates the complete proofs, and also helps to explain why
the partition of according to $\Scal_1$, i.e., the notion of conformant,
arises in the first place. The simplified setting is that all components of $G_0$ are homologous to
one another. By definition all components of $G_0$ are non-Gaussian (because
$P_1(\veceta^0) \neq 0$). 
Thus, we have $\theta_{1}^{0} = \ldots = 
\theta_{k_{0}}^{0}$ and $\dfrac{v_{1}^{0}}{1+(m_{1}^{0})^{2}}= \ldots = 
\dfrac{v_{k_{0}}^{0}}{1+(m_{k_{0}}^{0})^{2}}$. Additionally, $m_{i}^0 \neq 0$ for all $1 \leq i \leq k_0$. Since $G_{0}$ is conformant, 
$m_{i} ^{0}$ share the same sign for all $1 \leq i \leq k_{0}$. 
Without loss of generality, we assume $m_{i}^{0} >0$. To demonstrate that $(1,1,2)$ 
is the unique singularity index of $G_{0}$, we need
show that $G_0$ is 1-singular and (1,1,1)-singular, but not (1,1,2)-singular.

\paragraph{Claim: $G_0$ is 1-singular and (1,1,1)-singular} Given constraints on the parameters of $G_0$, 
it is simple to arrive at the following $1$-minimal and (1,1,1)-minimal form (cf. Eq.
\eqref{eqn:generallinearindependencerepresentation}):
\begin{eqnarray}
\dfrac{1}{W_{1}(G,G_{0})} \biggr \{\sum \limits_{i=1}^{k_{0}}{\biggr[\beta_{1i}^{(1)}+
\beta_{2i}^{(1)}(x-\theta_{1}^{0})+\beta_{3i}^{(1)}(x-\theta_{1}^{0})^{2})\biggr]f
\left(\dfrac{x-\theta_{1}^{0}}{\sigma_{i}^{0}}\right)\Phi\left(\dfrac{m_{i}^{0}(x-
\theta_{1}^{0})}{\sigma_{i}^{0}}\right)} \nonumber \\
+\biggr[\gamma_{1}^{(1)}+\gamma_{2}^{(1)}(x-\theta_{1}^{0})\biggr] \exp\left(-
\dfrac{(m_{1}^{0})^{2}+1}{2v_{1}^{0}}(x-\theta_{1}^{0})^{2}\right)\biggr \}
+o(1), \label{eqn:taylorexpansion_first_order_homologous}
\end{eqnarray}
where coefficients $\beta_{1i}^{(1)}, \beta_{2i}^{(1)}, \beta_{3i}^{(1)}, \gamma_{1}^{(1)}, \gamma_{2}^{(1)}$ are the polynomials of
$\Delta \theta_{j}$, $\Delta v_{j}$, $\Delta m_{j}$, and $\Delta p_{j}$:
\begin{eqnarray*}
\beta_{1i}^{(1)}=\dfrac{2\Delta p_{i}}{\sigma_{i}^{0}}-\dfrac{p_{i}\Delta v_{i}}{(\sigma_{i}^{0})^{3}}, \ \beta_{2i}^{(1)}=\dfrac{2p_{i}\Delta \theta_{i}}{(\sigma_{i}^{0})^{3}},\ \beta_{3i}^{(1)}=\dfrac{p_{i}\Delta v_{i}}{(\sigma_{i}^{0})^{5}}, \\ 
\gamma_{1}^{(1)}=\sum \limits_{j=1}^{k_{0}}{-\dfrac{p_{j}m_{j}^{0}\Delta \theta_{j}}{\pi(\sigma_{j}^{0})^{2}}}, 
\gamma_{2}^{(1)}=\sum \limits_{j=1}^{k_{0}}{-\dfrac{p_{j}m_{j}^{0}\Delta v_{j}}{2\pi(\sigma_{j}^{0})^{4}}+\dfrac{p_{j}\Delta m_{j}}{\pi(\sigma_{j}^{0})^{2}}}. \nonumber
\end{eqnarray*} 
Note that, the conditions $m_{i}^0 \neq 0$ for all $1 \leq i \leq k_0$ allow us to have that $f
\left(\dfrac{x-\theta_{1}^{0}}{\sigma_{i}^{0}}\right)\Phi\left(\dfrac{m_{i}^{0}(x-
\theta_{1}^{0})}{\sigma_{i}^{0}}\right)$ and $\exp\left(-
\dfrac{(m_{1}^{0})^{2}+1}{2v_{1}^{0}}(x-\theta_{1}^{0})^{2}\right)$ are linearly independent. It is clear that if a sequence of $G$ (represented by 
Eq.~\eqref{eqn:representation_overfit}) is chosen such that
$\Delta \theta_{i} = \Delta v_{i} =\Delta p_{i} = 0$ for all $1 \leq i \leq k_{0}$, 
and $\sum \limits_{i=1}^{k_{0}}{p_{i}\Delta m_{i}/v_{i}^{0}}=0$,
then we obtain
$\beta_{1i}^{(1)}/W_{1}(G,G_{0}) = \beta_{2i}^{(1)}/W_{1}(G,G_{0})
= \beta_{3i}^{(1)}/W_{1}(G,G_{0}) = 
\gamma_{1}^{(1)}/W_{1}(G,G_{0}) = \gamma_{2}^{(1)}/W_{1}(G,G_{0}) = 0$. Hence,
$G_0$ is 1-singular and (1,1,1)-singular relative to $\mathcal{E}_{k_{0}}$.

\paragraph{Claim: $G_0$ is not (1,1,2)-singular} Indeed, suppose that this is not
true. Let $\kappa=(1,1,2)$. Then from Definition \ref{def-rsingular_set}, for any sequence of $G$ that tends to $G_0$ under $\widetilde{W}_{\kappa}$,
all coefficients of the $\kappa$-minimal form must vanish. A $\kappa$-minimal
form is given as follows:
\begin{eqnarray}
\dfrac{1}{\widetilde{W}_{\kappa}^{2}(G,G_{0})}\biggr[\sum \limits_{i=1}^{k_{0}}{\biggr(\sum 
\limits_{j=1}^{5}{\beta_{ji}^{(2)}(x-\theta_{1}^{0})^{j-1}}\biggr)f\left(\dfrac{x-\theta_{1}
^{0}}{\sigma_{i}^{0}}\right)\Phi\left(\dfrac{m_{i}^{0}(x-\theta_{1}^{0})}{\sigma_{i}
^{0}}\right)} \nonumber \\
+\biggr(\sum \limits_{j=1}^{4}{\gamma_{j}^{(2)}(x-\theta_{1}^{0})^{j-1}}\biggr)\exp
\left(-\dfrac{(m_{1}^{0})^{2}+1}{2v_{1}^{0}}(x-\theta_{1}^{0})^{2}\right)\biggr]+o(1), \label{eqn:taylorexpansionsecondorder}
\end{eqnarray}
where $\beta_{ji}^{(2)}, \gamma_{j}^{(2)}$ are polynomials of $\Delta \theta_{l}$, $\Delta 
v_{l}$, $\Delta m_{l}$, and $\Delta p_{l}$ for $l=1,\ldots, k_0$: 
\begin{eqnarray}
\beta_{1i}^{(2)}=\dfrac{2\Delta p_{i}}{\sigma_{i}^{0}}-\dfrac{p_{i}\Delta v_{i}}{(\sigma_{i}^{0})^{3}}-\dfrac{p_{i}(\Delta \theta_{i})^{2}}{(\sigma_{i}^{0})^{3}}+\dfrac{3p_{i}(\Delta v_{i})^{2}}{4(\sigma_{i}^{0})^{5}}, \ \beta_{2i}^{(2)}=\dfrac{2p_{i}\Delta \theta_{i}}{(\sigma_{i}^{0})^{3}}-\dfrac{6p_{i}\Delta \theta_{i}\Delta v_{i}}{(\sigma_{i}^{0})^{5}}, \nonumber \\
\beta_{3i}^{(2)}=\dfrac{p_{i}\Delta v_{i}}{(\sigma_{i}^{0})^{5}}+\dfrac{p_{i}(\Delta \theta_{i})^{2}}{(\sigma_{i}^{0})^{5}}-\dfrac{3p_{i}(\Delta v_{i})^{2}}{2(\sigma_{i}^{0})^{7}}, \ \beta_{4i}^{(2)}=\dfrac{2p_{i}\Delta \theta_{i}\Delta v_{i}}{(\sigma_{i}^{0})^{7}}, \ \beta_{5i}^{(2)}=\dfrac{p_{i}(\Delta v_{i})^{2}}{4(\sigma_{i}^{0})^{9}}, \nonumber \\
\gamma_{1}^{(2)}=\sum \limits_{j=1}^{k_{0}}{-\dfrac{p_{j}m_{j}^{0}\Delta \theta_{j}}{\pi(\sigma_{j}^{0})^{2}}+\dfrac{2p_{j}m_{j}^{0}(\Delta \theta_{j})(\Delta v_{j})}{\pi(\sigma_{j}^{0})^{4}}-\dfrac{2p_{j}\Delta \theta_{j}\Delta m_{j}}{\pi(\sigma_{j}^{0})^{2}}}, \nonumber \\
\gamma_{2}^{(2)}=\sum \limits_{j=1}^{k_{0}}{-\dfrac{p_{j}m_{j}^{0}\Delta v_{j}}{2\pi(\sigma_{j}^{0})^{4}}-\dfrac{p_{j}((m_{j}^{0})^{3}+2m_{j}^{0})(\Delta \theta_{j})^{2}}{2\pi(\sigma_{j}^{0})^{4}}+\dfrac{p_{j}\Delta m_{j}}{\pi(\sigma_{j}^{0})^{2}}+\dfrac{5p_{j}m_{j}^{0}(\Delta v_{j})^{2}}{8\pi(\sigma_{j}^{0})^{6}}-\dfrac{p_{j}\Delta v_{j}\Delta m_{j}}{\pi(\sigma_{j}^{0})^{4}}}, \nonumber \\
\gamma_{3}^{(2)}=\sum \limits_{j=1}^{k_{0}}{\dfrac{p_{j}(2(m_{j}^{0})^{2}+2)\Delta \theta_{j} \Delta m_{j}}{\pi(\sigma_{j}^{0})^{4}}-\dfrac{p_{j}((m_{j}^{0})^{3}+2m_{j}^{0})\Delta \theta_{j}\Delta v_{j}}{2\pi(\sigma_{j}^{0})^{6}}}, \nonumber \\
\gamma_{4}^{(2)}=\sum \limits_{j=1}^{k_{0}}{-\dfrac{p_{j}((m_{j}^{0})^{3}+2m_{j}^{0})(\Delta v_{j})^{2}}{8\pi(\sigma_{j}^{0})^{8}}-\dfrac{p_{j}m_{j}^{0}(\Delta m_{j})^{2}}{2\pi(\sigma_{j}^{0})^{4}}+\dfrac{p_{j}((m_{j}^{0})^{2}+1)\Delta v_{j}\Delta m_{j}}{\pi(\sigma_{j}^{0})^{6}}}. \nonumber 
\end{eqnarray}
Now, $\beta_{ji}^{(2)}/\widetilde{W}_{\kappa}^{2}(G,G_{0}) \to 0$ leads to $\Delta 
p_{i}/\widetilde{W}_{\kappa}^{2}(G,G_{0}), \Delta \theta_{i}/\widetilde{W}_{\kappa}^{2}(G,G_{0}), \Delta v_{i} /\widetilde{W}_{\kappa}^{2}(G,G_{0}) \to 0$ for all $1 \leq i  \leq 
k_{0}$ (The rigorous argument for that result is in Step 1.1 of the full proof of this theorem in Appendix F). Combining with Lemma \ref{lemma:bound_overfit_Wasserstein_first}, we obtain 
\begin{eqnarray}
\sum \limits_{i=1}^{k_{0}}{p_{i}|\Delta m_{i}|^{2}}/\widetilde{W}_{\kappa}^{2}(G,G_{0}) \not \to 0. \label{eqn:singularitys1second}
\end{eqnarray}  
Additionally, the vanishing of coefficients $\gamma_{j}^{(2)}/\widetilde{W}_{\kappa}^{2}
(G,G_{0})$ for $1 \leq j \leq 4$ entails 
\begin{eqnarray}
\biggr(\sum \limits_{i=1}^{k_{0}}{p_{i}\Delta m_{i}/v_{i}^{0}}\biggr)/\widetilde{W}_{\kappa}^{2}
(G,G_{0}) \to 0, \nonumber \\
\biggr(\sum \limits_{i=1}^{k_{0}}{p_{i}m_{i}^{0}(\Delta m_{i})^{2}/(v_{i}^{0})^{2}}
\biggr)/\widetilde{W}_{\kappa}^{2}(G,G_{0}) \to 0. \label{eqn:singularitys1third}
\end{eqnarray}
Combining \eqref{eqn:singularitys1second} and \eqref{eqn:singularitys1third}, it
follows that
\begin{eqnarray}
\biggr(\sum \limits_{i=1}^{k_{0}}{p_{i}m_{i}^{0}(\Delta m_{i})^{2}/(v_{i}^{0})^{2}}\biggr)/\sum \limits_{i=1}^{k_{0}}{p_{i}|\Delta m_{i}|^{2}} \to 0, \nonumber
\end{eqnarray}
which is a contradiction due to $m_{i}^{0} >0$ for all $1 \leq i \leq k_{0}$. 
Hence, $G_0$ is not (1,1,2)-singular relative to $\mathcal{E}_{k_{0}}$. 
We conclude that $\lev(G_{0}|\mathcal{E}_{k_{0}})=1$ and $\singset(G_{0}|\Ecal_{k_{0}})=\left\{(1,1,2)\right\}$.
\subsection{Singularity structure of $\Scal_2$}
\label{Section:singularity_level_S2_setting}
To illustrate the singularity level and singularity index of $G_{0} \in \Scal_{2}$, we consider a simple
setting of $G_{0} \in \mathcal{S}_{2}$ in which $m_{1}^{0},m_{2}^{0},\ldots,m_{k_{0}}^{0}=0$, leaving out
the possible setting of conformant homologous sets and generic components. 

\paragraph{Claim: $G_0$ is 2-singular and (2,2,2)-singular} To establish this, we look at
2-minimal and (2,2,2)-minimal form
for $(p_{G}(x)-p_{G_{0}}(x))/W_{2}^{2}(G,G_{0})$, which is asymptotically
equal to
\begin{eqnarray}
\dfrac{1}{W_{2}^{2}(G,G_{0})}\biggr[\sum \limits_{i=1}^{k_{0}}{\biggr(\sum 
\limits_{j=1}^{5}{\zeta_{ji}^{(2)}(x-\theta_{i}^{0})^{j-1}}\biggr)} f\left(\dfrac{x-\theta_{i}
^{0}}{\sigma_{i}^{0}}\right)\biggr], \label{eqn:taylorexpansionsecondordertwo}
\end{eqnarray}
where $\zeta_{li}^{(2)}$ are the polynomials in terms of $\Delta \theta_{j}$, $\Delta v_{j}$, $\Delta 
m_{j}$, and $\Delta p_{j}$ as $1 \leq i,j \leq k_{0}$ and $1 \leq l \leq 5$. 
To make all the coefficients vanish, 
it suffices to have 
$(\Delta v_{i})^{2}/W_{2}^{2}(G,G_{0}) \to 0$ and 
\begin{eqnarray}
\biggr[-\dfrac{p_{i}\Delta v_{i}}{2(\sigma_{i}^{0})^{3}}-\dfrac{p_{i}(\Delta 
\theta_{i})^{2}}{2(\sigma_{i}^{0})^{3}}+\dfrac{3p_{i}(\Delta v_{i})^{2}}{8(\sigma_{i}
^{0})^{5}}-\dfrac{2p_{i}\Delta \theta_{i}\Delta m_{i}}{\sqrt{2\pi}(\sigma_{i}^{0})^{2}} + 
\dfrac{\Delta p_{i}}{\sigma_{i}^{0}}\biggr]/W_{2}^{2}(G,G_{0}) \to 0, \nonumber \\
\biggr[\dfrac{\Delta \theta_{i}}{(\sigma_{i}^{0})^{3}}+\dfrac{2\Delta m_{i}}{\sqrt{2\pi}
(\sigma_{i}^{0})^{2}}-\dfrac{3\Delta \theta_{i}\Delta v_{i}}{2(\sigma_{i}^{0})^{5}}-
\dfrac{2\Delta v_{i}\Delta m_{i}}{\sqrt{2\pi}(\sigma_{i}^{0})^{4}}\biggr]/W_{2}^{2}(G,G_{0}) \to 0, \nonumber \\
\biggr[\dfrac{\Delta v_{i}}{2(\sigma_{i}^{0})^{5}}+\dfrac{(\Delta \theta_{i})^{2}}
{2(\sigma_{i}^{0})^{5}}+\dfrac{2\Delta \theta_{i}\Delta m_{i}}{\sqrt{2\pi}(\sigma_{i}
^{0})^{4}}\biggr]/W_{2}^{2}(G,G_{0}) \to 0, \nonumber \\
\biggr[\dfrac{\Delta \theta_{i}\Delta v_{i}}{2(\sigma_{i}^{0})^{7}}+\dfrac{\Delta v_{i}
\Delta m_{i}}{\sqrt{2\pi}(\sigma_{i}^{0})^{6}}\biggr]/W_{2}^{2}(G,G_{0}) \to 0. \label{eqn:singularitys2first}
\end{eqnarray}
This can be achieved by choosing a sequence of $G \rightarrow G_0$ in $W_2$ such that
$\Delta \theta_{i} = \Delta v_{i} = \Delta m_{i} =\Delta p_{i} = 0$ for 
all $2 \leq i \leq k_{0}$; only for component 1 do we set $\Delta \theta_{1} = -2 \Delta m_{1} \sigma_{1}^{0}/\sqrt{2\pi}$ and $\Delta v_{1}= (\Delta \theta_{1})^{2}/2$. It follows that $G_0$ is 2-singular and (2,2,2)-singular relative to $\mathcal{E}_{k_{0}}$.

\paragraph{Claim: $G_0$ is not (3,2,3)-singular} It also entails that $G_{0}$ is not 3-singular relative to $\Ecal_{k_{0}}$. Now, let $\kappa=(3,2,3)$. The $\kappa$-minimal form of $(p_{G}(x)-
p_{G_{0}}(x))/\widetilde{W}_{\kappa}^{3}(G,G_{0}) $ is asymptotically equal to
\begin{eqnarray}
\dfrac{1}{\widetilde{W}_{\kappa}^{3}(G,G_{0})}\biggr[\sum \limits_{i=1}^{k_{0}}{\biggr(\sum 
\limits_{j=1}^{7}{\zeta_{ji}^{(3)}(x-\theta_{i}^{0})^{j-1}}\biggr)} f\left(\dfrac{x-\theta_{i}
^{0}}{\sigma_{i}^{0}}\right)\biggr], \label{eqn:taylorexpansionthirdorder}
\end{eqnarray}
where $\zeta_{li}^{(3)}$ are the polynomials in terms of $\Delta \theta_{j}$, $\Delta v_{j}$, $\Delta 
m_{j}$, and $\Delta p_{j}$ as $1 \leq i,j \leq k_{0}$ and $1 \leq l \leq 7$. 
Suppose that there exists a sequence $G \rightarrow G_0$ under $\widetilde{W}_{\kappa}$ such that
all the coefficients of the $\kappa$-minimal form vanish. For any $1 \leq i \leq k_{0}$, it follows after some calculations 
that 
\begin{eqnarray}
C_{1}^{(i)} := \biggr[-\dfrac{p_{i}\Delta v_{i}}{2(\sigma_{i}^{0})^{3}}-\dfrac{p_{i}(\Delta \theta_{i})^{2}}{2(\sigma_{i}^{0})^{3}}+\dfrac{3p_{i}(\Delta v_{i})^{2}}{8(\sigma_{i}^{0})^{5}}-\dfrac{2p_{i}\Delta \theta_{i}\Delta m_{i}}{\sqrt{2\pi}(\sigma_{i}^{0})^{2}}+\dfrac{3p_{i}(\Delta \theta_{i})^{2}\Delta v_{i}}{4(\sigma_{i}^{0})^{5}}+\nonumber \\
\dfrac{2p_{i}\Delta \theta_{i}\Delta v_{i}\Delta m_{i}}{\sqrt{2\pi}(\sigma_{i}^{0})^{4}}+\dfrac{\Delta p_{i}}{\sigma_{i}^{0}}\biggr]/\widetilde{W}_{\kappa}^{3}(G,G_{0}) \to 0, \nonumber \\
C_{2}^{(i)} := \biggr[\dfrac{p_{i}\Delta \theta_{i}}{(\sigma_{i}^{0})^{3}}+\dfrac{2p_{i}\Delta m_{i}}{\sqrt{2\pi}(\sigma_{i}^{0})^{2}}-\dfrac{3p_{i}\Delta \theta_{i}\Delta v_{i}}{2(\sigma_{i}^{0})^{5}}-\dfrac{2p_{i}\Delta v_{i}\Delta m_{i}}{\sqrt{2\pi}(\sigma_{i}^{0})^{4}}-\dfrac{p_{i}(\Delta \theta_{i})^{3}}{2(\sigma_{i}^{0})^{5}}- \nonumber \\
\dfrac{3p_{i}(\Delta \theta_{i})^{2}\Delta m_{i}}{\sqrt{2\pi}(\sigma_{i}^{0})^{4}}+\dfrac{15p_{i}\Delta \theta_{i}(\Delta v_{i})^{2}}{8(\sigma_{i}^{0})^{7}}+\dfrac{2p_{i}(\Delta v_{i})^{2}\Delta m_{i}}{\sqrt{2\pi}(\sigma_{i}^{0})^{6}}\biggr]/\widetilde{W}_{\kappa}^{3}(G,G_{0}) \to 0, \nonumber \\
C_{3}^{(i)} := \biggr[\dfrac{p_{i}\Delta v_{i}}{2(\sigma_{i}^{0})^{5}}+\dfrac{p_{i}(\Delta \theta_{i})^{2}}{2(\sigma_{i}^{0})^{5}}-\dfrac{3p_{i}(\Delta v_{i})^{2}}{4(\sigma_{i}^{0})^{7}}+\dfrac{2p_{i}\Delta \theta_{i}\Delta m_{i}}{\sqrt{2\pi}(\sigma_{i}^{0})^{4}}-\dfrac{3p_{i}(\Delta \theta_{i})^{2}\Delta v_{i}}{2(\sigma_{i}^{0})^{7}} - \nonumber \\
\dfrac{5\Delta \theta_{i}\Delta v_{i}\Delta m_{i}}{\sqrt{2\pi}(\sigma_{i}^{0})^{6}} \biggr]/\widetilde{W}_{\kappa}^{3}(G,G_{0}) \to 0, \nonumber \\
C_{4}^{(i)} := \biggr[\dfrac{p_{i}\Delta \theta_{i}\Delta v_{i}}{2(\sigma_{i}^{0})^{7}}+\dfrac{p_{i}\Delta v_{i}\Delta m_{i}}{\sqrt{2\pi}(\sigma_{i}^{0})^{6}}+\dfrac{p_{i}(\Delta \theta_{i})^{3}}{6(\sigma_{i}^{0})^{7}}-\dfrac{p_{i}(\Delta m_{i})^{3}}{3\sqrt{2\pi}(\sigma_{i}^{0})^{4}}+\nonumber \\
\dfrac{p_{i}(\Delta \theta_{i})^{2}\Delta m_{i}}{\sqrt{2\pi}(\sigma_{i}^{0})^{6}}-\dfrac{5p_{i}\Delta \theta_{i}(\Delta v_{i})^{2}}{4(\sigma_{i}^{0})^{9}}-\dfrac{2p_{i}(\Delta v_{i})^{2}\Delta m_{i}}{\sqrt{2\pi}(\sigma_{i}^{0})^{8}}\biggr]/\widetilde{W}_{\kappa}^{3}(G,G_{0}) \to 0, \nonumber \\
C_{5}^{(i)} := \biggr[\dfrac{p_{i}(\Delta v_{i})^{2}}{8(\sigma_{i}^{0})^{9}}-\dfrac{5p_{i}(\Delta v_{i})^{3}}{16(\sigma_{i}^{0})^{11}}+\dfrac{p_{i}(\Delta \theta_{i})^{2}\Delta v_{i}}{4(\sigma_{i}^{0})^{9}}+\dfrac{p_{i}\Delta \theta_{i}\Delta v_{i}\Delta m_{i}}{\sqrt{2\pi}(\sigma_{i}^{0})^{8}}\biggr]/\widetilde{W}_{\kappa}^{3}(G,G_{0}) \to 0, \nonumber \\
C_{6}^{(i)} := \biggr[\dfrac{p_{i}\Delta \theta_{i}(\Delta v_{i})^{2}}{8(\sigma_{i}^{0})^{11}}+\dfrac{p_{i}(\Delta v_{i})^{2}\Delta m_{i}}{4\sqrt{2\pi}(\sigma_{i}^{0})^{10}}\biggr]/\widetilde{W}_{\kappa}^{3}(G,G_{0}) \to 0, \nonumber \\
C_{7}^{(i)} := p_{i}(\Delta v_{i})^{3}/48(\sigma_{i}^{0})^{3}/\widetilde{W}_{\kappa}^{3}(G,G_{0}) \to 0. \label{eqn:system_limits_symmetry_emixtures}
\end{eqnarray} 
Since the system of limits in \eqref{eqn:system_limits_symmetry_emixtures} holds for any 
$1 \leq i \leq k_{0}$, to further simplify the argument without loss of generality, we consider 
$k_{0}=1$. Under that scenario, we can rewrite $\widetilde{W}_{\kappa}^{3}(G,G_{0})=p_{1}(|\Delta \theta_{1}|
^{3}+|\Delta v_{1}|^{2}+|\Delta m_{1}|^{3})$ where $p_{1}=1$. Additionally, for the 
simplicity of the presentation, we denote $C_{i} := C_{i}^{(1)}$ for any $1 \leq i \leq 7$. 
Now, our argument is organized into the following key steps
\paragraph{Step 1.1:} We will argue that $\Delta \theta_{1}, \Delta v_{1}, \Delta 
m_{1} \neq 0$. If $\theta_{1}=0$, by combining the vanishing of $C_{5}$, we achieve 
$(\Delta v_{1})^{2}/\widetilde{W}_{\kappa}^{3}(G,G_{0}) \to 0$. Since $C_{3} \to 
0$, we further obtain that $\Delta v_{1}/\widetilde{W}_{\kappa}^{3}(G,G_{0}) \to 0$. 
Combining the previous result with $C_{4} \to 0$ eventually yields that $(\Delta 
m_{1})^{3}/\widetilde{W}_{\kappa}^{3}(G,G_{0}) \to 0$. 
Hence, $1 = p_{1}(|\Delta v_{1}|^{2}+|\Delta 
m_{1}|^{3})/\widetilde{W}_{\kappa}^{3}(G,G_{0}) \to 0$, which is a contradiction. 

If $\Delta v_{1}=0$, then $C_{1}+\Delta 
\theta_{1}C_{2} \to 0$ 
implies that $(\Delta \theta_{1})^{2}/\widetilde{W}_{\kappa}^{3}(G,G_{0}) \to 0$. 
Combining this result with $C_{4} \to 0$, we achieve $(\Delta m_{1})^{3}/
\widetilde{W}_{\kappa}^{3}(G,G_{0}) \to 0$, which also leads to a contradiction.

If $\Delta m_{1}=0$, then $C_{4} \to 0$ leads to 
\begin{eqnarray}
\biggr[\dfrac{\Delta \theta_{1}\Delta v_{1}}{2(\sigma_{1}^{0})^{7}}+\dfrac{(\Delta \theta_{1})^{3}}{6(\sigma_{1}^{0})^{7}}\biggr]/\widetilde{W}_{\kappa}^{3}(G,G_{0}) \to 0. \label{eqn:system_limits_symmetry_emixtures_first}
\end{eqnarray}
The combination of the above result and $C_{3} \to 0$ implies that $\Delta v_{1}/\widetilde{W}_{\kappa}^{3}(G,G_{0}) \to 0$. Combine the former results with \eqref{eqn:system_limits_symmetry_emixtures_first}, we obtain $(\Delta \theta_{1})^{3}/\widetilde{W}_{\kappa}^{3}(G,G_{0}) \to 0$, which is also a contradiction. Overall, we obtain the conclusion of this step.
\paragraph{Step 1.2:} If $|\Delta v_{1}|^{2}$ is the maximum among $|
\Delta \theta_{1}|^{3}$, $|\Delta v_{1}|^{2}$, and $|\Delta m_{1}|^{3}$. Since $C_{5} \to 0$, it leads to $(\Delta v_{1})^{2}/\widetilde{W}_{\kappa}^{3}(G,G_{0}) \to 0$, which is a contradiction.
\paragraph{Step 1.3:} If $|\Delta \theta_{1}|^{3}$ is the maximum among $|
\Delta \theta_{1}|^{3}$, $|\Delta v_{1}|^{2}$, and $|\Delta m_{1}|^{3}$. 
Denote $(\Delta v_{1})^{2}/(\Delta \theta_{1})^{3} \to k_{1}$ and $\Delta 
m_{1}/\Delta \theta_{1} \to k_{2}$. Since $C_{5} \to 0$ leads to $(\Delta v_{1})^{2}/\widetilde{W}_{\kappa}^{3}(G,G_{0}) \to 0$, we obtain $k_{1}=0$. As $C_{2} \to 0$, we obtain 
\begin{eqnarray}
\left[-\Delta \theta_{1}/(\sigma_{1}^{0})^{3}+2\Delta m_{1}/\sqrt{2\pi}(\sigma_{1}^{0})^{2}\right]/ 
(|\Delta \theta_{1}|+|\Delta v_{1}|+|\Delta m_{1}|) \to 0. \nonumber
\end{eqnarray}
By diving both the numerator and denominator of this ratio by $\Delta \theta_{1}$, 
we quickly obtain the equation $1/(\sigma_{1}^{0})^{3}+2k_{2}/\sqrt{2\pi}(\sigma_{1}^{0})^{2} =0$, which yields the solution $k_{2}=-\sqrt{\pi}/\sqrt{2}\sigma_{1}^{0}$.
  
Applying the result $(\Delta v_{1})^{2}/\widetilde{W}_{\kappa}^{3}(G,G_{0}) \to 0$ to $C_{3} \to 0$ and $C_{4} \to 0$, we have $M_{1} \to 0$ and $M_{2} \to 0$ where the formations of $M_{1}, M_{2}$ are as follows:
\begin{eqnarray}
M_{1} & : = & \biggr(\dfrac{\Delta v_{1}}{2(\sigma_{1}^{0})^{5}}+\dfrac{(\Delta \theta_{1})^{2}}{2(\sigma_{1}^{0})^{5}}
+\dfrac{2(\Delta \theta_{1})(\Delta m_{1})}{\sqrt{2\pi}(\sigma_{1}^{0})^{4}}\biggr)/(|\Delta \theta_{1}|^{3}+|\Delta v_{1}|^{2}+|\Delta m_{1}|^{3}), \nonumber \\
M_{2} & : = &  \biggr(\dfrac{(\Delta \theta_{1})(\Delta v_{1})}{2(\sigma_{1}^{0})^{7}}
+\dfrac{(\Delta v_{1})(\Delta m_{1})}{\sqrt{2\pi}(\sigma_{1}^{0})^{6}}+\dfrac{(\Delta \theta_{1})^{3}}{6(\sigma_{1}^{0})^{7}}-\dfrac{(\Delta m_{1})^{3}}{3\sqrt{2\pi}(\sigma_{1}^{0})^{4}}  +  \nonumber \\
& + & \dfrac{(\Delta \theta_{1})^{2}(\Delta m_{1})}{\sqrt{2\pi}(\sigma_{1}^{0})^{6}}\biggr)/(|\Delta \theta_{1}|^{3}+|\Delta v_{1}|^{2}+|\Delta m_{1}|^{3}). \nonumber 
\end{eqnarray}
Now, $\left(\dfrac{\Delta \theta_{1}}{(\sigma_{1}^{0})^{2}}+
\dfrac{2\Delta m_{1}}{\sqrt{2\pi}\sigma_{1}^{0}}\right)M_{1} - M_{2}$ yields that 
\begin{eqnarray}
\left[\dfrac{(\Delta m_{1})^{3}}{3\sqrt{2\pi}}+\dfrac{2(\theta_{1})(\Delta m_{1})^{2}}{\pi \sigma_{1}^{0}}
+\dfrac{2(\Delta \theta_{1})^{2}(\Delta m_{1})}{\sqrt{2\pi}(\sigma_{1}^{0})^{2}}+\dfrac{(\Delta \theta_{1})^{3}}{3(\sigma_{1})^{3}}\right]/
(|\Delta \theta_{1}|^{3}+|\Delta v_{1}|^{2}+|\Delta m_{1}|^{3}) \to 0. \nonumber
\end{eqnarray}
By dividing both the numerator and denominator of this term by $(\Delta \theta_{1})^{3}$, 
we obtain the equation $\dfrac{k_{2}^{3}}{3\sqrt{2\pi}}+\dfrac{2k_{2}^{2}}{\pi 
\sigma_{1}^{0}}+\dfrac{2k_{2}}{\sqrt{2\pi}(\sigma_{1}^{0})^{2}}+
\dfrac{1}{3(\sigma_{1}^{0})^{3}}=0$. 
Since $k_{2}=-\dfrac{\sqrt{\pi}}{\sqrt{2}\sigma_{1}^{0}}$, this equation yields $\pi/6-1/3=0$, which is a contradiction. Therefore, this step cannot hold.
\paragraph{Step 1.4:} If $|\Delta m_{1}|^{3}$ is the maximum among $|
\Delta \theta_{1}|^{3}$, $|\Delta v_{1}|^{2}$, and $|\Delta m_{1}|^{3}$. 
The argument in this step is similar to that of Step 1.3. In fact, by denoting $\Delta \theta_{1}/\Delta m_{1} \to k_{3}$ and $(\Delta 
v_{1})^{2}/(\Delta m_{1})^{3} \to k_{4}$
then we also achieve $k_{4} = 0$ and $k_{3}=-\dfrac{\sqrt{2}}{\sqrt{\pi}\sigma_{1}^{0}}$ (by $C_{2} \to 0$). 
Now by using the limits $C_{3}, C_{4} \to 0$ as that of Step 1.3 and after some calculations, we obtain the 
equation $\dfrac{k_{3}^{3}}{3(\sigma_{1}^{0})^{3}}+\dfrac{2k_{3}^{2}}{\sqrt{2\pi}(\sigma_{1})^{2}}
+\dfrac{2k_{3}}{\pi \sigma_{1}^{0}}+\dfrac{1}{3\sqrt{2\pi}}=0$, 
which also does not admit $k_{3}=-\dfrac{\sqrt{2}}{\sqrt{\pi}\sigma_{1}^{0}}$ as a 
solution --- a contradiction. As a consequence, $G_{0}$ is not (3,2,3)-singular relative 
to $\Ecal_{k_{0}}$.

Since $G_{0}$ is 2-singular but not 3-singular relative to $\Ecal_{k_{0}}$, we obtain 
that $\lev(G_{0}|\mathcal{E}_{k_{0}})=2$. To demonstrate that (3,2,3) is the 
singularity index of $G_{0}$, we need to verify that $G_{0}$ is (3,1,3)-singular, (3,2,2)-
singular, and (2,2,3)-singular relative to $\Ecal_{k_{0}}$. In particular, combining these 
results with the fact that $G_{0}$ is (2,2,2)-singular relative to $\Ecal_{k_{0}}$, it 
implies that for any $\kappa' \prec (3,2,3)$, $G_{0}$ is $\kappa'$-singular relative to $
\Ecal_{k_{0}}$.
\paragraph{Claim: $G_0$ is (3,1,3)-singular} Here, we choose $k_{0}=1$ and denote $\kappa_{1}=(3,1,3)$. Similar to the argument for $G_{0}$ is not (3,2,3)-singular, the vanishing of all coefficients of the (3,1,3)-minimal form leads to
\begin{eqnarray}
\biggr(\dfrac{\Delta \theta_{1}}{(\sigma_{1}^{0})^{3}}+\dfrac{2\Delta m_{1}}{\sqrt{2\pi}(\sigma_{1}^{0})^{2}}-\dfrac{(\Delta \theta_{1})^{3}}{2(\sigma_{1}^{0})^{5}}-\dfrac{3(\Delta \theta_{1})^{2}\Delta m_{1}}{\sqrt{2\pi}(\sigma_{1}^{0})^{4}}\biggr)/\widetilde{W}_{\kappa_{1}}^{3}(G,G_{0}) \to 0, \nonumber \\
\biggr(\dfrac{\Delta v_{1}}{2(\sigma_{1}^{0})^{5}}+\dfrac{(\Delta \theta_{1})^{2}}{2(\sigma_{1}^{0})^{5}}+\dfrac{2\Delta \theta_{1}\Delta m_{1}}{\sqrt{2\pi}(\sigma_{1}^{0})^{4}}\biggr)/\widetilde{W}_{\kappa_{1}}^{3}(G,G_{0}) \to 0, \nonumber \\
\biggr(\dfrac{(\Delta \theta_{1})^{3}}{6(\sigma_{1}^{0})^{7}}-\dfrac{(\Delta m_{1})^{3}}{3\sqrt{2\pi}(\sigma_{1}^{0})^{4}}+\dfrac{(\Delta \theta_{1})^{2}\Delta m_{1}}{\sqrt{2\pi}(\sigma_{1}^{0})^{6}}\biggr)/\widetilde{W}_{\kappa_{1}}^{3}(G,G_{0}) \to 0 \nonumber
\end{eqnarray}
where $\widetilde{W}_{\kappa_{1}}^{3}(G,G_{0})=|\Delta \theta_{1}|
^{3}+|\Delta v_{1}|+|\Delta m_{1}|^{3}$. By choosing $\Delta \theta_{1}=-
\dfrac{2\sigma_{1}^{0}\Delta m_{1}}{\sqrt{2\pi}}$ and $\Delta v_{1} = (\Delta 
\theta_{1})^{2}$, then all the above limits satisfy as long as $\Delta \theta_{1} \to 0$. 
Therefore, $G_{0}$ is (3,1,3)-singular relative to $\Ecal_{k_{0}}$.
\paragraph{$G_{0}$ is (3,2,2)-singular} Here, we choose $k_{0}=1$ and denote $\kappa_{2}=(3,2,2)$. The vanishing of all coefficients of the (3,2,2)-minimal form leads to
\begin{eqnarray}
\biggr(\dfrac{\Delta \theta_{1}}{(\sigma_{1}^{0})^{3}}+\dfrac{2\Delta m_{1}}{\sqrt{2\pi}(\sigma_{1}^{0})^{2}}-\dfrac{3\Delta \theta_{1}\Delta v_{1}}{2(\sigma_{1}^{0})^{5}}-\dfrac{2\Delta v_{1}\Delta m_{1}}{\sqrt{2\pi}(\sigma_{1}^{0})^{4}}-\dfrac{(\Delta \theta_{1})^{3}}{2(\sigma_{1}^{0})^{5}}\biggr)/\widetilde{W}_{\kappa_{2}}^{3}(G,G_{0}) \to 0, \nonumber \\
\biggr(\dfrac{\Delta v_{1}}{2(\sigma_{1}^{0})^{5}}+\dfrac{(\Delta \theta_{1})^{2}}{2(\sigma_{1}^{0})^{5}}+\dfrac{2\Delta \theta_{1}\Delta m_{1}}{\sqrt{2\pi}(\sigma_{1}^{0})^{4}}\biggr)/\widetilde{W}_{\kappa_{2}}^{3}(G,G_{0}) \to 0, \nonumber \\
\biggr(\dfrac{\Delta \theta_{1}\Delta v_{1}}{2(\sigma_{1}^{0})^{7}}+\dfrac{\Delta v_{1}\Delta m_{1}}{\sqrt{2\pi}(\sigma_{1}^{0})^{6}}+\dfrac{(\Delta \theta_{1})^{3}}{6(\sigma_{1}^{0})^{7}}\biggr)/\widetilde{W}_{\kappa_{2}}^{3}(G,G_{0}) \to 0, \nonumber \\
(\Delta v_{1})^{2}/\widetilde{W}_{\kappa_{2}}^{3}(G,G_{0}) \to 0 \nonumber
\end{eqnarray}
where $\widetilde{W}_{\kappa_{2}}^{3}(G,G_{0})=|\Delta \theta_{1}|
^{3}+|\Delta v_{1}|^{2}+|\Delta m_{1}|^{2}$. We can verify that by choosing $\Delta 
\theta_{1}=-\dfrac{2\sigma_{1}^{0}\Delta m_{1}}{\sqrt{2\pi}}$ and $\Delta v_{1} = 
(\Delta \theta_{1})^{2}$, then all the above limits satisfy as long as $\Delta \theta_{1} 
\to 0$. Therefore, $G_{0}$ is (3,2,2)-singular relative to $\Ecal_{k_{0}}$.
\paragraph{$G_{0}$ is (2,2,3)-singular} Here, we choose $k_{0}=1$ and denote $
\kappa_{3}=(2,2,3)$. The vanishing of all coefficients of the (2,2,3)-minimal form leads to
\begin{eqnarray}
\biggr(\dfrac{\Delta \theta_{1}}{(\sigma_{1}^{0})^{3}}+\dfrac{2\Delta m_{1}}{\sqrt{2\pi}(\sigma_{1}^{0})^{2}}-\dfrac{3\Delta \theta_{1}\Delta v_{1}}{2(\sigma_{1}^{0})^{5}}-\dfrac{2\Delta v_{1}\Delta m_{1}}{\sqrt{2\pi}(\sigma_{1}^{0})^{4}}\biggr)/\widetilde{W}_{\kappa_{3}}^{3}(G,G_{0}) \to 0, \nonumber \\
\biggr(\dfrac{\Delta v_{1}}{2(\sigma_{1}^{0})^{5}}+\dfrac{(\Delta \theta_{1})^{2}}{2(\sigma_{1}^{0})^{5}}+\dfrac{2\Delta \theta_{1}\Delta m_{1}}{\sqrt{2\pi}(\sigma_{1}^{0})^{4}}\biggr)/\widetilde{W}_{\kappa_{3}}^{3}(G,G_{0}) \to 0, \nonumber \\
\biggr(\dfrac{\Delta \theta_{1}\Delta v_{1}}{2(\sigma_{1}^{0})^{7}}+\dfrac{\Delta v_{1}\Delta m_{1}}{\sqrt{2\pi}(\sigma_{1}^{0})^{6}}-\dfrac{(\Delta m_{1})^{3}}{3\sqrt{2\pi}(\sigma_{1}^{0})^{4}}\biggr)/\widetilde{W}_{\kappa_{3}}^{3}(G,G_{0}) \to 0, \nonumber \\
(\Delta v_{1})^{2}/\widetilde{W}_{\kappa_{3}}^{3}(G,G_{0}) \to 0 \nonumber
\end{eqnarray}
where $\widetilde{W}_{\kappa_{3}}^{3}(G,G_{0})=|\Delta \theta_{1}|
^{2}+|\Delta v_{1}|^{2}+|\Delta m_{1}|^{3}$. By using $\Delta \theta_{1}=-
\dfrac{2\sigma_{1}^{0}\Delta m_{1}}{\sqrt{2\pi}}$ and $\Delta v_{1} = (\Delta 
\theta_{1})^{2}$, we can verify that all the above limits satisfy as long as $\Delta 
\theta_{1} \to 0$. Therefore, $G_{0}$ is (2,2,3)-singular relative to $\Ecal_{k_{0}}$.

In sum, we have shown under the setting of $G_{0} \in \mathcal{S}
_{2}$, (3,2,3) is the singularity index of $G_{0}$ relative to $\Ecal_{k_{0}}$. To 
demonstrate further that $(3,2,3)$ is the unique singularity index of $G_{0}$, we need 
to show that $G_{0}$ is $\kappa$-singular relative to $\Ecal_{k_{0}}$ where $\kappa 
\in \left\{(r,r,2),(2,r,r),(r,1,r)\right\}$ for any $r \geq 1$. As $r \leq 3$, we have 
verified these results hold above with the choice of $k_{0}=1$, $\Delta \theta_{1}=-
\dfrac{2\sigma_{1}^{0}\Delta m_{1}}{\sqrt{2\pi}}$, and $\Delta v_{1} = (\Delta 
\theta_{1})^{2}$. We now argue that these choices of $k_{0}$ and $G$ are also 
sufficient to obtain previous results for any $r \geq 4$. In fact, with the previous choices 
of $k_{0}$ and $G$, it is clear that $(\Delta \theta_{1})^{\alpha_{1}}(\Delta 
v_{1})^{\alpha_{2}}(\Delta m_{1})^{\alpha_{3}}/\widetilde{W}_{\kappa}^{r}
(G,G_{0}) \to 0$ for any $\kappa \in \left\{(r,r,2),(2,r,r),(r,1,r)\right\}$ where $r \geq 
4$ and $|\alpha| \geq 4$. With these results, for any $\kappa$ from the previous set, if 
all coefficients of any $\kappa$-minimal form of $G$ vanish, they will eventually lead to 
one of three systems of limits (the denominator is changed to $\widetilde{W}
_{\kappa}^{r}(G,G_{0})$) that we used to demonstrate that $G_{0}$ is (3,1,3)-
singular, (3,2,2)-singular, and (2,2,3)-singular relative to $\Ecal_{k_{0}}$ above. These 
systems of limits with new denominator still hold with our choices of $k_{0}$ and $G$.

As a consequence, (3,2,3) is the unique singularity index of $G_{0}$ relative to $
\Ecal_{k_{0}}$, i.e., $\singset(G_{0}|\Ecal_{k_{0}})=\left\{(3,2,3)\right\}$. Therefore, 
we achieve the conclusion of the theorem under the setting of $G_{0} \in \Scal_2$.

\subsection{Singularity structure of $\Scal_3$}
\label{Section:singularity_level_S3_setting}

To develop intuition and obtain bounds for singularity structure of $G_0 \in \Scal_3$,
we start by considering a simple case 
similar to the exposition of subsection
\ref{Section:singularity_level_S1_setting} and subsection~\ref{Section:singularity_level_S2_setting}. That is, $G_0$ has only one homologous set 
of size $k_0$. $G_0\in \Scal_3$ means that $m_{i}^{0}$ do not share the same 
signs for all $i=1,\ldots, k_0$. 
\comment{
Thus, it is possible to construct a sequence
of $G$ tending to $G_0$ such that (cf. Eq. \eqref{eqn:singularitys1third})
\begin{eqnarray}
\sum \limits_{i=1}^{k_{0}}{p_{i}\Delta m_{i}/v_{i}^{0}} = 0,\;
\sum \limits_{i=1}^{k_{0}}{p_{i}m_{i}^{0}(\Delta m_{i})^{2}/(v_{i}^{0})^{2}} = 0. \nonumber
\end{eqnarray}
It follows that $G_0$ is 2-singular.
}
To investigate the singularity structure for $G_0$, we first obtain
an $\kappa$-minimal form, for any $\|\kappa\|_{\infty}=r$ such that $\kappa_{3}=r$ where $r \geq 2$, of $(p_{G}(x)-p_{G_{0}}(x))/\widetilde{W}_{\kappa}^{r}
(G,G_{0})$ by
\begin{eqnarray}
\dfrac{1}{\widetilde{W}_{\kappa}^{r}
(G,G_{0})}\biggr[\sum \limits_{i=1}^{k_{0}}{\biggr(\sum \limits_{j=1}
^{2r+1}{\beta_{ji}^{(r)}(x-\theta_{1}^{0})^{j-1}}\biggr)f\left(\dfrac{x-\theta_{1}^{0}}
{\sigma_{i}^{0}}\right)\Phi\left(\dfrac{m_{i}^{0}(x-\theta_{1}^{0})}{\sigma_{i}^{0}}\right)} \nonumber \\
+\biggr(\sum \limits_{j=1}^{2r}{\gamma_{j}^{(r)}(x-\theta_{1}^{0})^{j-1}}\biggr)\exp
\left(-\dfrac{(m_{1}^{0})^{2}+1}{2v_{1}^{0}}(x-\theta_{1}^{0})^{2}\right)\biggr]+o(1), \label{eqn:Taylorexpansionnonconformant}
\end{eqnarray}
where $\beta_{ji}^{(r)}, \gamma_{j}^{(r)}$ are polynomials of $\Delta \theta_{l}$, $\Delta 
v_{l}$, $\Delta m_{l}$, and $\Delta p_{l}$ as $1 \leq i,l \leq k_{0}$ and $1 \leq j \leq 2r
+1$. 
For concrete formulas of
$\beta_{ji}^{(r)}, \gamma_{j}^{(r)}$, we note that for
any $\alpha=(\alpha_{1},\alpha_{2},\alpha_{3})$ such that $|\alpha| \leq r$, 
there holds
\begin{eqnarray}
\dfrac{\partial^{|\alpha|}{f}}{\partial{\theta^{\alpha_{1}}}\partial{v^{\alpha_{2}}}
\partial{m^{\alpha_{3}}}}=\biggr(\sum \limits_{i=1}^{2r}{\dfrac{U_{i}^{\alpha_{1},\alpha_{2},
\alpha_{3}}(m)}{V_{i}^{\alpha_{1},\alpha_{2},\alpha_{3}}(v)}(x-\theta)^{i-1}\biggr)f
\left(\dfrac{x-\theta}{\sigma}\right)f\left(\dfrac{m(x-\theta)}{\sigma}\right)}+ \nonumber \\
\dfrac{1}{\sigma}\biggr(\sum \limits_{i=1}^{2r+1}{\dfrac{L_{i}^{\alpha_{1},\alpha_{2},
\alpha_{3}}}{N_{i}^{\alpha_{1},\alpha_{2},\alpha_{3}}(v)}(x-\theta)^{i-1}\biggr)f\left(\dfrac{x-
\theta}{\sigma}\right)\Phi\left(\dfrac{m(x-\theta)}{\sigma}\right)}.\nonumber
\end{eqnarray}
In the above display $U_{i}^{\alpha_{1},\alpha_{2},\alpha_{3}}(m), V_{i}^{\alpha_{1},\alpha_{2},
\alpha_{3}}(v), N_{i}^{\alpha_{1},\alpha_{2},\alpha_{3}}(v)$ are polynomials in terms of 
$m,v$ and $L_{i}^{\alpha_{1},\alpha_{2},\alpha_{3}}$ are some constant numbers. As $
\alpha_{3} \geq 1$, we can further check that $L_{i}^{\alpha_{1},\alpha_{2},\alpha_{3}}
=0$ for all $1 \leq i \leq 2r$ and $\alpha_{1},\alpha_{2}$ such that $|\alpha| \leq r$. 
It follows that
\begin{eqnarray}
\beta_{ji}^{(r)} & = & \dfrac{2\Delta p_{i}}{\sigma_{j}^{0}}1_{\left\{j=1\right\}}+\dfrac{1}{\sigma_{i}^{0}}\sum \limits_{|\alpha| \leq r}{\dfrac{L_{j}^{\alpha_{1},\alpha_{2},\alpha_{3}}}{N_{j}^{\alpha_{1},\alpha_{2},\alpha_{3}}(v_{i}^{0})}\dfrac{p_{i}(\Delta \theta_{i})^{\alpha_{1}}(\Delta v_{i})^{\alpha_{2}}(\Delta m_{i})^{\alpha_{3}}}{\alpha_{1}!\alpha_{2}!\alpha_{3}!}}, \nonumber \\
\gamma_{j}^{(r)} & = & \sum \limits_{i=1}^{k_{0}}{\sum \limits_{|\alpha| \leq r}{\dfrac{U_{j}^{\alpha_{1},\alpha_{2},\alpha_{3}}(m_{i}^{0})}{V_{j}^{\alpha_{1},\alpha_{2},\alpha_{3}}(v_{i}^{0})}\dfrac{p_{i}(\Delta \theta_{i})^{\alpha_{1}}(\Delta v_{i})^{\alpha_{2}}(\Delta m_{i})^{\alpha_{3}}}{\alpha_{1}!\alpha_{2}!\alpha_{3}!}}}, \nonumber
\end{eqnarray} 
where $1 \leq i \leq k_{0}$ and $1 \leq j \leq 2r+1$. Since $L_j^{\alpha_{1},\alpha_{2},\alpha_{3}}=0$ as $\alpha_{3} \geq 1$, we further obtain that
\begin{eqnarray}
\beta_{ji}^{(r)} & = & \dfrac{2\Delta p_{i}}{\sigma_{j}^{0}}1_{\left\{j=1\right\}}+\dfrac{1}{\sigma_{i}^{0}}\sum \limits_{\alpha_{1}+\alpha_{2} \leq r}{\dfrac{L_{j}^{\alpha_{1},\alpha_{2},0}}{N_{j}^{\alpha_{1},\alpha_{2},0}(v_{i}^{0})}\dfrac{p_{i}(\Delta \theta_{i})^{\alpha_{1}}(\Delta v_{i})^{\alpha_{2}}}{\alpha_{1}!\alpha_{2}!}}. \nonumber
\end{eqnarray}
Therefore, $\beta_{ji}^{(r)}$ are polynomials of $\Delta p_{i}, \Delta 
\theta_{i}, \Delta v_{i}$, while $\gamma_{j}^{(r)}$ are polynomials 
of $\Delta \theta_{i}, \Delta v_{i}, \Delta m_{i}$, for 
$1 \leq i \leq k_{0}$, $1 \leq j \leq 2r+1$.

Suppose that there is a sequence of $G$ tending to $G_0$ (in $\widetilde{W}_{\kappa}$ distance) such that
all coefficients of its $\kappa$-minimal form in \eqref{eqn:Taylorexpansionnonconformant} vanish.
It can be checked that $\beta_{ji}^{(r)}/\widetilde{W}_{\kappa}^{r}
(G,G_{0}) \to 0$ for all
even number $j \in [1,2r+1]$ entails that 
$\Delta \theta_{i}/\widetilde{W}_{\kappa}^{r}
(G,G_{0}) \to 0$ for all $1 \leq i \leq k_{0}$. 
Similarly, $\beta_{ji}^{(r)}/\widetilde{W}_{\kappa}^{r}
(G,G_{0}) \to 0$ for all
odd $j \in [3,2r+1]$ entails that 
$\Delta v_{i}/\widetilde{W}_{\kappa}^{r}
(G,G_{0}) \to 0$ for all $1 \leq i \leq 
k_{0}$. So, as $\beta_{1i}^{(r)}/\widetilde{W}_{\kappa}^{r}
(G,G_{0}) \to 0$, we obtain $\Delta 
p_{i}/\widetilde{W}_{\kappa}^{r}
(G,G_{0}) \to 0$. It follows that, as $\beta_{ji}^{(r)}/\widetilde{W}_{\kappa}^{r}
(G,G_{0}) \to 0$ for all $1 \leq j \leq 2r+1$, we must have
 $\Delta p_{i}/\widetilde{W}_{\kappa}^{r}
(G,G_{0}) \to 0$, $\Delta \theta_{i}/\widetilde{W}_{\kappa}^{r}
(G,G_{0}) \to 0$, and $\Delta v_{i}/\widetilde{W}_{\kappa}^{r}
(G,G_{0}) \to 0$ 
for all $1 \leq i \leq k_{0}$. Note that, these results hold for \emph{any choice} of $1 \leq \kappa_{1}, \kappa_{2} \leq r$. Additionally, they also imply that
\begin{eqnarray}
\dfrac{\sum \limits_{i=1}^{k_{0}}{|\Delta p_{i}|+p_{i}(|\Delta \theta_{i}|^{\kappa_{1}}+|\Delta v_{i}|^{\kappa_{2}})}}{\widetilde{W}_{\kappa}^{r}
(G,G_{0})} \to 0. \nonumber
\end{eqnarray}
If $\Delta m_{i} = 0$ for all $1 \leq i \leq k_{0}$, then by means of Lemma 
\ref{lemma:bound_overfit_Wasserstein_first}, $\sum \limits_{i=1}^{k_{0}}{|\Delta p_{i}|+p_{i}(|
\Delta \theta_{i}|^{\kappa_{1}}+|\Delta v_{i}|^{\kappa_{2}}) }= D_{\kappa}(G_{0},G) \asymp \widetilde{W}_{\kappa}^{r}
(G,G_{0})$, 
which contradicts with the above limit. Therefore, we have $\mathop {\max }\limits_{1 
\leq i \leq k_{0}}{|\Delta m_{i}|}>0$.

Turning to $\gamma_{l}^{(r)}$ and the fact that 
$\Delta p_{i}/\widetilde{W}_{\kappa}^{r}
(G,G_{0}) \to 0$, $\Delta \theta_{i}/\widetilde{W}_{\kappa}^{r}
(G,G_{0}) \to 0$, and $\Delta v_{i}/\widetilde{W}_{\kappa}^{r}
(G,G_{0})$\\$ \to 0$, if $\gamma_{l}^{(r)}/\widetilde{W}_{\kappa}^{r}
(G,G_{0}) \to 0$ as $1 \leq l \leq 2r$, we also have that
\begin{eqnarray}
\left(\sum \limits_{i=1}^{k_{0}}{\sum \limits_{\alpha_{3} \leq r}{\dfrac{U_{l}^{0,0,\alpha_{3}}(m_{i}^{0})}{V_{l}^{0,0,\alpha_{3}}(v_{i}^{0})}\dfrac{p_{i}(\Delta m_{i})^{\alpha_{3}}}{\alpha_{3}!}}}\right) /\widetilde{W}_{\kappa}^{r}
(G,G_{0})\to 0. \nonumber 
\end{eqnarray}
We can verify that as $1 \leq l \leq 2r$ is odd, $U_{j}^{0,0,\alpha_{3}}(m_{i}^{0})=0$ for all $\alpha_{3} \leq r$ and $1 \leq i \leq k_{0}$. Additionally, as $1 \leq l \leq 2r$ is even, the above system of limits becomes
\begin{eqnarray}
\biggr(\sum \limits_{i_{1}-i_{2}=l/2}{\dfrac{q_{i_{1},i_{2}}}{i_{1}!}\sum \limits_{i=1}
^{k_{0}}{\dfrac{p_{i}(m_{i}^{0})^{i_{1}-2i_{2}-1}(\Delta m_{i})^{i_{1}}}{(\sigma_{i}
^{0})^{l}}}}\biggr)/\widetilde{W}_{\kappa}^{r}
(G,G_{0}) \to 0, \label{eqn:singularitynonconformant_two}
\end{eqnarray}
where $1 \leq i_{1} \leq r$, $i_{2} \leq (i_{1}-1)/2$ as $i_{1}$ is odds or $i_{2} \leq 
i_{1}/2-1$ as $i_{1}$ is even. Here, $q_{i,j}$ are the integer coefficients 
that appear in the high order derivatives of $f(x|\theta,\sigma,m)$ with respect 
to $m$:
\begin{eqnarray}
\dfrac{\partial^{s+1}{f}}{\partial{m}^{s+1}}=\left[\sum \limits_{j=0}^{(s-1)/2}{\dfrac{q_{(s+1),j}m^{s-2j}}{\sigma^{2s+2-2j}}(x-\theta)^{2s-2j+1}}\right]f\left(\dfrac{x-\theta}{\sigma}\right)f\left(\dfrac{m(x-\theta)}{\sigma}\right) \nonumber
\end{eqnarray}
when $s$ is an odd number and
\begin{eqnarray}
\dfrac{\partial^{s+1}{f}}{\partial{m}^{s+1}}=\left[\sum \limits_{j=0}^{s/2}{\dfrac{q_{(s+1),j}m^{s-2j}}{\sigma^{2s+2-2j}}(x-\theta)^{2s-2j+1}}\right]f\left(\dfrac{x-\theta}{\sigma}\right)f\left(\dfrac{m(x-\theta)}{\sigma}\right) \nonumber
\end{eqnarray}
when $s$ is an even number. For instance, when $s=0$, we have $q_{1,0}=2$ and when 
$s=1$, we have $q_{2,0}=-2$.

Summarizing, under that simple setting of $G_{0}$ in order for all the coefficients in the $\kappa$-minimal form
\eqref{eqn:Taylorexpansionnonconformant} to
vanish, i.e., we have $\beta_{ji}^{(r)}/\widetilde{W}_{\kappa}^{r}
(G,G_{0}) \to 0$ and $\gamma_{l}^{(r)}/\widetilde{W}_{\kappa}^{r}
(G,G_{0}) \to 0$, only the third component of $\kappa$, i.e., $\kappa_{3}=r$, plays a 
key role while the first two components $\kappa_{1}$ and $\kappa_{2}$ of $\kappa$ 
can be of any values from 1 to $r$. Additionally, the value $\kappa_{3}=r$ is 
determined by the system 
of limits \eqref{eqn:singularitynonconformant_two}, i.e., that system is the important factor to determine
the singularity structure of $G_{0} \in \Scal_{3}$. Let $r_{\max}$ to be the maximum 
number $r$ such that system of limits \eqref{eqn:singularitynonconformant_two} holds. 
From the definition of singularity level, it is clear that $r_{\max} = \lev(G_{0}|
\Ecal_{k_{0}})$. These observations under this simple setting of $G_{0} \in \Scal_{3}$ 
shed light on the following important result regarding singularity index of any $G_{0} 
\in \Scal_{3}$ relative to $\Ecal_{k_{0}}$ whose rigorous proof is deferred to Appendix F.
\begin{theorem} \label{theorem:singularity_index_S3}
Suppose that $G_{0} \in \Scal_{3}$.
\begin{itemize}
\item[(a)] If $\lev(G_{0}|\Ecal_{k_{0}})<\infty$ and $P_1(\veceta^0) \neq 0$, i.e., there are no Gaussian components in $G_{0}$, then we have $\singset(G_{0}|\Ecal_{k_{0}})=\left\{(1,1,\lev(G_{0}|\Ecal_{k_{0}})+1)\right\}$.
\item[(b)] If $\lev(G_{0}|\Ecal_{k_{0}})<\infty$ and $P_1(\veceta^0) = 0$, i.e., there are some Gaussian components in $G_{0}$, then we have $\singset(G_{0}|\Ecal_{k_{0}})=\left\{(3,2,\max\left\{2,\lev(G_{0}|\Ecal_{k_{0}})\right\}+1)\right\}$.
\item[(c)] If $\lev(G_{0}|\Ecal_{k_{0}})=\infty$, then $\singset(G_{0}|\Ecal_{k_{0}})=(\infty,\infty,\infty)$.
\end{itemize}
\end{theorem} 
The above results imply that we only need to focus on studying the singularity level to
understand singularity structure of $G_{0} \in \Scal_{3}$, i.e., we can choose $
\kappa=(r,r,r)$. To illustrate the behaviors of singularity levels of $G_{0} \in \Scal_{3}$, 
we will continue exploring the structure of 
system of limits \eqref{eqn:singularitynonconformant_two} under the simple setting of $G_0 \in \Scal_3$ when it has only one homologous set of size $k_{0}$.

\subsubsection{Singularity structure of $G_{0} \in \Scal_{31}$} 
\label{Section:nonconformant_no_typeC}
\comment{
Under this setting, we will give the insight into the condition of C.1 singularity. Due to the 
assumption of $G_{0}$, the condition $P_{3}(\vecp,\veceta) \neq 0$ becomes
\begin{eqnarray}
\prod \limits_{S \subseteq \left\{1,\ldots,k_{0}\right\}, |S| \geq 2}{\biggr(\sum \limits_{i \in S}{p_{i}^{0}\prod \limits_{j \neq i}{m_{j}^{0}}}\biggr)} \neq 0. \label{eqn:type3singularitycondition}
\end{eqnarray}
}

Recall that $G_{0}$ has only one homologous set of size $k_{0}$. Let $\kappa=(r,r,r)$. From the above argument that, as we have $\beta_{ji}^{(r)}/\widetilde{W}_{\kappa}^{r}
(G,G_{0}) \to 0$ when $1 \leq i,l \leq k_{0}$ and $1 \leq 
j \leq 2r+1$, we obtain $\Delta p_{i}/\widetilde{W}_{\kappa}^{r}
(G,G_{0}) \to 0$, $\Delta \theta_{i}/\widetilde{W}_{\kappa}^{r}
(G,G_{0}) \to 0$, and $\Delta v_{i}/\widetilde{W}_{\kappa}^{r}
(G,G_{0}) \to 0$ for all $1 \leq i \leq k_{0}$. Combining with Lemma 
\ref{lemma:bound_overfit_Wasserstein_first}, it follows that
\begin{eqnarray}
\sum \limits_{i=1}^{k_{0}}{p_{i}|\Delta m_{i}|^{r}}/\widetilde{W}_{\kappa}^{r}
(G,G_{0}) \not \to 0. \label{eqn:singularitynonconformant_one}
\end{eqnarray}  
Since we have $\mathop {\max }\limits_{1 \leq i \leq k_{0}}{|\Delta m_{i}|}>0$, a combination of \eqref{eqn:singularitynonconformant_two} and 
\eqref{eqn:singularitynonconformant_one} leads to
\begin{eqnarray}
\biggr(\sum \limits_{i_{1}-i_{2}=l/2}{\dfrac{q_{i_{1},i_{2}}}{i_{1}!}\sum \limits_{i=1}
^{k_{0}}{\dfrac{p_{i}(m_{i}^{0})^{i_{1}-2i_{2}-1}(\Delta m_{i})^{i_{1}}}{(\sigma_{i}
^{0})^{l}}}}\biggr) \biggr /
\sum \limits_{i=1}^{k_{0}}{p_{i}|\Delta m_{i}|^{r}} \to 0, 
\label{eqn:singularitynonconformant_three}
\end{eqnarray}
for any even $l$ such that $1 \leq l \leq 2r$. 
Let $q_{i}=p_{i}/\sigma_{i}^{0}$, $t_{i}^{0}
=m_{i}^{0}/\sigma_{i}^{0}$, and $\Delta t_{i}=\Delta m_{i}/\sigma_{i}^{0}$ for all $1 
\leq i \leq k_{0}$, then the above limits can be rewritten as
\begin{eqnarray}
\biggr(\sum \limits_{i=1}^{k_{0}}
\sum \limits_{i_{1}-i_{2}=l/2}{\dfrac{q_{i_{1},i_{2}}}{i_{1}!}
{q_{i}(t_{i}^{0})^{i_{1}-2i_{2}-1}(\Delta t_{i})^{i_{1}}}}\biggr)/\sum 
\limits_{i=1}^{k_{0}}{q_{i}|\Delta t_{i}|^{r}} \to 0, 
\label{eqn:singularitynonconformant_four}
\end{eqnarray}
where in the summation of the above display, 
$1 \leq i_{1} \leq r$, $i_{2} \leq (i_{1}-1)/2$ as $i_{1}$ is odd, or $i_{2} \leq 
i_{1}/2-1$ as $i_{1}$ is even and $l$ is an even number ranging from 2 to $2r$. These are the 
limits of the ratio of two semipolynomial functions. The existence of these limits
will be shown to entail the existence of zeros of a system of polynomial
equations. 

\paragraph{Greedy extraction of limiting polynomials} 
As explained in the main text, it is generally difficult to obtain all polynomial limits of 
the system of rational semipolynomial functions given by \eqref{eqn:singularitynonconformant_four}. 
However, it is possible to obtain a subset of polynomial limits via
a greedy method of extraction. We shall demonstrate this technique
for the specific case $r=3$, and then present a general result, not
unlike what we have done in subsections \ref{Section:illustration_omixture_byone}
and \ref{Section:general_bound_omixtures_generic} for o-mixtures.
For $r=3$, we only have three possible choices of $l$ in 
\eqref{eqn:singularitynonconformant_four}, which are $l=2, 4$ and $6$. As $l=2$, we have 
$(i_{1}, i_{2})=(1,0)$. As $l=4$, we obtain $(i_{1},i_{2}) \in \left\{(2,0),(3,1)\right\}$. 
Finally, as $l=6$, we get $(i_{1},i_{2}) = (3,0)$. Here, we can compute that $q_{1,0}=2, 
q_{2,0}=-2, q_{3,1}=-2, q_{3,0}=2$. Therefore, as $r=3$, the system of limits 
\eqref{eqn:singularitynonconformant_four} becomes
\begin{eqnarray}
\biggr(\sum \limits_{i=1}^{k_{0}}{q_{i}\Delta t_{i}}\biggr)/\sum \limits_{i=1}^{k_{0}}{q_{i}|\Delta t_{i}|^{3}} \to 0,  \nonumber \\
\biggr(\sum \limits_{i=1}^{k_{0}}{q_{i}t_{i}^{0}(\Delta t_{i})^{2}+\dfrac{1}{3}q_{i}(\Delta t_{i})^{3}}\biggr)/\sum \limits_{i=1}^{k_{0}}{q_{i}|\Delta t_{i}|^{3}} \to 0, \nonumber \\
\biggr(\sum \limits_{i=1}^{k_{0}}{q_{i}(t_{i}^{0})^{2}(\Delta t_{i})^{3}}\biggr)/\sum \limits_{i=1}^{k_{0}}{q_{i}|\Delta t_{i}|^{3}} \to 0. \label{eqn:singularityconformant_specific}
\end{eqnarray}
Denote $|\Delta t_{k_0}| := \mathop {\max }\limits_{1 \leq i \leq k_{0}}
{\left\{|\Delta t_{i}|\right\}}$. 
In each of the limiting expressions
in the above display, we shall divide both the numerator and denominator
of the left hand side by $|\Delta t_{k_0}|^{\alpha}$, where $\alpha$
is the smallest degree that appears in one of the monomials in the
numerator. Since $|\Delta t_i|/|\Delta t_{k_0}|$ is bounded,
there exist a subsequence according to which $\Delta t_i /|\Delta t_{k_0}|$
tends to a constant, say $k_i$, for each $i=1,\ldots, k_0$. Note that
at least one of the $k_i$ is non-zero.
Moreover, we obtain the following equations in the limit
\begin{eqnarray}
\sum \limits_{i=1}^{k_{0}}{q_{i}^{0}k_{i}}=0,\;
\sum \limits_{i=1}^{k_{0}}{q_{i}^{0}t_{i}^{0}(k_{i})^{2}}=0,\;
\sum \limits_{i=1}^{k_{0}}{q_{i}^{0}(t_{i}^{0})^{2}(k_{i})^{3}}=0. \nonumber
\end{eqnarray}
Since $q_{i}^{0}=p_{i}^{0}/\sigma_{i}^{0}$, $t_{i}^{0}=m_{i}^{0}/\sigma_{i}^{0}$ for all $1 \leq i 
\leq k_{0}$, by rescaling $k_i$, 
the above system of polynomial equations can be rewritten as 
\begin{eqnarray}
\sum \limits_{i=1}^{k_{0}}{p_{i}^{0}k_{i}}=0, \;
\sum \limits_{i=1}^{k_{0}}{p_{i}^{0}m_{i}^{0}(k_{i})^{2}}=0, \;
\sum \limits_{i=1}^{k_{0}}{p_{i}^{0}(m_{i}^{0})^{2}(k_{i})^{3}}=0. \nonumber
\end{eqnarray}

Now we shall apply the greedy extraction technique to
the general system \eqref{eqn:singularitynonconformant_four}. This involves
dividing both the numerator and the denominator of the left hand side in each
equation of the system by $(\Delta t_{k_{0}})^{l/2}$ for any $2 \leq l \leq 2r$ 
and $l$ is even. This leads to the existence of solution for the following
system of polynomial equations
\begin{eqnarray}
\sum \limits_{i=1}^{k_{0}}{p_{i}^{0}(m_{i}^{0})^{l/2-1}k_{i}^{l/2}} = 0,  
\label{eqn:singularitynonconformant_equation}
\end{eqnarray}
where the index $l$ is even and $2 \leq l \leq 2r$. In this sytem,
at least one of $k_i$ is non-zero.

At this point, by a contrapositive argument we immediately
deduces that if system of polynomial equations \eqref{eqn:singularitynonconformant_equation} does \emph{not} have
a valid solution for the $k_i$, one of which must be non-zero, 
then $G_0$ is \emph{not} $\kappa$-singular relative to $\mathcal{E}_{k_{0}}$. It follows that 
$\lev(G_0|\Ecal_{k_0}) \leq r-1$ and thus the singularity index $(1,1,\lev(G_0|\Ecal_{k_0})+1)$ of $G_{0}$ relative to $\Ecal_{k_{0}}$ is bounded by $(1,1,r)$ according to Theorem \ref{theorem:singularity_index_S3}.
This connection motivates a deeper investigation into the behavior
of the system of real polynomial 
equations~\eqref{eqn:singularitynonconformant_equation}.

\paragraph{Behavior of system of limiting polynomial equations} 
We proceed to study the solvability of the system of 
polynomial equations like \eqref{eqn:singularitynonconformant_equation}. 
Consider two parameter sequences 
$\veca = \left\{a_{i}\right\}_{i=1}^{k_0}$, $\vecb = 
\left\{b_{i}\right\}_{i=1}^{k_0}$ such that 
$a_{i}>0, b_{i} \neq 0$ for all $1 \leq i \leq l$ and $b_{i}$ are pairwise 
different. Additionally, there exists two indices $1 \leq i_{1} \neq j_{1} \leq l$ 
such that $b_{i_{1}}b_{j_{1}}<0$. 
We can think of $a_{i}$ as taking the role of $p_{i}^{0}$ and 
$b_{i}$ the role of $m_{i}^{0}$. 

Define $\overline{s}(k_0,\veca, \vecb)$ 
to be the \emph{minimum} value of $s \geq 1$ such that 
the following system of polynomial equations
\begin{eqnarray}
\mathop {\sum }\limits_{i=1}^{k_0}{a_{i}b_{i}^{u}c_{i}^{u+1}}=0, \;
\textrm{for}\; u=0,1, \ldots,s
\label{eqn:noncannonicalexactfittedskewnorma}
\end{eqnarray}
does not admit any \emph{non-trivial} solution, by which we require that at 
least one of $c_{i}$ is non-zero.
For example, if $s=2$, and $k_0=2$, the above 
system of polynomial equations is
\begin{eqnarray}
a_{1}c_{1}+a_{2}c_{2}=0, \;
a_{1}b_{1}c_{1}^{2}+a_{2}b_{2}c_{2}^{2}=0, \;
a_{1}b_{1}^{2}c_{1}^{3}+a_{2}b_{2}^{2}c_{2}^{3}=0. \nonumber
\end{eqnarray}

In general, it is difficult to determine
the exact value of $\overline{s}(k_0,\veca, \vecb)$
since it depends on the specific values of parameter sequences
$\veca$ and $\vecb$. However, it is possible to obtain some nontrivial
bounds:

\begin{proposition}
\label{proposition:upperboundskewnormalfirst} Let $k_0 \geq 2$.
\begin{itemize}
\item[(a)] If for any subset I of $\left\{1,2,\ldots,k_0\right\}$ we have $\sum \limits_{i \in I}{a_{i}\prod \limits_{j \in I \setminus \{i\}}{b_{j}}} \neq 0$, 
then $\overline{s}(k_0,\veca, \vecb) \leq k_0-1$. 
\item[(b)] If there is a subset $I$ of $\left\{1,2,\ldots, k_0\right\}$ 
such that $\sum \limits_{i \in I}{a_{i}\prod \limits_{j \in I \setminus \{i\}}{b_{j}}} = 
0$, then $\overline{s}(k_0,\veca, \vecb)=\infty$. 
\item[(c)] Under the same condition as that of part (a):
\begin{itemize}
\item[] If $k_0=2$, then $\overline{s}(k_0,\veca,\vecb) = 1$.
\item[] If $k_0=3$, and $\sum \limits_{i=1}^{k_0}{a_{i}\prod \limits_{j \neq i, j 
\leq k_0}{b_{j}}}>0$, then $\overline{s}(k_0,\veca, \vecb) = 1$. 
Otherwise, $\overline{s}(k_0,\veca, \vecb) = 2$. 
\end{itemize}
\end{itemize}
\end{proposition}

\paragraph{Remarks}
(i) Applying part (a) of this proposition to system~\eqref{eqn:singularitynonconformant_equation}, since $G_{0} \in \mathcal{S}_{31}$, i.e $P_{3}(\vecp^0, \veceta^0)=\sum \limits_{i=1}^{k_{0}}{p_{i}^0\prod \limits_{j \neq i}{m_{j}^{0}}} \neq 0$, 
$G_0$ is \emph{not} $\overline{s}(k_0,\{p_i^0\}_{i=1}^{k_0},
\{m_i^0\}_{i=1}^{k_0})+1$-singular relative to $\mathcal{E}_{k_{0}}$. Therefore,
the singularity level of $G_0$ is at most $\overline{s}(k_0,\{p_i^0\}_{i=1}^{k_0},
\{m_i^0\}_{i=1}^{k_0})$ and the singularity index of $G_{0}$ is at most $(1,1,\overline{s}(k_0,\{p_i^0\}_{i=1}^{k_0},
\{m_i^0\}_{i=1}^{k_0})+1)$ according to Theorem \ref{theorem:singularity_index_S3}. 
(ii) Part (a) provides a mild condition of parameter sequences $\veca,\vecb$ under
which a nontrivial finite upper bound can be obtained. 
A closer investigation of the proof
establishes that this bound is tight, i.e., there exists $(\veca,\vecb)$ such
that $\overline{s}(k_0,\veca,\vecb) = k_0-1$ holds. This motivates
the definition of $\Scal_{31}$.
(iii) Part (b) suggests the possibility of infinite level of singularity as well as singularity index, 
even as $k_0$ is fixed. We will show that this happens when $G_{0} \in \Scal_{33}$.
(iv) Part (c) suggests that the singularity levels and singularity indices of $G_{0}$ may be different for
different values of $(\vecp^0,\veceta^0)$ for the same $k_0$.

\paragraph{General bounds for singularity level and singularity index of $G_0 \in \Scal_{31}$}
So far, we assume that $G_{0}$ has exactly one homologous set without C(1) singularity of 
size $k_{0}$. Now, we suppose that $G_0$ has more than one nonconformant homologous set without C(1) singularity of components,
and that there are no Gaussian components (i.e., $P_1(\veceta^0)
= \prod_{j=1}^{k_0} m_j^0\neq 0$).
It can be observed that the singularity level of $G_0$ can
be bounded in terms of a number of system of polynomial equations
of the same form as Eq.~\eqref{eqn:singularitynonconformant_equation},
which are applied to \emph{disjoint} subsets of noncomformant homologous 
components. The application to each subset yields a corresponding system of polynomial
limits like \eqref{eqn:singularitynonconformant_four}. If none of such systems admit non-trivial solutions, then 
we are absolutely certain that their corresponding systems of limiting equations
cannot hold. As a consequence, we obtain that
$\lev(G_0|\Ecal_{k_0}) \leq \overline{s}(G_0)$,
where
\begin{eqnarray}
\overline{s}(G_0):= \max_{I} 
\overline{s}(|I|, \{p_i^0\}_{i\in I},\{m_i^0\}_{i\in I}), \label{eqn:sufficient_order_nonconformant_emixture}
\end{eqnarray}
where the maximum is taken over all nonconformant homologous subsets $I$ of
components of $G_0$.

If, on the other hand, $G_0$ has one or more Gaussian components,
in addition to having some nonconformant homologous subsets, then
by combining the argument presented in Section \ref{Section:singularity_level_S2_setting} 
with the foregoing argument, we deduce that the singularity level of $G_0$
is at most $\max\{2, \overline{s}(G_0)\}$. Summarizing, combining with Theorem \ref{theorem:singularity_index_S3} we have the following theorem regarding the upper bounds of singularity levels and singularity indices of $G_{0} \in \Scal_{31}$ whose rigorous proof is deferred to Appendix F.


\begin{theorem}
\label{theorem:nonconformant_no_typeC_setting} 
Suppose that $G_0 \in \Scal_{31}$.
\begin{itemize}
\item[(a)] If $P_1(\veceta^0) \neq 0$, then 
$\lev(G_0| \Ecal_{k_0}) \leq \overline{s}(G_{0})
\leq k^{*}-1 \leq k_0-1$ and singularity index $(1,1,\lev(G_{0}|\Ecal_{k_{0}})+1) \preceq (1,1,\overline{s}(G_{0})+1)$.
\item[(b)] If $P_1(\veceta^0) = 0$, then 
$\lev(G_0| \Ecal_{k_0}) \leq \max\{2,\overline{s}(G_{0})\}
\leq \max\{2,k^{*}-1\} \leq \max\{2,k_0-1\}$ and singularity index $(3,2,\max\left\{2,\lev(G_{0}|\Ecal_{k_{0}})\right\}+1) \preceq (3,2,\max\{2,\overline{s}(G_{0})\}+1)$.
\end{itemize}
where $k^*$ is the 
maximum length among all nonconformant homologous sets without C(1) singularity of $G_0$.
\end{theorem}

\paragraph{Exact calculations in special cases} 
Since our proof method was to extract only an (incomplete) 
subset of polynomial limits, we could only speak of upper bounds
of the singularity level and singularity index, not lower bounds in general.
For some special cases of $G_0 \in \Scal_{31}$, with extra work
we can determine the exact singularity level and singularity index of $G_0$. This is 
based on the specific value of $k^*$, which is defined to be the 
maximum length among all nonconformant homologous sets without C(1) singularity of $G_0$ in Theorem \ref{theorem:nonconformant_no_typeC_setting}:
\begin{proposition} 
{\bf (Exact singularity structure)}\label{proposition:tight_singularity_noC.1}
Assume that $G_{0} \in \Scal_{31}$ and $P_1(\veceta^0) \neq 0$.
\begin{itemize}
\item[(a)] If $k^{*} = 2$, then $\lev(G_0|\Ecal_{k_0}) = 1$ and $\singset(G_{0}|\Ecal_{k_{0}})=\left\{(1,1,2)\right\}$.
\item[(b)] Let $k^{*} = 3$. In addition, if
all homologous sets $I$ of $G_{0}$ such that $|I| = k^*$ satisfy
$\sum \limits_{i \in I}{p_{i}^{0}\prod \limits_{j \in I\setminus \{i\}}
{m_{j}^{0}}}>0$, then $\lev(G_0|\Ecal_{k_0}) = 1$ and $\singset(G_{0}|\Ecal_{k_{0}})=\left\{(1,1,2)\right\}$.
Otherwise, $\lev(G_0|\Ecal_{k_0}) = 2$ and $\singset(G_{0}|\Ecal_{k_{0}})=\left\{(1,1,3)\right\}$.
\end{itemize}
\end{proposition}

\subsubsection{Singularity structure of $\Scal_{32}$} 
\label{Section:nonconformant_typeC}
For the simplicity of the argument in this section, we go back to the simple setting of $G_{0}$, 
i.e., $G_{0}$ has only one homologous set 
of size $k_{0}$. Since $G_0 \in \Scal_{32}$, we have $P_3(\vecp^0,\veceta^0) =\sum 
\limits_{i=1}^{k_{0}}{p_{i}^0\prod \limits_{j \neq i}{m_{j}^{0}}}= 0$. 
This entails that $\overline{s}(k_{0},
\left\{p_{i}^{0}\right\},\left\{m_{i}^{0}\right\})= \infty$ according 
to part (b) of Proposition \ref{proposition:upperboundskewnormalfirst}. 
As a result, $\overline{s}(G_0) = \infty$, i.e., the upper bound 
given by Theorem~\ref{theorem:nonconformant_no_typeC_setting}, that is,
$\lev(G_0|\Ecal_{k_0}) \leq \overline{s}(G_0)$, is no
longer meaningful for $\Scal_{32}$.
This does not necessarily imply that the singularity level and singularity index for $G_0\in \Scal_{32}$
is infinite. It simply means that the system of polynomial equations 
in \eqref{eqn:singularitynonconformant_equation} will not lead to any 
contradiction for any order $r$. In fact,
these equations described by~\eqref{eqn:singularitynonconformant_equation} are no longer 
sufficient to express the polynomial limits of the 
system~\eqref{eqn:singularitynonconformant_three}. The issue is that our
greedy extraction of polynomial limits for the 
system~\eqref{eqn:singularitynonconformant_three} treats each equation of the 
system separately. For instance, in system \eqref{eqn:singularityconformant_specific}, a 
special case of system~\eqref{eqn:singularitynonconformant_three} when $r=3$, we do not 
consider the interaction between two summations $\sum \limits_{i=1}^{k_{0}}{q_{i}t_{i}
^{0}(\Delta t_{i})^{2}}$ and $\sum \limits_{i=1}^{k_{0}}{\dfrac{1}{3}q_{i}(\Delta 
t_{i})^{3}}$ in the numerator of the second limit.
As a result, the limiting polynomials
obtained are dependent only on the lowest order monomial terms that
appear in the numerator of each of the $(r,r,r)$-minimal form's coefficients.

To go further with $\Scal_{32}$, we introduce a more sophisticated technique
for the polynomial limit extraction, 
which seeks to partially account for the interactions among 
different summations in the numerators of all the limits in system~\eqref{eqn:singularitynonconformant_three}.
This can be achieved by keeping not only the lowest order monomial in the 
numerator of the $(r,r,r)$-minimal form's coefficient, but also the second lowest order monomials.
As a result, we can extract a larger set of polynomial limits than
\eqref{eqn:singularitynonconformant_equation}. This would 
allow us to obtain
a tighter bound of the singularity level and singularity index for elements of $\Scal_{32}$.
Although our extraction technique is general, the system of limiting
polynomials that can be extracted is difficult to express explicitly for large values of $k_0$. For this reason
in the following we shall illustrate this technique of polynomial limit
extraction on a specific case of $k_{0}=2$.

\begin{proposition}
Assume that $G_0 \in \Scal_{32}$ and $G_{0}$ has only one homologous set of size $k_{0}$. Then as $k_0 = 2$, we have
$\lev(G_0|\Ecal_{k_0}) = 3$ and $\singset(G_{0}|\Ecal_{k_{0}})=\left\{(1,1,4)\right\}$.
\end{proposition}
\paragraph{Remark:} (i) The assumption that $G_{0}$ has only one homologous set is just for 
the convenience of the argument. The conclusion of this proposition still holds when $G_{0} 
\in \Scal_{32}$ has multiple homologous sets and the maximum length of homologous sets 
with C(1) singularity is 2. (ii) By using the same technique, we can demonstrate that $
\lev(G_0|\Ecal_{k_0}) = k_{0}+1$ and $\singset(G_{0}|\Ecal_{k_{0}})=\left\{(1,1,k_{0}+2)\right\}$ when $k_{0} \leq 5$ and $G_{0} \in \Scal_{32}$ has 
only one homologous set of size $k_{0}$. We conjecture that this result also holds for general $k_{0}$.
\begin{proof}
From Theorem \ref{theorem:singularity_index_S3}, it is sufficient to demonstrate that $\lev(G_0|\Ecal_{k_0}) = 3$. The proof proceeds in two main steps
\paragraph{Step 1:} We will demonstrate that $G_{0}$ is $3$-singular and $(3,3,3)$-singular relative to 
$\Ecal_{k_0}$. As $r=3$, the 
system~\eqref{eqn:singularitynonconformant_three} consists of the following
limiting equations, as $q_i \rightarrow q_i^0 > 0$ and $\Delta t_i \rightarrow 0$
for all $i=1,2$,
\begin{eqnarray}
\sum \limits_{i=1}^{2}{q_{i}\Delta t_{i}}/\sum \limits_{i=1}^{2}{q_{i}|\Delta t_{i}|^{3}} \to 0, \nonumber \\
\left(\sum \limits_{i=1}^{2}{q_{i}t_{i}^{0}(\Delta t_{i})^{2}}+\dfrac{1}{3}q_{i}(\Delta t_{i})^{3}\right)/\sum \limits_{i=1}^{2}{q_{i}|\Delta t_{i}|^{3}} \to 0, \nonumber \\
\left(\sum \limits_{i=1}^{2}{q_{i}(t_{i}^{0})^{2}(\Delta t_{i})^{3}} \right)/\sum \limits_{i=1}^{2}{q_{i}|\Delta t_{i}|^{3}} \to 0, \nonumber
\end{eqnarray}
where $q_{i}=p_{i}/\sigma_{i}^{0}, 
q_{i}^{0}=p_{i}^{0}/\sigma_{i}^{0}, t_{i}^{0}=m_{i}^{0}/\sigma_{i}^{0}$, 
and $\Delta t_{i}=\Delta m_{i}/\sigma_{i}^{0}$ for all $i=1,2$. 
The condition of C(1) singularity means $P_3(\vecp^0,\veceta^0) = 0$.
That is $p_{1}^{0}m_{2}^{0}+p_{2}^{0}m_{1}^{0}=0$. So,
$q_{1}^{0}t_{2}^{0}+q_{2}^{0}t_{1}^{0}=0$. 
By choosing $\Delta t_{2}=1/n$, $\Delta t_{1}=\dfrac{1}{n}\left(-\dfrac{q_{2}}{q_{1}}+\dfrac{1}{n^{4}}\right)$ 
where $q_{1}=q_{1}^{0}+1/n$ and $q_{2}=-q_{1}t_{2}^{0}/t_{1}^{0}+1/n^{2}$, we can check that all of the above limits are satisfied. 
Hence, $G_{0}$ is 3-singular and (3,3,3)-singular relative to $\mathcal{E}_{k_{0}}$.
\paragraph{Step 2:} It remains to show that $G_0$ is \emph{not} 4-singular and (4,4,4)-singular relative to 
$\Ecal_{k_0}$, and hence, $G_0$'s singularity level is 3. 
Let $r=4$, the 
system~\eqref{eqn:singularitynonconformant_three} consists of the following
limiting equations
\begin{eqnarray}
 \sum \limits_{i=1}^{2}{q_{i}^{n}\Delta t_{i}^{n}}/\sum \limits_{i=1}^{2}{q_{i}^{n}|\Delta t_{i}^{n}|^{4}} \to 0, \nonumber \\
 \left(\sum \limits_{i=1}^{2}{q_{i}t_{i}^{0}(\Delta t_{i})^{2}}+\dfrac{1}{3}q_{i}(\Delta t_{i})^{3}\right)/\sum \limits_{i=1}^{2}{q_{i}|\Delta t_{i}|^{4}} \to 0, \nonumber \\
\left(\sum \limits_{i=1}^{2}{\dfrac{1}{3}q_{i}(t_{i}^{0})^{2}(\Delta t_{i})^{3}+\dfrac{1}{4}q_{i}t_{i}^{0}(\Delta t_{i})^{4}}\right)/\sum \limits_{i=1}^{2}{q_{i}|\Delta t_{i}|^{4}} \to 0, \nonumber \\
\sum \limits_{i=1}^{2}{q_{i}(t_{i}^{0})^{3}(\Delta t_{i})^{4}}/\sum \limits_{i=1}^{2}{q_{i}|\Delta t_{i}|^{4}} \to 0. \nonumber
\end{eqnarray}
In order to account for the second-lowest order monomials of the numerator 
in each of the equations, we raise the order of the denominator in each equation
to the former. That is,
\begin{eqnarray}
K_{1} &:= & \sum \limits_{i=1}^{2}{q_{i}\Delta t_{i}}/\sum \limits_{i=1}^{2}{q_{i}|\Delta t_{i}|^{2}} \to 0, \nonumber \\
K_{2} & := & \left(\sum \limits_{i=1}^{2}{q_{i}t_{i}^{0}(\Delta t_{i})^{2}}+\dfrac{1}{3}q_{i}(\Delta t_{i}^{n})^{3}\right)/\sum \limits_{i=1}^{2}{q_{i}|\Delta t_{i}|^{3}} \to 0, \nonumber \\
K_{3} & := & \left(\sum \limits_{i=1}^{2}{\dfrac{1}{3}q_{i}(t_{i}^{0})^{2}(\Delta t_{i})^{3}+\dfrac{1}{4}q_{i}t_{i}^{0}(\Delta t_{i})^{4}}\right)/\sum \limits_{i=1}^{2}{q_{i}|\Delta t_{i}|^{4}} \to 0, \nonumber \\
K_{4} & := & \sum \limits_{i=1}^{2}{q_{i}(t_{i}^{0})^{3}(\Delta t_{i})^{4}}/\sum \limits_{i=1}^{2}{q_{i}|\Delta t_{i}|^{4}} \to 0. \nonumber
\end{eqnarray}
We assume without loss of generality that
$|\Delta t_{2}|$ is the maximum between $|\Delta t_{1}|$ and 
$|\Delta t_{2}|$.
Denote $\Delta t_{1}=k_{1}\Delta t_{2}$ where 
$k_{1} \in [-1,1]$ and $k_{1} \to k_{1}'$.
The vanishing of $K_{1}$ yields $q_{1}^{0}k_{1}'+q_{2}^{0}=0$. 
So, $k_{1}'=-q_{2}^{0}/q_{1}^{0}=t_{2}^{0}/t_{1}^{0}$. 
 
Divide both the numerator and denominator of $K_1$ by 
$(\Delta t_{2})^2$, we obtain $(q_1 k_1 + q_2)/\Delta t_2
\rightarrow 0$. Write $u= k_1 + q_2/q_1$, then 
$q_1 u/\Delta t_2 \rightarrow 0$, which implies that
$u/\Delta t_2 \rightarrow 0$.

Next, divide both the numerator and denominator of $K_{2}$ by
$(\Delta t_2)^3$, we obtain
\begin{eqnarray}
\left(\sum \limits_{i=1}^{2}{q_{i}t_{i}^{0}(\Delta t_{i})^{2}}+\dfrac{1}{3}q_{i}(\Delta t_{i})^{3}\right)/(\Delta t_{2})^{3} \to 0. \nonumber
\end{eqnarray}
%
Plug in the formula of $k_{1}$ and the fact that $u/\Delta t_2 \to 0$,
it follows that
 \begin{eqnarray}
 \left(q_{1}t_{1}^{0}\left(\dfrac{q_{2}}{q_{1}}\right)^{2}+q_{2}t_{2}^{0}\right)/(\Delta t_{2}) \to -\dfrac{1}{3}(q_{1}^{0}(k_{1}')^{3}+q_{2}^{0}). \nonumber
 \end{eqnarray}
Thus, we get
$P_{1} := (t_{1}^{0}q_{2}+t_{2}^{0}q_{1})/\Delta t_{2} \to -\dfrac{q_{1}^{0}}{3q_{2}^{0}}(q_{1}^{0}(k_{1}')^{3}+q_{2}^{0})$. It is simple to verify that
this limit is non-zero, otherwise we would 
have $q_1^0 = q_2^0$, which violates the definition that $G_{0}$ does not have C(2) singularity, i.e., $G_{0} \in \Scal_{32}$.

Continuing, divide both the numerator and denominator of $K_{3}$ 
by $(\Delta t_2)^4$, and with the same argument, we obtain
$P_{2} := (t_{1}^{0}q_{2}-t_{2}^{0}q_{1})(t_{1}^{0}q_{2}+t_{2}^{0}q_{1})/\Delta t_{2} \to -\dfrac{3(q_{1}^{0})^{2}}{4q_{2}^{0}}(q_{1}^{0}t_{1}^{0}(k_{1}')^{4}+q_{2}^{0}t_{2}^{0})$.

By dividing $P_2$ by $P_1$ and let it to vanish, we 
can extract the following polynomial in the limit:
\[4(q_1^0(k_1')^3+q_2^0)(t_1^0q_2^0 - t_2^0q_1^0) = 
9q_1^0(q_1^0t_1^0(k_1')^4+q_2^0t_2^0).\]
By plugging in $k_{1}'=-q_{2}^{0}/q_{1}^{0}$ and 
$t_{1}^{0}q_{2}^{0}+t_{2}^{0}q_{1}^{0}=0$, we can deduce
that $q_1^0=q_2^0$, which is a contradiction.
Thus, we conclude that $G_0$ is not 4-singular and (4,4,4)-singular relative to
$\Ecal_{k_0}$. 
\end{proof}

\subsubsection{Singularity structure of $G_{0} \in \Scal_{33}$} \label{Section:nonconformant_typeC_typeIV}

As we can see from the proof of Proposition \ref{proposition:tight_singularity_noC.1}, the 
condition of without C(2) singularity plays a major role in guaranteeing that $G_{0} \in  
\Scal_{32}$ is not (4,4,4)-singular relative to $\mathcal{E}_{k_{0}}$ when $G_{0}$ has only one 
homologous set of $k_{0}=2$. Therefore, for elements $G_{0}$ in $\Scal_{33}$, we expect 
the singularity level and singularity index of $G_{0}$ may be very large. In fact, we can show that

\begin{theorem}
\label{theorem:nonconformant_typeC_typeIV_setting} 
If $G_{0} \in \Scal_{33}$, then 
$\lev(G_{0}|\Ecal_{k_0}) = \infty$ and $\singset(G_{0}|\Ecal_{k_{0}})=\left\{(\infty,\infty,\infty)\right\}$.
\end{theorem}

\begin{proof} Here, we present the proof for $k_0=2$.
For general values of $k_0$, the proof is similar and deferred to a complete proof in Section~\ref{section:remaining-emixtures}. For $k_0=2$, 
the condition that $G_0 \in \Scal_{33}$ entails $P_{4}(\vecp^{0},\veceta^0)=0$, i.e
$p_{1}^{0}/\sigma_{1}^{0}=p_{2}^{0}/\sigma_{2}^{0}$ and 
$m_{1}^{0}/\sigma_{1}^{0}=-m_{2}^{0}/\sigma_{2}^{0}$. By 
choosing $\Delta m_{1}/\sigma_{1}^{0}=-\Delta m_{2}/\sigma_{2}^{0}$, $p_{1}=p_{2}
=p_{1}^{0}=p_{2}^{0}$, we can check that
\begin{eqnarray}
\sum \limits_{i=1}^{2}{\dfrac{p_{i}(m_{i}^{0})^{u}(\Delta m_{j})^{v}}{(\sigma_{i}
^{0})^{u+v+1}}}=0, \nonumber
\end{eqnarray}
for all odd numbers $u \in [1,v]$ when $v$ is even number, or 
for all even numbers $u \in [0,v]$ when $v$ is odd number.
 
Take order $r\geq 1$ to be an arbitrary natural number and let $\kappa=(r,r,r)$. 
Incorporating the identity in the previous display into
\eqref{eqn:Taylorexpansionnonconformant} and 
\eqref{eqn:singularitynonconformant_two}, we obtain the vanishing of 
all $\gamma_{l}^{(r)}/\widetilde{W}_{\kappa}^{r}(G_{1},G)$ for all $1 \leq l \leq 2r$ and 
$l$ is even. If we choose $\Delta 
\theta_{i}= \Delta v_{i} = 0$ for all $1 \leq i \leq 2$, we also have the coefficients 
$\beta_{ji}^{(r)}/\widetilde{W}_{\kappa}^{r}(G_{1},G)=0$ for all $1 \leq i \leq 2$ 
and $1 \leq j \leq 2r+1$. 
Additionally, we also have 
$\gamma_{l}^{(r)}/\widetilde{W}_{\kappa}^{r}(G_{1},G)=0$ for all $1 \leq l \leq 2r
$ and $l$ is odd. Hence, $G_0$ is $r$-singular and $(r,r,r)$-singular relative to $\mathcal{E}_{k_{0}}$ for any $r\geq 1$. As a consequence, $\lev(G_{0}|\Ecal_{k_0}) = \infty$. Combining with the result of Theorem \ref{theorem:singularity_index_S3}, we achieve the conclusion of the theorem.
\end{proof}

\section{Remaining proofs of technical results on skew-normal e-mixtures}
\label{section:remaining-emixtures}

\subsection{Proofs of statements in Appendix E}

\paragraph{FULL PROOF OF THEOREM \ref{theorem:conformant_setting}}

Here, we shall complete the proof of Theorem 
\ref{theorem:conformant_setting}, which is the generalization of the argument in Section 
\ref{Section:singularity_level_S1_setting} for a special case of $G_0$. 
Note that, the idea of this generalization is also used to the other settings of 
$G_{0} \not \in \mathcal{S}_{1}$. Now, we consider the possible existence of 
generic components in $G_{0}$, i.e., there are no homologous sets or symmetry components. 
Let $u_{1}=1 < u_{2} < \ldots < u_{\overline{i}_{1}} \in [1,k_{0}+1]$ 
such that $(\dfrac{v_{j}^{0}}{1+(m_{j}^{0})^{2}},\theta_{j}^{0})= (\dfrac{v_{l}^{0}}{1+(m_{l}^{0})^{2}},\theta_{l}^{0})$ 
and $m_{j}^{0}m_{l}^{0}>0$ for all $u_{i} \leq j,l \leq u_{i+1}-1$, $1 \leq i \leq \overline{i}
_{1}-1$. The constraint $m_{j}^{0}m_{l}^{0}>0$ is due to the conformant property of the 
homologous sets of $G_{0}$.
By definition, we have 
$|I_{u_{i}}|=u_{i+1}-u_{i}$ for all $1 \leq i \leq \overline{i}_{1}-1$ where $I_{u_{i}}$ denotes 
the set of all components homologous to component $u_{i}$. 

To show that $G_{0}$ is 1-singular and (1,1,1)-singular, we construct a sequence of $G \in \mathcal{E}
_{k_{0}}$ such that $(p_{i},\theta_{i},v_{i},m_{i})=(p_{i}^{0},\theta_{i}^{0},v_{i}
^{0},m_{i}^{0})$ for all $u_{2} \leq i \leq k_{0}$, i.e., all the components of $G$ and $G_{0}
$ are identical from index $u_{2}$ up to $k_{0}$. Hence, in the construction of the 
components from index $u_{1}$ to $u_{2}-1$ of $G$ we consider only the 
homologous set $I_{u_{1}}$ of $G_{0}$. Utilizing the argument from the special case proof of 
Theorem \ref{theorem:conformant_setting} in Section \ref{Section:conformant_setting}, 
the construction of the sequence of $G$ is specified by
$\Delta \theta_{i}=\Delta v_{i}=\Delta p_{i}=0$ and $\sum 
\limits_{i=u_{1}}^{u_{2}-1}{p_{i}\Delta m_{i}}/v_{i}^{0}=0$. Thus $G_0$ is 1-singular and (1,1,1)-singular.
It remains to demonstrate that $G_{0} \in \mathcal{S}_{1}$ 
is not (1,1,2)-singular relative to $\mathcal{E}_{k_{0}}$.
 
Indeed, let $\kappa=(1,1,2)$ and consider any sequence $G \in \mathcal{E}_{k_{0}} \to G_{0}$ under $\widetilde{W}_{\kappa}$ distance. Since $\widetilde{W}_{\kappa}^{2}(G,G_{0}) \asymp D_{\kappa}(G_{0},G)$ (cf. 
Lemma \ref{lemma:bound_overfit_Wasserstein_first}), we have the $\kappa$-minimal form for the sequence $G$ as 
\begin{eqnarray}
\dfrac{p_{G}(x)-p_{G_{0}}(x)}{\widetilde{W}_{\kappa}^{2}(G,G_{0})} \asymp \dfrac{A_{1}(x)+A_{2}(x)}{D_{\kappa}(G_{0},G)}, \nonumber
\end{eqnarray} 
where $A_{1}(x)/D_{\kappa}(G_{0},G)$ and $A_{2}(x)/D_{\kappa}(G_{0},G)$ are linear combinations 
of the elements of the forms $\dfrac{\partial^{|\alpha|}{f}}{\theta^{\alpha_{1}}v^{\alpha_{2}}m^{\alpha_{3}}}(x|\eta_{i}^{0})$ for any $1 \leq i \leq k_{0}$ and $0 \leq |\alpha| \leq 2$. 
In $A_{1}(x)/D_{\kappa}(G_{0},G)$, 
the indices of the components range from $1$ to $s_{\overline{i}_{1}}-1$.
In $A_{2}(x)/D_{\kappa}(G_{0},G)$, the indices of the components 
range from $u_{\overline{i}_{1}}$ to $k_{0}$.
It is convenient to think of the term $A_{1}(x)/D_{\kappa}(G_{0},G)$ as the linear 
combination of homologous components, and $A_{2}(x)/D_{\kappa}(G_{0},G)$ as the 
linear combination of generic components, i.e., no Gaussian nor homologous components. 

Regarding $A_{2}(x)/D_{\kappa}(G_{0},G)$, since we have the system of partial differential 
equations in \eqref{eqn:overfittedskewnormaldistributionzero}, the collection
of functions in $\biggr\{\dfrac{\partial^{|\alpha|}{f}}{\partial{\theta}^{\alpha_{1}}v^{\alpha_{2}}m^{\alpha_{3}}}
(x|\eta_{i}^{0}): \ |\alpha| \leq 2, \ 1 \leq i \leq k_{0} \biggr\}$ 
are not linearly independent. 
Employing the same strategy described in Section
\ref{Section:overfitskew}, we obtain a reduced system
of linearly independent partial derivatives in Lemma
\ref{lemma:reduced_linearly_independent}. This is the set 
$\biggr\{\dfrac{\partial^{|\alpha|}{f}}
{\partial{\theta}^{\alpha_{1}}v^{\alpha_{2}}m^{\alpha_{3}}}(x|\eta_{i}^{0}): \ \alpha \in 
\mathcal{F}_{2}, \ 1 \leq i \leq k_{0} \biggr\}$. Let $\lambda_{\alpha_{1}\alpha_{2}
\alpha_{3}}^{(2)}(\eta_{i}^{0})/D_{\kappa}(G_{0},G)$ be the coefficients of the terms
$\dfrac{\partial^{|\alpha|}{f}}{\theta^{\alpha_{1}}v^{\alpha_{2}}m^{\alpha_{3}}}(x|\eta_{i}^{0})$ 
for any $s_{\overline{i}_{1}} \leq i \leq k_{0}$ and $\alpha \in \mathcal{F}_{2}$. 
The formulae for $\lambda_{\alpha_{1},\alpha_{2},\alpha_{3}}^{(2)}$ will be given later in Case 2.

Regarding $A_{1}(x)/D_{\kappa}(G_{0},G)$, 
by exploiting the fact that 
$(\dfrac{v_{j}^{0}}{1+(m_{j}^{0})^{2}},\theta_{j}^{0})= (\dfrac{v_{l}^{0}}{1+(m_{l}^{0})^{2}},\theta_{l}^{0})$ 
for all $u_{i} \leq j,l \leq u_{i+1}-1$, $1 \leq i \leq \overline{i}_{1}-1$, 
the term $A_{1}(x)/D_{\kappa}(G_{0},G)$ can be written as
\begin{eqnarray}
\dfrac{A_{1}(x)}{D_{\kappa}(G_{0},G)}= \dfrac{1}{D_{\kappa}(G_{0},G)}\biggr(\sum_{l=1}^{\overline{i}_{1}-1} \biggr \{\mathop {\sum }\limits_{i=u_{l}}^{u_{l+1}-1}{\biggr[\sum \limits_{j=1}^{5}{\beta_{jil}^{(2)}(x-\theta_{u_{l}}^{0})^{j-1}}\biggr]}f\biggr(\dfrac{x-\theta_{u_{l}}^{0}}{\sigma_{i}^{0}}\biggr)\Phi\biggr(\dfrac{m_{i}^{0}(x-\theta_{u_{l}}^{0})}{\sigma_{i}^{0}}\biggr)\biggr\}+ \nonumber \\
\biggr[\sum \limits_{j=1}^{4}{\gamma_{jl}^{(2)}(x-\theta_{u_{l}}^{0}) ^{j-1}}\biggr]
\exp\biggr(-\dfrac{(m_{u_{l}}^{0})^{2}+1}{2v_{u_{l}}^{0}}(x-\theta_{u_{l}}^{0})^{2}\biggr)\biggr), \nonumber
\end{eqnarray}
where $f(x)=\dfrac{1}{\sqrt{2\pi}}\exp(-\dfrac{x^{2}}{2})$. 
(This form is a general version of Eq.
\eqref{eqn:taylorexpansionsecondorder} in Section \eqref{Section:conformant_setting} when $\overline{i}_{1}=2,u_{1}=1,u_{2}=k_{0}+1$). The detailed formulas of $\beta_{jil}^{(2)}$ and $
\gamma_{jl}^{(2)}$ for $1 \leq l \leq \overline{i}_{1}-1, u_{l} \leq i \leq u_{l+1}-1$, and $1 \leq j \leq 5$ are thus similar to that of \eqref{eqn:taylorexpansionsecondorder}. 
Here, we rewrite their general fomulations for the transparency of subsequent arguments:

\begin{eqnarray}
\beta_{1il}^{(2)} & = & \dfrac{2\Delta p_{i}}{\sigma_{i}^{0}}-\dfrac{p_{i}\Delta v_{i}}{(\sigma_{i}^{0})^{3}}-\dfrac{p_{i}(\Delta \theta_{i})^{2}}{(\sigma_{i}^{0})^{3}} + \dfrac{3p_{i}(\Delta v_{i})^{2}}{4(\sigma_{i}^{0})^{5}}, \beta_{2il}^{(2)} = \dfrac{2p_{i}\Delta \theta_{i}}{(\sigma_{i}^{0})^{3}}-\dfrac{6p_{i}\Delta \theta_{i}\Delta v_{i}}{(\sigma_{i}^{0})^{5}}, \nonumber \\
\beta_{3il}^{(2)} & = & \dfrac{p_{i}\Delta v_{i}}{(\sigma_{i}^{0})^{5}}+ \dfrac{p_{i}(\Delta \theta_{i})^{2}}{(\sigma_{i}^{0})^{5}}-\dfrac{3p_{i}(\Delta v_{i})^{2}}{2(\sigma_{i}^{0})^{7}}, \ \beta_{4il}^{(2)} = \dfrac{2p_{i}\Delta \theta_{i}\Delta v_{i}}{(\sigma_{i}^{0})^{7}}, \ \beta_{5il}^{(2)} = \dfrac{p_{i}(\Delta v_{i})^{2}}{4(\sigma_{i}^{0})^{9}}, \nonumber \\
\gamma_{1l}^{(2)} &=& \mathop {\sum }\limits_{j=u_{l}}^{u_{l+1}-1}{-\dfrac{p_{j}m_{j}^{0}\Delta \theta_{j}}{\pi(\sigma_{j}^{0})^{2}}
+\dfrac{2p_{j}m_{j}^{0}\Delta \theta_{j}\Delta v_{j}}{\pi(\sigma_{j}^{0})^{4}}}- \dfrac{2p_{j}\Delta \theta_{j}\Delta m_{j}}{\pi(\sigma_{j}^{0})^{2}}, \nonumber \\
\gamma_{2l}^{(2)} &=& \mathop {\sum }\limits_{j=s_{l}}^{s_{l+1}-1}{ -\dfrac{p_{j}m_{j}^{0}\Delta v_{j}}{2\pi(\sigma_{j}^{0})^{4}}
-\dfrac{p_{j}((m_{j}^{0})^{3}+2m_{j}^{0})(\Delta \theta_{j})^{2}}{2\pi(\sigma_{j}^{0})^{4}}} + \dfrac{p_{j}\Delta m_{j}}{\pi(\sigma_{j}^{0})^{2}} \nonumber \\
&+& \dfrac{5p_{j}m_{j}^{0}(\Delta v_{j})^{2}}{8\pi(\sigma_{i}^{0})^{6}}-\dfrac{p_{j}\Delta m_{j}\Delta v_{j}}{\pi(\sigma_{j}^{0})^{4}}, \nonumber \\
\gamma_{3l}^{(2)} &=&\mathop {\sum }\limits_{j=s_{l}}^{s_{l+1}-1}{\dfrac{p_{j}(2(m_{j}^{0})^{2}+2)\Delta m_{j}\Delta \theta_{j}}{\pi(\sigma_{j}^{0})^{4}}} - \dfrac{p_{j}((m_{j}^{0})^{3}+2m_{j}^{0})\Delta \theta_{j}\Delta v_{j}}{2\pi(\sigma_{j}^{0})^{6}}, \nonumber \\
\gamma_{4l}^{(2)} &=& \mathop {\sum }\limits_{j=s_{l}}^{s_{l+1}-1}{-\dfrac{p_{j}((m_{j}^{0})^{3}+2m_{j}^{0})(\Delta v_{j})^{2}}{8\pi(\sigma_{j}^{0})^{8}}-\dfrac{p_{j}m_{j}^{0}(\Delta m_{j})^{2}}{2\pi(\sigma_{j}^{0})^{4}}} \nonumber \\
&+& \dfrac{p_{j}((m_{j}^{0})^{2}+1)\Delta m_{j}\Delta v_{j}}{\pi(\sigma_{j}^{0})^{6}}, \nonumber
\end{eqnarray}
where $1 \leq l \leq \overline{i}_{1}-1$ and $u_{l} \leq i \leq u_{l+1}-1$. Now, suppose 
that all the coefficients of $A_{1}(x)/D_{\kappa}(G_{0},G)$ and $A_{2}(x)/D_{\kappa}(G_{0},G)$ go 
to 0. It implies that $\gamma_{jl}^{(2)}/D_{\kappa}(G_{0},G)$ ($1 \leq j \leq 4$, $1 \leq l \leq \overline{i}_{1}-1$), $\beta_{jil}^{(2)}/D_{\kappa}(G_{0},G)$
($1 \leq j \leq 5$, $u_{l} \leq i \leq u_{l+1}-1$, $1 \leq l \leq \overline{i}_{1}-1$), and $\lambda_{\alpha_{1}\alpha_{2}\alpha_{3}}^{(2)}(\eta_{i}^{0})/D_{\kappa}(G_{0},G)$
(for all $|\alpha| \leq 2$) go to $0$. From the formation of $D_{\kappa}(G_{0},G)$, we can find 
at least one index $1 \leq i^{*} \leq k_{0}$ such that $\biggr(|\Delta p_{i^{*}}|+p_{i^{*}}(|\Delta \theta_{i^{*}}|
+|\Delta v_{i^{*}}|+|\Delta m_{i^{*}}|^{2})\biggr)/D_{\kappa}(G_{0},G) \not \to 0$. Let 
\begin{eqnarray}
\tau(p_{i^{*}},\theta_{i^{*}},v_{i^{*}},m_{i^{*}})=|\Delta p_{i^{*}}|+p_{i^{*}}(|\Delta \theta_{i^{*}}|
+|\Delta v_{i^{*}}|+|\Delta m_{i^{*}}|^{2}). \nonumber
\end{eqnarray}
Now, there are two possible cases for $i^{*}$:
\paragraph{Case 1} $u_{1} \leq i^{*} \leq u_{\overline{i}_{1}}-1$. Without loss of generality, we assume that $u_{1} \leq i^{*} \leq u_{2}-1$. Denote 
\begin{eqnarray}
d(p_{i^{*}},\theta_{i^{*}},v_{i^{*}},m_{i^{*}})=\mathop {\sum }\limits_{j=u_{1}}^{u_{2}-1}{|\Delta p_{j}|+p_{j}(|\Delta \theta_{j}|+|\Delta v_{j}|+|\Delta m_{j}|^{2})}. \nonumber
\end{eqnarray}
Since $\tau(p_{i^{*}},\theta_{i^{*}},v_{i^{*}},m_{i^{*}})/D_{\kappa}(G_{0},G) \not \to 0$, we have
\begin{eqnarray}
d(p_{i^{*}},\theta_{i^{*}},v_{i^{*}},m_{i^{*}})/D_{\kappa}(G_{0},G) \not \to 0. \nonumber
\end{eqnarray}
Therefore, for $1 \leq j \leq 5$ and $u_{1} \leq i \leq u_{2}-1$, $D_{j}:=\dfrac{\alpha_{ji1}}{d(p_{i^{*}},\theta_{i^{*}},v_{i^{*}},m_{i^{*}})} \to 0$. Now, our argument for this case is organized further into two steps:
\paragraph{Step 1.1} From the vanishes of $D_{2}$ and $D_{3}$, we obtain 
\begin{eqnarray}
p_{i}\Delta \theta_{i}/d(p_{i^{*}},\theta_{i^{*}},v_{i^{*}},m_{i^{*}}) \to 0, \ \text{and} \ p_{i}\Delta v_{i}/d(p_{i^{*}},\theta_{i^{*}},v_{i^{*}},m_{i^{*}}) \to 0 \nonumber
\end{eqnarray}
for all $u_{1} \leq i \leq u_{2}-1$. Combining this result with $D_{1} \to 0$, we achieve for all $u_{1} \leq i \leq u_{2}-1$ that
\begin{eqnarray}
\Delta p_{i}/d(p_{i^{*}},\theta_{i^{*}},v_{i^{*}},m_{i^{*}}) \to 0.\nonumber
\end{eqnarray}
These results eventually show that 
\begin{eqnarray}
U : = \biggr(\mathop {\sum }\limits_{j=u_{1}}^{u_{2}-1}{p_{j}(\Delta m_{j})^{2}}\biggr)/d(p_{i^{*}},\theta_{i^{*}},v_{i^{*}},m_{i^{*}}) \not \to 0. \nonumber
\end{eqnarray}
\paragraph{Step 1.2} Since $p_{i}\Delta \theta_{i}/d(p_{i^{*}},\theta_{i^{*}},v_{i^{*}},m_{i^{*}}) \to 0$ and  $p_{i}\Delta v_{i}/d(p_{i^{*}},\theta_{i^{*}},v_{i^{*}},m_{i^{*}}) \to 0$, by using the result that $\gamma_{41}^{(2)}/d(p_{i^{*}},\theta_{i^{*}},v_{i^{*}},m_{i^{*}}) \to 0$, we have 
\begin{eqnarray}
V := \left[\mathop {\sum }\limits_{j=u_{1}}^{u_{2}-1}{\dfrac{p_{j}m_{j}^{0}(\Delta m_{j})^{2}}{(\sigma_{j}^{0})^{4}}}\right]/d(p_{i^{*}},\theta_{i^{*}},v_{i^{*}},m_{i^{*}}) \to 0. \nonumber
\end{eqnarray}
As $U \not \to 0$, we obtain
\begin{eqnarray}
V/U=\left[\mathop {\sum }\limits_{j=u_{1}}^{u_{2}-1}{\dfrac{p_{j}^{n}m_{j}^{0}(\Delta m_{j})^{2}}{(\sigma_{j}^{0})^{4}}}\right]/\mathop {\sum }\limits_{j=u_{1}}^{u_{2}-1}{p_{j}(\Delta m_{j})^{2}} \to 0. \label{eqn:exactfittedskewnormalgeneralkey}
\end{eqnarray}
Since $m_{i}^{0}m_{j}^{0}>0 $ for all $u_{1} \leq i,j \leq u_{2}-1$, 
without loss of generality we assume that $m_{j}^{0}>0$ for all $s_{1} \leq j \leq s_{2}-1$. However, it implies that
\begin{eqnarray}
\left[\mathop {\sum }\limits_{j=u_{1}}^{u_{2}-1}{\dfrac{p_{j}m_{j}^{0}(\Delta m_{j})^{2}}{(\sigma_{j}^{0})^{4}}}\right]/
\mathop {\sum }\limits_{j=u_{1}}^{u_{2}-1}{p_{j}(\Delta m_{j})^{2}} \geq m_{\min} \mathop {\sum }\limits_{j=u_{1}}^{u_{2}-1}{p_{j}(\Delta m_{j})^{2}}/
\mathop {\sum }\limits_{j=u_{1}}^{u_{2}-1}{p_{j}(\Delta m_{j})^{2}} , \label{eqn:exactfittedskewnormalgeneral}
\end{eqnarray}
where $m_{\min} := \mathop {\min }\limits_{u_{1} \leq j \leq u_{2}-1}{\left\{\dfrac{m_{j}^{0}}{(\sigma_{j}^{0})^{4}}\right\}}$. 
Combining with \eqref{eqn:exactfittedskewnormalgeneralkey}, $m_{\min}=0$ --- a contradiction. In sum, Case 1 cannot happen.

\paragraph{Case 2} $u_{\overline{i}_{1}} \leq i^{*} \leq k_{0}$. We can write down the formation of $A_{2}(x)/D_{\kappa}(G_{0},G)$ as follows
\begin{eqnarray}
\dfrac{A_{2}(x)}{D_{\kappa}(G_{0},G)} = \dfrac{1}{D_{\kappa}(G_{0},G)}\biggr(\sum \limits_{i=u_{\overline{i}_{1}}}^{k_{0}}{\sum \limits_{\alpha \in \mathcal{F}_2}{\lambda_{\alpha_{1},\alpha_{2},\alpha_{3}}^{(2)}(\eta_{i}^{0})\dfrac{\partial^{|\alpha|}{f}}{\partial{\theta^{\alpha_{1}}}\partial{v^{\alpha_{2}}}\partial{m^{\alpha_{3}}}}(x|\eta_i^0)}}\biggr), \nonumber
\end{eqnarray}
where $\lambda_{\alpha_{1},\alpha_{2},\alpha_{3}}^{(2)}(\eta_{i}^{0})$ are given by
\begin{eqnarray}
\lambda_{0,0,0}^{(2)}(\eta_{i}^{0})=\Delta p_{i}, \ \lambda_{1,0,0}^{(2)}(\eta_{i}^{0})=p_{i}\Delta \theta_{i}, \ \lambda_{0,1,0}^{(2)}(\eta_{i}^{0})= p_{i}\Delta v_{i}+
p_{i}(\Delta \theta_{i})^{2}, \nonumber \\
\lambda_{0,0,1}^{(2)}(\eta_{i}^{0})=-\dfrac{(m_{1}^{0})^{3}+m_{1}^{0}}{2v_{1}^{0}}
p_{i}(\Delta \theta_{1i})^{2}-\dfrac{1}{v_{1}^{0}}p_{i}\Delta v_{i}\Delta m_{i}+ p_{i}
\Delta m_{i}, \nonumber \\
\lambda_{0,2,0}^{(2)}(\eta_{i}^{0})= p_{i}(\Delta v_{i})^{2}, \;
\lambda_{0,0,2}^{(2)}(\eta_{i}^{0})=-\dfrac{(m_{1}^{0})^{2}+1}{2v_{1}^{0}m_{1}^{0}}
p_{i}\Delta v_{i}\Delta m_{i} +
p_{i}\Delta (m_{i})^{2}, \nonumber \\
\lambda_{1,1,0}^{(2)}(\eta_{i}^{0})= p_{i}\Delta \theta_{i}\Delta v_{i}, \;
\lambda_{1,0,1}^{(2)}(\eta_{i}^{0})= p_{i}\Delta \theta_{i}\Delta m_{i}. \nonumber
\end{eqnarray} 
From the assumption with the coefficients of $A_{2}(x)/D_{\kappa}(G_{0},G)$, we have $
\lambda_{\alpha_{1}, \alpha_{2}, \alpha_{3}}^{(2)}(\eta_{i}^{0})/D_{\kappa}
(G_{0},G) \to 0$ for any $u_{\overline{i}_{1}} \leq i \leq k_{0}$. From the hypothesis with $i^{*}$, we have $
\tau(p_{i^{*}},\theta_{i^{*}},v_{i^{*}},m_{i^{*}})/D_{\kappa}(G_{0},G) \not \to 
0$. Therefore, it leads to $\lambda_{\alpha_{1}, \alpha_{2}, \alpha_{3}}^{(2)}(\eta_{i}
^{0})/\tau(p_{i^{*}},\theta_{i^{*}},v_{i^{*}},m_{i^{*}})$ for any 
$u_{\overline{i}_{1}} \leq i \leq k_{0}$ and $\alpha \in \mathcal{F}_{2}$.

Now, since $\lambda_{1,0,0}^{(2)}(\eta_{i^{*}}^{0})/\tau(p_{i^{*}},
\theta_{i^{*}},v_{i^{*}},m_{i^{*}}) \to 0$, we obtain $\Delta \theta_{i^{*}}/
\tau(p_{i^{*}},\theta_{i^{*}},v_{i^{*}},m_{i^{*}}) \to 0$. Combining this 
result with $\lambda_{0,1,0}^{(2)}(\eta_{i^{*}}^{0})/\tau(p_{i^{*}},
\theta_{i^{*}},v_{i^{*}},m_{i^{*}}) \to 0$, we have $\Delta v_{i^{*}}/
\tau(p_{i^{*}},\theta_{i^{*}},v_{i^{*}},m_{i^{*}}) \to 0$. Furthermore, as $
\lambda_{0,0,1}^{(2)}(\eta_{i^{*}}^{0})/\tau(p_{i^{*}},
\theta_{i^{*}},v_{i^{*}},m_{i^{*}}) \to 0$, we get $\Delta m_{i^{*}}/ 
\tau(p_{i^{*}},\theta_{i^{*}},v_{i^{*}},m_{i^{*}}) \to 0$. Hence, since $
\lambda_{0,0,0}^{(2)}(\eta_{i^{*}}^{0})/\tau(p_{i^{*}},
\theta_{i^{*}},v_{i^{*}},m_{i^{*}}) \to 0$, we ultimately obtain
\begin{eqnarray}
1 = \dfrac{|\Delta p_{i^{*}}|+p_{i^{*}}(|\Delta \theta_{i^{*}}|
+|\Delta v_{i^{*}}|+|\Delta m_{i^{*}}|^{2})}{\tau(p_{i^{*}},\theta_{i^{*}},v_{i^{*}},m_{i^{*}})} \to 0, \nonumber
\end{eqnarray}
which is a contradiction. As a consequence, Case 2 cannot happen.

Summarizing,  not all the coefficients $\gamma_{jl}^{(2)}/D_{\kappa}(G_{0},G)$ ($1 \leq j \leq 4$, $1 \leq l \leq \overline{i}_{1}-1$), $\beta_{jil}^{(2)}/D_{\kappa}(G_{0},G)$
($1 \leq j \leq 5$, $u_{l} \leq i \leq u_{l+1}-1$, $1 \leq l \leq \overline{i}_{1}-1$), $\lambda_{\alpha_{1}\alpha_{2}\alpha_{3}}^{(2)}(\eta_{i}^{0})/D_{\kappa}(G_{0},G)$
(for all $\alpha \in \mathcal{F}_{2}$) go to $0$. From Definition \ref{def-rsingular_set}, 
$G_{0}$ is not (1,1,2)-singular relative to $\mathcal{E}_{k_{0}}$. 
This concludes our proof.
\paragraph{FULL PROOF OF THEOREM \ref{theorem:conformant_symmetry_setting}}
Here, we address the general setting of 
$G_{0} \in \mathcal{S}_{2}$, which accounts for the possible presence
of both generic components and conformant homologous sets. The proof idea is the generalization of the argument in Section 
\ref{Section:singularity_level_S2_setting} for a special case of $G_0$. Without loss of generality, 
we assume that $m_{1}^{0},m_{2}^{0},\ldots,m_{\overline{i}_{2}}=0$ where $1 \leq \overline{i}_{2} \leq k_{0}$ 
denotes the largest index $i$ such that $m_{i}^{0}=0$. The remaining components are 
either conformant homologous sets or generic components. Using the exact same constructions 
as that of Section \ref{Section:singularity_level_S2_setting}, we establish easily that $G_{0}$ is $(r,1,r)$-singular, $(2,r,r)$-singular, and $(r,r,2)$-singular relative to $
\mathcal{E}_{k_{0}}$ for any $r \geq 1$. It remains to
show that $G_{0}$ is not (3,2,3)-singular relative to $\mathcal{E}_{k_{0}}$.

Let $\kappa=(3,2,3)$. Consider the $\kappa$-minimal form for any sequence $G \in \mathcal{E}_{k_{0}} \to G_{0}$ 
under $\widetilde{W}_{\kappa}$ distance. Since 
$\widetilde{W}_{\kappa}^{3}(G,G_{0}) \asymp D_{\kappa}(G_{0},G)$ (cf. Lemma \ref{lemma:bound_overfit_Wasserstein_first}), 
we have
\begin{eqnarray}
\dfrac{p_{G}(x)-p_{G_{0}}(x)}{W_{\kappa}^{3}(G,G_{0})} \asymp \dfrac{A_{1}'(x)+A_{2}'(x)}{D_{\kappa}(G_{0},G)}, \nonumber
\end{eqnarray}  
where $A_{1}'(x)/
D_{\kappa}(G_{0},G)$ is the linear combination of Gaussian components, i.e., the indices of 
components range from 1 to $\overline{i}_{2}$, 
while $A_{2}'(x)/D_{\kappa}(G_{0},G)$ is the linear combination of conformant 
homologous components and generic components. 

Suppose that all the coefficients of 
$A_{1}'(x)/D_{\kappa}(G_{0},G), A_{2}'(x)/D_{\kappa}(G_{0},G)$ go to $0$. 
Similar to the argument in the proof of Theorem \ref{theorem:conformant_setting}, observe that there
is some index $\underline{i} \in [1,k_{0}]$ such that 
$(|\Delta p_{\underline{i}}|+p_{\underline{i}}
(|\Delta \theta_{\underline{i}}|^{3}
+|\Delta v_{\underline{i}}|^{2}+|\Delta m_{\underline{i}}|^{3}))/D_{\kappa}(G_{0},G) \not \to 0$. 
There are two possible cases regarding $\underline{i}$ .
\paragraph{Case 2.1} $\underline{i} \in [1,\overline{i}_{2}]$. Applying a
similar argument as that from Section \ref{Section:singularity_level_S2_setting} where we have only 
Gaussian components, we conclude that not all of the coefficients of 
$A_{1}'(x)/D_{\kappa}(G_{0},G)$ vanish, which is a contradiction. Therefore, Case 2.1 cannot happen.

\paragraph{Case 2.2} $\underline{i} \in [\overline{i}_{2}+1,k_{0}]$. Define
\begin{eqnarray}
D_{\kappa',new}(G_{0},G)=\sum \limits_{i=\overline{i}_{2}+1}^{k_{0}}{(|\Delta p_{i}|+p_{i}(|\Delta \theta_{i}|^{\kappa_{1}'}
+|\Delta v_{i}|^{\kappa_{2}'}+|\Delta m_{i}|^{\kappa_{3}'}))}, \nonumber
\end{eqnarray}
for any $\kappa' \in \left\{(1,1,2),(3,2,3)\right\}$. The idea of $D_{\kappa',new}(G_{0},G)$ is that we truncate the value of $
D_{\kappa'}(G_{0},G)$ from the index $1$ to $\overline{i}_{2}$, i.e., all the indices 
correspond to Gaussian components. 

Let $\kappa_{1}=(1,1,2)$. Since $\kappa=(3,2,3) \succ \kappa_{1}$, it is clear that $D_{\kappa,new}(G_{0},G) \lesssim D_{\kappa_{1},new}(G_{0},G)$. Since $D_{\kappa,new}
(G_{0},G)/D_{\kappa}(G_{0},G) \not \to 0$, we have $D_{\kappa_{1},new}(G_{0},G)/D_{\kappa}(G_{0},G) 
\not \to 0$. By multiplying all the coefficients of $A_{2}'(x)/D_{\kappa}(G_{0},G)$ with 
$D_{\kappa}(G_{0},G)/D_{\kappa_{1},new}(G_{0},G)$, we eventually obtain all the coefficients of $A_{2}'(x)/
D_{\kappa_{1},new}(G_{0},G)$ go to 0. However, by utilizing the same argument as in the proof of 
Theorem \ref{theorem:conformant_setting}, we reach to the conclusion that the second order Taylor 
expansion is sufficient to have all the coefficients of $A_{2}'(x)/D_{\kappa_{1},new}(G_{0},G)
$ not vanish. Thus, not all the coefficients of $A_{2}'(x)/D_{\kappa}(G_{0},G)$ 
go to 0, which is a contradiction. As a consequence, Case 2.2 also cannot happen.

In sum, under no circumstance can all 
the coefficients of $A_{1}'(x)/D_{\kappa}(G_{0},G)$ and 
$A_{2}'(x)/D_{\kappa}(G_{0},G)$ be made to vanish. 
Hence, $G_{0} \in \mathcal{S}_{2}$ is not (3,2,3)-singular relative to $\mathcal{E}_{k_{0}}$. As a consequence, $(3,2,3)$ is the unique singularity index of $G_{0}$ relative to $\Ecal_{k_{0}}$, which concludes the proof.

\subsection{Proofs of statements in Appendix F}

\paragraph{PROOF OF PROPOSITION \ref{proposition:upperboundskewnormalfirst}}
(a) The proof proceeds by induction on $l$.
When $l=1$, the conclusion clearly holds. 
Assume that that conclusion of the proposition holds for $l-1$. 
We will demonstrate that it also holds for $l$. 
Denote $y_{i}=a_{i}c_{i}$ and $z_{i}=b_{i}c_{i}$ for all $1 \leq i \leq l+1$. 
Then, we can rewrite system of polynomial equations \eqref{eqn:noncannonicalexactfittedskewnorma} as follows: $\sum \limits_{i=1}^{l+1}{z_{i}^{u}y_{i}}=0$ for any $0 \leq u \leq l$. 
If there exists some $1 \leq i_{1} \leq l+1$ such that $c_{i_{1}}=0$, then we go back to the case $l-1$, 
which we have already known from the hypothesis that we do not have non-trivial solution. 
Therefore, we assume that $c_{i} \neq 0$ for all $1 \leq i \leq l+1$, which implies that $y_{i} \neq 0$ for all $1 \leq i \leq l+1$. 
Now, the system of equations has the form of Vardermonde matrix, which is
$\begin{bmatrix}
1 & 1 & \ldots & 1 \\
z_{1} & z_{2} & \ldots & z_{l+1} \\
\vdots & \vdots & \ddots & \vdots \\
z_{1}^{s} & z_{2}^{s} & \ldots & z_{l+1}^{s} 
\end{bmatrix} $. By suitable linear transformations, we can rewrite the original 
system of equations as the following equivalent equations 
$\prod \limits_{j \neq i}{(z_{j}-z_{i})}y_{i}=0$ for all $1 \leq i \leq l+1$. 
Since $y_{i} \neq 0$ for all $1 \leq i \leq l+1$, we obtain $\prod \limits_{j \neq i}{(z_{j}-z_{i})}=0$ for all $1 \leq i \leq l+1$. 
As a consequence, there exists a partition $J_{1},J_{2},\ldots,J_{s}$  of $\left\{1,2,\ldots,l+1\right\}$ for some $1 \leq s \leq [l/2]$ 
such that if $i_{2},i_{3} \in J_{u}$ for $1 \leq u \leq s$, we have $z_{i_{2}}=z_{i_{3}}$ and for any $1 \leq i \neq j \leq s$, 
any two elements $z_{i_{4}} \in J_{i}$, $z_{j_{4}} \in J_{j}$ are different. 
Choose any $j_{i} \in J_{i}$ for all $1 \leq i \leq s$. It is clear that the system of equations 
can be rewritten as $\sum \limits_{i=1}^{s}{z_{j_{i}}^{u}\sum \limits_{j \in J_{i}}{y_{j}}}=0$ for all $0 \leq u \leq l+1$. 
If $s \geq 2$, it indicates that $|J_{i}| \leq l$ for all $1 \leq i \leq s$. 
Now, if we have some $1 \leq i_{4} \leq s$ such that $\sum \limits_{j \in J_{i_{4}}}{y_{j}}=0$ 
then we obtain $\sum \limits_{j \in J_{i_{4}}}{a_{j}c_{j}}=0$. Since $z_{i_{1}}=z_{i_{2}}$ for any $i_{1},i_{2} \in J_{i_{4}}$, 
this equation can be rewritten as $\sum \limits_{j \in J_{i_{4}}}{a_{j} \prod \limits_{v \neq j}{b_{v}}}=0$, 
which is a contradiction to the assumption of part (a) of the proposition. 
Therefore, $\sum \limits_{j \in J_{i}}{y_{j}} \neq 0$ for all $1 \leq i \leq s$. 
However, by using the same argument as before, again by linear transformation, 
we can rewrite the new system of polynomial equations as $\sum \limits_{j \in J_{i}}{y_{j}}\prod \limits_{v \neq i}{(z_{j_{u}}-z_{j_{i}})}=0$ for all $1 \leq i \leq s$. 
This implies that there should be some $1 \leq u_{1} \neq u_{2} \leq s$ such that $z_{j_{u_{1}}}=z_{j_{u_{2}}}$, which is a contradiction. 

As a consequence, we have $s=1$, i.e., $|I_{1}|=l+1$. Hence, $b_{1}c_{1}=b_{2}c_{2}=\ldots=b_{l+1}c_{l+1}$. 
Combining this fact with the equation $\sum \limits_{i=1}^{l+1}{a_{i}c_{i}}=0$, we obtain $\sum \limits_{i=1}^{l+1}{a_{i}\prod \limits_{j \neq i}{b_{j}}}=0$, 
which is a contradiction to the assumption of the proposition. This concludes the proof.

(b) We choose $c_{i}=0$ for all $i \not \in I \subset \left\{1,\ldots,l \right\}$. 
The system of polynomial equations \eqref{eqn:noncannonicalexactfittedskewnorma} 
becomes $\sum \limits_{i \in s}{a_{i}b_{i}^{u}c_{i}^{u+1}}=0$ for all $u \geq 0$. 
Notice that by choosing $b_{i}c_{i}=b_{j}c_{j}$ for all $i,j \in I$, 
we have $\sum \limits_{i \in I}{a_{i}b_{i}^{u}c_{i}^{u+1}}=b_{j}c_{j}\sum \limits_{i \in I}{a_{i}c_{i}}= 0$ for some $j \in I$ and for all $u \geq 1$ 
as long as $\sum \limits_{i \in I}{a_{i}c_{i}}= 0$. 
Combining all the conditions, we obtain $\sum \limits_{i \in J}{a_{i}\prod \limits_{j \neq i}{b_{j}}}=0$, which completes the proof.

(c) The result for the case $l=1$ is obvious. For the case $l=2$, 
after replacing $c_{3}$ in terms of $c_{1},c_{2}$, 
we obtain the following quadratic equation $(a_{1}a_{3}b_{1}+a_{1}^{2}b_{3})c_{1}^{2}+2a_{1}a_{2}b_{3}c_{1}c_{2}+(a_{2}a_{3}b_{2}+a_{2}^{2}b_{3})c_{2}^{2}=0$. 
Note that, $c_{1},c_{2} \neq 0$ due to the assumption of part (c). 
Therefore, we does not have solution of this quadratic equation 
when $a_{1}^{2}a_{2}^{2}b_{3}^{2}<(a_{1}a_{3}b_{1}+a_{1}^{2}b_{1})(a_{2}a_{3}b_{2}+a_{2}^{2}b_{3})$.  It is equivalent to $\sum \limits_{i=1}^{3}{a_{i}\prod \limits_{j \neq i}{b_{j}}}>0$, which 
confirms our hypothesis. We are done.

\paragraph{FULL PROOF OF THEOREM \ref{theorem:singularity_index_S3}}
Here, we only provide the proof for part (b) as the proofs for part (a) and part (c) are similar. This is a 
generalization of the argument in Section \ref{Section:singularity_level_S3_setting}. Under 
this situation, apart from the nonconformant homologous sets, we also 
have in $G_0$ the presence of Gaussian components components and possibly some 
conformant homologous sets, in addition to some generic components. 

Let $u_{1}=1 < u_{2} < \ldots < u_{\overline{i}_{3}} \in [1,k_{0}+1]$ 
such that $(\dfrac{v_{j}^{0}}{1+(m_{j}^{0})^{2}},\theta_{j}^{0})= (\dfrac{v_{l}^{0}}{1+
(m_{l}^{0})^{2}},\theta_{l}^{0})$ for all $u_{i} \leq j,l \leq u_{i+1}-1$, $1 \leq i \leq 
\overline{i}_{3}-1$, i.e., all the nonconformant homologous components 
are from index $1$ to $u_{\overline{i}_{3}}$. 
The remaining components are either Gaussian ones or conformant 
homologous sets or generic ones. It follows that
$|I_{u_{i}}|=u_{i+1}-u_{i}$ for all $1 \leq i \leq \overline{i}_{3}-1$ and all $I_{u_{i}}$ are 
nonconformant homologous sets. We divide the argument for our proof into two main steps
\paragraph{Step 1} $G_{0}$ is not $(3,2,\max \left\{2,\lev(G_{0}|\Ecal_{k_{0}})\right\}+1)$-singular relative to $\Ecal_{k_{0}}$. In fact, for any $\kappa=(\kappa_{1},\kappa_{2},\kappa_{3}) \in \mathbb{N}^{3}$ such that $\|\kappa\|_{\infty}=r$ and $\kappa_{3}=r$ where $r \geq 1$, consider the $\kappa$-minimal form for any sequence $G \in \Ecal_{k_{0}} \to G_{0}$ under $\widetilde{W}_{\kappa}$ distance. Since $\widetilde{W}_{\kappa}^{r}(G,G_{0}) \asymp D_{\kappa}
(G_{0},G)$ (cf. Lemma \ref{lemma:bound_overfit_Wasserstein_first}), we have
\begin{eqnarray}
\dfrac{p_{G}(x)-p_{G_{0}}(x)}{\widetilde{W}_{\kappa}^{r}(G,G_{0})} \asymp \dfrac{B_{1}(x)+B_{2}(x)}{D_{\kappa}(G_{0},G)}, \nonumber
\end{eqnarray}  
where $B_{1}(x)/
D_{\kappa}(G_{0},G)$ is the linear combination of nonconformant homologous 
components, i.e., the indices of 
components range from 1 to $\overline{i}_{3}$ 
while $B_{2}(x)/D_{\kappa}(G_{0},G)$ is the linear combination of conformant 
homologous components, Gaussian components, and generic components. 

Now, suppose that all the coefficients of $B_{1}(x)/D_{\kappa}(G_{0},G), B_{2}(x)/
D_{\kappa}(G_{0},G)$ go to $0$. Similar 
to the argument employed in the proof of
Theorem \ref{theorem:conformant_setting}, there is some index $
\underline{i} \in [1,k_{0}]$ such that $(|\Delta p_{\underline{i}}|+p_{\underline{i}}
(|\Delta \theta_{\underline{i}}|^{\kappa_{1}}
+|\Delta v_{\underline{i}}|^{\kappa_{2}}+|\Delta m_{\underline{i}}|^{\kappa_{3}}))/D_{\kappa}(G_{0},G) \not \to 0$. Now, there are two possible scenarios regarding $\underline{i}$
\paragraph{Case 1.1} $\underline{i} \in [1,u_{\overline{i}_{3}}-1]$. Under that case, we can check that
\begin{eqnarray}
\dfrac{B_{1}(x)}{D_{\kappa}(G_{0},G)}= \dfrac{1}{D_{\kappa}(G_{0},G)}\biggr(\sum_{l=1}^{\overline{i}_{3}-1} \biggr \{\mathop {\sum }\limits_{i=u_{l}}^{u_{l+1}-1}{\biggr[\sum \limits_{j=1}^{2r+1}{\beta_{jil}^{r}(x-\theta_{u_{l}}^{0})^{j-1}}\biggr]}f\biggr(\dfrac{x-\theta_{u_{l}}^{0}}{\sigma_{i}^{0}}\biggr)\Phi\biggr(\dfrac{m_{i}^{0}(x-\theta_{u_{l}}^{0})}{\sigma_{i}^{0}}\biggr)\biggr\}+ \nonumber \\
\biggr[\sum \limits_{j=1}^{2r}{\gamma_{jl}^{r}(x-\theta_{u_{l}}^{0}) ^{j-1}}\biggr]
\exp\biggr(-\dfrac{(m_{u_{l}}^{0})^{2}+1}{2v_{u_{l}}^{0}}(x-\theta_{u_{l}}^{0})^{2}\biggr)\biggr). \nonumber
\end{eqnarray}
This representation 
of $B_{1}(x)/D_{\kappa}(G_{0},G)$ is the general formulation of the equation 
\eqref{eqn:Taylorexpansionnonconformant} in Section \ref{Section:singularity_level_S3_setting} 
where $\overline{i}_{3}=2,u_{1}=1$, and $u_{2}=k_{0}+1$. Since $
\underline{i} \in [1,u_{\overline{i}_{3}}-1]$, there exists some index $l^{*} \in [1,
\overline{i}_{3}-1]$ such that $\underline{i} \in [u_{l^{*}},u_{l^{*}+1}-1]$. By means of the 
same argument as that of Section \ref{Section:singularity_level_S3_setting} for $\beta_{jil}
^{(r)}/D_{\kappa}(G_{0},G) \to 0$ and $\gamma_{jl}^{(r)}/
D_{\kappa}(G_{0},G) \to 0$, we can extract the 
following system of limits
\begin{eqnarray}
\biggr(\sum \limits_{i_{1}-i_{2}=l/2}{\dfrac{q_{i_{1},i_{2}}}{i_{1}!}\sum \limits_{i=u_{l^{*}}}^{u_{l^{*}+1}-1}{\dfrac{p_{i}(m_{i}^{0})^{i_{1}-2i_{2}-1}(\Delta m_{i})^{i_{1}}}{(\sigma_{i}
^{0})^{l}}}}\biggr) \biggr /
\sum \limits_{i=u_{l^{*}}}^{u_{l^{*}+1}-1}{p_{i}|\Delta m_{i}|^{\kappa_{3}}} \to 0, \nonumber
\end{eqnarray}
for any even $l$ such that $1 \leq l \leq 2r$ where $q_{i,j}$ are the integer coefficients 
that appear in the high order derivatives of $f(x|\theta,\sigma,m)$ with respect to $m$ 
defined as those after equation \eqref{eqn:singularitynonconformant_two}, and $1 \leq 
i_{1} \leq r$, $i_{2} \leq (i_{1}-1)/2$ as $i_{1}$ is odds or $i_{2} \leq i_{1}/2-1$ as 
$i_{1}$ is even. Let $r_{\max}$ to be maximum number that the above system of limits 
holds for any $l^{*} \in [1,\overline{i}_{3}-1]$. It would indicate that $r=r_{max}$ is 
the best possible value that $G_{0}$ is $(\kappa_{1},\kappa_{2},r)$-singular relative 
to $\Ecal_{k_{0}}$ for any $\kappa_{1}, \kappa_{2} \leq r$. Hence, from the definition 
of singularity level, we would have $\lev(G_{0}|\Ecal_{k_{0}})=r_{\max}$. If we choose 
$\kappa=(3,2,\max \left\{2,\lev(G_{0}|\Ecal_{k_{0}})\right\}+1)$, then as $
\kappa_{3}=\max \left\{2,\lev(G_{0}|\Ecal_{k_{0}})\right\}+1 \geq \lev(G_{0}|
\Ecal_{k_{0}})+1$ we achieve that $G_{0}$ is not $\kappa$-singular relative to $
\Ecal_{k_{0}}$.
\paragraph{Case 1.2} $\underline{i} \in [u_{\overline{i}_{3}},k_{0}]$. Using the same 
argument as that in the proof of 
Theorem \ref{theorem:conformant_symmetry_setting}, the third order 
Taylor expansion is sufficient so that not all the coefficients of $B_{2}(x)/D_{\kappa',new}
(G_{0},G)$ go to 0 where $\kappa'=(3,2,3)$ and
\begin{eqnarray}
D_{\kappa',new}(G_{0},G)=\sum \limits_{i=u_{\overline{i}_{3}}}^{k_{0}}{(|\Delta p_{i}|+p_{i}(|\Delta \theta_{i}|^{3}
+|\Delta v_{i}|^{2}+|\Delta m_{i}|^{3}))}. \nonumber
\end{eqnarray}
If we choose $\kappa=(3,2,\max \left\{2,\lev(G_{0}|\Ecal_{k_{0}})\right\}+1)$, then we have $\kappa'=(3,2,3) \prec \kappa$, which leads to $D_{\kappa',new}(G_{0},G)/D_{\kappa}(G_{0},G) \not \to 0$. As
all the coefficients of $B_{2}(x)/D_{\kappa}(G_{0},G)$ vanish, it leads to
all the coefficients of $B_{2}(x)/D_{\kappa',new}(G_{0},G)$ go to 0, which is a contradiction. Thus, Case 1.2 cannot happen.

In summary, if we choose $\kappa=(3,2,\max \left\{2,\lev(G_{0}|\Ecal_{k_{0}})\right\}+1)$, then not all the coefficients of $B_{1}(x)/D_{\kappa}(G_{0},G), B_{2}(x)/
D_{\kappa}(G_{0},G)$ go to $0$. Therefore, $G_{0}$ is not $(3,2,\max \left\{2,\lev(G_{0}|\Ecal_{k_{0}})\right\}+1)$-singular relative to $\Ecal_{k_{0}}$. 
\paragraph{Step 2} To demonstrate that $(3,2,\max \left\{2,\lev(G_{0}|\Ecal_{k_{0}})
\right\}+1)$ is the unique singularity index of $G_{0}$, we need to verify that $G_{0}$ 
is $(r,r,\max \left\{2,\lev(G_{0}|\Ecal_{k_{0}})\right\})$-singular, $(2,r,r)$-singular, 
and $(r,1,r)$-singular relative to $\Ecal_{k_{0}}$ for any $r \geq 1$. The later two 
results are straightfoward from the fact that there is at least one Gaussian component in 
$G_{0}$. In particular, by choosing sequence $G$ such that all the masses of $G$ and 
$G_{0}$ are identical while all the atoms of $G$ and $G_{0}$ are identical except for one 
Gaussian component of $G_{0}$. With that Gaussian component of $G_{0}$, we choose 
the corresponding component of $G$ similar to that in Section 
\ref{Section:singularity_level_S2_setting}. According to this choice of $G$, we can 
easily check that $G_{0}$ is $(2,r,r)$-singular and $(r,1,r)$-singular relative to $
\Ecal_{k_{0}}$. Finally, to demonstrate that $G_{0}$ is $(r,r,\max \left\{2,\lev(G_{0}|
\Ecal_{k_{0}})\right\})$, it comes directly from our analysis in Case 1.1 with the 
definition of $r_{\max}$, which guarantees the existence of $G$ to make all of these 
systems of limits vanish. Therefore, we achieve the conclusion of this step.

As a consequence, $(3,2,\max \left\{2,\lev(G_{0}|\Ecal_{k_{0}})\right\}+1)$ is the 
unique singularity index of $G_{0}$, i.e., $\singset(G_{0}|\Ecal_{k_{0}})=\left\{(3,2,
\max \left\{2,\lev(G_{0}|\Ecal_{k_{0}})\right\}+1)\right\}$. We achieve the 
conclusion of the theorem.

\paragraph{FULL PROOF OF THEOREM \ref{theorem:nonconformant_no_typeC_setting}}
Here, we only provide the proof for part (b) as the proof for part (a) is similar. This is a 
generalization of the argument in Section \ref{Section:singularity_level_S3_setting} and similar to the proof argument of that of Theorem \ref{theorem:singularity_index_S3}. Under 
this situation, apart from the nonconformant homologous sets without C(1) singularity, we also 
have for $G_0$ the presence of Gaussian components components and possibly some 
conformant homologous sets, in addition to some generic components. 

Let $u_{1}=1 < u_{2} < \ldots < u_{\overline{i}_{3}} \in [1,k_{0}+1]$ 
such that $(\dfrac{v_{j}^{0}}{1+(m_{j}^{0})^{2}},\theta_{j}^{0})= (\dfrac{v_{l}^{0}}{1+
(m_{l}^{0})^{2}},\theta_{l}^{0})$ for all $u_{i} \leq j,l \leq u_{i+1}-1$, $1 \leq i \leq 
\overline{i}_{3}-1$, i.e., all the nonconformant homologous components 
without type C(1) singularity are from index $1$ to $u_{\overline{i}_{3}}$. 
The remaining components are either Gaussian ones or conformant 
homologous sets or generic ones. It follows that
$|I_{u_{i}}|=u_{i+1}-u_{i}$ for all $1 \leq i \leq \overline{i}_{3}-1$ and all $I_{u_{i}}$ are 
nonconformant homologous sets without C(1) singularity.

From the result of Theorem \ref{theorem:singularity_index_S3}, it is sufficient to focus on obtaining the upper bound on the singularity level of $G_{0}$. Let $\kappa=(\overline{r},\overline{r},\overline{r})$ where $\overline{r}=\mathop {\max}{\biggr\{3,\overline{s}(G_{0})+1\biggr\}}$. Consider the $\overline{r}$-minimal and $\kappa$-th minimal form for any sequence $G \in \mathcal{E}_{k_{0}} \to G_{0}$ under $\widetilde{W}_{\kappa}$ 
distance.
Since $\widetilde{W}_{\kappa}^{\overline{r}}(G,G_{0}) \asymp D_{\kappa}
(G_{0},G)$ (cf. Lemma \ref{lemma:bound_overfit_Wasserstein_first}), we have
\begin{eqnarray}
\dfrac{p_{G}(x)-p_{G_{0}}(x)}{\widetilde{W}_{\kappa}^{\overline{r}}(G,G_{0})} \asymp \dfrac{B_{1}(x)+B_{2}(x)}{D_{\kappa}(G_{0},G)}, \nonumber
\end{eqnarray}  
where $B_{1}(x)/
D_{\kappa}(G_{0},G)$ is the linear combination of nonconformant homologous 
components, i.e., the indices of 
components range from 1 to $\overline{i}_{3}$ 
while $B_{2}(x)/D_{\kappa}(G_{0},G)$ is the linear combination of conformant 
homologous components, Gaussian components, and generic components. 

Now, suppose that all the coefficients of $B_{1}(x)/D_{\kappa}(G_{0},G), B_{2}(x)/
D_{\kappa}(G_{0},G)$ go to $0$. Similar 
to the argument employed in the proof of
Theorem \ref{theorem:conformant_setting}, there is some index $
\underline{i} \in [1,k_{0}]$ such that $(|\Delta p_{\underline{i}}|+p_{\underline{i}}
(|\Delta \theta_{\underline{i}}|^{\overline{r}}
+|\Delta v_{\underline{i}}|^{\overline{r}}+|\Delta m_{\underline{i}}|^{\overline{r}}))/D_{\kappa}(G_{0},G) \not \to 0$. Now, there are two possible scenarios regarding $\underline{i}$
\paragraph{Case 1.1} $\underline{i} \in [1,u_{\overline{i}_{3}}-1]$. Similar to Case 1.1 in the proof of Theorem \ref{theorem:singularity_index_S3}, we can check that
\begin{eqnarray}
\dfrac{B_{1}(x)}{D_{\kappa}(G_{0},G)}= \dfrac{1}{D_{\kappa}(G_{0},G)}\biggr(\sum_{l=1}^{\overline{i}_{3}-1} \biggr \{\mathop {\sum }\limits_{i=u_{l}}^{u_{l+1}-1}{\biggr[\sum \limits_{j=1}^{2\overline{r}+1}{\beta_{jil}^{(\overline{r})}(x-\theta_{u_{l}}^{0})^{j-1}}\biggr]}f\biggr(\dfrac{x-\theta_{u_{l}}^{0}}{\sigma_{i}^{0}}\biggr)\Phi\biggr(\dfrac{m_{i}^{0}(x-\theta_{u_{l}}^{0})}{\sigma_{i}^{0}}\biggr)\biggr\}+ \nonumber \\
\biggr[\sum \limits_{j=1}^{2\overline{r}}{\gamma_{jl}^{(\overline{r})}(x-\theta_{u_{l}}^{0}) ^{j-1}}\biggr]
\exp\biggr(-\dfrac{(m_{u_{l}}^{0})^{2}+1}{2v_{u_{l}}^{0}}(x-\theta_{u_{l}}^{0})^{2}\biggr)\biggr). \nonumber
\end{eqnarray}
Since $
\underline{i} \in [1,u_{\overline{i}_{3}}-1]$, there exists some index $l^{*} \in [1,
\overline{i}_{3}-1]$ such that $\underline{i} \in [u_{l^{*}},u_{l^{*}+1}-1]$. By means of the 
same argument as that of Section \ref{Section:singularity_level_S3_setting} for $\beta_{jil}
^{(\overline{r})}/D_{\kappa}(G_{0},G) \to 0$ and $\gamma_{jl}^{(\overline{r})}/
D_{\kappa}(G_{0},G) \to 0$, we can extract the 
following system of polynomial limits:
\begin{eqnarray}
\sum \limits_{i=u_{l^{*}}}^{u_{l^{*}+1}-1}{p_{i}^{0}(m_{i}^{0})^{l/2-1}(k_{i})^{l/2}}=0, \nonumber
\end{eqnarray} 
where at least one of $k_{i}$ differs 
from 0. Here, $l$ is any even number such that $2 \leq l \leq 2\overline{r}$. From the formulation of $\overline{s}(G_{0})$, since $\overline{r} \geq \overline{s}
(G_{0})+1 \geq \overline{s}(|I_{u_{l^{*}}}|,\left\{p_{i}^{0}\right\}_{i \in I_{u_{I^{*}}}},
\left\{m_{i}^{0}\right\}_{i \in I_{u_{I^{*}}}})+1$, we can guarantee that the above system 
of polynomial equations does not have any non-trivial solution, which is a contradiction. 
Therefore, Case 1.1 cannot happen.
\paragraph{Case 1.2} $\underline{i} \in [u_{\overline{i}_{3}},k_{0}]$. Using the same 
argument as that in the proof of 
Theorem \ref{theorem:conformant_symmetry_setting}, the third order 
Taylor expansion is sufficient so that not all the coefficients of $B_{2}(x)/D_{\kappa',new}
(G_{0},G)$ go to 0 where $\kappa'=(3,2,3)$ and
\begin{eqnarray}
D_{\kappa',new}(G_{0},G)=\sum \limits_{i=u_{\overline{i}_{3}}}^{k_{0}}{(|\Delta p_{i}|+p_{i}(|\Delta \theta_{i}|^{3}
+|\Delta v_{i}|^{2}+|\Delta m_{i}|^{3}))}. \nonumber
\end{eqnarray}
Since $\overline{r} \geq 3$, we have $\kappa'=(3,2,3) \prec \kappa=(\overline{r},\overline{r},\overline{r})$, which leads to $D_{\kappa',new}(G_{0},G)/D_{\kappa}(G_{0},G) \not \to 0$. As
all the coefficients of $B_{2}(x)/D_{\kappa}(G_{0},G)$ vanish, it follows that
all the coefficients of $B_{2}(x)/D_{\kappa',new}$\\$(G_{0},G)$ go to 0, which is a contradiction. Thus, Case 1.2 cannot happen.

In sum, for any sequence of $G$ tending to $G_0$ in $\widetilde{W}_{\kappa}$,
not all the coefficients of $B_{1}(x)/D_{\kappa}(G_{0},G)
$ and $B_{2}(x)/D_{\kappa}(G_{0},G)$ go to 0. By Definition \ref{def-rsingular_set}, we 
conclude that $G_{0} \in \mathcal{S}_{2}$ is not $\kappa$-singular relative to $
\mathcal{E}_{k_{0}}$. As a consequence, $\lev(G_0|\Ecal_{k_0}) \leq \overline{r}-1 = 
\mathop {\max }{\biggr\{2,\overline{s}(G_{0})\biggr\}}$, which leads to the singularity index $(3,2,\max\left\{2,\lev(G_{0}|\Ecal_{k_{0}})\right\}+1) \preceq (3,2,\max\{2,\overline{s}(G_{0})\}+1)$. This concludes part (b) of the theorem.

\paragraph{PROOF OF PROPOSITION \ref{proposition:tight_singularity_noC.1}}
Here, we utilize the same assumption on $G_{0}$ as that in the proof of Theorem 
\ref{theorem:nonconformant_no_typeC_setting}, i.e., all the nonconformant homologous sets 
without C(1) singularity are from index $1$ to $u_{\overline{i}_{3}}$. We also rearrange 
the components of $G_{0}$ such that the first nonconformant homologous set without C(1) 
singularity $I_{u_{1}}$ has exactly $k^{*}$ elements, i.e., $u_{2}-u_{1}=k^{*}$. As 
$u_{1}=1$, we have $u_{2}=k^{*}+1$.

(a) We will demonstrate that $G_{0}$ is 1-singular and (1,1,1)-singular relative to $\mathcal{E}_{k_{0}}$. Indeed, 
the sequence of $G$ is constructed as follows:  $p_{i}=p_{i}^{0},  
\theta_{i}=\theta_{i}^{0}, v_{i}=v_{i}^{0}$ for all $u_{2}=k^{*}+1 \leq i \leq k_{0}$, i.e., we 
match all the components of $G$ and $G_{0}$ except the first $k^{*}$ components of 
$G_{0}$. Now, by proceeding in the same way as described in Section 
\ref{Section:singularity_level_S3_setting} up to Eq. \eqref{eqn:singularitynonconformant_four}, 
to verify that $G_{0}$ is indeed 1-singular and (1,1,1)-singular, the choice of the 
first $k^{*}$ components of $G$ needs to satisfy
\begin{eqnarray}
\sum \limits_{i=u_{1}}^{u_{2}-1}{q_{i}\Delta t_{i}}/\sum \limits_{i=u_{1}}^{u_{2}-1}{q_{i}|\Delta t_{i}|} \to 0, \nonumber
\end{eqnarray} 
where $q_{i}=p_{i}/\sigma_{i}^{0}$ and $\Delta t_{i}=\Delta m_{i}/\sigma_{i}^{0}$ as $u_{1} 
\leq i \leq u_{2}-1$. A simple choice is to take the first $k^{*}$ components of $G$ 
by $\sum \limits_{i=u_{1}}^{u_{2}-1}{q_{i}\Delta t_{i}}=q_{1}\Delta t_{1}+q_{2}\Delta t_{2}=0$, 
which is always possible. We conclude that 
$G_{0}$ is 1-singular and (1,1,1)-singular relative to $\mathcal{E}_{k_{0}}$. 
Since $\overline{s}(G_{0})=1$ as $k^{*}=2$, by combining with the upper 
bound of Theorem \ref{theorem:nonconformant_no_typeC_setting}, we have $\lev(G_0|\Ecal_{k_0})=1$. According to the result of Theorem \ref{theorem:singularity_index_S3}, we achieve $\singset(G_{0}|\Ecal_{k_{0}})=\left\{(1,1,2)\right\}$.

(b) There are two cases to consider in this part
\paragraph{Case 1:} All the homologous sets $I$ of $G_{0}$ such that $|I|=k^{*}$ satisfy $
\sum \limits_{i \in I}{p_{i}^{0}\prod \limits_{j \in I\setminus \{i\}}
{m_{j}^{0}}}>0$. To demonstrate that $G_{0}$ is 1-singular and (1,1,1)-singular
relative to $\mathcal{E}_{k_{0}}$, we utilize the same construction 
of $G$ as that of part (a), i.e., $p_{i}=p_{i}^{0}, \theta_{i}=\theta_{i}^{0}, v_{i}=v_{i}^{0}$ 
for all $u_{2}=k^{*}+1 \leq i \leq k_{0}$ and $\sum \limits_{i=u_{1}}^{u_{2}-1}{q_{i}\Delta t_{i}}=0$. 
Next, we will show that $G_{0}$ is not 2-singular and (2,2,2)-singular relative to $\mathcal{E}_{k_{0}}$. Using 
the same argument as that of the proof of
Theorem \ref{theorem:nonconformant_no_typeC_setting}, we obtain the 
following system of limiting rational polynomial functions:
\begin{eqnarray}
\sum \limits_{i=u_{l^{*}}}^{u_{l^{*}+1}-1}{q_{i}\Delta t_{i}}/\sum \limits_{i=u_{l^{*}}}^{u_{l^{*}+1}-1}{q_{i}|\Delta t_{i}|^{2}} \to 0, \nonumber \\
\sum \limits_{i=u_{l^{*}}}^{u_{l^{*}+1}-1}{q_{i}t_{i}^{0}(\Delta t_{i})^{2}}/\sum \limits_{i=u_{l^{*}}}^{u_{l^{*}+1}-1}{q_{i}|\Delta t_{i}|^{2}} \to 0, \nonumber
\end{eqnarray} 
where $l^{*}$ is some index in $[1,\overline{i}_{3}-1]$ and $q_{i}=p_{i}/\sigma_{i}^{0}$, 
$\Delta t_{i}=\Delta m_{i}/\sigma_{i}^{0}$, $t_{i}^{0}
=m_{i}^{0}/\sigma_{i}^{0}$ for all $u_{l^{*}} \leq i \leq u_{l^{*}+1}-1$. 
By employing the greedy extraction technique being described in Section \ref{Section:nonconformant_no_typeC}, we obtain the following system of polynomial 
equations:
\begin{eqnarray}
\sum \limits_{i=u_{l^{*}}}^{u_{l^{*}+1}-1}{p_{i}^{0}c_{i}}=0, \ \sum \limits_{i=u_{l^{*}}}^{u_{l^{*}+1}-1}{p_{i}^{0}m_{i}^{0}c_{i}^{2}}=0, \nonumber
\end{eqnarray}
where at least one of $c_{i}$ differs from 0. Now, we have two possible scenarios:
\paragraph{Case 1.1:} $|I_{u_{l^{*}}}|=u_{l^{*}+1}-u_{l^{*}}=2$. Then, by solving the 
above system of equations, we obtain $\sum \limits_{i \in \mathcal{I}_{u_{l^{*}}}}{p_{i}^0
\prod \limits_{j \in I_{u_{l^{*}}}\backslash \left\{i\right\}}{m_{j}^{0}}}=0$, which means 
$I_{u_{l^{*}}}$ is nonconformant homologous set with C(1) singularity of $G_{0}$ --- a contradiction to the fact that $G_{0} \in \Scal_{31}$.
\paragraph{Case 1.2:} $|I_{u_{l^{*}}}|=u_{l^{*}+1}-u_{l^{*}}=k^{*}=3$. Then, by solving 
the above system of equations, we obtain $\sum \limits_{i \in \mathcal{I}_{u_{l^{*}}}}
{p_{i}^{0}\prod \limits_{j \in I_{u_{l^{*}}}\backslash \left\{i\right\}}{m_{j}^{0}}}<0$ --- a 
contradiction to the assumption of Case 1. 

Thus, $G_{0}$ is not 2-singular and (2,2,2)-singular relative to $\mathcal{E}_{k_{0}}$.
As a consequence, $\lev(G_0|\Ecal_{k_0})=1$ and $\singset(G_{0}|\Ecal_{k_{0}})=\left\{(1,1,2)\right\}$ according to Theorem \ref{theorem:singularity_index_S3}.

\paragraph{Case 2:} There exists at least one nonconformant homologous set $I$ of $G_{0}
$ such that $|I|=k^{*}$ satisfies $\sum \limits_{i \in I}{p_{i}^{0}\prod \limits_{j \in I
\setminus \{i\}}
{m_{j}^{0}}}<0$. Without loss of generality, we assume the homologous set $I_{u_{1}}
$ of $G_{0}$ to have the property $\sum \limits_{i \in I_{u_{1}}}{p_{i}^{0}\prod \limits_{j 
\in I_{u_{1}}\setminus \{i\}}
{m_{j}^{0}}}<0$. We will show that $G_{0}$ is 2-singular and (2,2,2)-singular relative to 
$\mathcal{E}_{k_{0}}$. In fact, we construct the 
sequence of $G$ by letting 
$p_{i}=p_{i}^{0}, \theta_{i}=\theta_{i}^{0}, v_{i}=v_{i}^{0}$ for all 
$u_{2}=k^{*}+1 \leq i \leq k_{0}$. In order for $G_{0}$ to be 2-singular and (2,2,2)-singular, 
it is sufficient that
\begin{eqnarray}
\sum \limits_{i=u_{1}}^{u_{2}-1}{q_{i}\Delta t_{i}}/\sum \limits_{i=u_{1}}^{u_{2}-1}{q_{i}|\Delta t_{i}|^{2}} \to 0, \nonumber \\
\sum \limits_{i=u_{1}}^{u_{2}-1}{q_{i}t_{i}^{0}(\Delta t_{i})^{2}}/\sum \limits_{i=u_{1}}^{u_{2}-1}{q_{i}|\Delta t_{i}|^{2}} \to 0. \nonumber
\end{eqnarray} 
The simple solution to the above system of limits is $\sum \limits_{i=u_{1}}^{u_{2}-1}{q_{i}
\Delta t_{i}}=0$ and $\sum \limits_{i=u_{1}}^{u_{2}-1}{q_{i}t_{i}^{0}(\Delta t_{i})^{2}}=0$.
One solution to these two equations is $p_{i}=p_{i}^{0}$ and $\Delta m_{i}=(\sigma_{i}^0)^{2}d_{i}/n$ for all $u_{1} \leq i \leq u_{2}-1$ where $d_{1},d_{2},d_{3}$ satisfy
\begin{eqnarray}
\sum \limits_{i=u_{1}}^{u_{2}-1}{p_{i}^{0}d_{i}}=0, \ \sum \limits_{i=u_{1}}^{u_{2}-1}{p_{i}^{0}m_{i}^{0}d_{i}^{2}}=0, \nonumber
\end{eqnarray}
which is guaranteed to have non-trivial solution as $\sum \limits_{i \in I_{u_{1}}}{p_{i}^{0}
\prod \limits_{j \in I_{u_{1}} \backslash \left\{i\right\}}{m_{j}^{0}}}<0$. Therefore, 
$G_{0}$ is 2-singular and (2,2,2)-singular relative to $\mathcal{E}_{k_{0}}$. Since $\overline{s}(G_{0})=2$ as $k^{*}=3$, combining with the 
upper bound of Theorem \ref{theorem:nonconformant_no_typeC_setting}, we obtain $\
\lev(G_0|\Ecal_{k_0})=2$  under Case 2. According to Theorem \ref{theorem:singularity_index_S3}, it also indicates that $\singset(G_{0}|\Ecal_{k_{0}})=\left\{(1,1,3)\right\}$, which concludes our proof.
\paragraph{FULL PROOF OF THEOREM \ref{theorem:nonconformant_typeC_typeIV_setting}}
Here, we shall provide the complete proof of Theorem 
\ref{theorem:nonconformant_typeC_typeIV_setting}, which is also the generalization of the 
argument in Section \ref{Section:nonconformant_typeC_typeIV}. Indeed, without loss of generality, 
we assume that $(p_{1}^{0}/\sigma_{1}^{0},m_{1}^{0}/\sigma_{1}^{0})=(p_{2}^{0}/
\sigma_{2}^{0},-m_{2}^{0}/\sigma_{2}^{0})$. 
Next, we proceed to choosing a sequence of $G \in \Ecal_{k_0}$ as follows:  $p_{i}=p_{i}^{0},  
\theta_{i}=\theta_{i}^{0}, v_{i}=v_{i}^{0}$ for all $1 \leq i \leq k_{0}$, and 
$m_{1}=m_{1}^{0}+1/n$, $m_{2}=m_{2}^{0}-\sigma_{2}^{0}/n\sigma_{1}^{0}
$, $m_{i}=m_{i}^{0}$ for all $3 \leq i \leq k_{0}$. The choice of $m_{1},m_{2}$ is
taken to guarantee that $\Delta m_{1}/\sigma_{1}^{0}+\Delta m_{2}/\sigma_{2}^{0}=0$ 
as we have discussed in Section 
\ref{Section:nonconformant_typeC_typeIV}. Then, we can check that $\mathop {\sum }
\limits_{j=1}^{2}{p_{j}(m_{j}^{0})^{u}(\Delta m_{j})^{v}/(\sigma_{j}^{0})^{u+v+1}}=0$ 
for all odd numbers $u \leq v$ when $v$ is even number or for all even numbers 
$0 \leq u \leq v$ when $v$ is odd number. 
From here, the completion of the proof follows in the same way
as that of the special case previously described.

\end{document}